\documentclass[11pt, a4paper]{article}

\usepackage[utf8]{inputenc}
\usepackage[american]{babel}
\usepackage{amsmath,amssymb,amsthm,color}
\usepackage{calrsfs}
\usepackage[shortlabels]{enumitem}
\usepackage{comment}
\usepackage{braket}
\usepackage{graphicx} 
\usepackage[hidelinks]{hyperref}
\usepackage{authblk}
 \makeatletter

\def\blfootnote{\gdef\@thefnmark{}\@footnotetext}

\makeatother

\title{Coarse rigidity in von Neumann algebras via derivations}
\author[1]{Manish Kumar}
\author[2]{Melchior Wirth}
\affil[1]{KU Leuven, Department of Mathematics, Celestijnenlaan 200B, 3001 Leuven, Belgium, manish.kumar@kuleuven.be}
\affil[2]{Mathematical Institute, Leipzig University,  Augustusplatz 10, 04109 Leipzig, Germany, melchior.wirth@uni-leipzig.de}
\date{}

\newcommand{\cL}{\mathcal L}

\renewcommand{\epsilon}{\varepsilon}
\renewcommand{\phi}{\varphi}

\newcommand{\dom}{\operatorname{dom}}
\renewcommand{\Re}{\operatorname{Re}}
\renewcommand{\Im}{\operatorname{Im}}

\newcommand{\tr}{\operatorname{tr}}

\newcommand{\abs}[1]{\lvert#1\rvert}
\newcommand{\norm}[1]{\lVert#1\rVert}
\numberwithin{equation}{section}

\newtheorem{theorem}{Theorem}
\numberwithin{theorem}{section}
\newtheorem{proposition}[theorem]{Proposition}
\newtheorem{lemma}[theorem]{Lemma}
\newtheorem{corollary}[theorem]{Corollary}

\theoremstyle{definition}
\newtheorem{definition}[theorem]{Definition}
\newtheorem{notation}[theorem]{Notation}
\newtheorem*{fullprimecriterion}{Fullness and Primeness criterion}
\newtheorem*{relativeamenability}{Relative amenability criterion}
\theoremstyle{remark}
\newtheorem{example}[theorem]{Example}
\newtheorem{remark}[theorem]{Remark}

\newcommand{\R}{\mathbb{R}}
\newcommand{\C}{\mathbb{C}}

\newcommand{\A}{\mathsf A}
\newcommand{\B}{\mathsf B}
\renewcommand{\H}{\mathsf{H}}
\newcommand{\K}{\mathsf{K}}
\newcommand{\M}{\mathsf{M}}
\newcommand{\N}{\mathsf{N}}
\newcommand{\G}{\mathcal{G}}
\newcommand{\Q}{\mathsf{Q}}
\newcommand{\I}{\mathcal{I}}
\newcommand{\op}{\mathrm{op}}
\newcommand{\Ad}{\mathrm{Ad}}

\DeclareMathOperator{\Op}{op}
\DeclareMathOperator{\Lim}{Lim}
\DeclareMathOperator{\alg}{alg}

\DeclareMathOperator{\id}{id} 
\DeclareMathOperator{\Null}{Null}

\begin{document}

\maketitle

\begin{abstract}
Motivated by the interaction between amenability and first cohomology into tempered representations of groups, we propose a notion of relative  coarse rigidity for $\sigma$-finite von Neumann algebras  through modular derivations and their associated GNS-symmetric quantum Markov semigroups. This  extends Peterson’s framework of $L^2$-rigidity to the nontracial setting. Our main result shows that the absence of relatively amenable summands forces the associated deformations to converge on relative commutants in ultrapowers. For diffuse von Neumann algebras with separable predual, this yields a dichotomy between the presence of an amenable summand and rigidity in terms   of central sequences. Combining this principle with mixing and compactness, we provide  a novel unified approach for proving indecomposability properties such as fullness, primeness and stable solidity. Applications yield new examples of relatively $\omega$-solid and relatively stably solid type III von Neumann algebras arising from amalgamated free products, operator-valued free Araki–Woods constructions, cocycle crossed products, and twisted nonsingular discrete measured groupoids.
\end{abstract}


\section{Introduction}

Rigidity and amenability are two opposing themes in the study of groups and von Neumann algebras. Kazhdan's property (T)  is a basic example of rigidity:  approximate invariance must eventually become uniform.  For a discrete group $G$, property (T) requires every net of normalized positive-definite functions
converging pointwise to $1$ to converge uniformly on $G$.
Analogy to group representation theory has been a guiding principle in the structure theory of von Neumann algebras. Here, group representations are replaced by Hilbert correspondences, while
positive-definite functions are replaced by unital completely positive maps. 
Using this analogy, Connes and
Jones introduced property (T) for von Neumann algebras \cite{CJ85}. 
For  tracial von Neumann algebras, property (T) can be described with trace-preserving unital completely
positive maps. If such maps converge pointwise to the identity in the $L^2$-norm, then property (T) requires
uniform convergence on the unit ball.

Relative property  (T)  is analogously defined for a subgroup $H\leq G$ (resp. for an inclusion $\N\subseteq\M$ of tracial von Neumann algebras) by requiring uniform convergence only on $H$ (resp.  on the unit ball of $\N$) \cite{dHV89, Pop06, PP05, Pop86}.
A cohomological formulation of relative property  (T)  for $H\leq G$ can be obtained by considering Schoenberg semigroups where one requires that  a semigroup $(\psi_t)_{t>0}$ of normalized  positive definite functions  converges uniformly on $H$ as $t\to 0$. The Delorme--Guichardet theorem gives an equivalent criterion in terms of vanishing of first cohomology.



In  von Neumann algebras, the analogues of $1$-cocycles are densely defined  closable derivations into Hilbert correspondences. Peterson established a cohomological characterization of property  (T)  for separable II$_1$ factors by requiring that any symmetric derivation with `nice' domain must be inner \cite{Pet09b}. This is the von Neumann algebraic analogue of vanishing of first cohomology. At the heart of this equivalence are results of Sauvageot \cite{Sau90} and Cipriani--Sauvageot \cite{CS03}, which establish a bijective relation between densely defined real closable derivations and semigroups of tracially symmetric unital completely positive maps.

A more selective cohomological question is obtained by restricting the coefficient representations to the left regular representations. For a finitely generated nonamenable group, the vanishing of the first cohomology into the left regular representation is equivalent to the vanishing of the first $L^2$-Betti number \cite{BV97}. The seminal work of Peterson's $L^2$-rigidity \cite{Pet09a} is the tracial von Neumann algebraic realization of this 
approach for derivations into amplifications of the coarse correspondence.
This gives
structural rigidity under assumptions substantially weaker than property (T).

Another selective cohomological question is obtained by restricting the coefficient representations to the tempered representations, those which are  weakly
contained in the left regular representations. Amenability is intrinsic to this choice of coefficients. Vanishing of the first cohomology into tempered representations is equivalent to the so called cohomological Property (T), which in turn implies   nonamenability for finitely generated exact groups \cite{MOSS15}. 
The aim of this paper is to consider the von Neumann algebraic analogue of this point of view, where we consider coefficients and deformations associated to derivations into Hilbert correspondences weakly contained in the coarse correspondence.

We extend this concept to general $\sigma$-finite von Neumann algebras, including type III
algebras. Two changes are necessary in the nontracial setting. First, the derivation structure
must respect modular theory. Results of Sauvageot and the second author \cite{Sau90, Wir22, Wir24} associate GNS-symmetric quantum Markov semigroups with densely defined closable
modular derivations taking values in Tomita correspondences. These correspondences carry a conjugation and a modular action compatible with the standard form of the algebra (see Section \ref{sec:modular derivation}). Second,
uniform convergence on the entire unit ball in the $L^2$-norm of a fixed faithful normal state
is too restrictive as a general replacement for the tracial condition, because the modular automorphism group creates a serious obstruction. We instead formulate rigidity through strong convergence of
the induced deformations on $L^2$-spaces of Ocneanu ultrapowers. Proposition \ref{prop:equivalent conditions for L2-rigidity} identifies this
condition with uniform convergence on suitable subsets of the unit balls. In the tracial case these are
the usual unit balls. The ultrapower formulation therefore retains the rigidity in finite von Neumann algebras
while incorporating the modular restrictions needed in type III.

 To state the definition, let $\M$ be a $\sigma$-finite von Neumann algebra and  let $\delta\colon\dom\delta\subseteq L^2(\M)\to\H$ be a  densely defined closable modular derivation.
For $\alpha>0$, write $\theta_\alpha^{(2)}$ for the resolvent operator $\alpha(\alpha+\delta^\ast\bar\delta)^{-1}$, and  denote  its ultrapower implementation  on $L^2(\M^\omega)$ by $(\theta_\alpha^{(2)})^\omega$, for a free ultrafilter $\omega\in\beta\mathbb N\setminus\mathbb N$. We use  $\Null(\bar\delta)$ for the set $\{x\in \M\mid x\phi^{1/2}\in \ker\bar\delta\}$. The precise notation and conventions are given in Section \ref{sec:modular derivation} and \ref{sec:definition of coarse regifity}.
\begin{definition}[{Definition \ref{def: L2 rigidity}}]\label{def:coarse_rigidity_Introduction}
    Let $\Q,\B\subseteq\M$ be inclusions of von Neumann algebras with expectation. We say that $\Q$ is {\em coarsely rigid in $\M$ relative to $\B$} if for every $\omega\in\beta\mathbb N\setminus\mathbb N$ and for every  modular derivation $\delta\colon \dom(\delta)\subseteq L^2(\M)\to \H$ with $\B\subseteq\Null(\bar\delta)$  and ${_\M}\H_\M\prec {_\M}L^2(\M)\otimes_\B\K\otimes_\B L^2(\M)_\M$ for some $\B$-$\B$-correspondence $\K$, the operators $(\theta_\alpha^{(2)})^\omega$ converge strongly to the identity on $L^2(\Q^\omega
)$ as $\alpha
\to\infty$. We say that $\M$ is {\em coarsely rigid} if $\M$ is coarsely rigid in $\M$ relative to $\mathbb C$.
\end{definition}
Thus relative coarse rigidity means that every deformation whose gradient correspondence is weakly coarse over $\B$ becomes asymptotically trivial on $\Q^\omega$. Restricting this definition in the tracial setting to $\B=\mathbb C$ and  the Hilbert correspondences to amplifications of the
coarse correspondences recovers Peterson’s $L^2$-rigidity (Proposition \ref{prop:equivalent conditions for L2-rigidity}).

Just like in the group case, we connect this rigidity to amenability. Their interaction is especially visible in three questions about a nonamenable factor: whether it admits nontrivial central sequences, whether it decomposes as
a tensor product of two non-type I factors, and whether a diffuse subalgebra can have a
nonamenable relative commutant. These questions lead, respectively, to {\em fullness, primeness},
and {\em solidity}.
The historical significance of these questions goes back to the foundations of the subject, where Murray and von Neumann introduced property $\Gamma$ to distinguish the hyperfinite II$_1$
factor from the free group factors. Connes’ notion of fullness  in terms of the absence
of nontrivial asymptotic centralizers \cite{Con76} and Ozawa’s solidity theorem \cite{Oza04} revealed a stronger indecomposability phenomenon in the free group factors.  Popa’s deformation/rigidity theory provided a powerful way of studying these questions
through the interaction of completely positive deformations, spectral gap, and intertwining-by-bimodules \cite{Pop06}.
Subsequent work extended important parts of this theory
beyond the tracial setting \cite{HV13, HR15, HI17, Iso25}.

 The purpose of this work is to develop a derivation-based theory that connects
these structural properties to relative amenability for general $\sigma$-finite von Neumann algebras,
including type III algebras.
The  main result  is  Theorem \ref{main_theorem} which connects relative amenability with coarse rigidity. In simple terms, if an expected subalgebra $\Q$ of $\M$  has no  amenable summand relative to $\B$, then every allowed coarse deformation as in Definition \ref{def:coarse_rigidity_Introduction} must `vanish'
asymptotically on the algebra commuting with $\Q$. A particular case of this theorem for amenable  $\B$ reads as follows:

\begin{theorem}[Corollary \ref{cor:dichotomy result in amenable case}]\label{main_theorem_amenable_for introduction}
   If $\B,\Q\subseteq\M$ are inclusions of $\sigma$-finite von Neumann algebras with expectations such that $\B$ is amenable and $\Q$ does not have any non-zero amenable summand, then $\Q'\cap\M$ is coarsely rigid relative to $\B$.
\end{theorem}

Theorem \ref{main_theorem} is the main rigidity principle used throughout the paper.  Nonamenability on one side therefore forces rigidity on every algebra
that commutes with it. Combining this
relative-commutant rigidity with mixing for nonamenable factors with Property $\Gamma$ yields the following nontracial counterpart of Peterson’s  spectral-gap argument (see Sec. \ref{sec:coarse rigidity and spectral gap} for the notion of mixing correspondence). 

\begin{theorem}[Corollary \ref{cor:fullness and rigidity}]
     If $\M$ is  a nonamenable  factor with separable predual which is not full, in particular,  if $\M$ is a type III$_0$ nonamenable factor, then $\M$ is  rigid with respect to modular derivations taking values in mixing weakly coarse correspondences.
\end{theorem}

This statement turns the existence of nontrivial central sequences into a restriction on
all mixing weakly coarse modular deformations. More generally, we obtain an exact characterization of the absence of amenable summands in terms of deformation rigidity on central
sequences.

\begin{theorem}[Theorem \ref{thm:dichotomy result for subalgebras}]\label{thm:dichotomy result_Introduction}
     Let $\M$ be a diffuse von Neumann algebra with separable predual, and let $\omega\in\beta\mathbb N\setminus\mathbb N$. Exactly one of the following holds true:
     \begin{enumerate}
         \item $\M$ has a non-zero amenable summand.
         \item For every modular derivation  $\delta:\dom(\delta)\subseteq L^2(\M)\to \H$ with $\H\prec L^2(\M)\otimes L^2(\M)$, $(\theta_\alpha^{(2)})^\omega$ converges strongly to the identity on $L^2(\M'\cap\M^\omega)$ as $\alpha\to\infty$.
     \end{enumerate}
 \end{theorem}

The dichotomy makes the role of amenability precise. Nonamenability need not imply
rigidity of the whole algebra, but the absence of amenable summands is exactly the obstruction to nontrivial weakly coarse deformations on its central-sequence algebra. This
is a von Neumann algebraic  version of the tension between amenability of a group and vanishing of the first cohomology into tempered representations that motivates
the theory.

A principal ingredient of our main theorem is a nontracial characterization of relative amenability.   Recall from \cite{AD95, OP10} that if $\A,\B\subseteq\M$ are inclusions of von Neumann algebras with expectation,  $\A$ is {\em amenable relative to $\B$ inside $\M$} if there is a norm-one projection $E$ from the Jones basic construction $\langle\M,\B\rangle$ to $\A$  such that $E_{|_\M}$ is normal and faithful. In the finite setting, relative amenability can be detected by  central states on the basic construction \cite{OP10}. In the nontracial case, centrality must be replaced by a KMS relation, as the following result shows.

\begin{relativeamenability}[Theorem \ref{thm:conditions for relative amenability}]
    Let $\A,\B\subseteq\M$ be  inclusions of $\sigma$-finite von Neumann algebras such that   $\A$ is globally invariant under $\sigma^\phi$ for some normal faithful state $\phi$ on $\M$. Then $\A$ is amenable relative to $\B$ inside $\M$ if and only if there exists a  state $\psi:\langle\M,\B\rangle\to \mathbb  C$,  such that $\psi_{|_{\M}}\leq C\varphi$ for some $C>0$, $\psi_{|_{\mathcal{Z}(\A'\cap\M)}}$ is faithful and satisfies 
     $\psi(Ta)=\psi(\sigma_i^\varphi(a)T)$
            for $T\in \langle\M,\B\rangle$ and $\sigma^\phi$-analytic $a\in \A$.
\end{relativeamenability}

When the modular action on $\A$ is trivial, the displayed KMS relation reduces to centrality. In
general, it retains the modular information needed to recover relative amenability. This theorem is of independent interest and is the mechanism through which a failure of deformation
rigidity produces a relatively amenable summand.

The resulting structural criteria for fullness and primeness are particularly significant in the type III setting. Compactness supplies the second ingredient: relative commutant-rigidity forces deformation to converge on a sufficiently large algebra whereas a compact deformation cannot approximate the identity there unless the algebra has the requisite `smallness'. This tension produces the required indecomposability results. We record two consequences explicitly.

\begin{fullprimecriterion}
Let $\M$ be a von Neumann algebra with separable predual admitting a densely defined closable modular derivation $\delta:\dom(\delta)\subseteq L^2(\M)\to \H$.
\begin{enumerate}
    \item \textup{(Corollary \ref{cor:compactness implies fullness} and Proposition \ref{prop: compactness implies solidity})} If ${}_\M\H_\M{}\prec {}_\M L^2(M)\otimes L^2(M)_\M{}$ and $\delta^*\overline\delta$ has compact resolvent, then  $\M$ is $\omega$-solid for $\omega\in\beta\mathbb N\setminus\mathbb N$, and hence solid. In particular, if $M$ is a nonamenable factor, then $M$ is full. 
    \item \textup{(Theorem \ref{thm:primeness for type I base and compact resolvent})} Let $\B\subseteq\M$ be an expected subalgebra such that $\B$ is type I such that  ${}_\M\H_\M{}\prec {}_\M L^2(M)\otimes_\B  L^2(M)_\M{}$ and  $\delta^*\overline\delta$ has compact resolvent relative to $B$. If $M$ is a nonamenable factor, then $M$ is prime.
\end{enumerate}

\end{fullprimecriterion}

 We next develop relative $\omega$-solidity and relative stable solidity (sec. \ref{sec:solidity}) using an ultrapower notion of relative diffuseness. For expected subalgebras $\A,\B$ of $\M$ and $\omega\in\beta\mathbb N\setminus\mathbb N$, we say that $\A$ is {\em diffuse relative to $\B$ along $\omega$} if there is a unitary $u\in\A^\omega$ such that $E_{\B^\omega}(x^\ast uy)=0$ for all $x,y\in \M$. For a finite $\A$ with separable predual,  this is equivalent to failure of Popa's intertwining-by-bimodules \cite{Pop06}; see Proposition \ref{prop: comparion of two inytertwining notion}. The ultraproduct formulation is designed to work with relatively
compact deformations in type III algebras. The $\omega$-solid conclusions strengthen control of ordinary relative commutants to control of ultrapower relative commutants, following the perspective of~\cite{HR15}. The stable conclusions address relative commutants after tensoring with auxiliary algebras, in the sense developed in~\cite{DV25} and extended here to the nontracial relative setting.

The general results apply to several classes of algebras arising from free constructions and cocycles. They provide a common modular method for amalgamated free products, operator-valued free Araki-Woods algebras, cocycle crossed products, and discrete measured groupoids. The principal applications can be assembled as follows.

\begin{theorem}
In the following examples, for any $\omega\in \beta\mathbb N\setminus\mathbb N$,     $\M$ is $\omega$-solid and stably solid relative to $\B$ along $\omega$.
\begin{itemize}
\item \textup{(Theorem \ref{thm:solidity for amalgamated free products})} $\M=\M_1\bar\ast_\B\M_2$,  where $\M_1,\M_2$ are $\sigma$-finite von Neumann algebras with a common von Neumann subalgebra $\B$ with expectations such that $\mathrm{id}_{L^2(\M_j)}$ is compact relative to $\B$ for $j=1,2$. 
    \item \textup{(Theorem \ref{prop:omega solidity in crossed products})} $\M=\B\rtimes_{\alpha,\nu}\Gamma$, where $(\alpha,\nu)$ is a cocycle action of a discrete countable group $\Gamma$ on a von Neumann algebra $\B$  such that $\Gamma$ a proper 1-cocycle $c$ into a representation weakly contained in the left regular representation.
    \item \textup{(Corollary \ref{cor:solidity of free Araki-Woods with finitely generated bimodule})} $\M$ = $\Gamma(_\B\H_\B)$, the operator-valued free Araki–Woods von Neumann algebra , where $(\H,\mathcal J,(\mathcal U_t)_{t\in\mathbb R})$ is a Tomita correspondence over a von Neumann algebra $\B$ with separable predual such that $\H$ is a countable direct sum of Tomita sub-correspondences which are  finitely-generated as right $\B$-modules.
    \item \textup{(Theorem \ref{thm:solidity_measured_groupoid})}
    $\M=L_\sigma(\G,\mu)$ 
    where $(\G,\mu)$ is a discrete measured groupoid
    with unit space $X$, $\sigma$ is a scalar $2$-cocycle, $\B=L^\infty(X,\mu)$, and $\mathcal G$ admits a uniformly proper $1$-cocycle into a representation weakly contained in its left regular representation.
\end{itemize}
\end{theorem}
The novelty of these applications is concentrated in their nontracial and relative scope.  Stronger results are known in several finite special cases~\cite{PV14,KV19,DV25}, and the scalar free Araki-Woods $\omega$-solidity theorem is already established in \cite{HR15}. Here the operator-valued construction permits a nontracial base; the crossed-product theorem permits a possibly nonamenable base and a cocycle action without an invariant faithful normal state; and the groupoid construction treats scalar twists and nonsingular measures. These features make the modular derivation method applicable to settings in which the standard tracial formulation cannot be used directly, and thus  provides many new examples of type III relative $\omega$-solid and relative solid factors.

The general primeness criterion also gives prime nonamenable cocycle crossed-product factors over type~I bases under the proper-cocycle hypothesis; see Theorem~8.23. For measured groupoids we obtain, in particular, the following fullness and primeness result.
\begin{theorem}[Theorem \ref{thm:fullness and primeness_groupoid}]
        Let $(\mathcal G,\mu,\sigma)$ be a discrete measured groupoid with a scalar $2$-cocycle $\sigma$ such that $\mathcal G$ admits a uniformly proper $1$-cocycle into a mixing representation weakly contained in its left regular representation. If $L_\sigma(\mathcal G,\mu)$ is a nonamenable factor, then it is full and prime.
\end{theorem}

The article is organized as follows. Section \ref{chp:preliminary} recalls standard forms, modular theory, GNS-symmetric quantum Markov semigroups, Tomita correspondences, modular derivations, and Ocneanu ultraproducts. Section~\ref{chp:main_technical_lemma} establishes the analytic estimates and technical tools used in our deformation arguments. Section \ref{chp:relative_amenability} develops relative amenability and proves its KMS characterization. Section~\ref{sec:definition of coarse regifity} introduces coarse rigidity and its equivalent formulations.
Section~\ref{sec:relative compactness} studies relatively compact operators and the associated operator algebras.
Section~\ref{sec:new notion of deffuness} introduces an ultrapower notion of relative diffuseness, compares it with Popa's intertwining-by-bimodules, and examines its interaction with coarse rigidity and relative compactness. Section~\ref{chp:main theorem} proves the main theorem connecting the absence of relatively amenable summands with deformation rigidity on relative commutants in ultrapowers.
Section~\ref{sec: a dichotomy result} establishes the dichotomy between the existence of an amenable summand and rigidity on the central-sequence algebra.
Section~\ref{sec:coarse rigidity and spectral gap} combines coarse rigidity with mixing and compactness to obtain spectral gap and fullness results.
Section~\ref{sec:coarse rigidity and primeness} derives primeness criteria for nonamenable factors under compactness or mixing assumptions.
Section~\ref{sec:solidity} proves relative $\omega$-solidity and relative stable solidity results.
Finally, Section~\ref{chp:Examples} applies the theory to amalgamated free products, operator-valued free Araki--Woods algebras, cocycle crossed products, and twisted discrete measured groupoid von Neumann algebras.

\section{Preliminaries}\label{chp:preliminary}

\subsection{Basic Notation}
Let $\M$ be a von Neumann algebra. In this paper, all von Neumann algebras are $\sigma$-finite. We denote by $\M_\ast$  its predual which is identified with the space of normal linear functionals on $\M$.  We denote by $\mathcal{U}(\M)$ the group of unitaries on $\M$,  by $\mathcal{Z}(\M)$ the center of $\M$, and by $(\M)_1$ the unit ball of $\M$. For a positive functional $\phi\in \M_\ast^+$, we write $\|x\|_\phi=\phi(x^\ast x)^{1/2}$ and $\|x\|_\phi^\sharp=\phi(x^\ast x+xx^\ast)^{1/2}$ for $x\in \M$. We write $\M^{\op}$ for the opposite von Neumann algebra, and $x^{\op}$ for the corresponding element in $\M^{\op}$ for $x\in \M$. We denote by $1_\M$ the unit of $\M$.

By an inclusion $\N\subseteq\M$ of von Neumann algebras, we shall always mean a unital inclusion.
An inclusion $\N\subseteq\M$ is said to be {\em with expectation} if there is a normal faithful conditional expectation $E_\N:\M\to\N$.

\subsection{Standard form of von Neumann algebras}\label{SubSec: Standard form}
A {\em standard form} for a von Neumann algebra $\M$ is  a triple $(\H, \H_+, J)$ consisting of a Hilbert space $\H$ with $\M\subseteq B(\H)$, an anti-linear isometry $J$ with $J^2=1$ and a self-dual cone $\H_+$ in $\H$ such that $\M'=J\M J, J\xi=\xi$ and $aJaJ\xi\in \H_+$ for $a\in \M, \xi\in\H_+$. 

Let $(L^2(\M), L^2(\M)_+, J)$ be the canonical standard form of $\M$ \cite[Chapter 3]{Hia21}. The conjugation $J$ induces a bimodule action of $\M$ given by $\xi x=Jx^\ast J\xi.$ For $\xi\in L^2(\M)$, we will write  $\xi^\ast=J\xi$ and denote by $\xi\xi^\ast\in \M_\ast^+$ the positive functional on $\M$ given by $x\mapsto\langle\xi,x\xi\rangle$. For any positive functional $\phi\in \M_\ast^+$, there exists a unique $\phi^{1/2}\in L^2(\M)_+$ such that $\phi(x)=\langle\phi^{1/2}, x\phi^{1/2}\rangle$.
For each automorphism $\theta$ on $\M$, there is a unitary $U_\theta$ on $L^2(\M)$ such that $U_\theta(L^2(\M)_+)=L^2(\M)_+$, $U_\theta J=J U_\theta$ and $\theta(x)=U_\theta xU_\theta^\ast$ for $x\in \M$.

We also mention the following  Radon-Nikodym type theorem: If $\psi,\phi$ are positive linear functionals then  $\psi\leq C\phi$ for some $C>0$ if and only if there is  $a \in \M$ such that $0\leq a\leq C$ and $\psi(x)=\langle\varphi^{1/2}, x\varphi^{1/2}a\rangle=|\eta^\ast|^2(x)=\eta\eta^\ast(x)$, $x\in \M$ where $\eta=\varphi^{1/2}\sqrt{a}$.

\subsection{Modular operators}\label{sub:modular operator}
One realization of the standard form of a $\sigma$-finite von Neumann algebra $\M$ is the GNS space of a faithful normal state $\phi$ on $\M$.
Let $(L^2(\M,\phi),\pi_\phi,\xi_\phi)$ be its GNS representation. The modular operators $J_\phi$ and $\Delta_\phi$ are obtained from the polar decomposition of the closure of the operator
 $S_\phi$ on $L^2(\M,\phi)$ given by $x\xi_\phi\mapsto x^\ast\xi_\phi, x\in \M$ such that $S_\phi=J_\phi\Delta_\phi^{1/2}$. The triple $(L^2(\M,\phi), L^2(\M,\phi)_+, J_\phi)$ is a standard form  for $\M$ where $L^2(\M)_+=\overline{\Delta_\phi^{1/4}\Lambda_\phi(\M_+)}$. Moreover, $\Delta^{it}\pi_\phi(\M)\Delta^{-it}=\pi_\phi(\M)$  for $t\in\mathbb R$, and  $\sigma_t^\phi(x):=\pi_\phi^{-1}(\Delta_\phi^{it}\pi_\phi(x)\Delta_\phi^{-it})$ is called the {\em modular automorphism group} on $\M$ associated to $\phi$. 

 For any $z\in\mathbb C$, denote by $\dom(\sigma_z^\varphi)$ the set of all elements $x\in\M$ such that $t\mapsto\sigma_t^\phi(x)$ extends to a continuous function on the strip $\{\lambda\in\mathbb C; -\Im\alpha\leq \Im\lambda\leq \Im\alpha\}$ which is analytic on its interior. An element $x\in\M$ is {\em $\sigma^\phi$-analytic} if $x\in \cap_{z\in\mathbb C}\dom(\sigma^\phi_z)$. We denote by $\M_{\infty,\sigma^\varphi}$ (or simply $\M_\infty$ if the state is obvious)  the set of all $\sigma^\varphi$-analytic elements of $\M$. The set $\M_{\infty,\sigma^\phi}$ is a unital strongly dense $\ast$-subalgebra of $\M$. By Takesaki's theorem \cite[Theorem 5.6]{Hia21}, if $\B\subseteq\M$ is an inclusion of von Neumann algebras with a normal faithful conditional expectation $E_\B$ such that $\varphi=\phi\circ E_\B$, then $\sigma_t^\varphi(\B)=\B$ for all $t\in\mathbb R$ and     $\B_{\infty, \sigma^{\phi_{|_\B}}}=\M_{\infty,\sigma^\phi}\cap\B$.  Conversely, if $\B$ is globally invariant under $\sigma^\phi$ then there is a unique normal faithful conditional expectation $E:\M\to\B$ such that $\phi=\phi\circ E$.

For any $y\in \M$, we have $\varphi(y^\ast ay)\leq C\varphi(a)$ for all $a\in\M_+$ if and only if $y\in \dom(\sigma_{i/2}^\phi)$ with $\|\sigma_{i/2}^\phi(y)\|\leq C^{1/2}$ \cite[Proposition 2.14]{Str20}.

\subsection{GNS-symmetric quantum Markov semigroups}\label{subsec:QMS}

Let $\M$ be a von Neumann algebra. A \emph{quantum Markov semigroup} is a family $(\Phi_t)_{t\geq 0}$ of normal unital completely positive maps on $\M$ such that
\begin{itemize}
    \item $\Phi_0=\mathrm{id}_\M$, $\Phi_{s+t}=\Phi_s\Phi_t$ for $s,t\geq 0$,
    \item $\Phi_t(x)\to x$ in weak$^\ast$ topology as $t\to 0$ for all $x\in \M$.
\end{itemize}
The \emph{generator} of $(\Phi_t)_{t\geq 0}$ is the weak$^\ast$ densely defined and closed operator $\cL$ given by
\begin{align*}
    \dom(\cL)&=\{x\in \M\mid \lim_{t\to 0}\frac 1 t (x-\Phi_t(x))\text{ exists in weak$^\ast$ topology}\},\\
    \cL(x)&=\text{weak$^\ast$-}\lim_{t\to 0}\frac 1 t(x-\Phi_t(x)).
\end{align*}
For all $\lambda>0$, the operator $\cL+\lambda$ is invertible and its inverse is given by the weak$^\ast$ integral
\begin{equation*}
    (\cL+\lambda)^{-1}(x)=\int_0^\infty e^{-t\lambda}\Phi_t(x)\,dt
\end{equation*}
for $x\in\M$. For every $\epsilon>0$, the bounded normal operator $\cL(1+\epsilon\cL)^{-1}$ generates a quantum Markov semigroup $(\Phi_{\epsilon,t})_{t\geq 0}$ and $\Phi_{\epsilon,t}(x)\to \Phi_t(x)$ in weak$^\ast$ topology as $\epsilon\searrow 0$ for all $x\in\M$, $t\geq 0$.

Let $\phi$ be a faithful normal state on $\M$. A linear map $\Phi\colon \M\to \M$ is called \emph{GNS-symmetric} with respect to $\phi$ if
\begin{equation*}
    \phi(\Phi(x)^\ast y)=\phi(x^\ast\Phi(y)),\qquad x,y\in\M.
\end{equation*}
One important consequence of the GNS symmetry condition is that if a normal unital completely positive map is GNS-symmetric, then it commutes with the modular group $\sigma^\phi$ (see \cite[Proposition 2.1]{KFGV77} for example).

By definition, a GNS-symmetric ucp map preserves the state $\phi$. By Kadison--Schwarz, there exists a unique contraction $\Phi^{(2)}$ on $L^2(\M)$ such that $\Phi^{(2)}(x\phi^{1/2})=\Phi(x)\phi^{1/2}$ for $x\in \M$. The operator $\Phi^{(2)}$ commutes strongly with $\Delta_\phi$, leaves $\phi^{1/2}$ invariant and satisfies $\Phi^{(2)}(L^2_+(\M))\subseteq L^2_+(\M)$. Moreover, under the canonical identification $L^2(\M\otimes M_n(\mathbb C))\cong L^2(\M)\otimes S^2_n$ one has $(\Phi^{(2)}\otimes\mathrm{id})(L^2_+(\M\otimes M_n(\mathbb C)))\subseteq L^2_+(\M\otimes M_n(\mathbb C))$.

A bounded linear map $T\colon L^2(\M)\to L^2(\M)$ is called a \emph{completely Markov operator} (with respect to $\phi^{1/2})$ if $T\phi^{1/2}=\phi^{1/2}$ and $(T\otimes\mathrm{id})(L^2_+(\M\otimes M_n(\mathbb C)))\subseteq L^2_+(\M\otimes M_n(\mathbb C))$ for all $n\in \mathbb N$. It follows from the techniques in \cite[Section 2]{Cip97} that the map $\Phi\mapsto \Phi^{(2)}$ is a bijection between the GNS-symmetric ucp maps on $\M$ and the symmetric completely Markov operators on $L^2(\M)$ that commute strongly with $\Delta_\phi$. Moreover, $(\Phi_t)_{t\geq 0}\mapsto (\Phi_t^{(2)})_{t\geq 0}$ is a bijection between the GNS-symmetric quantum Markov semigroups on $\M$ and the strongly continuous semigroups of symmetric completely Markov operators  on $L^2(\M)$ that commute strongly with $\Delta_\phi$ (see again \cite[Section 2]{Cip97}).

Moreover, we define $L^1(\M)=\M_\ast$ and $\Phi_t^{(1)}=(\Phi_t)_\ast$ for $t\geq 0$. The family $(\Phi_t^{(1)})_{t\geq 0}$ is a strongly continuous contraction semigroup on $L^1(\M)$ such that $\Phi_t^{(1)}(x\phi)=\Phi_t(x)\phi$ for all $x\in\M$, $t\geq 0$.

More generally, $(\Phi_t)_{t\geq 0}$ can also be extended to $L^p(\M)$ for $p\in[1,\infty]$, but we will only deal with the cases $p\in \{1,2,\infty\}$ in this article. Note that $\Phi_t^{(1)}$ and $\Phi_t$ do not depend on the choice of the invariant state $\phi$. The next lemma shows that the same is true for $p=2$.

\begin{lemma}
    If $(\Phi_t)_{t\geq 0}$ is a quantum Markov semigroup on $\M$ that is GNS-symmetric with respect to $\phi$ and $\psi$ is a faithful normal state on $\M$ that is invariant under $(\Phi_t)_{t\geq 0}$, then $\Phi_t^{(2)}(x\psi^{1/2})=\Phi_t(x)\psi^{1/2}$ for all $x\in\M$.
\end{lemma}
\begin{proof}
    By \cite[Theorem A]{Wat79}, the fixed-point set $\B=\{x\in \M\mid \Phi_t(x)=x\text{ for all }t\geq 0\}$ is a von Neumann subalgebra of $\M$. For every $x\in\M$, the limit $E(x)=\lim_{T\to\infty}\frac 1 T\int_0^T \Phi_t(x)\,dt$ exists in the strong operator topology and $E$ is a faithful normal conditional expectation from $\M$ onto $\B$. Clearly, $\phi\circ E=\phi$ and $\psi\circ E=\psi$.

    Let $\phi_\B$, $\psi_\B$ denote the restrictions of $\phi$ and $\psi$ to $\B$. Since $\phi_\B^{1/2}\in L^2(\B)$ is cyclic, there exists a sequence $(b_n)$ in $\B$ such that $b_n\phi_\B^{1/2}\to \psi_\B^{1/2}$. If $\iota_\B\colon L^2(\B)\to L^2(\M)$ denotes the isometry that maps $b\phi_\B^{1/2}$ to $b\phi^{1/2}$, then $\iota_\B^\ast(x\phi^{1/2})=E(x)\phi_\B^{1/2}$ and
    \begin{equation*}
        \langle \iota_\B \psi_\B^{1/2},x\iota_\B\psi_\B^{1/2}\rangle=\langle \psi^{1/2},E(x)\psi^{1/2}\rangle=\psi(x).
    \end{equation*}
    Thus $\psi^{1/2}=\iota_\B \psi_\B^{1/2}=\lim_{n\to\infty}b_n \phi^{1/2}$ and
    \begin{equation*}
        \Phi_t^{(2)}(x\psi^{1/2})=\lim_{n\to\infty}\Phi_t(xb_n)\phi^{1/2}=\lim_{n\to \infty}\Phi_t(x)b_n\phi^{1/2}=\Phi_t(x)\psi^{1/2},
    \end{equation*}
    where we used that $\B$ is contained in the multiplicative domain of $\Phi_t$.
\end{proof}

The generator $\cL_p$ of $(\Phi_t^{(p)})$ is the closed densely defined operator given by
\begin{align*}
    \dom(\cL_p)&=\{a\in L^p(\M)\mid \lim_{t\to 0}\frac 1 t(a-\Phi_t^{(p)}(a))\text{ exists}\}\\
    \cL_p(a)&=\lim_{t\to 0}\frac 1 t(a-\Phi_t^{(p)}(a)).
\end{align*}
We call $\cL_p$ the $L^p$ generator of $(\Phi_t)_{t\geq 0}$.

In the case $p=2$, the operator $\cL_2$ is positive self-adjoint and $\Phi_t^{(2)}=e^{-t\cL_2}$ for $t\geq 0$ in the sense of functional calculus. As a positive self-adjoint operator, $\cL_2$ is uniquely determined by its quadratic form $\mathcal E$ given by
\begin{align*}
    \dom(\mathcal E)=\dom(\cL_2^{1/2}),\,\mathcal E(a)=\norm{\cL^{1/2}a}_2^2.
\end{align*}
The quadratic form $\mathcal E$ is called the \emph{quantum Dirichlet form associated with $(\Phi_t)_{t\geq 0}$}. As a consequence of the strong commutation of $\cL_2$ and $\Delta_\phi$, the form $\mathcal E$ is invariant under $(\Delta_\phi^{is})_{s\in\mathbb R}$. An intrinsic characterization of the quadratic forms that occur as quantum Dirichlet forms associated with a GNS-symmetric quantum Markov semigroup can be deduced from \cite[Section 4]{Cip97}.

For $\epsilon>0$, the quantum Dirichlet form associated with the quantum Markov semigroup $(\Phi_{\epsilon,t})_{t\geq 0}$ introduced above is the \emph{Moreau regularization} $\mathcal E_\epsilon$ of $\mathcal E$, given by
\begin{equation*}
    \dom(\mathcal E_\epsilon)=L^2(\M),\,\mathcal E_\epsilon(a)= \inf_{b\in \dom(\mathcal E)}\mathcal E(b)+\frac 1 \epsilon \norm{a-b}_2^2.
\end{equation*}
For every $\epsilon>0$, the net $(\mathcal E_\epsilon(a))_{\epsilon>0}$ is decreasing in $\epsilon$ with $\lim_{\epsilon\searrow 0}\mathcal E_\epsilon(a)=\mathcal E(a)$ if $a\in\dom(\mathcal E)$ and $\lim_{\epsilon\searrow 0}\mathcal E_\epsilon(a)=\infty$ otherwise.

Moreover, if $\cL_2$ is the $L^2$ generator of a GNS-symmetric quantum Markov semigroup with associated quantum Dirichlet form $\mathcal E$ and $\beta\in (0,1)$, then $\cL^\beta$ is the $L^2$ generator of a GNS-symmetric quantum Markov semigroup and the associated quantum Dirichlet form $\mathcal E^{(\beta)}$ is given by
\begin{align*}
    \dom(\mathcal E^{(\beta)})&=\left\{a\in L^2(\M): \int_0^\infty \epsilon^{-\beta }\mathcal E_\epsilon(a)\,d\epsilon<\infty\right\},\\
    \mathcal E^{(\beta)}(a)&=\frac{\sin \beta \pi}{\pi}\int_0^\infty \epsilon^{-\beta}\mathcal E_\epsilon(a)\,d\epsilon.
\end{align*}

\subsection{Hilbert bimodules, Tomita bimodules, Connes fusion tensor products}\label{subsec:bimodules}

Let $\M$, $\N$ be von Neumann algebras. An {\em $\M$-$\N$ correspondence}, also called {\em Hilbert bimodule},  is a Hilbert space $\H $ together with normal unital $\ast$-representations $\pi_l^{\H }\colon \M\to B(\H )$, $\pi_r^{\H }\colon \N^{\mathrm{op}}\to B(\H )$ with commuting images. We use bimodule notation and write $x\xi y$ for $\pi_l^{\H }(x)\pi_r^{\H }(y^{\mathrm{op}})\xi$ if $x\in \M$, $y\in\N$ and $\xi\in\H $.

For example, the standard Hilbert space $L^2(\M)$ is an $\M$-$\M$ correspondence with the standard representation of $\M$ as the left action and with the right action given by $\xi y=Jy^\ast J\xi$ for $y\in\M$ and $\xi\in L^2(\M)$.

Now let $\phi\in \M_\ast$ and $\psi\in \N_\ast$ be faithful normal states. A vector $\xi\in \H $ is called \emph{left $\psi$-bounded} if the map
\begin{equation*}
    \psi^{1/2}\N\to \H ,\,\psi^{1/2}y\mapsto \xi y
\end{equation*}
extends to a bounded linear operator $L_\psi(\xi)$ from $L^2(\N)$ to $\H $. By definition, $L_\psi(\xi)$ is a right $\N$-module map. The set of all left $\psi$-bounded vectors in $\H $ is denoted by $\mathcal D(\H ,\psi)$.

A vector $\xi\in \H $ is called \emph{right $\phi$-bounded} if the map
\begin{equation*}
    \M\phi^{1/2}\to\H ,\,x\phi^{1/2}\mapsto x\xi
\end{equation*}
extends to a bounded linear operator $R_\phi(\xi)\colon L^2(\M)\to \H $. By definition, $R_\phi(\xi)$ is a left $\M$-module map. The set of all right $\phi$-bounded vectors in $\H $ is denoted by $\mathcal D^\prime(\H ,\phi)$. A vector $\xi\in\H $ that is both left $\psi$-bounded and right $\phi$-bounded is called $(\phi,\psi)$-bounded.

For example, $\mathcal D(L^2(\M),\phi)=\M\phi^{1/2}$ with $L_\phi(x\phi^{1/2})=x$ and $\mathcal D^\prime(L^2(\M),\phi)=\phi^{1/2}\M$ with $R_\phi(\phi^{1/2}y)=Jy^\ast J$.

Now let $\Q$ be another von Neumann algebra with faithful normal state $\omega\in \Q_\ast$. If $\H $ is an $\M$-$\Q$ correspondence and $\K $ is a $\Q$-$\N$ correspondence, then
\begin{align*}
    s\colon (\mathcal D(\H ,\omega)\odot \K )\times (\mathcal D(\H ,\omega)\odot \K )&\to \mathbb C,\\
    (\xi_1\otimes \eta_1,\xi_2\otimes \eta_2)&\mapsto \langle \eta_1, \pi_l^{\K }(L_\omega(\xi_1)^\ast L_\omega(\xi_2))\eta_2\rangle
\end{align*}
is a positive semi-definite sesquilinear form. The \emph{relative tensor product} $\H \otimes_\omega\K $ is the Hilbert space obtained from $\mathcal D(\H ,\omega)\odot \K $ after separation and completion with respect to $s$. If $\xi\in \mathcal D(\H,\omega)$ and $\eta\in\K $, then the image of $\xi\otimes \eta$ in $\H \otimes_\omega\K $ is denoted by $\xi\otimes_\omega\eta$. The relative tensor product $\H \otimes_\omega \K $ becomes an $\M$-$\N$ correspondence with the bimodule actions $x(\xi\otimes_\omega\eta)y=x\xi\otimes_\omega \eta y$.

Up to isomorphism, the relative tensor product $\H\otimes_\omega \K$ does not depend on the choice of the faithful normal state $\omega$. If the precise realization is  immaterial, we also write $\H\otimes_\Q \K$ (or ${_\M\H}\otimes_\Q \K_\N$ if we want to stress the bimodule actions) for $\H\otimes_\omega\K$. We note that $L^2(\M)$ satisfies the absorption property ${_\M L^2(\M)}\otimes_\M \H_\N\cong {_\M \H_\N}$ and ${_\N \K}\otimes_\M L^2(\M)_\M\cong {_\N\K_\M}$ for every $\M$-$\N$ correspondence $\H$ and every $\N$-$\M$ correspondence $\K$.

If $S\colon \H_1\to\H_2$ is a bounded right $\Q$-module map and $T\colon \K _1\to\K _2$ is a bounded left $\Q$-module map, then there exists a unique bounded linear map $S\otimes_\omega T\colon \H _1\otimes_\omega\K _1\to\H_2\otimes_\omega \K _2$ such that $(S\otimes_\omega T)(\xi\otimes_\omega \eta)=S\xi\otimes_\omega T\eta$ for all $\xi\in\mathcal D(\H _\Q,\omega)$ and $\eta\in\K $. Moreover, $\norm{S\otimes_\omega T}\leq \norm{S}\norm{T}$ and $(S\otimes_\omega T)^\ast=S^\ast\otimes_\omega T^\ast$.

A \emph{Tomita correspondence} over $(\M,\phi)$ is a triple $(\H ,\mathcal J,(\mathcal U_t)_{t\in \mathbb R})$ consisting of an $M$-$M$ correspondence $\H $, an anti-unitary involution $\mathcal J\colon \H \to \H $ and a strongly continuous unitary group $(\mathcal U_t)_{t\in\mathbb R}$ on $\H $ such that
\begin{itemize}
    \item $\mathcal J(x\xi y)=y^\ast(\mathcal J\xi)x^\ast$ for $x,y\in \M$, $\xi\in\H $,
    \item $\mathcal U_t(x\xi y)=\sigma^\phi_t(x)(\mathcal U_t\xi)\sigma^\phi_t(y)$ for $x,y\in\M$, $\xi\in \H $, $t\in\mathbb R$,
    \item $\mathcal U_t\mathcal J=\mathcal J\mathcal U_t$ for $t\in\mathbb R$.
\end{itemize}
If $\H $ is a Tomita correspondence over $(\M,\phi)$, then the relative tensor powers $\H ^{\otimes_\phi n}$ become Tomita correspondences when endowed with the anti-unitary involution $\mathcal J^{(n)}$ and the strongly continuous unitary group $(\mathcal U^{(n)}_t)_{t\in\mathbb R}$ given by
\begin{align*}
    \mathcal J^{(n)}(\xi_1\otimes_\phi\dots\otimes_\phi \xi_n)&=\mathcal J\xi_n\otimes_\phi\dots \mathcal J\xi_1\\
    \mathcal U_t^{(n)}(\xi_1\otimes_\phi\dots\otimes_\phi \xi_n)&=\mathcal U_t\xi_1\otimes_\phi\dots\otimes_\phi \mathcal U_t\xi_n
\end{align*}
for $\xi_1,\dots,\xi_n\in \mathcal D(\H ,\phi)\cap \mathcal D^\prime(\H ,\phi)$.

Using an operator-valued semi-circular construction (see \cite{Shl99,KW25}), one can show that there exists a von Neumann algebra $\N$ containing $\M$ with a faithful normal conditional expectation $E\colon \N\to \M$ such that $\H\subseteq L^2(\N)\ominus L^2(\M)$ as $\M$-correspondences, $\mathcal J=J_{\phi\circ E}|_{\H}$ and $\mathcal U_t=\Delta_{\phi\circ E}^{it}|_{\H}$ for $t\in\mathbb R$.

\subsection{Modular derivations}\label{sec:modular derivation}

Let $\M$ be a von Neumann algebra and $\phi$ a faithful normal state on $\M$. The \emph{left Hilbert algebra} $\mathfrak A_\phi$ associated with $\phi$ is $M\phi^{1/2}$. The \emph{maximal Tomita algebra} associated with $\phi$ is $\mathfrak T_\phi=\{a\in\M \phi^{1/2}\cap \bigcap_{z\in\mathbb C}\dom(\Delta_\phi^z)\mid\Delta_\phi^z a\in M\phi^{1/2}\text{ for all }z\in\mathbb C\}$. We say that a linear subspace $\mathfrak T$ of $\mathfrak T_\phi$ is a \emph{unital Tomita subalgebra} if it is invariant under $J$, invariant under $\Delta_\phi^z$ for all $z\in\mathbb C$,  contains $\phi^{1/2}$ and if $a,b\in \mathfrak T$ implies $\pi_l(a)b=\pi_r(b)a\in \mathfrak T$.

Let $(\H ,\mathcal J,(\mathcal U_t))$ be a Tomita correspondence over $(\M,\phi)$. A linear map $\delta\colon \dom(\delta)\subseteq L^2(\M)\to \H $ is called a \emph{$\phi$-modular derivation} if
\begin{itemize}
    \item $\dom(\delta)$ is a unital Tomita subalgebra of $\mathfrak T_\phi$,
    \item $\delta(Ja)=\mathcal J\delta(a)$ for $a\in \dom(\delta)$,
    \item $\delta(\Delta_\phi^{iz}a)=\mathcal U_z\delta(a)$ for $a\in \dom(\delta)$ and $z\in\mathbb C$,
    \item $\delta(ab)=\pi_l(a)\delta(b)+\delta(a)\cdot J\pi_r(b)^\ast J$ for $a,b\in \dom(\delta)$.
\end{itemize}
We say that $\delta\colon\dom(\delta)\subseteq L^2(\M)\to\H$ is a \emph{modular derivation} if there exists a faithful normal state $\phi$ on $\M$, an anti-unitary involution $\mathcal J\colon \H\to \H$ and a strongly continuous unitary group $(\mathcal U_t)_{t\in\mathbb R}$ such that $(\H,\mathcal J,(\mathcal U_t)_{t\in\mathbb R})$ is a Tomita correspondence over $(\M,\phi)$ and $\delta$ is a $\phi$-modular derivation. Note that given $\delta$ and $\phi$, the Tomita correspondence structure is uniquely determined on the closed bimodule generated by $\operatorname{ran}\delta$.

Note that if $\delta$ is a densely defined $\phi$-modular derivation, then $\pi_l(\dom(\delta))^{\prime\prime}=\M$. Indeed, since $\dom(\delta)$ is invariant under $(\Delta_\phi^{it})_{t\in\mathbb R}$, the von Neumann subalgebra $\pi_l(\dom(\delta))^{\prime\prime}$ is invariant under $\sigma^\phi$. By Takesaki's theorem, there exists a $\phi$-preserving conditional expectation $E\colon \M\to \pi_l(\dom(\delta))^{\prime\prime}$ such that $E(x)\phi^{1/2}$ is the orthogonal projection of $x\phi^{1/2}$ onto $\overline{\dom(\delta)}$ for all $x\in\M$. If $\dom(\delta)$ is dense in $L^2(\M)$, then $x=E(x)\in \pi_l(\dom(\delta))^{\prime\prime}$ for all $x\in \M$.

If $\delta\colon\dom(\delta)\subseteq L^2(\M)\to\H $ is a densely defined closable $\phi$-modular derivation, then $\delta^\ast\bar\delta$ is a positive self-adjoint operator on $L^2(\M)$, and it was shown in \cite{Wir24} that $\delta^\ast\bar\delta$ is the $L^2$-generator of a GNS-symmetric quantum Markov semigroup $(\Phi_t)_{t\geq 0}$ on $\M$. The following result shows how the fixed-point algebra of $(\Phi_t)_{t\geq 0}$ can be recovered from $\delta$. 

\begin{lemma}
    If $\delta\colon\dom(\delta)\subseteq L^2(\M)\to\H $ is a densely defined closable $\phi$-modular derivation with associated GNS-symmetric quantum Markov semigroup $(\Phi_t)_{t\geq 0}$ and $x\in \M$, then $\Phi_t(x)=x$ for all $t\geq 0$ if and only if $x\phi^{1/2}\in \ker \delta$.
\end{lemma}
\begin{proof}
    If $\Phi_t(x)=x$ for all $t\geq 0$, then $\Phi_t^{(2)}(x\phi^{1/2})=x\phi^{1/2}$ for all $t\geq 0$ and thus $x\phi^{1/2}\in \ker \mathcal L_2$. Hence $\norm{\bar\delta(x\phi^{1/2})}^2=\langle x\phi^{1/2},\mathcal L_2(x\phi^{1/2})\rangle=0$.
    
    Conversely, if $x\phi^{1/2}\in \ker \bar\delta\subseteq \ker \mathcal L_2$, then $\Phi_t(x)\phi^{1/2}=\Phi_t^{(2)}(x\phi^{1/2})=x\phi^{1/2}$, hence $\Phi_t(x)=x$ for all $t\geq 0$.
\end{proof}

\begin{notation}
    If $\delta$ is a densely defined closable modular derivation, we write $\Null(\bar\delta)$ for $\{x\in \M\mid x\phi^{1/2}\in \ker\bar\delta\}$.
\end{notation} As discussed before, the space $\Null(\bar\delta)$ is a unital von Neumann subalgebra of $\M$ that is globally invariant under $\sigma^\phi$. Moreover, $\Phi_t$ is a $\Null(\bar\delta)$-bimodule map for all $t\geq 0$.

Conversely, let $(\Phi_t)_{t\geq0}$ be a GNS-symmetric quantum Markov semigroup on $\M$ with associated quantum Dirichlet form $\mathcal E$. It was shown in \cite{Wir22} that 
\begin{equation*}
    \mathfrak A_{\mathcal E}=\{a\in \M\phi^{1/2}\cap\bigcap_{z\in \mathbb C}\dom(\Delta_\phi^{z})\mid \Delta_\phi^z a\in \dom(\mathcal E)\text{ for all }z\in\mathbb C\}
\end{equation*}
is a unital Tomita subalgebra of $\mathfrak T_\phi$ and a form core for $\mathcal E$. Let $A_{\mathcal E}=\overline{J\pi_r(\mathfrak A_{\mathcal E})J}^{\norm\cdot}$, which is a weak$^\ast$-dense $C^\ast$-subalgebra of $\M$.

For $a,b\in \mathfrak A_{\mathcal E}$, the carré du champ $\Gamma(a,b)$ is defined as
\begin{equation*}
    \Gamma(a,b)\colon J\pi_r(\mathfrak A_{\mathcal E})J\to \mathbb C,\,J\pi_r(c)^\ast J\mapsto \frac 1 2(\mathcal E(a,bc)+\mathcal E(ac^\flat,b)-\mathcal E(c^\flat,a^\sharp b)),
\end{equation*}
where $a^\sharp =J\Delta_\phi^{1/2} a$ and $c^\flat=J\Delta_\phi^{-1/2}c$. The carré du champ extends to a bounded linear functional on $A_{\mathcal E}$ with $\norm{\Gamma(a,b)}_{A_{\mathcal E}^\ast}\leq \mathcal E(a)^{1/2}\mathcal E(b)^{1/2}$.

The quantum Markov semigroup $(\Phi_t)_{t\geq 0}$ or the associated quantum Dirichlet form $\mathcal E$ is called \emph{$\Gamma$-regular} if $\Gamma(a,b)$ extends to a normal linear functional on $\M$ for all $a,b\in\mathfrak A_{\mathcal E}$. In particular, every bounded quantum Dirichlet form is $\Gamma$-regular.

It was shown in \cite{Wir22} that $(\Phi_t)_{t\geq 0}$ is $\Gamma$-regular if and only if there exists a Tomita correspondence $(\H ,\mathcal J,(\mathcal U_t)_{t\in\mathbb R})$ over $(\M,\phi)$ and a closable densely defined $\phi$-modular derivation $\delta\colon \mathfrak A_{\mathcal E}\to \H $ such that the $L^2$-generator $\cL_2$ of $(\Phi_t)_{t\geq 0}$ is given by $\cL_2=\delta^\ast\bar\delta$. In this case,
\begin{equation*}
    \langle \delta(a),\delta(b)y\rangle=\Gamma(a,b)(y)
\end{equation*}
for $a,b\in \mathfrak A_{\mathcal E}$ and $y\in \M$.

If $\cL_2$ is the $L^2$ generator of a GNS-symmetric quantum Markov semigroup and $\beta\in (0,1)$, then the quantum Dirichlet form $\mathcal E^{(\beta)}$ described in Subsection \ref{subsec:QMS} is $\Gamma$-regular, even if the quantum Dirichlet form $\mathcal E$ itself is not. The proof of this fact in the tracially symmetric case in \cite[Section 10.4]{CS03} carries over to the GNS-symmetric case without much difficulty.

\subsection{The Ocneanu ultraproduct of von Neumann algebras}
We recall the notion of ultraproduct of von Neumann algebras and Banach spaces.  For more on ultraproducts, we refer the reader to \cite{AH14} and \cite{HI17}. 

Let $I$ be a directed set, and let $\omega$ be a cofinal ultrafilter on $I$, that is, $\{i\in I; i\geq i_0\}\in\omega$ for all $i_0\in I$. If $I=\mathbb N$, then $\omega$ is cofinal if and only if $\omega$ is free, that is, $\omega\in\beta\mathbb N\setminus\mathbb N$. Let $(X,\|\cdot\|)$ be a Banach space. The {\em  Banach space ultraproduct} $X^\omega$ of $X$ with respect to $\omega$ is defined as the quotient of $\ell^\infty(I,X)$ by the closed subspace 
\[\{(x_i)\in\ell^\infty(I,X); \lim_{i\to\omega}\|x_i\|=0\}\]
If $(x_i)\in \ell^\infty(I,X)$,  we denote its class in $X^\omega$ by $(x_i)^\omega$. The norm on $X^\omega$ is given by $\|(x_i)^\omega\|=\lim_{i\to\omega}\|x_i\|$. If $\H$ is a Hilbert space, then $\H^\omega$ is also a Hilbert space with the inner product given by 
$\langle(\xi_i)^\omega, (\eta_i)^\omega\rangle=\lim_{i\to\omega}\langle\xi_i,\eta_i\rangle.$ 

Now let $\M$ be a $\sigma$-finite von Neumann algebra. Define
\begin{align*}
   & \mathcal{I}_\omega(\M)=\{(x_i)\in \ell^\infty(I,\M); x_i\to 0\ast\text{-strongly as }i\to\omega\}\\
   & \mathcal{M}^\omega(\M)=\{(x_i)\in \ell^\infty(I,\M); (x_i)\I_\omega\subseteq\I_\omega, \I_\omega(\M)(x_i)_i\subseteq\I_\omega(\M)\}
\end{align*}
The multiplier algebra $\mathcal{M}^\omega(\M)$ is a $C^\ast$-algebra and $\I_\omega(\M)$ is a norm-closed two-sided ideal of $\mathcal{M}^\omega(\M)$. Following \cite{AH14}, we define the {\em Ocneanu  ultraproduct von Neumann algebra} $\M^\omega$ by $\M^\omega:=\mathcal{M}^\omega/\I_\omega$. The image of $(x_i)\in \mathcal{M}^\omega(\M)$ in $\M^\omega$ is denoted by $(x_i)^\omega$.

We regard $\M$ as a von Neumann subalgebra of $\M^\omega$ via the constant net $(x)^\omega$ for $x\in \M$. The map $E_\M:\M^\omega\to\M$ defined by $(x_i)^\omega\mapsto\lim_{i\to\omega}x_i$ (in $\sigma$-weak topology) is a faithful normal conditional expectation. Each  faithful normal state $\varphi$ on $\M$ gives rise to a faithful normal state $\phi^\omega:=\phi\circ E_\M$.

If $\N\subseteq\M$ is a subalgebra with expectation $E_\N$, then $\I_\omega(\N)\subseteq\I_\omega(\M)$ and $\mathcal{M}^\omega(\N)\subseteq\mathcal{M}^\omega(\M)$ so that we identify $\N^\omega=\mathcal{M}^\omega(\N)/\I_\omega(\N)$ with $(\mathcal{M}^\omega(\N)+\I_\omega(\M))/\I^\omega(\M)$ and hence regard $\N^\omega\subseteq\M^\omega$ as a von Neumann subalgebra with expectation $E_{\N^\omega}:(x_i)^\omega\mapsto(E_\N(x_i))^\omega$ for $(x_i)^\omega\in \M^\omega$.
If $p\in\M$ is a projection, then $p\M^\omega p\cong(p\M p)^\omega$ via the unital isomorphism  $\pi:p\M^\omega p \to  (p\M p)^\omega$ by $\pi (p(x_n)^\omega p)=(px_np)^\omega$.
The well-definedness  and injectivity of $\pi$ is easy to verify. For surjectivity, take any $(p x_n p)\in\mathcal{M}^\omega(p\M p)$, and show that $(px_n p)\in\mathcal{M}^\omega(\M)$ using the normal faithful state $\psi:=\phi(p\cdot p)+\phi((1-p)\cdot(1-p))$ for some fixed normal faithful state $\phi$ on $\M$. 

The standard space $L^2(\M^\omega)$ is a closed subspace of the Hilbert space ultrapower $L^2(\M)^\omega$.
The standard form for $\M^\omega$ is given by $(L^2(\M^\omega), L^2(\M^\omega)\cap (L^2(\M)_+)^\omega, J^\omega)$. We shall often use the following consequence of \cite[Proposition 4.11]{AH14}.
\begin{proposition}\label{prop: criteria for elements in the idealizer}
    Let $\M$ be a von Neumann algebra equipped with a normal faithful state $\varphi$, and let $\omega$ be a cofinal ultrafilter on a directed set $I$. For any net $(x_i)\in \ell^\infty(I,\M)$, we have $(x_i)\in\mathcal{M}^\omega(\M)$ if and only if for each $\epsilon>0$, there is $c>0$ and $(y_i)\in\mathcal{M}^\omega(\M)$ such that $\lim_{i\to\omega}\|x_i-y_i\|_\phi^\sharp\leq \epsilon$ and each $y_i\in \M(\sigma^\varphi, [-a,a])$.
\end{proposition}

\section{Main Technical Lemmas}\label{chp:main_technical_lemma}

Throughout this section, let $\M$ be a von Neumann algebra and $\phi$ a faithful normal state on $\M$. GNS symmetry in this section is always understood with respect to the reference state $\phi$.

Let us briefly recall some aspects of Haagerup $L^p$-spaces (see \cite[Chapter 9]{Hia21} for details). Consider the triplet $(c(\M),\theta,\tau)$ where $c(\M)=\M\rtimes_{\sigma^{\varphi}}\mathbb R$, the dual action $\theta_s=\widehat{\sigma_s^{\varphi}}=v_s\cdot v_s^\ast$ where $(v_s\xi)(t)=e^{-ist}\xi(t)$ for $\xi\in L^2(\mathbb R, L^2(\M))$ and $\tau=\hat{\varphi}(\cdot h^{-1})$ is the canonical trace which satisfies $\tau\circ\theta_s=e^{-s}\tau$. The dual action $\theta_s$ extends uniquely  to the space of ($\tau$-measurable) operators affiliated with $c(\M)$.

For $p\in [1,\infty]$, the Haagerup $L^p$-space $L^p(\M)$ is a subspace of the $\tau$-measurable operators on $c(\M)$. It carries a norm such that the induced topology coincides with the measure topology on $L^p(\M)$. The space $L^2(\M)$ with the positive cone of positive operators and the conjugation given by taking adjoints provides a standard form for $\M$, justifying our previously introduced notation.

These spaces satisfy the usual Hölder inequality: If $a\in L^p(\M)$ and $b\in L^q(M)$ such that $\frac 1 r=\frac 1p+\frac 1q$, then $ab\in L^r(M)$ and $\norm{ab}_r\leq \norm{a}_p\norm{b}_q$. Moreover, if $a\in L^p(\M)$ is positive and $0<r\leq p$, then $a^r\in L^{p/r}(\M)$.

There is a canonical order-preserving isometric isomorphism $\M_\ast\to L^1(\M)$, $\omega\mapsto h_\omega$. The trace functional $\mathrm{tr}$ on $L^1(\M)$ is defined by $\mathrm{tr}(h_\omega)=\omega(1)$. In the following, we will identify $\omega$ with $h_\omega$. The notation $\phi^{1/2}\in L^2(\M)$ introduced before is consistent with this notation.

\begin{lemma}\label{lem:analytic_core_generator}
    If $(\Phi_t)$ is a GNS-symmetric QMS on $\M$ with $L^2$-generator $\mathcal L_2$, then $\mathfrak A_{\mathcal E}\cap\dom(\cL_2)$ is dense in $\dom(\cL_2)$ with respect to the graph norm of $\cL_2$.
\end{lemma}
\begin{proof}
    Since $\mathfrak A_{\mathcal E}$ is  dense in $\mathfrak A_{\phi}$, and is invariant under $(\Phi_t^{(2)})$, this follows from \cite[Lemma 3.7]{Wir22}.
\end{proof}

\begin{lemma}\label{lem:dom_L1_gen}
    Let $(\Phi_t)_{t\geq 0}$ be a GNS-symmetric quantum Markov semigroup with associated quantum Dirichlet form $\mathcal E$ and $L^1$-generator $\cL_1$. If $a\in \dom(\mathcal E)$ and there exists $b\in L^1(\M)$ such that
    \begin{equation*}
        \mathcal E(a,\phi^{1/4}x\phi^{1/4})=\tr(b^\ast x)
    \end{equation*}
    for all $x\in \M$ with $\phi^{1/4}x\phi^{1/4}\in \dom(\mathcal E)$, then $\phi^{1/4}a\phi^{1/4}\in \dom(\cL_1)$ and 
    \begin{equation*}
        \cL_1(\phi^{1/4}a\phi^{1/4})=b.
    \end{equation*}
\end{lemma}
\begin{proof}
    If $y\in \M$, then $(\cL+1)^{-1}(y)\in \dom(\cL)$ and $\phi^{1/4}(\cL+1)^{-1}(y)\phi^{1/4}=(\cL_2+1)^{-1}(\phi^{1/4}x \phi^{1/4})\in \dom(\cL_2)\subseteq \dom(\mathcal E)$. By assumption,
    \begin{equation*}
        \mathcal E(a,\phi^{1/4}(\cL+1)^{-1}(y)\phi^{1/4})=\tr(b^\ast (\cL+1)^{-1}(y))=\tr((\cL_1+1)^{-1}(b)^\ast y).
    \end{equation*}
    On the other hand,
    \begin{align*}
        \mathcal E(a,(\cL_2+1)^{-1}(\phi^{1/4}y\phi^{1/4}))
        &=\langle a,(1-(\cL_2+1)^{-1})(\phi^{1/4}y\phi^{1/4})\rangle_2\\
        &=\tr((1-(\cL_1+1)^{-1})(\phi^{1/4}a\phi^{1/4})^\ast y)
    \end{align*}
    Thus $\phi^{1/4}a\phi^{1/4}=(\cL_1+1)^{-1}(b+\phi^{1/4}a\phi^{1/4})\in \dom(\cL_1)$ and $\cL_1(\phi^{1/4}a\phi^{1/4})=b$.
\end{proof}

The following is a generalization of \cite[Theorem I.4.2.2]{BH91}.

\begin{proposition}\label{prop:char_Gamma_regular}
    Let $(\Phi_t)$ be a GNS-symmetric QMS on $M$ with $L^p$-generator $\cL_p$ for $p\in \{1,2\}$. The quantum Dirichlet form $\mathcal E$ associated with $(\Phi_t)$ is $\Gamma$-regular if and only if $a,b\in \dom(\cL_2)$ implies $ab\in\dom(\cL_1)$, and in this case
    \begin{equation*}
        \cL_1(ab)=\cL_2(a)b+a\cL_2(b)-2\Gamma(a^\ast,b),
    \end{equation*}
    where the products are taken in the Haagerup $L^2$ space.
\end{proposition}
\begin{proof}
    Assume that $\mathcal E$ is $\Gamma$-regular. We first consider the case $a,b\in\mathfrak A_{\mathcal E}\cap\dom(\mathcal L_2)$. In particular, there are $x,y\in\M$ such that $a=\phi^{1/4}x\phi^{1/4}$, $b=\phi^{1/4}y\phi^{1/4}$. If $x\in\M$ with $\phi^{1/4}x\phi^{1/4}\in\mathfrak A_{\mathcal E}$, then it follows from \cite[Remark 7.1]{Wir22} that
    \begin{align*}
        \mathcal E((x\phi^{1/2}y)^\ast,\phi^{1/4}z\phi^{1/4})
        &=\mathcal E(\phi^{1/4}z^\ast \phi^{1/4},x\phi^{1/2}y)\\
        &=\mathcal E(a^\ast,bz)+\mathcal E(a^\ast z^\ast, b)-2\tr(\Gamma(a^\ast,b)y)\\
        &=\tr((\mathcal L_2(a)b+a\mathcal L_2(b)-2\Gamma(a^\ast b))z)
    \end{align*}
    for all $z\in\M$ with $\phi^{1/4}z\phi^{1/4}\in \mathfrak A_{\mathcal E}$. If we only have $z\in\M$ with $\phi^{1/4}z\phi^{1/4}\in \dom(\mathcal E)$, then it follows analogous to \cite[Theorem 6.3]{Wir22} that there exists a bounded sequence $(d_n)$ in $\M$ with $\phi^{1/4}d_n\phi^{1/4}\in\mathfrak A_{\mathcal E}$ such that $d_n\to z$ strongly and $\phi^{1/4}d_n\phi^{1/4}\to \phi^{1/4}z\phi^{1/4}$ with respect to the form norm of $\dom(\mathcal E)$. Hence $\mathcal E((x\phi^{1/2}y)^\ast,\phi^{1/4}z\phi^{1/4})=\tr((\mathcal L_2(a)b+a\mathcal L_2(b)-2\Gamma(a^\ast b))z)$ remains valid. Now we deduce from Lemma \ref{lem:dom_L1_gen} that $ab=\phi^{1/4}x\phi^{1/2}y\phi^{1/4}\in \dom(\cL_1)$ and $\cL_1(ab)=\cL_2(a)b+a\cL_2(b)-2\Gamma(a^\ast,b)$.

    In the general case $a,b\in \dom(\cL_2)$, we can choose sequences $(a_n)$, $(b_n)$ in $\mathfrak A_{\mathcal E}\cap \dom(\cL_2)$ such that $a_n\to a$, $b_n\to b$ in the graph norm of $\cL_2$ by Lemma \ref{lem:analytic_core_generator}. In particular, $a_n\to a$ and $b_n\to b$ in the form norm of $\mathcal E$, which implies $\Gamma(a_n^\ast,b_n)\to \Gamma(a^\ast,b)$ in $L^1(\M)$. Since $\cL_1$ is closed, we conclude $ab\in \dom(\cL_1)$ and $\cL_1(ab)=\cL_2(a)b+a\cL_2(b)-2\Gamma(a^\ast,b)$.

    Now assume conversely that $a,b\in\dom(\cL_2)$ implies $ab\in\dom(\cL_1)$. Let $a,b\in \mathfrak A_{\mathcal E}\cap\dom(\cL_2)$ and $y\in A_{\mathcal E}$. We have
    \begin{align*}
        \mathcal E(\phi^{1/2}y^\ast,\pi_l(a)^\ast b)
        &=\lim_{t\to 0}\frac 1 t\tr(y\phi^{1/2}(\pi_l(a)^\ast b-P_t^{(2)}(\pi_l(a)^\ast b)))\\
        &=\lim_{t\to 0}\frac1 t \tr(y(a^\ast b-P_t^{(1)}(a^\ast b)))\\
        &=\tr(y\cL_1(a^\ast b)).
    \end{align*}
    It follows from \cite[Remark 7.1]{Wir22} that 
    \begin{align*}
        \tr(\Gamma(a,b)y)&=\frac 1 2(\mathcal E(a,by)+\mathcal E(ay^\ast,b)-\mathcal E(\phi^{1/2}y^\ast,\pi_l(a)^\ast b))\\
        &=\frac 1 2\tr((\cL_2(a)^\ast b+a^\ast\cL_2(b)-\cL_1(a^\ast b))y).
    \end{align*}
    Therefore $\Gamma(a,b)=\frac1  2(\cL_2(a)^\ast b+a^\ast\cL_2(b)-\cL_1(a^\ast b))\in L^1(\M)$. If $a,b\in \dom(\mathcal E)$, then by Lemma \ref{lem:analytic_core_generator} and the density of $\dom(\cL_2)$ in $\dom(\mathcal E)$, there exist sequences $(a_n)$, $(b_n)$ in $\mathfrak A_{\mathcal E}\cap \dom(\cL_2)$ such that $a_n\to a$, $b_n\to b$ in the form norm. Since
    \begin{align*}
        \norm{\Gamma(a_m-a_n,b_m-b_n)}_1&\leq \norm{\Gamma(a_m-a_n,b_m)}_1+\norm{\Gamma(a_n,b_m-b_n)}_1\\
        &\leq \mathcal E(a_m-a_n)^{1/2}\mathcal E(b_m)^{1/2}+\mathcal E(a_n)^{1/2}\mathcal E(b_m-b_n)^{1/2}
    \end{align*}
    there exists $h_{a,b}\in L^1(\M)$ such that $\Gamma(a_n,b_n)\to h_{a,b}$ in $L^1(\M)$. Thus
    \begin{align*}
        \langle \bar\delta(a),\bar\delta(b)y\rangle_\H=\lim_{n\to\infty}\langle \bar\delta(a_n),\bar\delta(b_n)y\rangle=\lim_{n\to\infty}\tr(\Gamma(a_n,b_n)y)=\tr(h_{a,b}y)
    \end{align*}
    for all $y\in A_{\mathcal E}$, which implies $\Gamma(a,b)=h_{a,b}\in L^1(\M)$.
\end{proof}

\begin{lemma}
    If $(\Phi_t)_{t\geq0}$ is a $\Gamma$-regular GNS-symmetric quantum Markov semigroup on $\M$ with $L^2$-generator $\cL_2$ and $a,b\in\dom(\cL_2)$, then
    \begin{equation*}
        \norm{\cL_2(a)b+a\cL_2(b)-\cL_1(ab)}_1\leq 2\norm{\cL_2^{1/2}(a)}_2\norm{\cL_2^{1/2}(b)}_2.
    \end{equation*}
\end{lemma}
\begin{proof}
    By the previous lemma and the definition of the carré du champ,
    \begin{align*}
        \norm{\cL_2(a)b+a\cL_2(b)-\cL_1(ab)}_1&=2\sup_{\norm{z}\leq 1}\abs{\Gamma(a^\ast b)(z)}\\
        &=2\sup_{\abs{z}\leq 1}\abs{\langle\delta(a^\ast),\delta(b)z\rangle}\\
        &\leq 2\norm{\delta(a)}_{\H}\norm{\delta(b)}_{\H}\\
        &=2\norm{\cL_2^{1/2}(a)}_2\norm{\cL_2^{1/2}(b)}_2\qedhere
    \end{align*}
    as required.
\end{proof}

\begin{lemma}\label{lem:L2_L4_duality}
    If $b\in L^2(\M)$, then
    \begin{equation*}
        \norm{b}_2=\sup\{\abs{\tr(abc)}:a,c\in L^4(M),\,\norm{a}_4,\norm{c}_4\leq 1\}.
    \end{equation*}
\end{lemma}
\begin{proof}
    One inequality is a direct consequence of Hölder's inequality. To prove the reverse inequality, let $b=u\abs{b}$ be the polar decomposition of $b$ with $u\in c_\phi(\M)$ and $\abs{b}\in L^2(\M)$. We can assume $b\neq 0$. Let $a=\abs{b}^{1/2}u^\ast/\norm{b}_2^{1/2}$, $c=\abs{b}^{1/2}/\norm{b}_2^{1/2}$. We have $\norm{a}_4,\norm{c}_4\leq 1$ and
    \begin{equation*}
        \abs{\tr(abc)}=\frac 1{\norm{b}_2}\abs{\tr(\abs{b}^{1/2}u^\ast u\abs{b}\abs{b}^{1/2})}=\frac 1{\norm{b}_2}\tr(\abs{b}^2)=\norm{b}_2.\qedhere
    \end{equation*}
\end{proof}

For the following lemma, recall that if $\mathcal L_2$ is the $L^2$-generator of a GNS-symmetric quantum Markov semigroup, then $\mathcal L_2^{1/2}$ is also the $L^2$-generator of a GNS-symmetric quantum Markov semigroup as discussed in Subsection~\ref{subsec:QMS}.

\begin{lemma}\label{lem:prod_rule}
    If $\delta\colon \dom(\delta)\subseteq L^2(M)\to \H$ is a densely defined closable $\phi$-modular derivation, $a\in \dom(\bar\delta)\cap \M\phi^{1/2}$ and $b\in \dom(\bar\delta)\cap \phi^{1/2}\M$, then $\pi_l(a)b\in \dom(\bar\delta)$ and 
    \begin{equation*}
        \bar\delta(\pi_l(a)b)=\pi_l(a) \bar\delta(b)+\bar\delta(a)\cdot J\pi_r(b)^\ast J.
    \end{equation*}
\end{lemma}
\begin{proof}
    Since $a$ is left-bounded and $b$ is right-bounded, we have $a=x\phi^{1/2}$ and $b=\phi^{1/2}y$ with $x=\pi_l(a)$ and $y=J\pi_r(b)^\ast J$. Let $\mathcal E$ be  the quantum Dirichlet form associated with $\bar\delta$. By \cite[Theorem 6.3]{Wir22} there exist sequences $(a_n)$, $(b_n)$ in $\mathfrak A_\mathcal E$ such that $a_n\to a$, $b_n\to b$ in $(\dom(\mathcal E),\norm{\cdot}_{\mathcal E})$ and $\pi_l(a_n)\to \pi_l(a)$, $\pi_r(b_n)\to \pi_r(b)$ in the strong operator topology.

    By \cite[Proposition 4.6]{Wir24}, one has
    \begin{equation*}
        \bar\delta(a_n b_n)=\pi_l(a_n)\bar\delta(b_n)+\bar\delta(a_n)\cdot J\pi_r(b_n)^\ast J.
    \end{equation*}
    Since $a_n b_n\to ab$ in $L^2(M)$ and the right side converges to $\pi_l(a)\bar\delta(b)+\bar\delta(a)\cdot J\pi_r(b)^\ast J$ in $\H $, the claim follows from the fact that $\bar\delta$ is closed.
\end{proof}

\begin{lemma}\label{lem:carre_du_champ_L2}
    If $\delta\colon \dom(\delta)\subseteq L^2(\M)\to \H $ is a densely defined closable modular derivation and $\mathcal L_2=\delta^\ast\bar\delta$, then
    \begin{align*}
        &\norm{x\mathcal L_2^{1/2}(\phi^{1/2}y)+\mathcal L_2^{1/2}(x\phi^{1/2})y-\mathcal L_2^{1/2}(x\phi^{1/2}y)}_2\\
        &\quad\leq 4\norm{x}^{1/2}\norm{y}^{1/2}\norm{\bar\delta(x\phi^{1/2})}^{1/2}\norm{\bar\delta(\phi^{1/2}y)}^{1/2}.
    \end{align*}
     for all $x,y\in \M$ with $x\phi^{1/2}$, $\phi^{1/2}y\in \dom(\bar\delta)$.
\end{lemma}

\begin{remark}
    If $x\phi^{1/2},\phi^{1/2}y\in \dom(\bar\delta)$, then $x\phi^{1/2}y\in \dom(\bar\delta)$ by Lemma \ref{lem:prod_rule}, so that the expression in the lemma is well-defined.
\end{remark}

\begin{proof}
    Let $\mathcal E$ be the quantum Dirichlet form associated with $\delta$. For $x,y\in\pi_l(\mathfrak A_{\mathcal E})$ define
    \begin{equation*}
        a_{x,y}=x^\ast\mathcal L_2^{1/2}(\phi^{1/2}y)+\mathcal L_2^{1/2}(x^\ast\phi^{1/2})y-\mathcal L_2^{1/2}(x^\ast\phi^{1/2}y).
    \end{equation*}
   By Proposition  \ref{prop:char_Gamma_regular} we have
    \begin{align*}
        \phi^{1/4}a_{x,y}\phi^{1/4}&=\phi^{1/4}x^\ast \phi^{1/4}\mathcal L_2^{1/2}(\phi^{1/4}y\phi^{1/4})+\mathcal L_2^{1/2}(\phi^{1/4}x^\ast \phi^{1/4})\phi^{1/4}y \phi^{1/4}\\
        &\quad-\mathcal L_2^{1/2}(\phi^{1/4}x^\ast\phi^{1/2}y\phi^{1/4})\\
        &=2\Gamma_{1/2}(\phi^{1/4}x\phi^{1/4},\phi^{1/4}y\phi^{1/4}),
    \end{align*}
    where $\Gamma_{1/2}$ is the carré du champ associated with $\mathcal L_2^{1/2}$.
    
    In particular, the sesquilinear form 
    \begin{align*}
        s\colon (\pi_l(\mathfrak A_{\mathcal E})\odot \M)\times (\pi_l(\mathfrak A_{\mathcal E})\odot \M)&\to \mathbb C,\\
        (x\otimes z_1,y\otimes z_2)&\mapsto \tr(z_1^\ast \phi^{1/4}a_{x,y}\phi^{1/4}z_2)
    \end{align*}
    is  positive semi-definite.
    
    By Cauchy--Schwarz and Lemma \ref{lem:L2_L4_duality}, we have
    \begin{align*}
        \norm{a_{x,y}}_2&=\sup\{\abs{\tr(z_1^\ast \phi^{1/4}a_{x,y}\phi^{1/4}z_2): z_1,z_2\in\M,\,\norm{\phi^{1/4}z_j}_4\leq 1}\\
        &=\sup_{\norm{\phi^{1/4} z_j}_4\leq 1}\abs{s(x\otimes z_1,y\otimes z_2)}\\
        &\leq \sup_{\norm{\phi^{1/4}z_1}_4\leq 1}s(x\otimes z_1)^{1/2}\sup_{\norm{\phi^{1/4}z_2}\leq 1}s(y\otimes z_2)^{1/2}\\
        &=\norm{a_{x,x}}_2^{1/2}\norm{a_{y,y}}_2^{1/2}.
    \end{align*}
    Moreover,
    \begin{align*}
        \norm{a_{x,x}}_2\leq 2\norm{\cL_2^{1/2}(\phi^{1/2}x)}_2\norm{x}+\norm{\cL_2^{1/2}(x^\ast \phi^{1/2}x)}_2
    \end{align*}
    By definition, $\norm{\cL_2^{1/2}(\phi^{1/2}x)}_2=\norm{\delta(\phi^{1/2}x)}^{1/2}$. Furthermore, by the product rule,
    \begin{align*}
        \norm{\cL_2^{1/2}(x^\ast \phi^{1/2}x)}_2=\norm{\delta(x^\ast\phi^{1/2})x+x^\ast \delta(\phi^{1/2}x)}\leq 4\norm{\delta(\phi^{1/2}x)}\norm{x},
    \end{align*}
    therefore,
    \begin{equation*}
        \norm{a_{x,x}}_2\leq 2\norm{x}\norm{\delta(\phi^{1/2}x)}.
    \end{equation*}
    If we combine this with the analogous bound for $\norm{a_{y,y}}_2$, we arrive at the desired bound for $x,y\in\pi_l(\mathfrak A_{\mathcal E})$.

    In the general case, when we only assume $x\phi^{1/2}$, $\phi^{1/2}y\in \dom(\bar\delta)$, the result follows via approximation as in \cite[Theorem 6.3]{Wir22}.
\end{proof}

\begin{lemma}\label{lem:bound_Stinespring}
    If  a ucp map $\Phi\colon \M\to \M$ is GNS-symmetric with respect to $\phi$  and $x,y\in \M$, then
    \begin{align*}
        \phi(\abs{\Phi(xy)-\Phi(x)y}^2)&\leq 2\norm{x}^2\phi(\abs{y}^2)^{1/2}\phi(\abs{y-\Phi(y)}^2)^{1/2}
    \end{align*}
    and
    \begin{equation*}
        \norm{y\Phi^{(2)}(\phi^{1/2}x)-\Phi^{(2)}(y\phi^{1/2}x)}^2\leq 2\norm{x}^2\phi(\abs{y}^2)^{1/2}\phi(\abs{y-\Phi(y)}^2)^{1/2}.
    \end{equation*}
\end{lemma}
\begin{proof}
    Let $\Phi=V^\ast\pi(\cdot)V$ be a Stinespring dilation of $\Phi$. We have
    \begin{align*}
        \phi(\abs{\Phi(xy)-\Phi(x)y}^2)&=\norm{V^\ast\pi(x)(\pi(y)V-V y)\phi^{1/2}}^2\\
        &\leq \norm{x}^2\norm{(\pi(y)V-V y)\phi^{1/2}}^2\\
        &=\norm{x}^2(\phi(\Phi(y^\ast y))+\phi(y^\ast y)-2\Re\phi(\Phi(y)^\ast y)\\
        &=2\norm{x}^2\phi(y^\ast(y-\Phi(y))\\
        &\leq 2\norm{x}^2\phi(\abs{y}^2)^{1/2}\phi(\abs{y-\Phi(y)}^2)^{1/2}.
    \end{align*}
    Thus, if $y\in \M_{\infty,\sigma^\phi}$, then
    \begin{align*}
        &\norm{\Phi^{(2)}(x\phi^{1/2}\sigma^\phi_{i/2}(y))-\Phi^{(2)}(x\phi^{1/2})\sigma^\phi_{i/2}(y)}^2\\
        &\quad\leq 2\norm{x}^2\phi(\abs{\sigma^\phi_{i/2}(y)^\ast}^2)^{1/2}\norm{\phi^{1/2}\sigma^\phi_{i/2}(y)-\Phi^{(2)}(\phi^{1/2}\sigma^\phi_{i/2}(y))}.
    \end{align*}
    Replacing $y$ by $\sigma^\phi_{-i/2}(y)$, we get
    \begin{equation*}
        \norm{\Phi^{(2)}(x\phi^{1/2}y)-\Phi^{(2)}(x\phi^{1/2})y}^2
        \leq 2\norm{x}^2\phi(\abs{y^\ast}^2)^{1/2}\norm{\phi^{1/2}y-\Phi^{(2)}(\phi^{1/2}y)}.
    \end{equation*}
    If we replace $x$ by $x^\ast$ and $y$ by $y^\ast$, we obtain the second inequality in the lemma. For general $y\in \M$, not necessarily analytic, the result follows by standard approximation arguments.
\end{proof}

\begin{notation}\label{notation}
    If $(\H ,\mathcal J,(\mathcal U_t))_{t\in\mathbb R})$ is a Tomita correspondence over $(\M,\phi)$, $\delta\colon \dom(\delta)\subseteq L^2(\M)\to\H $ is a densely defined closable $\phi$-modular derivation and $\alpha>0$, then we write $\zeta_\alpha$, $\theta_\alpha$ for the unique GNS-symmetric ucp maps on $\M$ such that
    \[\zeta_\alpha^{(2)}=\alpha^{1/2}(\cL_2+\alpha)^{-1/2},\;\;\;\;\theta_\alpha^{(2)}=\alpha(\cL_2+\alpha)^{-1} \;\; and \;\;\rho_\alpha=1-\mathcal L_2^{1/2}(\mathcal L_2+\alpha)^{-1/2}\]
    Moreover, we write
    \[\tilde\delta_\alpha =\bar\delta\circ(\cL_2+\alpha)^{-1/2}=\alpha^{-1/2}\bar\delta\circ\zeta_\alpha^{(2)}.\]
\end{notation}
We warn the readers that, although the above notations are inspired from those of \cite{Pet09a, OP10}, it may differ from the notations in those papers.

\begin{remark}\label{rmk:ucp_families}
    Let $(\Phi_t)_{t\geq 0}$ be a GNS-symmetric quantum Markov semigroup on $\M$. As discussed in \ref{subsec:QMS}, one has
    \begin{equation*}
        (\mathcal L+\alpha)^{-1}=\int_0^\infty e^{-\alpha t}\Phi_t\,dt
    \end{equation*}
    as a weak$^\ast$ integral. 
    Thus $\theta_\alpha=\alpha(\mathcal L+\alpha)^{-1}$ is a GNS-symmetric ucp map.
    With the help of the spectral theorem, one can show that (cf. \cite[Lemma 3.2]{Pet09a})
    \begin{align*}
        \alpha^{1/2}(\mathcal L_2+\alpha)^{-1/2}&=\frac 1 \pi\int_0^\infty \theta^{(2)}_{\alpha (1+t)/t}\frac{dt}{(1+t)\sqrt t},\\
        1-\mathcal L_2^{1/2}(\mathcal L_2+\alpha)^{-1/2}&=\frac 1 \pi \int_0^\infty \theta^{(2)}_{t\alpha/(1+t)}\frac{dt}{(1+t)\sqrt t}
    \end{align*}
    in the strong operator topology. It follows from Cipriani's correspondence \cite[Section 2]{Cip97} that the ucp maps $\zeta_\alpha$ and $\eta_\alpha$ exist.
\end{remark}

For the following lemma, note that $\tilde\delta_\alpha^\ast\tilde\delta_\alpha=\cL_2(\cL_2+\alpha)^{-1}\leq 1$ for all $\alpha>0$ and $\tilde\delta_\alpha\to 0$ in strong operator topology as $\alpha\to\infty$.

\begin{lemma}\label{lem: the main technical lemma}
    Let $(\H ,\mathcal J,(\mathcal U_t)_{t\in\mathbb R})$ be a Tomita correspondence over $(\M,\phi)$ and let $\delta\colon \dom(\delta)\subseteq L^2(\M)\to \H $ be a densely defined closable $\phi$-modular derivation. For $x,y\in \M$ and $\alpha>0$, one has
    \begin{align*}                      
        \norm{\zeta_\alpha(x)\tilde\delta_\alpha (\phi^{1/2}y)-\tilde\delta_\alpha(x\phi^{1/2}y)}&\leq 10\norm{x}^{1/2}\norm{y}\norm{\tilde\delta_\alpha(x\phi^{1/2})}^{1/2},\\
        \norm{\tilde\delta_\alpha(x\phi^{1/2})\zeta_\alpha (y)-\tilde\delta_\alpha(x\phi^{1/2}y)}&\leq 10\norm{y}^{1/2}\norm{x}\norm{\tilde\delta_\alpha(\phi^{1/2}y)}^{1/2}.
    \end{align*}
    In particular, 
    \begin{align*}
        \lim_{\alpha\to\infty} \norm{\zeta_\alpha(x)\tilde\delta_\alpha(a)-\tilde\delta_\alpha(xa)}=0\\
        \lim_{\alpha\to\infty} \norm{\tilde\delta_\alpha(a) \zeta_\alpha(x)-\tilde\delta_\alpha(a x)}=0\\
    \end{align*}
    for all $x\in \M$, $a\in L^2(\M)$.
\end{lemma}
\begin{proof}
    We may assume that $x$ and $y$ are analytic for $\sigma^\phi$. By the product rule for $\bar\delta$ from Lemma \ref{lem:prod_rule} we have
    \begin{equation*}
        \zeta_\alpha(x)\tilde\delta_\alpha(\phi^{1/2}y)=\alpha^{-1/2}\bar\delta(\zeta_\alpha(x)\phi^{1/2}\zeta_\alpha(y))-\tilde\delta_\alpha(x\phi^{1/2})\zeta_\alpha(y).
    \end{equation*}
    Note that $\norm{\tilde\delta_\alpha(x\phi^{1/2})\zeta_\alpha(y)}\leq \norm{\tilde\delta_\alpha(x\phi^{1/2})}\norm{y}$ since $\zeta_\alpha$ is a ucp map.

    To bound $\alpha^{-1/2}\bar\delta(\zeta_\alpha(x)\phi^{1/2}\zeta_\alpha(y))$, let $\bar\delta=V\cL_2^{1/2}$ be the polar decomposition of $\bar\delta$. As a consequence of Proposition \ref{prop:char_Gamma_regular} and Lemma \ref{lem:carre_du_champ_L2} we have
    \begin{align*}
        &V^\ast (\alpha^{-1/2}\bar\delta(\zeta_\alpha(x)\phi^{1/2}\zeta_\alpha(y)))\\
        &\quad=\alpha^{-1/2}\cL_2^{1/2}(\zeta_\alpha(x)\phi^{1/2}\zeta_\alpha(y))\\
        &\quad=\underbrace{\alpha^{-1/2}\zeta_\alpha (x)\cL_2^{1/2}(\phi^{1/2}\zeta_\alpha(y))}_{b_1}+\underbrace{\alpha^{-1/2}\cL_2^{1/2}(\zeta_\alpha(x)\phi^{1/2})\zeta_\alpha(y)}_{b_2}\\
        &\quad\quad-\underbrace{2\alpha^{-1/2}a_{\zeta_\alpha(x)^\ast,\zeta_\alpha(y)}}_{b_3},
    \end{align*}
    where $a_{\zeta_\alpha(x)^\ast,\zeta_\alpha(y)}\in L^2(\M)$ with 
    \begin{align*}
        \norm{a_{\zeta_\alpha(x)^\ast,\zeta_\alpha(y)}}_2
        &\leq 2\norm{\zeta_\alpha(x)}^{1/2}\norm{\zeta_\alpha(y)}^{1/2}\norm{\bar\delta(\phi^{1/2}\zeta_\alpha(x)^\ast)}^{1/2}\norm{\bar\delta(\phi^{1/2}\zeta_\alpha(y)}^{1/2}\\
        &\leq 2\alpha^{1/2}\norm{x}^{1/2}\norm{y}^{1/2}\norm{\tilde\delta_\alpha(\phi^{1/2}x)}^{1/2}\norm{\tilde\delta_\alpha(\phi^{1/2}y)}^{1/2}\\
        &\leq 2\alpha^{1/2}\norm{x}^{1/2}\norm{y}\norm{\tilde\delta_\alpha(\phi^{1/2}x)}^{1/2}.
    \end{align*}

    We have $\norm{b_2}\leq \norm{\alpha^{-1/2}\cL_2^{1/2}(\zeta_\alpha(x\phi^{1/2}))}\norm{y}=\norm{\tilde\delta_\alpha(x\phi^{1/2})}\norm{y}$ and
    \begin{align*}
        b_1-V^\ast\tilde\delta_\alpha(x\phi^{1/2}y)
        &=\alpha^{-1/2}(\zeta_\alpha(x)\cL_2^{1/2}\zeta_\alpha^{(2)}(\phi^{1/2}y)-\cL_2^{1/2}\zeta_\alpha^{(2)}(x\phi^{1/2}y))\\
        &=(\zeta_\alpha^{(2)}(x\phi^{1/2})y-x\phi^{1/2}y)+(\rho_\alpha^{(2)}(x\phi^{1/2}y)-\zeta_\alpha(x)\rho_\alpha^{(2)}(\phi^{1/2}y)).
    \end{align*}
    By the spectral theorem, if $E$ denotes the spectral measure of $\cL_2$, then
    \begin{align*}
        \norm{(\zeta_\alpha^{(2)}-1)(a)}^2
        &=\int_{[0,\infty)}(1-\alpha^{1/2}(\alpha+\lambda)^{-1/2})^2\,d\norm{E(\lambda)a}^2\\
        &\leq \int_{[0,\infty)}\frac{\lambda}{\alpha+\lambda}\,d\norm{E(\lambda)a}^2\\
        &=\alpha^{-1}\norm{\cL_2^{1/2}(\zeta_\alpha^{(2)}(a))}^2\\
        &=\norm{\tilde\delta_\alpha(a)}^2
    \end{align*}
    for all $a\in L^2(M)$. Thus
    \begin{equation*}
        \norm{\zeta_\alpha^{(2)}(x\phi^{1/2})\phi^{1/2}y-x\phi^{1/2}y}\leq \norm{\tilde\delta_\alpha(x\phi^{1/2})}\norm{y}.
    \end{equation*}
    Moreover,
    \begin{align*}
        &\norm{\rho_\alpha^{(2)}(x\phi^{1/2}y)-\zeta_\alpha(x)\rho_\alpha^{(2)}(\phi^{1/2}y)}\\
        &\quad\leq \norm{\rho_\alpha^{(2)}(x\phi^{1/2}y)-\rho_\alpha^{(2)}(x\phi^{1/2})\rho_\alpha(y)}+\norm{\rho_\alpha^{(2)}(x\phi^{1/2})-\zeta_\alpha^{(2)}(x\phi^{1/2})}\norm{y}.
    \end{align*}
    For the first summand on the right side, we can use the bound from Lemma~\ref{lem:bound_Stinespring} to get
    \begin{align*}
        \norm{\rho_\alpha^{(2)}(x\phi^{1/2}y)-\rho_\alpha^{(2)}(x\phi^{1/2})\rho_\alpha(y)}
        &\leq \norm{x}^{1/2}\norm{\phi^{1/2}y}\norm{x\phi^{1/2}-\rho_\alpha^{(2)}(x\phi^{1/2})}^{1/2}\\
        &\leq \norm{x}^{1/2}\norm{y}\norm{\tilde\delta_\alpha(x\phi^{1/2})}^{1/2}.
    \end{align*}
    Furthermore,
    \begin{align*}
        \norm{\rho_\alpha^{(2)}(x\phi^{1/2})-\zeta_\alpha^{(2)}(x\phi^{1/2})}
        &\leq \norm{\rho_\alpha(x\phi^{1/2})-x\phi^{1/2}}+\norm{x\phi^{1/2}-\zeta_\alpha^{(2)}(x\phi^{1/2})}\\
        &\leq 2\norm{\tilde\delta_\alpha(x\phi^{1/2})}.
    \end{align*}
    To conclude, we combine the bounds and note that $\norm{\tilde\delta_\alpha(x\phi^{1/2})}\leq \norm{x\phi^{1/2}}\leq \norm{x}$ since $\phi$ is a state.
\end{proof}

We can describe the above result in a more clean fashion using the language of ultrapowers. We first collect some facts. Let $\omega$ be a free ultrafilter on a directed set $I$.
Let $(\Phi_n)_{n\in I}$ be a net of $\phi$-symmetric ucp maps on a von Neumann algebra equipped with a normal faithful state $\phi$. 
 Let $\mathcal I_\omega=\{(x_n)\in\ell^\infty(I;\M):\norm{x_n}_\phi^\sharp\to 0\}$ and $\mathcal M^\omega$ the idealizer of $\mathcal I_\omega$ inside $\ell^\infty(I;M)$. If $(x_n)\in\ell^\infty(I;\M)$, then
    \begin{equation*}
        \norm{\Phi_n(x_n)}_\phi^\sharp=\phi(\Phi_n(x)^\ast\Phi_n(x)+\Phi_n(x)\Phi_n(x)^\ast)^{1/2}\leq \norm{x_n}_\phi^\sharp
    \end{equation*}
    by Kadison--Schwarz and $\phi\circ\Phi_n=\phi$. In particular, $\Phi(\mathcal I_\omega)\subseteq \mathcal I_\omega$.

    If $(x_n)\in\mathcal M^\omega$, then by Proposition \ref{prop: criteria for elements in the idealizer}, for every $\epsilon>0$ there exists $C>0$ and $(y_n)\in \ell^\infty(I;\M)$ such that $\lim_{n\to\omega}\norm{x_n-y_n}_\phi^\sharp<\epsilon$ and $y_n$ belongs to the spectral subspace $\M(\sigma^\phi,[-C,C])$. Since $\Phi_n$ commutes with $\sigma^\phi$, we have  $\Phi_n(y_n)\in \M(\sigma^\phi,[-C,C])$ and
    \begin{align*}
        \lim_{n\to\omega}\norm{\Phi_n(x_n)-\Phi_n(y_n)}_\phi^\sharp\leq \lim_{n\to\omega}\|x_n-y_n\|_\phi^\sharp<\epsilon.
    \end{align*} Thus, again by Proposition \ref{prop: criteria for elements in the idealizer}, $\Phi(x_n)\in\mathcal M^\omega$. Altogether we have shown that
    \[(\Phi_n)^\omega: (x_n)_n^\omega\mapsto (\Phi(x_n))_n^\omega\] is a well-defined map on $\M^\omega$. We moreover have $\phi^\omega\circ(\Phi_n)^\omega=(\phi\circ\Phi_n)^\omega=\phi^\omega$ and $(\Phi_n)^\omega\circ\sigma_t^{\phi^\omega}=\sigma_t^{\phi^\omega}\circ(\Phi_n)^\omega$. In particular, \[((\Phi_n)^\omega)^{(2)}:x(\phi^\omega)^{1/2}\mapsto(\Phi_n)^\omega(x)(\phi^\omega)^{1/2}, x\in \M^\omega\] is a well-defined positive map which clearly satisfies 
    \[(\Phi_n^{(2)})^\omega=((\Phi_n)^\omega)^{(2)}\;\;\;\mbox{on}\;\;L^2(\M^\omega).\] This also shows that $(\Phi_n^{(2)})^\omega$, which is a priori defined on $L^2(\M)^\omega$, leaves $L^2(\M^\omega)$ invariant.

If $T: L^2(\M)\to \H$ is a bounded linear map, then $T^\omega:(\xi_n)^\omega\mapsto(T\xi_n)^\omega$ is a well-defined operator from $L^2(\M^\omega)$ to $\H^\omega$ such that $\|T^\omega\|=\|T\|$ (since if $\lim_{n\to\omega}\|\xi_n\|=0$, then $\lim_{n\to\omega}\|T\xi_n\|=0$.) 
In particular, we consider $\tilde\delta_\alpha^\omega$ as an operator from $L^2(\M^\omega)$ to $\H^\omega$ in the sequel for  a densely defined closable $\phi$-modular derivation $\delta:\dom(\delta)\subseteq L^2(\M)\to\H$.

For $\xi\in L^2(\M)$, we have
\begin{align*}
    \|\tilde\delta_\alpha(\xi)\|^2=\langle\xi, \frac{\mathcal{L}_2}{\alpha+\mathcal{L}_2}\xi\rangle=\langle\xi, \xi-\theta_\alpha^{(2)}(\xi)\rangle\leq \|\xi\|\|\xi-\theta_\alpha^{(2)}(\xi)\|.
\end{align*}
On the other hand, since each $\theta_\alpha^{(2)}$ is a positive contraction, we have
\begin{align*}
    \|\xi-\theta_\alpha^{(2)}\xi\|^2=\langle\xi, (\mathrm{id}-\theta_\alpha^{(2)})^2\xi\rangle\leq \langle\xi, (\mathrm{id}-\theta_\alpha^{(2)})\xi\rangle=\|\tilde\delta_\alpha(\xi)\|^2.
\end{align*}
Therefore, for any $\|\cdot\|_2$-bounded set $S\subseteq L^2(\M)$, we have $\sup_{\xi\in S}\|\tilde\delta_\alpha(\xi)\|\to0$ if and only if $\sup_{\xi\in S}\|\xi-\theta_\alpha^{(2)}\xi\|\to0$ as $\alpha\to\infty$, a fact we will exploit quite often for the singleton $S=\{\xi\},\xi\in L^2(\M)$ or more generally for $S=\Q_1\psi^{1/2}$, where $\Q\subseteq\M$ is an inclusion of von Neumann algebras and  $\psi$ is any  normal faithful state on $\M$ such that $\Q$ is globally invariant under $\sigma^\psi$.
Moreover, if $\omega$ is a cofinal ultrafilter on a directed set, then for any $\|\cdot\|_2$-bounded subset $S\subseteq L^2(\M^\omega)$, one has $\sup_{\xi\in S}\|\tilde\delta_\alpha^\omega(\xi)\|\to0$ if and only if $\sup_{\xi\in S}\|(\theta_\alpha^{(2)})^\omega(\xi)-\xi\|\to0$ as $\alpha\to\infty$.

\begin{corollary}\label{cor; corollary to the technical lemma in the ultrapower language}
    Let $\delta:\dom\delta\to\H$ be a densely defined closable modular derivation and let $\omega$ be a cofinal ultrafilter on  a directed set. We have
    \begin{align*}
        &\lim_{\alpha\to\infty}\|\zeta_{\alpha}(x)\tilde\delta^\omega_{\alpha}(\xi)-\tilde\delta^\omega_{\alpha}(x\xi)\|=0\\
        &\lim_{\alpha\to\infty}\|\tilde\delta^\omega_{\alpha}(\xi)\zeta_{\alpha}(x)-\tilde\delta^\omega_{\alpha}(\xi x)\|=0
    \end{align*}
    for all $x\in \M$ and $\xi\in L^2(\M^\omega)$, where $\H^\omega$ is treated as a possibly non-normal bimodule over $\M$.
\end{corollary}
\begin{proof}
 We shall prove the first equality.   Let $x\in \M$ with $\|x\|_\infty\leq  1$  and let $\xi\in L^2(\M^\omega)$. Fix $\epsilon>0$ and choose $y\in\M^\omega$ such that $\|\xi-(\varphi^\omega)^{1/2} y\|\leq \epsilon$. Let  $(y_n)\in \mathcal{M}^\omega(\M)$ be  representatives of $y$, and let $C>0$ be a constant such that $C>\sup_n  \|y_n\|_\infty$. By Lemma \ref{lem: the main technical lemma}, we have
    \begin{align*}
     \|\zeta_{\alpha }(x)\tilde\delta^\omega_{\alpha }(\xi)-\tilde\delta^\omega_{\alpha }(x\xi)\|
     &\leq 2\epsilon \|x\|_\infty+ \|\zeta_{\alpha }(x)\tilde\delta^\omega_{\alpha }((\varphi^\omega)^{1/2}y)-\tilde\delta_{\alpha }(x(\varphi^\omega)^{1/2}y)\|\\
&=2\epsilon+\lim_{n\to\omega}\|\zeta_\alpha(x)\tilde\delta_\alpha(\varphi^{1/2}y_n)-\tilde\delta_\alpha(x\varphi^{1/2}y_n)\|\\
     &\leq 2\epsilon +10 C\|\tilde\delta_{\alpha }( x\varphi^{1/2})\|^{1/2}
    \end{align*}
 Because $\epsilon>0$ is arbitrary and $\tilde\delta_\alpha$ converges strongly to $0$ on $L^2(\M)$, the required assertion follows.
\end{proof}

\begin{lemma}\label{lemma:estimates on the values in derivation}
  Let $\delta:\dom\delta\subseteq L^2(\M)\to\H$ be a densely defined closable $\varphi$-modular derivation.   There exists $C>0$ such that for $\alpha\geq 1$ and $x,y\in \M$ with $x\phi^{1/2}\in \dom(\overline\delta)$ one has
    \begin{equation*}
        \norm{\tilde\delta_\alpha(x\phi^{1/2}y)-x\tilde\delta_\alpha(\phi^{1/2}y)}\leq C\alpha^{-1/4}(\norm{x}^{1/2}+\norm{\overline\delta(x\phi^{1/2})}^{1/2})\norm{\overline\delta(x\phi^{1/2})}^{1/2}\norm{y}.
    \end{equation*}
\end{lemma}
\begin{proof}
    By the product rule,
    \begin{equation*}
        x\tilde\delta_\alpha(\phi^{1/2}y)=\alpha^{-1/2}x\overline\delta(\phi^{1/2}\zeta_\alpha(y))=\alpha^{-1/2}\overline\delta(x\phi^{1/2}\zeta_\alpha(y))-\alpha^{-1/2}\overline{\delta}(x\phi^{1/2})\zeta_\alpha(y).
    \end{equation*}
    Thus
    \begin{align*}
        &\norm{\tilde\delta_\alpha(x\phi^{1/2}y)-x\tilde\delta_\alpha(\phi^{1/2}y)}\\
        &\quad\leq \alpha^{-1/2}\norm{\bar\delta(\zeta_\alpha(x\phi^{1/2}y)-x\phi^{1/2}\zeta_\alpha(y))}+\alpha^{-1/2}\norm{\bar\delta(x\phi^{1/2})\zeta_\alpha(y)}\\
        &\quad\leq \alpha^{-1/2}\norm{\bar\delta(\zeta_\alpha(x\phi^{1/2}y)-x\phi^{1/2}\zeta_\alpha(y))}+\alpha^{-1/2}\norm{\bar\delta(x\phi^{1/2})}\norm{y}.
    \end{align*}
    We have
    \begin{align*}
        &\norm{\bar\delta(\zeta_\alpha(x\phi^{1/2}y)-x\phi^{1/2}\zeta_\alpha(y))}\\
        &\quad=\norm{\mathcal L_2^{1/2}(\zeta_\alpha(x\phi^{1/2}y)-x\phi^{1/2}\zeta_\alpha(y))}\\
        &\quad\leq \norm{\mathcal L_2^{1/2}\zeta_\alpha(x\phi^{1/2}y)-x\mathcal L_2^{1/2}(\phi^{1/2}\zeta_\alpha(y))}\\
        &\qquad+\norm{x\mathcal L_2^{1/2}(\phi^{1/2}\zeta_\alpha(y))+\mathcal L_2^{1/2}(x\phi^{1/2})\zeta_\alpha(y)-\mathcal L_2^{1/2}(x\phi^{1/2}\zeta_\alpha(y))}\\
        &\qquad +\norm{\mathcal L_2^{1/2}(x\phi^{1/2})\zeta_\alpha(y)}.
    \end{align*}
   By Lemma \ref{lem:carre_du_champ_L2},  the second summand on the right side is bounded above by
    \begin{align*}
        4\norm{x}^{1/2}\norm{y}^{1/2}\norm{\bar\delta(x\phi^{1/2})}^{1/2}\norm{\bar\delta(\phi^{1/2}\zeta_\alpha(y)}^{1/2}\leq 4\alpha^{1/4}\norm{x}^{1/2}\norm{\bar\delta(x\phi^{1/2})}^{1/2}\norm{y}.
    \end{align*}
 For the third summand, we have
    \begin{align*}
        \norm{\mathcal L_2^{1/2}(x\phi^{1/2})\zeta_\alpha(y)}\leq \norm{\mathcal L_2^{1/2}(x\phi^{1/2})}\norm{y}=\norm{\bar\delta(x\phi^{1/2})}\norm{y}.
    \end{align*}
    It remains to bound the first summand. We have
    \begin{align*}
        \mathcal L_2^{1/2}\zeta_\alpha(x\phi^{1/2}y)-x\mathcal L_2^{1/2}(\phi^{1/2}\zeta_\alpha(y))&=\alpha^{1/2}x\rho_\alpha^{(2)}(\phi^{1/2}y)-\alpha^{1/2}\rho_\alpha^{(2)}(x\phi^{1/2} y).
    \end{align*}
    By Lemma \ref{lem:bound_Stinespring} we have
    \begin{align*}
        \norm{x\rho_\alpha^{(2)}(\phi^{1/2}y)-\rho_\alpha^{(2)}(x\phi^{1/2} y)}&\leq \sqrt 2\norm{y}\norm{x}^{1/2}\norm{\rho_\alpha^{(2)}(x\phi^{1/2})-x\phi^{1/2}}^{1/2}\\
        &=\sqrt 2 \norm{y}\norm{x}^{1/2}\norm{\tilde\delta_\alpha(x\phi^{1/2})}^{1/2}\\
        &\leq \alpha^{-1/4}\sqrt 2\norm{y}\norm{x}^{1/2}\norm{\bar\delta(x\phi^{1/2})}^{1/2}.
    \end{align*}
    Altogether, we have shown that
    \begin{align*}
        &\norm{\tilde\delta_\alpha(x\phi^{1/2}y)-x\tilde\delta_\alpha(\phi^{1/2}y)}\\
        &\quad\leq (2\alpha^{-1/2}\norm{\bar\delta(x\phi^{1/2})}^{1/2}+(4+\sqrt 2)\alpha^{-1/4}\norm{x}^{1/2})\norm{\bar\delta(x\phi^{1/2})}^{1/2}\norm{y}.\qedhere
    \end{align*}
\end{proof}

\begin{lemma}\label{lem: density of analytic unitaries}
    If $\M$ is a von Neumann algebra and $\alpha\colon \mathbb R\curvearrowright \M$ is a point-weak$^\ast$ continuous action by $\ast$-isomorphisms, then the set of $\alpha$-analytic unitaries is strongly dense in $\mathcal U(\M)$.
\end{lemma}
\begin{proof}
    If $u\in \mathcal U(\M)$, then there exists $x\in \M$ self-adjoint with spectrum contained in $[0,2\pi]$ such that $u=e^{ix}$. By standard approximation arguments, there is a sequence $(x_n)$ of self-adjoint elements in $M$ that are analytic for $\alpha$, satisfy $0\leq x_n\leq 2\pi$ for all $n\in\mathbb N$ and converge to $x$ strongly. For example, one can take $x_n=\sqrt{\frac n \pi}\int_{\mathbb R}e^{-nt^2}\alpha_t(x)\,dt$. If we let $u_n=e^{ix_n}$, then $u_n$ is unitary, analytic for $\alpha$ since $\exp$ is entire and $u_n\to u$ strongly by the strong continuity of functional calculus.
\end{proof}

\section{On relative amenability of subalgebras}\label{chp:relative_amenability}
In this section, we recall relative amenability and give
several equivalent criteria. These criteria may be known
to experts, nonetheless we include them below for completeness. We first recall the basic construction and its
standard form.

\subsection{Jones basic construction}
Let $\M$ be a von Neumann algebra in the standard form $(L^2(\M), L^2(\M)_+, J)$ and let $\B$ be a unital subalgebra of $\M$. The  {\em Jones  basic construction} is the von Neumann algebra  $\langle\M,\B\rangle:=(J\B J)'\cap B(L^2(\M))$. Its elements are precisely the right $\B$-linear maps on $L^2(\M)$.
Consider $L^2(\M)$ as an $\langle\M,\B\rangle$-$\B$ correspondence via the action:
 \[x\xi y=xJy^*J\xi, \;\;\; x\in \langle\M,\B\rangle, y\in \B\]
 so that $\overline{L^2(\M)}$ becomes a $\B$ -$\langle\M,\B\rangle$ correspondence.
 The Connes fusion tensor product $L^2(\M)\otimes_\B\overline{L^2(\M)}$  becomes the standard Hilbert space for $\langle\M,\B\rangle$ where the  positive cone is generated by $\xi\otimes_\B\bar\xi$ for $\xi\in L^2(\M)$  left bounded and the conjugation operator $J_{\langle\M,\B\rangle}$ is determined by $J_{\langle\M,\B\rangle}(\xi\otimes\bar\eta)=\eta\otimes_\B\bar\xi$ for $\xi,\eta\in L^2(\M)$ (see \cite{Sau83}).

For any $T,S\in \langle\M,\B\rangle$, write $T\otimes_\B S^{\op}$ for the unique operator on $L^2(\M)\otimes_\B \overline{L^2(\M)}$ determined by $\xi\otimes_\B\bar\eta\mapsto T\xi\otimes_\B \overline{S^\ast \eta}.$ 
 Denote by $\rho:\M\otimes_{\text{alg}} \M^{\Op}\to B(L^2(\M)\otimes_\B L^2(\M))$ the left-right action of $\M$ on $L^2(\M)\otimes_\B L^2(\M)$. Denote by $\M\otimes_{\B,\min}\M^{\Op}$ and $\M\overline{\otimes}_\B\M^{\Op}$ the following algebras:
 \begin{align*}
     &\M\otimes_{\B,\min}\M^{\Op}=C^\ast(\rho(\M\otimes_{\text{alg}}\M^{\Op})),\qquad \M\overline{\otimes}_\B\M^{\Op}=W^\ast(\rho(\M\otimes_{\text{alg}}\M^{\Op})).
 \end{align*}
 We now assume that $\M$ is $\sigma$-finite and  
 there is a normal faithful conditional expectation $E_\B:\M\to\B$. Then there is  a normal faithful state  $\varphi$ on $\M$ such that $\varphi\circ E_\B=\varphi$. Let $(L^2(\B), L^2(\B)_+, J_\B)$ be the standard form for $\B$.   If we denote $\xi_\varphi\in L^2(\B)_+$ and $\phi^{1/2}\in L^2(\M)_+$   the unique vectors such that $\varphi(b)=\langle\xi_\varphi,b\xi_\phi\rangle$ for $b\in \B$ and $\phi(x)=\langle\phi^{1/2},x\phi^{1/2}\rangle$ for $x\in\M$, then the assignment $b\xi_\phi\mapsto b\varphi^{1/2}$ provides a natural inclusion from $L^2(\B)$ to $L^2(\M)$. Moreover, $e_\B:x\phi^{1/2}\mapsto E_\B(x)\phi^{1/2}$ is an orthogonal projection from $L^2(\M)$ onto $L^2(\B)$, called the {\em Jones projection}. Remark that the above inclusion and the projection $e_\B$ are independent of the choice of $\phi$.

 In this case, we have $_\B L^2(\M)_{\langle\M,\B\rangle}\cong {_\B}\overline{L^2(\M)}_{\langle\M,\B\rangle}$ via the map $\xi\mapsto\overline{J\xi}$; consequently
 \[_{\langle\M,\B\rangle} L^2(\M)\otimes_\B L^2(\M)_{\langle\M,\B\rangle}\cong {_{\langle\M,\B\rangle}}L^2(\M)\otimes_\B \overline{L^2(\M)}_{\langle\M,\B\rangle}.\]
 Thus $(L^2(\M)\otimes_\B L^2(\M), (L^2(\M)\otimes_\B L^2(\M)_+, \tilde J)$ is the standard form for $\langle\M,\B\rangle$. 
Moreover, one can  identify $\langle\M,\B\rangle^{\Op}$ with $J\langle\M,\B\rangle J$ via the map $T^{\Op}\mapsto JT^\ast J$. 
 In particular, we have
 \[x(\xi\otimes_\B \eta)y=x\xi\otimes_\B Jy^\ast J\eta=(x\otimes_\B Jy^\ast J)(\xi\otimes_\B\eta),\;\;\; x,y\in \langle\M,\B\rangle\]
 that is, $\M\otimes_{\min,\B}\M^{\op}=C^\ast\{x\otimes_\B Jy^\ast J; x,y\in \M\}$. 
Furthermore, there is a normal semifinite faithful weight  $\hat{\phi}$ on $\langle\M,\B\rangle$ such that $\hat{\phi}(x^\ast e_\B x)=\phi(x^\ast x),x\in \M$.

The following lemma describes the structure of an $\langle\M,\B\rangle$-correspondence which extends the bimodule actions of $\M$.

\begin{lemma}\label{lem:ext_bimodule_action}
    Let $\B\subseteq \M$ be an inclusion of von Neumann algebras with expectation and let $\mathsf P=\langle\M,\B\rangle$. If $\H$ is an $\M$-$\B$ correspondence, then there exists a normal $\ast$-homomorphism $\rho\colon \mathsf P\to \pi_r^\H(\B)^\prime$ that extends $\pi_l^\H$ if and only if there exists a $\B$-$\B$ correspondence $\K$ such that ${_\M\H_\B}\cong {_\M L^2(\M)}\otimes_\B \K_\B$.

    Moreover, if $\H$ is an $\M$-$\M$ correspondence, then there exist normal $\ast$-homo\-morphisms $\rho_l\colon \mathsf P\to B(\H)$, $\rho_r\colon \mathsf P^{\mathrm{op}}\to B(\H)$ with commuting images that extend $\pi_l^\H$ and $\pi_r^\H$, respectively, if and only if there exists a $\B$-$\B$ correspondence $\K$ such that $_\M \H_\M\cong {_\M L^2(\M)}\otimes_\B \K\otimes_\B {L^2(\M)_\M}$.
\end{lemma}
\begin{proof}
We prove only the second part; the first is similar. 
Fix a faithful normal state $\omega$ on $\B$. If $_\M\H_\M={_\M L^2(\M)}\otimes_\omega \K\otimes_\omega {L^2(\M)_\M}$, one can define $\rho_l(S)=S\otimes_\omega 1_\K\otimes_\omega 1_{L^2(\M)}$ and $\rho_r(T^{\mathrm{op}})=1_{L^2(\M)}\otimes_\omega 1_\K\otimes_\omega J_\M T^\ast J_\M$.

    Conversely, let $_{\mathsf P}\H_{\mathsf P}$ be the $\mathsf P$-$\mathsf P$ correspondence defined by $\rho_l$ and $\rho_r$ and let $_\B \K_\B={_\B L^2(\M)}\otimes_{\mathsf P}\H\otimes_{\mathsf P} L^2(\M)_\B$. We have
    \begin{align*}
        {_\mathsf P}L^2(\M)\otimes_\B \K\otimes_\B L^2(\M)_{\mathsf P}&\cong {_\mathsf P}L^2(\M)\otimes_\B L^2(\M)\otimes_{\mathsf P}\H\otimes_{\mathsf P}L^2(\M)\otimes_\B L^2(\M)_{\mathsf P}\\
        &\cong{_\mathsf P}L^2(\mathsf P)\otimes_{\mathsf P}\H\otimes_{\mathsf P} L^2(\mathsf P)_{\mathsf P}\\
        &\cong{_\mathsf P}\H_{\mathsf P},
    \end{align*}
    where we used the isomorphism ${_\mathsf P}L^2(\M)\otimes_\B L^2(\M)_{\mathsf P}\cong {_\mathsf P}L^2(\mathsf P)_{\mathsf P}$ from \cite[Proposition 3.1]{Sau83}.
\end{proof}

Now let $\Phi$ be a $\phi$-symmetric ucp map on $\M$, and let 
 $\B\subseteq\M$ be an inclusion of von Neumann algebras  such that ${\Phi}_{|_\B}=\id_\B$. Then  $\Phi^{(2)}$ is a $\B$-$\B$ linear map so that $\Phi^{(2)}\in\B'\cap \langle\M,\B\rangle.$ Indeed, ${\Phi}_{|_\B}=\id_\B$ means that $\B$ lies in the multiplicative domain of $\Phi$ so that  $\Phi(b_1xb_2)=b_1\Phi(x)b_2$ for all $x\in \M, b_1,b_2\in \B$. Thus the $\phi$-symmetry of $\Phi$ yields  
\begin{align*}
   \Phi^{(2)}(\phi^{1/2}xb)=\phi^{1/2}\Phi(xb)=\phi^{1/2}\Phi(xb)=\phi^{1/2}\Phi(x)b=\Phi^{(2)}(\phi^{1/2}x)b,
\end{align*}
and since $\phi^{1/2}\M$ is dense in $L^2(\M)$, it follows that $\Phi^{(2)}$ is right $\B$-linear. Similarly, density of $\M\phi^{1/2}$ in $L^2(\M)$ will prove that $\Phi^{(2)}$ is left $\B$-linear.

Thus if $(\Phi_j)_{j\in J}$ is a net of GNS symmetric ucp maps on   $\M$ such that $(\Phi_j)_{|_\B}=\id_\B$ and  $\omega$ is a free ultrafilter on the directed set $J$, then $(\Phi_j^\omega)^{(2)}$ is $\B^\omega$-bilinear and thus $(\Phi_j^\omega)^{(2)}\in \langle\M^\omega,\B^\omega\rangle$ (see Subsec. \ref{subsec:QMS} for the notations).

\subsection{Relative amenability}
Let $\M$ be a von Neumann algebra, and let $\A,\B$ be unital subalgebras of $\M$ such that $\A$ admits a normal faithful conditional expectation $E_\A$. Following Ozawa-Popa \cite{OP10},
we say  that  $\A$ is {\em  amenable relative to $\B$ inside $\M$}, and write $\A\lessdot_\M\B$, if there exists a norm-one projection
$E:\langle\M,\B\rangle\to \A$ such that $E_{|_{\M}}=E_\A.$ 

 If $\B$ admits a normal faithful semifinite operator-valued weight, then Isono \cite[Theorem 3.2]{Iso22} shows  that the above definition does not depend on the choice of the expectation $E_\A$.
 One also defines a notion of semi-discreteness of $\A$ relative to $\B$ and  the two notions are equivalent when $\A$ is semifinite \cite[Lemma A.5]{Iso22} or when $\A=\M$ \cite[Corollary A.2]{BMO20}.
 
The following Proposition and its Corollary give a characterization of relative amenability in terms of  certain states on $\langle\M,\B\rangle$. This is a non-tracial version of the first three statements of \cite[Theorem 2.1]{OP10} with $\A$-central states being replaced by states satisfying  KMS-type conditions.

\begin{proposition}\label{Prop:Equivalent conditions on relative amenability in terms of states with extra assumption on summand}
    Let $\M$ be a  von Neumann algebra equipped with a normal faithful state $\varphi$, and let $\A$ be a subalgebra of $\M$  with $\varphi$-preserving expectation $E_\A$. Denote by $\M_\infty$ the algebra of $\sigma^\varphi$-analytic vectors of $\M$. Assume that $\A'\cap\M$ has no summand that is both semifinite and properly infinite. For a   $C^\ast$-algebra $\mathcal{B}$ containing $\M$ as a unital subalgebra, the following are equivalent:
    \begin{enumerate}
     \item There exists a norm-one projection $\Phi$ from $\mathcal{B}$ onto $\A$ such that $\Phi_{|_{\M}}=E_\A$.\label{existence of a conditional expectation}
      \item There exists a state $\psi:\mathcal{B}\to\mathbb C$ and a constant $C>0$ such that $\psi_{|_{\mathcal{Z}(\A'\cap\M)}}$ is faithful and   
        \[\psi_{|_{\M}}\leq C\varphi, \;\;\;\;\psi(Tb)=\psi(\sigma_i^\varphi (b)T)\]
        for $T\in \mathcal{B}$ and $b\in \A\cap \M_\infty$. \label{existence of a state dominated by phi}
        \item There exists a state $\psi:\mathcal{B}\to\mathbb C$ and a positive element  $z\in \mathcal{Z}(\A'\cap\M)$ with full support  such that  
       \[\psi(x)=\varphi(xz),\;\;\;\;\;\psi(Tb)=\psi(\sigma_i^\varphi(b)T)\]
        for  all $x\in \M$, $T\in \mathcal{B}$ and $b\in \A\cap \M_\infty$. \label{existence of a state conjugated by a centralizer element}
     \item There exists a  state $\psi:\mathcal{B}\to \mathbb  C$ such that $\psi_{|_{\M}}=\varphi$ and satisfies 
     \[\psi(Tb)=\psi(\sigma_i^\varphi(b)T)\]
            for $T\in \mathcal{B}$ and $b\in \A\cap \M_\infty$.\label{existence of a state whose restriction is phi}
    \end{enumerate}
    \end{proposition}
\begin{proof}
    (\ref{existence of a conditional expectation})$\implies$(\ref{existence of a state dominated by phi}) Take $\psi=\varphi\circ \Phi$.
    
    (\ref{existence of a state dominated by phi})$\implies$(\ref{existence of a state conjugated by a centralizer element}).  The condition $\psi_{|_\M}\leq C \varphi$ implies that there exists $a\in \M_+$ with    $0\leq a\leq C$ such that   $\psi(x)=\langle \varphi^{1/2},x\varphi^{1/2}a\rangle$  for $x\in \M$ (\cite[Proposition 2.3]{Str20}).
    We claim that $a\in \A'\cap\M$. Indeed, if $b\in \A,y\in \M$ and $b$ is analytic, then
    \begin{align*}
    \langle \varphi^{1/2}, y\varphi^{1/2}\sigma_{i/2}^\varphi(b)a\rangle&= \langle \varphi^{1/2}, yb\varphi^{1/2}a\rangle=\psi( yb)
    =\psi(\sigma_i^\varphi(b) y)\\
    &=\langle \sigma_{-i}^\varphi(b^\ast)\varphi^{1/2}, y\varphi^{1/2}a\rangle=\langle\varphi^{1/2},y \varphi^{1/2}a\sigma_{i/2}^\varphi(b)\rangle
    \end{align*}
    and since $y\in \M$  is arbitrary, it follows that $a\sigma_{i/2}^\varphi(b)=\sigma_{i/2}^\varphi(b)a$ for all $b\in \A\cap\M_\infty$. Since $\A\cap\M_\infty$ is globally invariant under $\sigma_z^\varphi, z\in \mathbb C$ and since $\A\cap\M_\infty$ is weakly dense in $\A$, our claim follows.

Now  denote by $G$ the subgroup of $\mathcal{U}(\A'\cap\M)$ consisting of $\sigma^\varphi$-analytic unitaries. Since $\A'\cap\M$ is globally invariant under $\sigma^\varphi$, it follows from Lemma \ref{lem: density of analytic unitaries} that $G$ is $\sigma$-weak dense in $\mathcal{U}(\A'\cap\M)$.
By weak Dixmier property (\cite[Theorem 2.2]{Iso25}, \cite[Theorem 5.2]{Mar25}), there is a normal faithful conditional expectation $F:\A'\cap\M\to\mathcal{Z}(\A'\cap\M)$ such that 
\begin{align*}
 F(a)\in \overline{\text{co}}^{\text{wk}}\{uau^\ast; u\in G\},
\end{align*}
where $\overline{\text{co}}^{\text{wk}}$ denotes the weak closure of the convex hull of a set. This is where we needed the fact that $\A'\cap\M$ contains no semifinite and properly infinite summand, otherwise $F$ need not be faithful.
Set
\[z=F(a).\]
Since $a\neq0$ is positive,  the faithfulness of $F$ yields $z\neq 0$; hence $\phi(z)\neq0$. We will define a new positive functional $\tilde\psi$ on $\mathcal{B}$ which will satisfy $\tilde\psi(x)=\varphi(zx)$ for $x\in \M$ and  the remaining
conditions in the statement. By normalizing $\tilde\psi$ by $\phi(z)$, we will get the required state.

 Let $\{a_j\}_{j\in \Lambda}$ be a net in $\text{co}\{uau^\ast; u\in G\}$ such that $z=\text{wk}-\lim_j a_j$ for some directed set $\Lambda$. Write  $a_j=\sum_{u\in F_j}\lambda_u uau^\ast$ for some finite subset $F_j$ of $G$ and nonnegative scalars $\{\lambda_u; u\in F_j\}$ with $\sum_{u\in F_j}\;\lambda
_u=1.$ For each $j\in \Lambda$, consider the positive linear functional $\psi_j$ on $\mathcal{B}$ defined by
\begin{align*}
    \psi_j(T)=\sum_{u\in F_j}\lambda_u\psi(\sigma_{-i/2}^\varphi(u)^\ast T\sigma_{-i/2}^\varphi(u)),\;\;T\in \mathcal{B}.
\end{align*}
For $x\in \M$, we have
\begin{align*}
    \psi_j(x)&=\sum_{u\in F_j}\lambda_u\psi(\sigma_{-i/2}^\varphi(u)^\ast x\sigma_{-i/2}^\varphi(u))=\sum_{u\in F_j}\lambda_u\langle\varphi^{1/2}, \sigma_{-i/2}(u)^\ast x\sigma_{-i/2}^\varphi(u)\varphi^{1/2}a\rangle\\
    &=\sum_{u\in F_j}\lambda_u\langle  \sigma_{-i/2}^\varphi(u)\varphi^{1/2}, x\sigma_{-i/2}^\varphi(u)\varphi^{1/2}a\rangle=\sum_{u\in F_j}\lambda_u\langle  \varphi^{1/2}u, x\varphi^{1/2}ua\rangle\\
    &=\sum_{u\in F_j}\lambda_u\langle\varphi^{1/2}, x\varphi^{1/2}uau^\ast\rangle=\langle\varphi^{1/2}, x\varphi^{1/2}a_j\rangle.
\end{align*}
Since $\|a_j\|\leq \|a\|\leq C$, we get ${\psi_j}_{|_\M}\leq C\varphi$; in particular, $\|\psi_j\|=\psi_j(1)\leq C$. Let $\tilde\psi$ be a weak$^\ast$ limit point of the net $\{\psi_j\}_{j\in \Lambda}$, which exists because the net $\psi_j$ is uniformly bounded. Since $\{a_j\}_{j\in \Lambda}$ converges to $z$ and since $z\in \mathcal{Z}(\A'\cap\M)\subseteq\M^\varphi$,  it follows for all $x\in \M$ that
\begin{align*}
    \tilde\psi(x)=\Lim_j\psi_j(x)=\Lim_j\langle\varphi^{1/2},x\varphi^{1/2}a_j\rangle=\langle\varphi^{1/2}, x\varphi^{1/2}z\rangle
    =\phi(xz)=\varphi(z^{1/2}xz^{1/2}).
\end{align*}
Moreover, for all $T\in \mathcal{B}$ and $b\in \A\cap\M_\infty$, since $\sigma_{i/2}^\varphi(u)$ commutes with $b$ for all $u\in G$, we obtain
\begin{align*}
  \psi(\sigma_{-i/2}^\varphi(u)^\ast Tb\sigma_{-i/2}^\varphi(u))&=\psi(\sigma_{-i/2}^\varphi(u)^\ast T\sigma_{-i/2}^\varphi(u)b)\\
  &=\psi(\sigma_i^\varphi(b)\sigma_{-i/2}^\varphi(u)^\ast T\sigma_{-i/2}^\varphi(u))\\
  &=\psi(\sigma_{-i/2}^\varphi(u)^\ast \sigma_i^\varphi(b)T\sigma_{-i/2}^\varphi(u))
\end{align*}
so that $\psi_j(Tb)=\psi_j(\sigma_i^\varphi(b)T)$ and thus $\tilde\psi(Tb)=\tilde\psi(\sigma_i^\varphi(b)T)$. Finally, if $p$ is the kernel projection of $z$ then $p\in \mathcal{Z}(\A'\cap\M)\subseteq\M^\varphi$ and we have
\begin{align*}
    \psi(p)=\langle\varphi^{1/2}, p\varphi^{1/2}a\rangle=\langle\varphi^{1/2}, \varphi^{1/2}pa\rangle
    =\phi(pa).
\end{align*}
However, by bimodularity of $F$, we have $F(pap)=pF(a)p=pzp=0$. Since $pap\geq0$, by faithfulness of $F$, we have $pap=0$; hence $pa=0$. Thus we get $\psi(p)=0$,
and since $\psi$ is faithful on $\mathcal{Z}(\A'\cap\M)$, it follows that $p=0$  proving that $z$ has full support. 

(\ref{existence of a state conjugated by a centralizer element})$\implies$(\ref{existence of a state whose restriction is phi}) Let $\psi$ and $z$ be as in the statement of (\ref{existence of a state conjugated by a centralizer element}). Set $z_\epsilon=\chi_{(\epsilon,\infty)}z^{-1/2}$ and define the positive functional $\psi_\epsilon$ on $\mathcal{B}$ by $\psi_\epsilon(T)=\psi(z_\epsilon Tz_\epsilon)$, $T\in\mathcal{B}$. Since
\[\psi_\epsilon(1)=\psi(z_\epsilon^{2})=\varphi(z_\epsilon^2z)=\varphi(\chi_{(\epsilon,\infty)}(z))\leq \varphi(1)\] it follows that $\psi_\epsilon$ is a uniformly bounded net. Let $\tilde\psi=\Lim_{\epsilon\to 0}\psi_\epsilon$ be a weak$^\ast$ limit of $\psi_\epsilon$. Since $z$ has full support, $\chi_{(\epsilon,\infty)}(z)\to 1$ $\sigma$-weakly as $\epsilon\to0$ which further yields 
\begin{align*}
    \tilde\psi(x)=\Lim_{\epsilon\to0}\varphi(z_\epsilon xz_\epsilon z)=\Lim_{\epsilon\to0}\varphi(\chi_{(\epsilon,\infty)}(z)x)=\varphi(x),\;\;\forall x\in \M.
\end{align*}Moreover, if $b\in\A\cap\M_\infty$  and $T\in \mathcal{B}$, then since $z$ commutes with $\A$, it follows that
    \begin{align*}
        \psi_\epsilon(Tb)=\psi(z_\epsilon Tbz_\epsilon)
        =\psi(z_\epsilon Tz_\epsilon b)=\psi(\sigma^\phi_i(b)z_\epsilon Tz_\epsilon)=\psi(z_\epsilon\sigma^\phi_i(b)Tz_\epsilon)
        =\psi_\epsilon(\sigma^\phi_i(b)T),
    \end{align*}
    and this property is clearly preserved in the weak$^\ast$ limit.

$(\ref{existence of a state whose restriction is phi})\implies(\ref{existence of a conditional expectation})$ 
Denote by $\xi_\varphi$ the cyclic and separating  vector in $L^2(\A,\varphi_{|_\A})$ such that $\varphi(b)=\langle\xi_\varphi,b\xi_\varphi\rangle$ for $b\in \A$.
For every $T\in \mathcal{B}$, consider the sesquilinear form on $\A\xi_\varphi\times \A\xi_\varphi$ given by
    \[(x\xi_\varphi, y\xi_\varphi)\mapsto \psi(x^*Ty),\;\;\;x,y\in \A.\]
    This form is bounded because
    \begin{align*}
        |\psi(x^*Ty)|^2\leq \|T\|^2\psi(x^*x)\psi(y^*y)= \|T\|^2\varphi(x^*x)\varphi(y^*y)=\|T\|^2\|x\xi_\varphi\|^2\|y\xi_\varphi\|^2.
    \end{align*}
    Hence there exists an operator $\Phi(T)\in B(L^2(\A,\varphi))$ with $\norm{\Phi(T)}\leq \norm{T}$ such that 
    \[\langle x\xi_\varphi, \Phi(T)y\xi_\varphi\rangle=\psi(x^*Ty),\;\;\;x,y\in \A.\]
    We identify $L^2(\A,\varphi_{|_\A})$ as a subspace of $L^2(\M,\varphi)$ via the inclusion $b\xi_\varphi\mapsto b\varphi^{1/2}$ and write $J_\varphi$ for the modular conjugation operator  in  $L^2(\A,\varphi_{|_\A})$.
   To show that $\Phi(T)\in \A$, one calculates the following  for all $x,y,z\in \A$ with $y\in\A\cap\M_\infty$:
    \begin{align*}
        \langle x\xi_\varphi, \Phi(T)J_\varphi yJ_\varphi z\xi_\varphi\rangle&=\langle x\xi_\varphi, \Phi(T)z\sigma_{-i/2}^\varphi(y^*)\xi_\varphi\rangle\\
        &=\psi(x^*Tz\sigma_{-i/2}^\varphi(y^*))\\
        &=\psi(\sigma_{i/2}^\varphi(y^*)x^*Tz)\;\;\;\;(\text{by hypothesis})\\
        &=\langle x\sigma_{-i/2}^\varphi(y)\xi_\varphi, \Phi(T)z\xi_\varphi\rangle\\
        &=\langle J_\varphi y^*J_\varphi x\xi_\varphi, \Phi(T)z\xi_\varphi\rangle\\
        &=\langle x\xi_\varphi, J_\varphi yJ_\varphi \Phi(T)z\xi_\varphi\rangle
    \end{align*}
    which shows that $\Phi(T)J_\varphi yJ_\varphi=J_\varphi yJ_\varphi \Phi(T)$. Since $\A\cap\M_\infty$ is weakly dense in $
\A$, it follows that $\Phi(T)\in (J_\varphi\A J_\varphi)'\cap B(L^2(\A,\varphi_{|_\A}))= \A$. For $a\in \M$, we moreover have
    \begin{align*}
        \langle x\xi_\varphi,\Phi(a)z\xi_\varphi\rangle=\psi(x^*az)=\varphi(x^\ast a z)=\langle x\varphi^{1/2},az\varphi^{1/2}\rangle=\langle x\varphi^{1/2}, E_\A(a)z\varphi^{1/2}\rangle
        =\langle x\xi_\varphi, E_\A(a)z\xi_\varphi\rangle
    \end{align*}
    proving that $\Phi(a)=E_\A(a)$. Therefore, $\Phi\colon \mathcal{B}\to \A$ is a norm-one projection such that $\Phi_{|_\M}=E_\A$.
\end{proof}

  We can remove in the previous Proposition  the condition that $\A'\cap\M$ contains no semifinite properly infinite summand in the following result.

\begin{theorem}\label{thm:conditions for relative amenability}
  Let $\M$ be a  von Neumann algebra equipped with a normal faithful state $\varphi$, and let $\A$ be a subalgebra of $\M$  with $\varphi$-preserving expectation $E_\A$. Denote by $\M_\infty$ the algebra of $\sigma^\varphi$-analytic vectors of $\M$.  For a   $C^\ast$-algebra $\mathcal{B}$ containing $\M$ as a unital subalgebra, the following are equivalent:
    \begin{enumerate}
     \item There exists a norm-one projection $\Phi$ from $\mathcal{B}$ onto $\A$ such that $\Phi_{|_{\M}}=E_\A$.\label{cor:existence of a conditional expectation1}
      \item There exists a state $\psi:\mathcal{B}\to\mathbb C$ and a constant $C>0$ such that $\psi_{|_{\mathcal{Z}(\A'\cap\M)}}$ is faithful and   
        \[\psi_{|_{\M}}\leq C\varphi, \;\;\;\;\psi(Tb)=\psi(\sigma_i^\varphi (b)T)\]
        for $T\in \mathcal{B}$ and $b\in \A\cap \M_\infty$. \label{cor:existence of a state dominated by phi}
     \item There exists a  state $\psi:\mathcal{B}\to \mathbb  C$ such that $\psi_{|_{\M}}=\varphi$ and satisfies 
     \[\psi(Tb)=\psi(\sigma_i^\varphi(b)T)\]
            for $T\in \mathcal{B}$ and $b\in \A\cap \M_\infty$.\label{cor:existence of a state whose restriction is phi}
    \end{enumerate}
\end{theorem}
\begin{proof}
    We only need to prove (\ref{cor:existence of a state dominated by phi})$\implies$(\ref{cor:existence of a state whose restriction is phi}) because the other implications  follow the same procedure as in Proposition \ref{Prop:Equivalent conditions on relative amenability in terms of states with extra assumption on summand}. 
    Let $\psi:\mathcal{B}\to\mathbb C$ be a given state on $\mathcal{B}$ satisfying the hypotheses of (\ref{existence of a state dominated by phi}).
    
    Let $q\in \mathcal{Z}(\A'\cap\M)$ be  the maximal projection such that $q(\A'\cap\M)$ is semifinite and properly infinite so that $(1-q)(\A'\cap\M)$ contains no semifinite properly infinite summand. If $q=0$,  the conclusion follows from Proposition \ref{Prop:Equivalent conditions on relative amenability in terms of states with extra assumption on summand}; so assume that $q\neq 0$.  We emphasize that $q\in \M^\varphi$ and we will repeatedly use the fact $\sigma_z^\varphi(q)=q$ for $z\in\mathbb C$.  Denote by \[\varphi_q:=\varphi(q\cdot q)/\varphi(q)\] the state on $q\M q$. 
By Lemma \ref{lem:existence of certain finite projections} below, there is a sequence of  projections in  the centralizer $q_n\in q(\A'\cap\M)^{\phi_{q}}=q(\A'\cap\M^\phi)$ such that $q_n\to q$ $\sigma$-strongly, each $q_nq(\A'\cap\M)q_n$ is a finite von Neumann algebra and $\psi(q_n\cdot q_n)_{|_{q_n\mathcal{Z}(\A'\cap\M)}}$ is faithful.

Consider the  states $\psi_n$ and $\varphi_n$ on the compression algebras $q_n\mathcal{B} q_n$ and $q_n\M q_n$ respectively given  by  \[\psi_n(\cdot)=\psi(q_n\cdot q_n)/\psi(q_n) \qquad\mbox{ and }\qquad \varphi_n:=\varphi(q_n\cdot q_n)/\varphi(q_n).\]
  By virtue of $q_n$ being in  $\M^\varphi$ and commuting with $\A$,   each $\psi_n$ satisfies 
   \[{\psi_n}_{|_{q_n\M q_n}}\leq C\frac{\phi(q_n)}{\psi(q_n)}\varphi_{n}\qquad\mbox{and}\qquad\psi_n(Tb)=\psi_n(\sigma_i^{\varphi_n}(b)T)\] for all $T\in q_n\mathcal{B} q_n$ and $\sigma^{\varphi_n}$-analytic $b\in q_n\A$. Since each
   $q_n(\A'\cap\M)q_n=(q_n\A)'\cap q_n\M q_n$ is finite and  ${\psi_n}_{|_{q_n\mathcal{Z}(\A'\cap\M)}}$ is faithful, Proposition \ref{Prop:Equivalent conditions on relative amenability in terms of states with extra assumption on summand} yields a new state, say $\zeta_n$ on $q_n\mathcal{B}  q_n$, such
   that
   \[{{\zeta}_n}_{|_{q_n\M q_n}}=\varphi_n\qquad\mbox{and}\qquad \zeta_n(Tb)=\zeta_n(\sigma_i^{\varphi_n}(b)T)\]
   for all $T\in q_n\mathcal{B} q_n$ and $\sigma^{\varphi_n}$-analytic $b\in q_n\A$.
   Let $\omega$ be a free ultrafilter on $\mathbb N$ and define the state $\zeta:q\mathcal{B} q\to\mathbb C$ by
   \[\zeta(T)=\lim_{n\to\omega}\zeta_n(q_n Tq_n),\;\;T\in q\mathcal{B} q.\]
Since $q_n\to q$ $\sigma$-strongly,  \[\zeta(x)=\lim_{n\to\omega}\varphi_n(q_n xq_n)=\varphi(qxq)/\varphi(q)=\varphi_q(x)\] for all $x\in q\M q.$ Since  $q_n\in q(\A'\cap\M^\varphi)$, we note that if   $b\in q\A q$ is $\sigma^{\varphi_q}$-analytic,  each $q_nbq_n$ is $\sigma^{\varphi_n}$ - analytic such that $\sigma_i^{\varphi_n}(q_n bq_n)=q_n\sigma_i^{\varphi_q}(b)q_n$, and since $q_n$ commutes with $\A$, the KMS condition of each $\zeta_n$  yields the following
\begin{align*}
    \zeta(Tb)=\zeta(\sigma_i^{\varphi_q}(b)T),\;\;\;\forall  T\in q\mathcal{B} q.
\end{align*}
If $q=1$, we are done. Otherwise, let $p:=1-q\neq 0$ and consider the state $\varphi_p:=\varphi(p\cdot p)/\varphi(p)$ on $p\M p$. Since $p\A'\cap p\M p=p(\A'\cap\M)$ contains no semifinite properly infinite summand by the maximality of $q$, we are now again in the position to invoke Proposition \ref{Prop:Equivalent conditions on relative amenability in terms of states with extra assumption on summand} for the unital inclusion $p\M p\subseteq p\mathcal{B} p$ to obtain a state $\eta$ on $p\mathcal{B} p$  satisfying 
\[\eta_{|_{p\M p}}=\varphi_{p}\qquad\mbox{and}\qquad\eta(Tb)=\eta(\sigma_i^{\varphi_p}(b)T)\] for all $T\in p\mathcal{B} p$ and $\sigma^{\varphi_p}$ - analytic $b\in p\A p$. This finishes the proof because the state $\tilde\psi:\mathcal{B}\to\mathbb C$ by
\[\tilde\psi(T)=\varphi(q)\zeta(qT q)+\varphi(p)\eta(pTp), \;\;T\in \mathcal{B}\] satisfies the required conditions.
\end{proof}

\begin{lemma}\label{lem:existence of certain finite projections}
  Let $\phi, \psi$ be normal positive functionals on a $\sigma$-finite semifinite von Neumann algebra $\N$ satisfying $\psi\leq \phi$ and such that $\phi$ and the restriction $\psi_{|_{\mathcal{Z}(\N)}}$ are faithful. There exists a sequence $\{q_n\}_{n\geq1}$ of projections in the centralizer $\N^\phi$ such that $q_n\to 1$ $\sigma$-strongly,  $q_n\N q_n$ is finite and $\psi(q_n\cdot q_n)$ is faithful on $\mathcal{Z}(q_n\N q_n)$ for each $n\geq 1$.
  \end{lemma}
  \begin{proof}
   Let $\tau$ be a normal faithful semifinite weight on $\N$, and let $a$ be the positive operator affiliated to $\N$ such that $\phi(x)=\tau(ax),\;\;x\in \N.$   
   For each $n\geq 1$, set $e_n={1}_{(1/n,\infty)}(a)\in \N$ and define the positive functionals $\mu_n$ on $\mathcal{Z}(\N)$ by 
   \begin{align*}
       \mu_n(z)=\psi(ze_n),\;\;z\in\mathcal{Z}(\N).
   \end{align*}
   Since $\phi$ is faithful, $a$ has full support; hence $e_n\to 1$ $\sigma$-strongly as $n\to\infty$. Moreover, $\psi(e_n)\uparrow\psi(1)$, so $\psi(e_n)\neq0$ for large enough $n$. We will therefore assume that $\psi(e_n)\neq0$ for all $n\geq1$, so that each $\mu_n$ is a non-zero functional. 
   Let $f_n$ be the support of $\mu_n$ in $\mathcal{Z}(\N)$ and set 
   \[q_n=e_nf_n.\]
   We show that $q_n$ satisfies the required assertion. Clearly, $e_n\in \N^\phi$ as $e_na=ae_n$, and since $f_n\in \mathcal{Z}(\N)\subseteq\N^\phi$, it follows that $q_n\in\N^\phi$.
   
  Write $b_n=g_n(a)$, where $g_n$ is the function defined by
  $g_n(t)=\frac{1}{t}1_{(1/n,\infty)}(t)$ for $ t\geq 0$.
  Then $\tau(e_n)=\tau(ab_n)=\psi(b_n)<\infty$, which means each $e_n$ is a finite projection. Hence each $q_n$ is a finite projection so that $q_n\N q_n$ is a finite von Neumann algebra. 
  
  Clearly,  $\psi(q_n\cdot q_n)_{|_{\mathcal{Z}(q_n\N q_n)}}$ is faithful. Indeed, we have $\mathcal{Z}(q_n\N q_n)=q_n\mathcal{Z}(\N)$ and  if  $z\in\mathcal{Z}(\N)$ is a positive element such that $\psi(q_nz)=0$, then
  \begin{align*}
      \mu_n(z)=\mu_n(zf_n)=\psi(zf_ne_n)=\psi(zq_n)=0,
  \end{align*}
  showing that $zf_n=0$ as $f_n$ is the support of $\mu_n$; hence $zq_n=0$. 

 Finally we show that $f_n\to1$ $\sigma$-strongly. Since $e_n\leq e_{n+1}$, we have $\mu_n\leq \mu_{n+1}$ as positive functionals. This implies $f_n\leq f_{n+1}$ as well as $q_n\leq q_{n+1}$. Let $f_n\to f$ $\sigma$-strongly which exists by the monotonicity of $f_n$. Using $\psi(e_nf_n)=\mu_n(f_n)=\mu_n(1)=\psi(e_n)$ and the inequality $f_n\geq e_nf_n$, we obtain
 \begin{align*}
     \psi(1-f)=\psi(1)-\lim_{n\to\infty}\psi(f_n)\leq \psi(1)-\lim_{n\to\infty}\psi(e_nf_n)    =\psi(1)-\lim_{n\to\infty}\psi(e_n)=0.
 \end{align*}
 Since $f\in \mathcal{Z}(\N)$ and $\psi_{|_{\mathcal{Z}(\N)}}$ is faithful, we infer $f=1$. Also since $e_n\uparrow1$ $\sigma$-strongly,  we conclude that $q_n=e_nf_n\uparrow1$ $\sigma$-strongly. 
  \end{proof}

\begin{remark}
    If we take $\mathcal{B}=\langle\M,\B\rangle$ in Theorem \ref{thm:conditions for relative amenability},  we get the equivalent conditions of relative amenability in terms of certain states satisfying KMS type conditions.

Note that $\B$ is amenable if and only if the basic construction $\langle \M,\B\rangle$ is amenable. Indeed, amenability of $\B$ is equivalent to amenability of $J\B J$, and hence to amenability of its commutant
$(J\B J)'=\langle \M,\B\rangle.
$

Now suppose  that $\B$ is amenable. Then there exists a conditional expectation
\[
F:\mathbb B(L^2(\M))\longrightarrow \langle \M,\B\rangle.
\]If $\A$ is amenable relative to $\B$ inside $\M$, there is a conditional expectation
$E:\langle \M,\B\rangle\to \A$
whose restriction to $\M$ is $E_{\A}$. Consequently,
$E\circ F:\mathbb B(L^2(\M))\longrightarrow \A $ is a conditional expectation, showing that $\A$ is amenable.
Conversely,  if $\A$ is amenable, then by the injectivity of $\A$, the conditional expectation $E_{\A}:\M\to\A$ extends to a conditional expectation
$\widetilde E:\langle \M,\B\rangle\longrightarrow\A$, which shows that $\A$ is amenable relative to $\B$ inside $\M$.
\end{remark}

As a corollary of Theorem \ref{thm:conditions for relative amenability}, we obtain the following characterization of injectivity of $\sigma$-finite von Neumann algebras, which is a generalization of the characterization for finite von Neumann algebras in terms of hypertraces (see \cite[Theorem 5.1]{Con76}).

\begin{corollary}
    Let $\M\subseteq \mathbb B(\H)$ be a $\sigma$-finite von Neumann algebra with faithful normal state $\phi$. The von Neumann algebra $\M$ is amenable if and only if there exists a state $\psi\colon \mathbb B(\H)\to\mathbb C$ such that $\psi|_\M=\phi$ and
    \begin{equation*}
        \psi(Tx)=\psi(\sigma^\phi_i(x)T)
    \end{equation*}
    for $T\in\mathbb B(\H)$ and $x\in \M_{\infty,\sigma^\phi}$.
\end{corollary}

\begin{remark}
The same argument shows the following result: Let $(A,\alpha)$ be a $C^\ast$-dynamical system. For a $\beta$-KMS state $\phi$ of $\alpha$ ($\beta\in \mathbb R$), the following properties are equivalent:
\begin{enumerate}[(i)]
    \item For every faithful representation $\pi:A\to  B(H)$, there exists a state $\psi\colon B(H)\to \mathbb C$ such that $\psi\circ\pi=\phi$ and
    \begin{equation*}
        \psi(aT)=\psi(T\alpha_{i\beta}(a))
    \end{equation*}
    for all $T\in B(H)$ and $\alpha$-analytic $a\in A$.
    \item For every faithful representation $A\subseteq B(H)$, there exists a ucp map $\Phi\colon B(H)\to \pi_\phi(A)^{\prime\prime}$ such that $\Phi|_A=\pi_\phi$.
\end{enumerate}
\end{remark}


\section{On the definition of coarse rigidity}\label{sec:definition of coarse regifity}
In \cite{Pet09a}, Peterson introduced the notion of $L^2$-rigidity for tracial von Neumann algebras. The convergence condition underlying this notion, formulated here for derivations on $(\M,\tau)$ itself, requires that for every densely defined, closable, $\tau$-symmetric real derivation
$\delta:\dom\delta\subseteq L^2(\M,\tau)
\longrightarrow
\bigl(L^2(\M,\tau)\otimes L^2(\M,\tau)\bigr)^{\oplus\infty},$
the associated resolvents $\theta_\alpha^{(2)}$ converge to the identity uniformly on ${x\tau^{1/2}\in\M,\ |x|\leq1}$ as $\alpha\to\infty$.

In this section, we characterize this uniform convergence through strong convergence of the induced operators on the $L^2$-space of the ultrapower. This formulation provides a natural framework for extending the rigidity condition to the nontracial setting, where modular theory plays an essential role. We then introduce \emph{coarse rigidity}, using deformations associated with modular derivations and allowing their target correspondences to be weakly contained in the coarse correspondence.

We begin by collecting some facts about GNS-symmetric unital completely positive maps.

\begin{proposition}
 \label{prop: uniform convergence of GNS-symmetric maps}
    Let $(\Psi_i)$ be a net of GNS-symmetric ucp maps on a $\sigma$-finite von Neumann algebra $\M$ and let $\omega$ be a cofinal ultrafilter on a directed set. For every inclusion $\Q\subseteq\M$ and every  normal faithful state $\varphi$ on $\M$,  the following are equivalent:
    \begin{enumerate}
        \item $\sup_{x\in\Q_1}{\|\Psi_i^{(2)}(x\varphi^{1/2})-x\varphi^{1/2}}\|\to^i 0$,
        \item  $\sup_{x\in(\Q^\omega)_1}{\|(\Psi_i^{(2)})^\omega(x\psi^{1/2})-x\psi^{1/2}}\|\to^i 0$, where $\psi=\varphi^\omega$.
    \end{enumerate}
\end{proposition}
    \begin{proof}
        Since $\Q\subseteq\Q^\omega$ and $(x\phi^{1/2})^\omega=x\psi^{1/2}$ for $x\in \Q$, $(2)\implies(1)$ is obvious. Conversely,  if $a\in \Q^\omega$ with $\|a\|_\infty\leq 1$, then there exists $(x_n)\in\mathcal{M}^\omega(\Q)$  such that $a=(x_n)_n^\omega $ and $\sup_n\norm{x_n}\leq 2$.
So we have
    \begin{align*}
        \norm{(\Psi_i^{(2)})^\omega(a\psi^{1/2})-a\psi^{1/2}}_2&=\lim_{n\to\omega}\norm{\Psi_i^{(2)}(x_n\phi^{1/2})-x_n\phi^{1/2}}\\
        &\leq \sup_{y\in\Q, \norm{y}\leq 2}\norm{\Psi_i^{(2)}(y\phi^{1/2})-y\phi^{1/2}}.
    \end{align*}
    The right side is independent of $a\in (\Q^\omega)_1$ and converges to zero along $i$.
    \end{proof}

We record the following simple fact, although we will not need it in this paper. See \cite{AH14} for the notation $\Q(\sigma^\phi, [-c, c])$.
\begin{proposition}
    Let $(\Psi_i)$ be a net of GNS-symmetric ucp maps on a $\sigma$-finite von Neumann algebra $\M$ and let $\omega\in\beta\mathbb N\setminus\mathbb N$. For every inclusion $\Q\subseteq\M$ and a normal faithful state $\varphi$ on $\Q$ such that $\Q$ is globally invariant under $\sigma^\varphi$,  the following are equivalent:
    \begin{enumerate}
        \item $\sup\{\|\Psi_i^{(2)}(x\varphi^{1/2})-x\varphi^{1/2}\|; x\in \Q(\sigma^\phi, [-c,c]), \|x\|\leq 1\}\to^i 0$ for every $c>0$,
        \item  $\sup\{\|(\Psi_i^{(2)})^\omega(y\varphi^{1/2})-y\psi^{1/2}\|; y\in \Q^\omega(\sigma^\psi, [-c,c]),\|x\|\leq 1\}\to^i 0$ for every $c>0$, where $\psi=\varphi^\omega$.
    \end{enumerate}
\end{proposition}
    \begin{proof}
        Since $\Q_{\infty,\sigma^\phi}\subseteq\Q^\omega_{\infty,\sigma^\psi}$, $(2)\implies(1)$ is obvious.  For the converse, let $y\in \Q^\omega(\sigma^\psi, [-c,c])$, then by \cite[Proposition 2.3]{HI17} there is a sequence $y_n\in \Q(\sigma^\varphi, [-4c, 4c])$ such that $(y_n)\in \mathcal{M}^\omega(\Q)$ and $y=(y_n)^\omega$. Therefore, 
        \begin{align*}
            \|(\Psi_i^{(2)})^\omega(y\psi^{1/2})-y\psi^{1/2}\|&=\lim_{n\to\omega}\|\Psi_i^{(2)}(y_n\phi^{1/2})-y_n\varphi^{y_n\varphi^{1/2}}\|\\
            &\leq \sup\{\|\Psi_i^{(2)}(z\varphi^{1/2})-z\psi^{1/2}\|; z\in \Q^\omega(\sigma^\psi, [-4c,4c])\}
        \end{align*}
        and the right hand side is independent of $y$ and goes to $0$ along $i$.
    \end{proof}

\begin{lemma}\label{lem:uniform convergence gives ultrapower pointwise convergence}
    Let $\omega$ be a cofinal ultrafilter on a directed set $\Lambda$, let $(\Psi_k)$ be a net of GNS-symmetric ucp maps on $\M$, let $\phi$ be a normal faithful state on $\M$ and let $\Q$ be a subalgebra of $\M$ globally invariant under $\sigma^\phi$.  If
    \begin{equation*}
        \lim_{k}\sup\{\norm{\Psi_k^{(2)}(a\phi^{1/2})-a\phi^{1/2}}: a\in \dom(\sigma_{\pm i/2}^\phi)\cap\Q, \|\sigma_{\pm i/2}^\phi(a)\|\leq 1\}= 0,
    \end{equation*}
    then $(\Psi_k^{(2)})^\omega$ converges strongly on $L^2(\Q^\omega)$.
\end{lemma}
\begin{proof}
 Write $\psi=\varphi^\omega$ and note that $\Q^\omega$ is globally invariant under $\sigma^\psi$.  Fix
        $x=(x_n)\in\Q^\omega$. Since $\Q^\omega\psi^{1/2}$ is norm dense in $L^2(\Q^\omega)$, it suffices to show that $(\Psi_k^{(2)})^\omega(x\psi^{1/2})$ converges to $x\psi^{1/2}$. For every $\epsilon>0$, there exist $y=(y_n)^\omega\in \Q^\omega$ such that each $y_n$ is $\sigma^\varphi$-analytic, $y$ is $\sigma^\psi$-analytic, $\|x\psi^{1/2}-y\psi^{1/2}\|\leq \epsilon$ and $\sup_n\|\sigma^\varphi_{\pm i/2}(y_n)\|_\infty<\infty$. Thus we have
    \begin{align*}
       \lim_{k}\|(\Psi_k^{(2)})^\omega(y\psi^{1/2})-y\psi^{1/2}\|&=\lim_{k}\lim_{n\to\omega}\|\Psi_k^{(2)}(y_n\varphi^{1/2})-y_n\varphi^{1/2}\|\\
       &\leq \lim_{k}\sup_n\|\Psi_k^{(2)}(y_n\varphi^{1/2})-y_n\varphi^{1/2}\|\\
       &=0 
    \end{align*}
where the last equality is by the hypothesis.  By the contractiveness of $\Psi_k^{(2)}$, this implies that 
   \begin{align*}
            \lim_{k} \|(\Psi_k^{(2)})^\omega(x\psi^{1/2})-x\psi^{1/2}\|\leq 2\epsilon+ \lim_{k}\|(\Psi_k^{(2)})^\omega(y\psi^{1/2})-y\psi^{1/2}\|=2\epsilon
   \end{align*}
   and we get our required assertion as $\epsilon>0$ is arbitrary.

\end{proof}

\begin{proposition}\label{prop:equivalent conditions for L2-rigidity}
 Let $\delta:\dom\delta\subseteq L^2(\M)\to\H$ be a modular derivation, and let $\theta_\alpha$ be the associated deformation. For an inclusion $\Q\subseteq\M$ with expectation, the following are equivalent:
    \begin{enumerate}
    \item For any normal faithful state $\varphi$ on $\M$ such that $\Q$ is globally invariant under $\sigma^\varphi$, we have
    \begin{align*}
        \lim_{\alpha\to\infty}\|\theta_\alpha^{(2)}(x\varphi^{1/2})-x\varphi^{1/2}\|; x\in\dom(\sigma_{\pm i/2}^\varphi)\cap\Q, \|\sigma^\varphi_{\pm i/2}(x)\|_\infty\leq 1\}=0.
    \end{align*}\label{st:for any faithful state}
    \item For some normal faithful state $\varphi$ on $\M$ such that $\Q$ is globally invariant under $\sigma^\varphi$, we have
    \begin{align*}
        \lim_{\alpha\to\infty}\|\theta_\alpha^{(2)}(x\varphi^{1/2})-x\varphi^{1/2}\|; x\in\dom(\sigma_{\pm i/2}^\varphi)\cap\Q, \|\sigma^\varphi_{\pm i/2}(x)\|_\infty\leq 1\}=0.
    \end{align*}\label{st:for some faithful state}
        \item $(\theta_\alpha^{(2)})^\omega$ converges strongly on the set $L^2(\Q^\omega)$ as $\alpha\to\infty$, for every cofinal ultrafilter $\omega$ on a directed set.\label{st: for any ultrafilter}
        \item $(\theta_\alpha^{(2)})^\omega$ converges strongly on the set $L^2(\Q^\omega)$ as $\alpha\to\infty$, for some $\omega\in \beta \mathbb N\setminus \mathbb N$.\label{st:for some ultrafilter}
    \end{enumerate}
    \end{proposition}
    \begin{proof}
        $(\ref{st:for any faithful state})\implies(\ref{st:for some faithful state})$ is obvious.
        
        $(\ref{st:for some faithful state})\implies(\ref{st: for any ultrafilter})$ follows from Lemma \ref{lem:uniform convergence gives ultrapower pointwise convergence}.
        
   $(\ref{st: for any ultrafilter})\implies(\ref{st:for some ultrafilter})$ is obvious.

  $(\ref{st:for some ultrafilter})\implies(\ref{st:for any faithful state})$  Suppose that $(\ref{st:for any faithful state})$ fails for some normal faithful state $\varphi$ on $\M$ for which $\Q$ is globally invariant under $\sigma^\varphi$, so  there exist $\delta>0$,  a sequence  $(\alpha_n)$ in $\mathbb N$ and a sequence $x_{n}\in\dom(\sigma^\phi_{i/2})\cap \dom(\sigma^\phi_{-i/2})\cap\Q$ with $\norm{\sigma^\phi_{\pm i/2}(x_n)}\leq 1$  such that $\|\theta_{\alpha_n}^{(2)}(x_{n}\varphi^{1/2})-x_n\varphi^{1/2}\|\geq \delta$.
  
We show that $(x_n)\in\mathcal{M}^\omega(\Q)$. First note that by the maximum principle for analytic functions, $\norm{x_n}\leq \max\{\norm{\sigma^\phi_{-i/2}(x_n)},\norm{\sigma^\phi_{i/2}(x_n)}\}\leq 1$. Now, if $(y_n)\in\Q$ is a net such that $\|y_n\|_\varphi^\sharp\to0$ as $n\to\omega$, then we have
    \begin{align*}
        \phi(y_n^\ast x_n^\ast x_n y_n+x_n y_n y_n^\ast x_n^\ast)
        &\leq \norm{x_n}^2\phi(y_n^\ast y_n)+\norm{\sigma^\phi_{-i/2}(x_n)}^2\phi(y_n y_n^\ast)\\
        &\leq \phi(y_n^\ast y_n+y_n y_n^\ast)\\
        &\to 0
    \end{align*}
    and
    \begin{align*}
        \phi(x_n^\ast y_n^\ast y_n x_n+y_n x_n x_n^\ast y_n^\ast)
        &\leq \norm{\sigma^\phi_{i/2}(x_n)}^2\phi(y_n^\ast y_n)+\norm{x_n}^2\phi(y_n y_n^\ast)\\
        &\leq \phi(y_n^\ast y_n+y_n y_n^\ast)\\
        &\to 0.
    \end{align*}
  Set $x=(x_n)^\omega\in\Q^\omega$. By assumption, $(\theta_{\alpha}^{(2)})^\omega(x(\varphi^\omega)^{1/2})$  converges to $x(\varphi^\omega)^{1/2}$ as $\alpha\to\infty$, so there is some $\alpha>0$ such that
    \begin{equation*}
\lim_{n\to\omega}\|{\theta_{\alpha}^{(2)}(x_n\varphi^{1/2})-x_n\varphi^{1/2}}\|=  \|(\theta_{\alpha}^{(2)})^\omega  (x(\varphi^\omega)^{1/2})-x(\varphi^\omega)^{1/2}\|\leq  \delta/2.
    \end{equation*}
However, for $\alpha_n\geq\alpha$, we have
\begin{align*}
  \|\theta^{(2)}_{\alpha}(x_n\varphi^{1/2})-x_n\varphi^{1/2}\|\geq  \|\theta^{(2)}_{\alpha_n}(x_n\varphi^{1/2})-x_n\varphi^{1/2}\|\geq\delta
\end{align*}
leading to a contradiction, where the last inequality follows because $(\theta_\beta^{(2)}-id)^2$ is decreasing; that is,  for any
$\beta\geq\alpha$ and $\xi\in L^2(\M)$, if $E$ denotes the spectral projection of $\mathcal{L}_2$, then
    \begin{align*}
       \|\theta_{\alpha}^{(2)}(\xi)-\xi\|^2
        &=\int_{[0,\infty)}\left\lvert1-\frac{\alpha}{\lambda+\alpha}\right\rvert^2\,d\norm{E(\lambda)\xi}^2\\
        &=\int_{[0,\infty)}\frac{\lambda^2}{(\lambda+\alpha)^2}\,d\norm{E(\lambda)\xi}^2\\
        &\geq \int_{[0,\infty)}\frac{\lambda^2}{(\lambda+\beta)^2}\,d\norm{E(\lambda)\xi}^2\\
        &=\|\theta_\beta^{(2)}(\xi)-\xi\|^2.
    \end{align*}    
This completes the proof. 
    \end{proof}

The following is the proposed definition of relative coarse rigidity. Note in the definition that we allow all modular derivations to be `vanishing'. However, one can also define the rigidity by allowing only $\phi$-modular derivations for a fixed normal faithful state $\phi$ and then get a weaker notion of rigidity.

\begin{definition}\label{def: L2 rigidity}
Let $\M$ be a $\sigma$-finite von Neumann algebra and let $\Q\subseteq 1_\Q\M 1_\Q$, $\B\subseteq \M$   be inclusions of von Neumann algebras with expectations. We say that $\Q$ is {\em coarsely rigid relative to $\B$ inside $\M$} if for every modular derivation $\delta\colon \dom(\delta)\subseteq L^2(\M)\to \H$  with $\B\subseteq\Null(\bar\delta)$ and ${_\M}\H_\M\prec {_\M}L^2(\M )\otimes_\B\K\otimes_\B L^2(\M)_\M$ for some $\B$-$\B$ correspondence $\K$, the associated deformation $(\theta_\alpha^{(2)})^\omega$ converges strongly on $L^2(\Q^\omega)$  to the identity as $\alpha\to\infty$   for some  free ultrafilter $\omega$ on $\mathbb N$. We call a von Neumann algebra $\M$  {\em coarsely rigid} if $\M\subseteq\M$ is coarsely rigid relative to $\mathbb C$.
\end{definition}

\begin{example}\label{ex:comparision of different rigidity notion}
\begin{enumerate}
    \item  If $\M$ is a tracial von Neumann algebra with Property (T), then $\M$ is coarsely rigid. More generally, if $\Q\subseteq\M$ is an inclusion of tracial von Neumann algebras with relative Property (T) (see \cite{Pop06}), then the inclusion $\Q\subseteq\M$ is  coarsely rigid relative to $\B$ for every subalgebra $\B$ of $\M$.
    
  However, Property (T) is a stronger notion  than coarse rigidity.  For example,  take $\M=L\mathbb F_2\bar\otimes L\mathbb F_2$, where $\mathbb F_2$ is the free group of 2 symbols. Then $\M$  does not have Property (T). Indeed this is because $L\mathbb F_2$ does not have Property (T), e.g. the trace preserving ucp map determined by the assignment $u_g\otimes u_h\mapsto e^{-t|g|}u_g\otimes u_h$, $g,h\in \mathbb F_2$ does not converge uniformly on the unit ball of $L\mathbb F_2\bar\otimes 1$ in $L^2$-norm as $t\to0$. 
    
    On the other hand, $\M$ is coarsely rigid by Corollary \ref{cor: tensor product of nonamenable factors are coarse rigid} below.
  \item   If $\M$ is a coarsely rigid tracial von Neumann algebra, then $\M$ is $L^2$-rigid in the sense of \cite{Pet09a}. However, coarse rigidity is a strictly stronger notion than $L^2$-rigidity.
  
  For example, take $\M=L\mathbb F_2\bar\otimes\mathcal{R}$, where  $\mathcal{R}$ is the separable hyperfinite II$_1$ factor. Since $\M$ is a nonamenable nonprime factor, it follows from \cite[Corollary 4.6]{Pet09a} that  $\M$ is $L^2$-rigid. 
  
  On the other hand, $\M$ is not coarsely rigid. Indeed, since  $L\mathbb F_2$ is not $L^2$-rigid (\cite[Remark 3.2]{Pet09a}), there is a trace-symmetric derivation 
$\delta:\dom(\delta)\subseteq \ell^2(\mathbb F_2)\to \left(\ell^2(\mathbb F_2)\otimes \ell^2(\mathbb F_2)\right)^{\oplus \infty}$ such that the associated deformation $(\theta_\alpha^{(2)})^\omega$ does not converge strongly to the identity on $L^2((L\mathbb F_2)^\omega)$, as $\alpha\to\infty$, for any $\omega\in \beta\mathbb N\setminus\mathbb N$ (in fact, one can take the derivation associated to the length function). Now consider  the modular derivation $\partial:=\delta\otimes 1$ on $L^2(\M)$ which takes values in $\left(\ell^2(\mathbb F_2)\otimes \ell^2(\mathbb F_2)\right)^{\oplus \infty}\otimes L^2(\mathcal{R})$ and whose associated deformation is $\eta_\alpha:=\theta_\alpha\otimes \id_{\mathcal{R}}$. Clearly, $(\eta_\alpha^{(2)})^\omega$ does not converge strongly to the identity on $L^2(\M^\omega)$,  even if $\left(\ell^2(\mathbb F_2)\otimes \ell^2(\mathbb F_2)\right)^{\oplus \infty}\otimes L^2(\mathcal{R})$ is weakly contained in $L^2(\M)\otimes L^2(\M)$ as $\M$-correspondence because $\mathcal{R}$ is amenable.
 \label{ex:L2 rigid bur not coarse rigid}

 \item More generally, for any $L^2$-rigid tracial von Neumann algebra $\N$ and any  hyperfinite factor $\Q$ (of type II$_\infty$ or of type III) the algebra $\M=\N\bar\otimes \Q$ is not coarsely rigid by the same argument as above.
 \item If a diffuse von Neumann algebra $\M$ admits a modular derivation  whose associated deformations $\theta_\alpha^{(2)}$ are compact for each $\alpha>0$, then $\M$ is not coarsely rigid by Proposition \ref{prop:convergence of deformation with compact resolvent gives atomicity}. For example, $\M=\Gamma(U_t, \H_{\mathbb R})$ if the free Araki-Woods factors for $\dim\H_{\mathbb R}<\infty$ (see Sec \ref{chp:Examples} for the construction.)

 \item By the same argument as in \cite[Remark 4.2]{Pet09a}, one shows that $\M_1\bar\ast\M_2$ is not coarsely rigid for any diffuse von Neumann algebras $\M_1,\M_2$ (see Sec \ref{chp:Examples} for construction of a derivation in the coarse bimodule).
\end{enumerate}
\end{example}

To see some more examples of von Neumann alegbras which are not coarsely rigid, we recall below the notion of relative compactness and introduce a new notion of relative diffuseness.

\section{Relatively compact operators}\label{sec:relative compactness}
We recall the notion of relatively compact operators and their $C^\ast$-algebras generated. Let $\M$ be a von Neumann algebra equipped with a normal faithful state $\phi$.
Let $\B\subseteq \M$ be an inclusion of von Neumann algebras such that $\B$ is globally invariant under $\sigma^\phi$, let $\phi_\B$ denote the restriction of $\phi$ to $\B$ and let $E_\B$ denote the $\phi$-preserving conditional expectation onto $\B$. The Jones projection $e_\B$ is given by $e_\B(x\phi^{1/2})=E_\B(x)\phi^{1/2}$. Following \cite{CKSVW23}, we let $\mathbb K(\B\subseteq\M,\phi)$ denote the closed linear span of $\{x e_\B y\mid x,y\in \M\}$ and call $\mathbb K(\B,\M,\phi)$ the space of \emph{compact operators relative to $\B$}.

\begin{remark}
    Contrary to what is claimed in \cite{CKSVW23}, the space $\mathbb K(\B\subseteq\M,\phi)$ is not necessarily an ideal in the basic construction $\langle \M,\B\rangle$ (regardless of whether $\B$ is finite or not).

    For example, let $\M=\ell^\infty(\mathbb N)\oplus \ell^\infty(\mathbb N)$,
     and 
     \[\B=\{(x_n,x_n)\mid x_n\in \ell^\infty(\mathbb N)\}, \phi((x_n,y_n))=\sum_{n\geq 1}2^{-n}((1-2^{-n})x_n+2^{-n}y_n).\]
     We have a natural identification $L^2(\M)\cong \ell^2(\mathbb N)\oplus \ell^2(\mathbb N)$ and $\langle \M,\B\rangle=M_2(\ell^\infty(\mathbb N))$, $\mathbb K(\B\subseteq \M,\phi)=\begin{pmatrix}\ell^\infty(\mathbb N)&c_0(\mathbb N)\\c_0(\mathbb N)&c_0(\mathbb N)\end{pmatrix}$ under this identification. Clearly, $\mathbb K(\B\subseteq \M,\phi)$ is not an ideal in $\langle \M,\B\rangle$.
\end{remark}

Let us connect this notion of relatively compact operators to that for Hilbert $C^\ast$-modules. Let $\mathcal X$ denote the Hilbert $C^\ast$-module that is obtained as completion of $\M$ with respect to the $\B$-valued inner product $\langle x\vert y\rangle_\B=E_\B(x^\ast y)$. For $\xi,\eta\in \mathcal X$ let 
\begin{equation*}
    \theta_{\xi,\eta}\colon \mathcal X\to \mathcal X,\,\zeta\mapsto \xi\langle\eta\vert\zeta\rangle_\B.
\end{equation*}
The space of finite-rank operators $\mathbb F(\mathcal X_\B)$ is the linear span of $\{\theta_{\xi,\eta}\mid \xi,\eta\in\mathcal X\}$. The space of compact operators $\mathbb K(\mathcal X_\B)$ is the norm closure of $\mathbb F(\mathcal X_\B)$ in $\mathbb B(\mathcal X)$.

\begin{lemma}
    There exists a unique $\ast$-isomorphism from $\mathbb K(\mathcal X_\B)$ onto $\mathbb K(\B\subseteq \M,\phi)$ that maps $\theta_{x,y}$ to $xe_\B y^\ast $ for $x,y\in \M$.
\end{lemma}
\begin{proof}
    We have a unitary $U\colon\mathcal X\mathbin{\overline\odot_\B} L^2(\B)\to L^2(\M)$ that sends $x\otimes_\B b\phi_\B^{1/2}$ to $xb\phi^{1/2}$ for $x\in \M$ and $b\in \B$. Indeed,
    \begin{align*}
        \langle x_1\otimes_\B b_1\phi_\B^{1/2},x_2\otimes_\B b_2\phi_\B^{1/2}\rangle
        &=\langle b_1\phi_\B^{1/2},E_\B(x_1^\ast x_2)b_2\phi_\B^{1/2}\rangle\\
        &=\phi_\B(b_1^\ast E_\B(x_1^\ast x_2)b_2)\\
        &=\phi(b_1^\ast x_1^\ast x_2 b_2)\\
        &=\langle x_1 b_1\phi^{1/2},x_2 b_2\phi^{1/2}\rangle
    \end{align*}
    for $x_1,x_2\in \M$ and $b_1,b_2\in \B$. 
    Since $\mathbb K(\mathcal X_\B)$ consists of adjointable operators, we get a $\ast$-homomorphism
    \begin{equation*}
        \sigma\colon \mathbb K(\mathcal X_\B)\to \mathbb B(L^2(\M)),\,S\mapsto U(S\otimes_\B 1)U^\ast.
    \end{equation*}
    The map $\sigma$ is injective: If $\sigma(S)=0$, then
    \begin{equation*}
        0=\norm{\sigma(S)x\phi^{1/2}}^2=\norm{Sx\otimes_\B \phi_\B^{1/2}}^2=\phi_\B(\langle Sx\vert Sx\rangle_\B),
    \end{equation*}
    hence $Sx=0$ for all $x\in \M$ since $\phi_\B$ is faithful. Thus $S=0$.
    Moreover, if $x,y,z\in \M$, then
    \begin{align*}
        \sigma(\theta_{x,y})z\phi^{1/2}=U(\theta_{x,y}(z)\otimes_\B \phi_\B^{1/2})=xE_\B(y^\ast z)\phi^{1/2}=xe_B y^\ast z\phi^{1/2}.
    \end{align*}
    Therefore, $\sigma(\theta_{x,y})=x e_\B y^\ast $. 
\end{proof}

The following result is probably well-known to experts, but we could not locate a precise reference.

\begin{proposition}
    Let $A$, $B$ be $C^\ast$-algebras, let $\mathcal X$ be a Hilbert $C^\ast$-module over $A$, $\mathcal Y$ a Hilbert $C^\ast$-module over $B$ and $\pi\colon A\to \mathcal L(\mathcal Y_B)$ a non-degenerate representation by adjointable operators. If $S\in \mathbb K(\mathcal X_A)$, $T\in \mathbb K(\mathcal Y_B)\cap \pi(A)^\prime$, then $S\otimes_A T\in \mathbb K(\mathcal X\overline\odot_A \mathcal Y_B)$.
\end{proposition}
\begin{proof}
    For $\xi\in\mathcal X$ we define
    \begin{equation*}
        a^\ast(\xi)\colon \mathcal Y\to \mathcal X\overline\odot_A \mathcal Y,\,\eta\mapsto \xi\otimes_A \eta.
    \end{equation*}
    A direct computation shows that $a^\ast(\xi)$ is adjointable with adjoint
    \begin{equation*}
        a(\xi)\colon \mathcal X\overline\odot_A \mathcal Y\to \mathcal Y,\,\zeta\otimes_A\eta \mapsto \pi(\langle \xi\vert \zeta\rangle_A)\eta.
    \end{equation*}
    If $(S_n)$ is a sequence in $\mathbb F(\mathcal X_A)$ such that $S_n\to S$ in norm, then
    \begin{equation*}
        \norm{S\otimes_A T-S_n \otimes_A T}\leq \norm{S-S_n}\norm{T}\to 0,
    \end{equation*}
    hence we may assume that $S=\theta_{\xi,\eta}$ for some $\xi,\eta\in \mathcal X$. We claim that $S\otimes_A T=a^\ast(\xi)T a(\eta)$. Indeed, if $\zeta_1\in \mathcal X$ and $\zeta_2\in \mathcal Y$, then
    \begin{equation*}
        (S\otimes_A T)(\zeta_1\otimes_ A \zeta_2)=\xi\langle \eta,\zeta_1\rangle_A\otimes_A T\zeta_2=\xi\otimes_A T\pi(\langle \eta\vert\zeta_1\rangle_A)\zeta_2=a^\ast(\xi)Ta(\eta)(\zeta_1\otimes_A \zeta_2).
    \end{equation*}
    It remains to show that if $\xi^\prime\in \mathcal X$ and $\eta^\prime\in \mathcal Y$, then $a^\ast(\xi)\theta_{\xi^\prime,\eta^\prime}a(\eta)\in \mathbb F(\mathcal X\overline\odot_A \mathcal Y_B)$. Indeed, if $\xi\in \mathcal X$ and $\xi^\prime\in \mathcal Y$, then
    \begin{align*}
        a^\ast(\xi)\theta_{\xi^\prime,\eta^\prime}a(\eta)(\zeta\otimes_A \zeta^\prime)&=a^\ast(\xi)(\xi^\prime\langle \eta^\prime\vert \pi(\langle \eta\vert \zeta\rangle_A)\zeta^\prime\rangle_B)\\
        &=\xi\otimes_A \xi^\prime\langle \eta\otimes_A \eta^\prime\vert\zeta\otimes_A \zeta^\prime\rangle_B\\
        &=\theta_{\xi\otimes_A \xi^\prime,\eta\otimes_A \eta^\prime}(\zeta\otimes_A \zeta^\prime).\qedhere
    \end{align*}
\end{proof}

\section{A new relative diffuseness notion}\label{sec:new notion of deffuness}

In this section, we introduce a  notion of diffuseness of a subalgebra $\A$ of a von Neumann algebra $\M$ relative to a subalgebra $\B$ along a cofinal ultrafilter. This condition is stronger than the failure of Popa's intertwining-by-bimodule \cite{Pop06}, and the two conditions agree when $\A$ is finite. We then study how this notion interacts with coarse rigidity and with the relative compactness of deformations arising from modular derivations.

We first recall the notion of intertwining-by-bimodules in the sense of Popa \cite{Pop06} introduced for finite von Neumann algebras and by Houdayer-Isono \cite{HI17} and Isono \cite{Iso25} in the general non-tracial settings.
\begin{definition}[\cite{Pop06, HI17}]\label{def:intertwining-by-bimodule}
    Let $\M$ be a $\sigma$-finite von Neumann algebra, and let $\A\subseteq1_\A\M1_\A,\B\subseteq 1_\B \M1_\B$  subalgebras with expectations. We say that $\A$ {\em embeds with expectation into $\B$ inside $M$} and write $\A\preceq_\M \B$ if there exist non-zero projections $p
    \in \A,q\in \B$, a unital faithful  normal $\ast$-homomorphism $\pi:p\A p\to q\B q$ and a partial isometry $w\in p\M q$ such that $\pi(p\A p)\subseteq q\B q$ is with expectation and $w\pi(a)=a w, \forall a\in p\A p.$
\end{definition}

Note that if $\B=\mathbb{C}$, then $\A$ is  diffuse if and only if $\A\npreceq_\M\mathbb C$. Hence, for general $\B$, one also says that $\A$ is diffuse relative to $\B$ inside $\M$ if $\A\npreceq_\M\B$ in the above sense. 

The following is a new notion of a negation of intertwining-by-bimodules for $\sigma$-finite von Neumann algebras in terms of ultraproducts, which agrees with that in \cite{Pop06} in separable tracial von Neumann algebras.

\begin{definition}\label{def: new definition of diffussness}
    Let $\M$ be a $\sigma$-finite von Neumann algebra, let $\omega$ be a free ultrafilter on a directed set, and let $\A\subseteq 1_\A\M1_\A$ and $\B\subseteq1_\B\M1_\B$ be inclusions of von Neumann algebras with expectation.  We say that $\A$ is {\em diffuse relative to $\B$ inside $\M$ along the ultrafilter $\omega$}, and write $\A\npreceq_{\M}^\omega\B$,  if there is a unitary $u\in \A^\omega$ such that $E_{\B^\omega}(x^\ast uy)=0$ for all $x,y\in 1_\A\M1_\B$. 
\end{definition}

We write $\A\preceq_\M^\omega\B$ is $\A$ is not diffuse relative to $\B$ inside $\M$ along $\omega$. We have the following comparison result between the two notions.

\begin{proposition}\label{prop: comparion of two inytertwining notion}
    Let $\M$ be a  von Neumann algebra with separable predual, and let $\A\subseteq1_\A\M1_\A, \B\subseteq1_\B\M1_\B$ be inclusion of von Neumann algebras with expectation. Fix $\omega\in\beta\mathbb N\setminus\mathbb N$. Consider the following statements:
    \begin{enumerate}
        \item $\A\preceq_\M\B$ in the sense of Definition \ref{def:intertwining-by-bimodule}.
        \item $\A\preceq_{\M}^\omega\B$.
    \end{enumerate}
    Then $(1)$ implies $(2)$. If $\A$ is finite, then $(2)$ implies $(1)$.
    \end{proposition}
    \begin{proof}
    Let $\phi$ be a normal faithful state on $1_\B\M1_\B$ such that $\B$ is globally invariant under $\sigma^\phi$.
    
    (1)$\implies(2)$ Assume that $\A\npreceq_{\M}^\omega\B$. Let $u\in \A^\omega$ be a unitary such that $E_{\B^\omega}(x^\ast uy)=0$ for $x,y\in 1_\A\M1_\B$. Let $(u_n)$ be a sequence of unitaries in $\A$ such that $(u_n)\in \mathcal{M}^\omega(\A)$ and $u=(u_n)^\omega$. Let $(x_n)$ be a countable set which spans a  $\sigma$-weakly  dense space in $1_\A\M1_\B$. By a recursive argument, there is a subsequence $(n_k)_{k\in\mathbb N}$ of $\mathbb N$ such that $\|E_\B(x_i^\ast u_{n_k}x_j)\|_{\phi}\leq 1/k$ for $1\leq i,j\leq k$. Clearly, then $E_\B(x^\ast u_{n_k}y)\to 0$ $\ast$-strongly as $k\to\infty$ for all $x,y\in \M$. Hence $\A\npreceq_\M\B$ by \cite[Theorem 4.3]{HI17}.
    
    Assume that $\A$ is finite and we prove  (2)$\implies(1)$. So assume that $\A\npreceq_\M\B$.   Since $\A$ has separable predual, there is a sequence $(u_n)$ of unitaries in $\A$ such that $E_\B(x^\ast u_ny)\to 0$ $\ast$-strongly as $n\to\infty$ for all $x,y\in 1_\A\M1_\B$. Since $\A$ is finite, $u:=(u_n)^\omega$ is a well-defined element of $\A^\omega$. Clearly, then for all $x,y\in \M$, we have 
        \begin{align*}
            \|E_{\B^\omega}(x^\ast u y)\|_{\phi^\omega}=\lim_{n\to\omega}\|E_\B(x^\ast u_ny)\|_\phi=0
        \end{align*}
      This finishes the proof. 
    \end{proof}

If we assume that $\B$ is trivial, then one can remove the finiteness condition on $\A$ in the above result from \cite{HR15}. Indeed, since diffuseness of a von Neumann algebra  is equivalent to the existence of a net of unitaries $u_i$ such that $u_i\to0$ $\sigma$-weakly, and from the proof of  \cite[Theorem A]{HR15}, one can arrange a unitary $u$ in the centralizer of  $ \M^\omega$ such that $E_\M(u)=0$. We record it below:
\begin{proposition}\label{prop; diffusenss in the scalar case}
    Let $\M$ be a  von Neumann algebra with separable predual, and let $\A\subseteq1_\A\M1_\A$ be an inclusion of von Neumann algebras with expectation. Let $\omega\in\beta\mathbb N\setminus\mathbb N$. Then $\A$ is diffuse (that is, $\A\npreceq_\M\mathbb C$) if and only if $\A\npreceq_{\M}^\omega\mathbb C$. 
\end{proposition}

\begin{example}\label{ex:the notions of intertwining are not equivalent}
 When $\A$ is not finite, $(2)\implies(1)$  in Proposition \ref{prop: comparion of two inytertwining notion} fails as the following example shows:

 Let ${\mathcal{R}_\lambda}$ be the Power's III$_\lambda$ injective factor for $0<\lambda<1$. Let $\N$ denote the separable hyperfinite II$_\infty$ factor and let $\tau$ denote the tracial normal faithful semifinite weight on $\N$.  By discrete decomposition (see \cite[Theorem 30.1]{Str20}) and the uniqueness of the hyperfinite II$_\infty$ factor, there is an automorphism $\theta$ on $\N$ such that
${\mathcal{R}_\lambda}=\N\rtimes_\theta\mathbb Z,$ $ \tau\circ\theta=\lambda\tau.$
  Denote by $v\in {\mathcal{R}_\lambda}$ the implementing unitary.  Set
 \begin{align*}
     \M={\mathcal{R}_\lambda}\bar\otimes{\mathcal{R}_\lambda},\,\,\A={\mathcal{R}_\lambda}\bar\otimes\mathbb C,\,\,\B=\N\bar\otimes{\mathcal{R}_\lambda}.
 \end{align*}
 The inclusions $\A,\B\subseteq\M$ admit normal faithful expectations $\id_{{\mathcal{R}_\lambda}}\otimes\psi$ and $E_\N\otimes\id_{{\mathcal{R}_\lambda}}$ respectively, where $E_\N$ is the canonical normal faithful expectation of ${\mathcal{R}_\lambda}$ onto $\N$ and $\psi$ is a normal faithful state on ${\mathcal{R}_\lambda}$ such that $\psi\circ E_\N=\psi$. We claim that $\A\npreceq_{\M}\B$ but $\A\preceq_\M^\omega\B$ for any $\omega\in\beta\mathbb N\setminus\mathbb N$.
 
 \textit{To show} $\A\preceq_\M^\omega\B$: By \cite[Proposition 6.23, Remark 6.24]{AH14}, discrete decomposition is preserved under the
Ocneanu ultrapower: $\M^\omega\cong \N^\omega\rtimes_{\theta^\omega}\mathbb Z.$
  Assume to the contrary that there is a unitary $u\in \A^\omega\cong \mathcal{R}_{\lambda}^\omega\bar\otimes\mathbb C$ such that $E_{\B^\omega}(x^\ast uy)=0$ for all $x,y\in \M$.  Fix $k\in\mathbb Z$ and set $y=v^k\otimes1\in \M$. Then $ 0=E_{\B^\omega}(u(v^k\otimes1))=E_{\N^\omega}(uv^k).$
  This implies $u=\sum_{k\in \mathbb Z}E_{\N^\omega}(uv^{-k})v^k=0,$
  an obvious contradiction.

  \textit{To show $\A\npreceq_\M\B$:} Set $\mathsf C=\langle {\mathcal{R}_\lambda},\N\rangle=J_\psi\N' J_\psi.$ Since  $\N'\cap\mathcal{R}_\lambda=\mathbb C$, we have
${{\mathcal{R}_\lambda}}'\cap \mathsf C=\mathbb C.$
Furthermore, we have
\[
\langle \M,\B\rangle=\mathsf C\bar\otimes {\mathcal{R}_\lambda},
\quad e_\B=e_\N\otimes1,\quad
\A'\cap\langle \M,\B\rangle=\mathbb C\bar\otimes {\mathcal{R}_\lambda}.
\tag{2}
\]
Let $\widehat E_\B:\langle \M,\B\rangle\longrightarrow \M$ be the canonical operator-valued weight, characterized by  $\widehat E_\B(xe_Bx^*)=xx^*,$ $x\in \M$. We claim that no nonzero positive element of the algebra in $\langle\M,\B\rangle$ has bounded image under $\widehat E_\B$.
To see this, let $0\ne d=1\otimes z$, where $z\in {\mathcal{R}_\lambda}_+$. For a finite subset $F\subseteq\mathbb Z$, put
$x_F=\sum_{k\in F}v^ke_\N v^{-k}.$
Then  $x_F\le1$, so that $0\le x_F\otimes z\le d.$ But,
\begin{align*}
\widehat E_\B(x_F\otimes z) =\sum_{k\in F}
\widehat E_\B\!\left(
(v^k\otimes z^{1/2})e_B(v^{-k}\otimes z^{1/2})
\right)=|F|\,(1\otimes z).
\end{align*}
which implies that $\|z\|\leq \frac{1}{|F|}\|\widehat E_\B(d)\|$. Since $z\neq0$, this proves that  $\widehat E_\B(d)$ is not bounded. Since $\A$ is properly infinite, by \cite[Theorem 4.2]{Iso25}, we have thus proved that $\A\npreceq_\M\B$.
\end{example}

\begin{proposition}\label{prop:reduction of intertwining to compresssions}
 Let $\M$ be a von Neumann algebra with separable predual, and   let $\A\subseteq 1_\A\M 1_\A, \B\subseteq1_\B\M1_\B$ be inclusions of von Neumann algebras with expectation, and let $\omega\in\beta\mathbb N\setminus\mathbb N$. If $p\in \A$ is a non-zero projection such that $p\A p\preceq_{\M}^\omega\B$, then $\A\preceq_{\M}^\omega\B$.
\end{proposition}
\begin{proof}
    Assume  that $\A\npreceq_{\M}^\omega\B$.  Let $u\in \A ^\omega$ be a unitary such that $E_{\B^\omega}(x^\ast uy)=0$ for all $x,y\in 1_\A\M1_\B$. Let $z$ be the central support of $p$ in $\A$. Let $z_1$ be the maximal central subprojection of $z$ in $\A$  such that $z_1p$ is a finite projection. 
     Set $z_2=z-z_1$ and $p_i=pz_i$ for $i=1,2$, so that $z_2p$ is a properly infinite projection.
     Moreover,  $p_1\A p_1$ is a finite von Neumann algebra while $p_2\A p_2$ is a properly infinite von Neumann algebra. 

Firstly we have $p_2\A p_2\npreceq_{\M}^\omega\B$; Indeed, since $p_2$ is a properly infinite projection in $z_2\A$, there is a partial isometry $v\in z_2\A$ such that $v^\ast v=p_2$ and $vv^\ast =z_2$. Set $u_2=v^\ast u v\in (p_2\A p_2)^\omega$. Clearly, $u_2$ is a unitary since $z_2$ commutes with $u$ and  we have $E_{\B^\omega}(x^\ast u_2 y)=0$ for all $x,y\in p_2\M1_\B$.

We next claim that $p_1\A p_1\npreceq_{\M}^\omega\B$. If not, then $p_\A p_1\preceq_\M^\omega\B$, and since $p_1\A p_1$ is finite, we have $p_1\A p_1\preceq_\M\B$ by Proposition \ref{prop: comparion of two inytertwining notion}. It then follows from \cite[Remark 4.2]{HI17} that $\A\preceq_{\M}\B$, and another use of Proposition \ref{prop: comparion of two inytertwining notion} says that $\A\preceq_{\M}^\omega\B$. This contradicts our assumption. Thus we have $p_1\A p_1\npreceq_{\M}^\omega\B$, so one gets a unitary $u_1\in (p_1\A p_1)^\omega$   such that $E_{\B^\omega}(x^\ast u_1 y)=0$ for all $x,y\in p_1\M1_\B$.

 Set $u=u_1+u_2\in \mathcal{U}((p\A p)^\omega)$. Clearly then we have $E_{\B^\omega}(x^\ast u y)=0$ for all $x,y\in p\M1_\B$. This proves that $p\A p\npreceq_{\M}^\omega\B$.
\end{proof}

\subsection{Coarse rigidity and relative compactness}

The purpose of this section is to study how this notion interacts with coarse rigidity and the relative compactness of deformations associated with modular derivations.
 One can compare the following results with \cite[Lemma 3.3]{OP10}, but with the  new notion of diffuseness.  A simple remark that   if $\Q\subseteq\M$ is an inclusion of von Neumann algebra with expectation. Then $p\Q p\subseteq p\M p$ is also with expectation for any projection $p\in \Q$ or $p\in \Q'\cap\M$ (\cite{HI17}).

\begin{proposition}\label{prop:convergence of deformation with compact resolvent gives intertwining}
Let $\B,\Q\subseteq\M$ be inclusion of von Neumann algebras with expectation, and let $\omega$ be a cofinal ultrafilter on a directed set. Let $T_i$ be a net of operators in $C^\ast(\M e_\B \M)$ such that $T_i^\omega$ strongly converges to identity on $L^2(\Q^\omega)$ along $i$. We have $p\Q p\preceq_{\M}^\omega\B$ for each projection $p\in \Q$.
\end{proposition}
\begin{proof}
Let $\phi$ be a normal faithful state such that $\phi\circ E_\B=\phi$. Set $\psi=\phi^\omega$.
Assume to the contrary that there is a non-zero projection in $p\in \Q$  such that $p\Q p\npreceq_{\M}^\omega\B$. Then there is a non-zero element $u\in (p\Q p)^\omega$ such that $E_{\B^\omega}(x^\ast u y)=0$ for all $x,y\in \M$. This implies that for all $z\in \M$, we have \[\|(ye_\B x)^\omega
(uz\psi^{1/2})\|=\|ye_{B^\omega}x(uz\psi^{1/2})\|\leq \|y\| \|E_{B^\omega}(xuz)\|_\psi=0\]
and hence
\[T^\omega u(z\phi^{1/2})^\omega=T^\omega (uz\psi^{1/2})=0\] for each $T\in C^\ast(\M e_{\B}\M)$ and $z\in \M$, which further implies that    $T^\omega u(\eta)^\omega=0$ for all $\eta\in L^2(\M)$, since $\M\phi^{1/2}$ is dense in $L^2(\M)$. Let $\phi_1$ be a normal faithful state on $\M$ such that $\Q$ is globally invariant under $\sigma^{\phi_1}$. Set $\psi_1=\phi_1^\omega$. Thus what we have proved implies in particular that 
\[T^\omega(u\psi_1^{1/2})=T^\omega u(\phi_1^{1/2})^\omega=0.\] In particular, we have $T_i^\omega(u\psi_1^{1/2})=0$ for each $i$. But then $u\psi_1^{1/2}\in L^2(\Q^\omega)$; so the hypothesis implies that $T_i^\omega(u\psi_1^{1/2})\to u\psi_1^{1/2}$ along $i$. Consequently, we obtain $u\psi_{1}^{1/2}=0$, an obvious contradiction.
\end{proof}

We rewrite the above results in the language of relative coarse rigidity as follows.

\begin{corollary}
    Let $\B\subseteq\M$ be an inclusion of von Neumann algebras with separable predual with expectation, let $\omega\in\beta\mathbb N\setminus\mathbb N$, and let   $\delta:\dom(\delta)\subseteq L^2(\M)\to \H$ be a modular derivation with ${}_\M\H_\M\prec{}_\M L^2(\M)\otimes_\B\K\otimes_\B L^2(\M)_\M{}$ for some $\B$-$\B$ correspondence $\K$.  If $p\M p\npreceq_\M^\omega\B$ for some projection $p\in\M$ and if  $\delta^\ast\bar\delta$ has compact resolvent relative to $\B$, then $\M$ is not coarsely rigid relative to $\B$. 
\end{corollary}

When $\Q$ is a finite von Neumann algebra in Proposition \ref{prop:convergence of deformation with compact resolvent gives intertwining} above,  one can state the result  as follows without the use of the ultrapower in view of Proposition \ref{prop:equivalent conditions for L2-rigidity} and Proposition \ref{prop: comparion of two inytertwining notion}.

\begin{proposition}\label{prop:convergence of deformation with compact resolvent for tracial Q}
Let $\B,\Q\subseteq\M$ be inclusion of von Neumann algebras with expectation, and let $\omega$ be a cofinal ultrafilter on a directed set.  Let $T_i$ be a net of operators in $C^\ast(\M e_\B \M)$. If $\Q$ is finite and such that $T_i$ uniformly converges along $i$ to identity on the unit ball of $\Q$, then  for each projection $p\in \Q$ we have $p\Q p\preceq_{\M}\B$ in the sense of Definition \ref{def:intertwining-by-bimodule}.
\end{proposition}

\begin{corollary}
    Let $\M$ be a diffuse von Neumann algebra with separable predual. If there is a modular derivation $\delta:\dom(\delta)\subseteq L^2(\M)\to \H$ with ${}_\M\H_\M\prec{}_\M L^2(\M)\otimes L^2(\M)_\M{}$ such that $\delta^\ast\bar\delta$ has compact resolvent, then $\M$ is not coarsely rigid.
\end{corollary}

A similar argument as in the proof of Proposition \ref{prop:convergence of deformation with compact resolvent gives intertwining} can be used to prove the following result:

\begin{proposition}\label{prop:convergence of deformation with compact resolvent gives intertwining on relative commutatant}
Let $\B,\Q\subseteq\M$ be inclusion of von Neumann algebras with expectation, and let $\omega$ be a cofinal ultrafilter on a directed set.  Let $T_i$ be a net of operators in $C^\ast(\M e_\B \M)$ such that $T_i^\omega$ strongly converges to identity on $L^2(\Q'\cap\M^\omega)$ along $i$. Then for every projection $p\in \Q'\cap\M^\omega$ and a unitary $u\in p(\Q'\cap\M^\omega)p$, there is $x\in \M$ such that   $E_{\B^\omega}(x^\ast u x)\neq 0$.
\end{proposition}

We also record below the special case when  $\B=\mathbb C$.

\begin{proposition}\label{prop:convergence of deformation with compact resolvent gives atomicity}
Let $\Q\subseteq\M$ be inclusion of von Neumann algebras with expectation, and let $\omega\in \beta\mathbb N\setminus\mathbb N$.  Let $T_i$ be a net of compact operators  such that $T_i^\omega\xi\to\xi$  for all $\xi\in L^2(\Q'\cap\M^\omega)$ (resp. $\Q^\omega$) along $i$. Then $\Q'\cap\M^\omega$ and hence $\Q'\cap\M$ (resp. $\Q^\omega$ and hence $\Q$) are completely atomic.
\end{proposition}
\begin{proof}
We prove the result only for $\Q'\cap\M^\omega$. For every compact operator $T$ on $L^2(\M)$, we have
$T^\omega(L^2(\M)^\omega)\subseteq L^2(\M).$ Indeed, this holds for finite-rank operators, and then follows for compact operators by norm approximation. Thus for all $\xi\in L^2(\Q'\cap\M^\omega)$, each $T_i\xi\in L^2(\M)$ which, by hypothesis, implies that $\xi\in L^2(\M)$. This proves 
$L^2(Q'\cap M^\omega)\subseteq L^2(\M)
$ which further implies $\Q'\cap\M^\omega\subseteq\M$, and hence $\Q'\cap\M=(\Q'\cap\M)^\omega$. By \cite[Theorem 2.3]{HR15}, it follows that $\Q'\cap\M$ is atomic.
\end{proof}

\begin{corollary}
    Let $\B\subseteq\M$ be an inclusion of tracial von Neumann algebras, let $\omega\in\beta\mathbb N\setminus\mathbb N$, and let   $\delta:\dom(\delta)\subseteq L^2(\M)\to \H$ be a modular derivation with ${}_\M\H_\M\prec{}_\M L^2(\M)\otimes_\B\K\otimes_\B L^2(\M)_\M{}$ for some $\B$-$\B$ correspondence $\K$.  If $\M\npreceq_\M\B$ and if  $\delta^\ast\bar\delta$ has compact resolvent relative to $\B$, then $\M$ is not coarsely rigid relative to $\B$. 
\end{corollary}

\section{The main theorem}\label{chp:main theorem}
This section contains one of the main results of the paper, Theorem \ref{main_theorem}. We first prove several norm estimates involving GNS symmetric ucp maps and their tensor products.

\subsection{Some norm estimates involving ucp maps}

 We collect some well-known norm estimate results about factorizable maps. The following lemma is taken from \cite[Corollary 23.11]{Pis}.
\begin{lemma}\label{An inequality on the norm of certain operator}
 If $\pi_\ell$ and $\pi_r$ are  $\ast$-representations of $\M$ and $\M^\op$, respectively, with commuting ranges, then $\norm{\sum_i\pi_\ell(x_i)\pi_r(x_i^\ast)}\leq \|\sum_ixJx_iJ\|$  for any $x_i\in \M, 1\leq i\leq n$. In particular, if  $\H$ is a Hilbert $\M$-correspondence, then we have $\norm{\sum_i\pi_\H(x_i^\ast\otimes x_i^{\Op})}\leq \norm{\sum_i x_i^\ast Jx_i^\ast J}$.  
\end{lemma}

\begin{lemma}\label{Norm estimate of involving certain ucp maps}
    Let $\B\subseteq\M$ be a unital inclusion of von Neumann algebras. For $i\in \{1,2\}$ let $\H_i$ be an $\M$-$\B$ correspondence, let $\K_i$ be a $\B$-$\B$ correspondence such that $_\M(\H_i)_\B\prec {_\M L^2(\M)}\otimes_\B (\K_i)_\B$ and let $V_i:L^2(\M)\to\H_i$ be  right $\B$-linear contractions. If $\theta_i:\M\to \langle\M,\B\rangle$ is the completely positive map given by $\theta_i(x)=V_i^\ast\pi_l^{\H_i}(x) V_i$, $x\in\M$, then we have
    \begin{align*}
        \left\lVert\sum_{j=1}^n\theta_1(x_j^\ast)\otimes_\B\theta_2(x_j)^{\Op}\right\rVert\leq \left\lVert\sum_{j=1}^n x_j^\ast\otimes_\B x_j^{\Op}\right\rVert
    \end{align*}
 for $x_i\in\M, 1\leq i\leq n$,   where $x\otimes_\B y^{\Op}$ is the operator on $L^2(\M)\otimes_\B L^2(\M)$ as defined before.
\end{lemma}
\begin{proof}
    Since $V_i$ is right $\B$-linear, it is clear that $\theta_i(x)\in \langle\M,\B\rangle$. By definition,
    \begin{equation*}
        \theta_1(x)\otimes_\B \theta_2(y)^{\mathrm{op}}=(V_1^\ast\otimes_\B V_2^\ast)(\pi_l^{\H_1}(x)\otimes_\B \pi_r^{\bar\H_2}(y^{\Op}))(V_1\otimes_\B V_2)
    \end{equation*}
    for all $x,y\in \M$.

    Since $\norm{V_1^\ast\otimes_\B V_2^\ast}=\norm{V_1\otimes_\B V_2}\leq \norm{V_1}\norm{V_2}\leq 1$, we get
    \begin{equation*}
        \left\lVert \sum_{j=1}^n \theta_1(x_j)^\ast\otimes_\B \theta_2(x_j)^\op\right\rVert\leq \left\lVert\sum_{j=1}^n \pi_l^{\H_1}(x_j)^\ast\otimes_\B \pi_r^{{\bar \H_2}}(x_j^\op) \right\rVert.
    \end{equation*}
    Since ${_\M (\H_1)}\otimes_\B (\overline{H_2})_\M\prec _\M L^2(\M)\otimes_\B \K_1\otimes_\B \overline{\K_2}\otimes_\B L^2(\M)_\M$ by \cite[Lemma 1.7]{AD95}, we obtain
    \begin{equation*}
        \left\lVert\sum_{j=1}^n \pi_l^{\H_1}(x_j)^\ast\otimes_\B \pi_r^{{\bar \H_2}}(x_j^\op) \right\rVert\leq \left\lVert\sum_{j=1}^n \pi_l^{L^2(\M)\otimes_\B \K_1}(x_j)^\ast\otimes_\B \pi_r^{{\bar\K_2\otimes_\B L^2(\M)}}(x_j^\op) \right\rVert.
    \end{equation*}
    Finally, as the left action of $\M$ on $_\M L^2(\M)\otimes_\B \K_1$ and the right action on $\bar\K_2\otimes_\B L^2(\M)_\M$ can be extended to actions of $\langle \M,\B\rangle$ and ${_{\langle M,B\rangle}L^2(\M)}\otimes_\B L^2(\M)_{\langle M,B\rangle}\cong {_{\langle M,B\rangle}L^2(\langle \M,\B\rangle)_{\langle M,B\rangle}}$ by \cite[Proposition 3.1]{Sau83}, the claim follows from Lemma \ref{An inequality on the norm of certain operator}.
\end{proof}

\begin{remark}\label{rmk:weak_containment_deformations}
    Let $\Phi\colon \M\to \M$ be a normal ucp map and $(\pi,\H,V)$ a minimal Stinespring triple for $\theta$, that is, $\H$ is a Hilbert space, $\pi\colon \M\to \mathbb B(\H)$ is a normal unital $\ast$-homomorphism and $V\colon L^2(\M)\to \H$ is an isometry such that $\text{span}\{\pi(x)V\eta\mid x\in \M,\eta\in L^2(\M)\}$ is dense in $\H$ and $\Phi=V^\ast\pi(\cdot)V$. One can turn $\H$ into an $\M$-$\M$ correspondence by defining $\pi_l^\H=\pi$ and
    \begin{equation*}
        \pi_r^\H(y^\op)\pi(x)V\eta=\pi(x)V(\eta y)
    \end{equation*}
    for $x,y\in \M$ and $\eta\in L^2(\M)$. We call $\H$ with these bimodule actions the {\em minimal Stinespring correspondence} and denote it by $\H_\Phi$.

    Now let $(\Phi_t)_{t\geq 0}$ be a GNS-symmetric quantum Markov semigroup and recall the definitions of the families of ucp maps  $(\zeta_\alpha)_{\alpha>0}$, $(\theta_\alpha)_{\alpha>0}$ and $(\rho_\alpha)_{\alpha>0}$ from Remark~\ref{rmk:ucp_families}. It follows easily from the integral formulas given there that
    \begin{align*}
        {_\M(\H_{\theta_\alpha})_\M}&\subseteq \int_{\mathbb R_+} ^\oplus {_\M (\H_{\Phi_t})_\M}\,\alpha e^{-\alpha t}dt,\\
        {_\M(\H_{\zeta_\alpha})_\M}&\subseteq \int_{\R_+}^\oplus {_\M(\H_{\theta_{\alpha(1+t)/t}})_\M}\frac{dt}{\pi(1+t)\sqrt t},\\
        {_\M(\H_{\rho_\alpha})_\M}&\subseteq \int_{\mathbb R_+}^\oplus {_\M(\H_{\theta_{t\alpha/(1+t)}})_\M}\frac{dt}{\pi(1+t)\sqrt t}.
    \end{align*}
    In particular, if there exists a $\B$-$\B$ correspondence $\K_t$ such that ${_\M (\H_{\Phi_t})_\B}\prec _\M L^2(\M)\otimes_\B (\K_t)_\B$ for all $t\geq 0$, then the same is true for $\H_{\eta_\alpha}$, $\H_{\zeta_\alpha}$ and $\H_{\theta_\alpha}$.
\end{remark}

The condition $_\M\H_\B\prec {_\M L^2(\M)}\otimes_\B \K_\B$ from the previous lemma is satisfied in various instances. 


\begin{example}\label{ex:weak_containment_amenable}
    If $\B$ is amenable, then $ {_\B}L^2(\B)_\B\prec  {_\B}L^2(\B)\otimes L^2(\B)_\B$, so for any $\M$-$\B$ correspondence $\H$ we have \begin{align*}
        {_\M}\H_{\B}&\cong  {_\M}\H\otimes_\B L^2(\B)_\B\prec  {_\M}\H\otimes_\B L^2(\B)\otimes L^2(\B)_\B\\
        &\cong {_\M}\H\otimes L^2(\B)_\B\prec  {_\M}L^2(\M)\otimes L^2(\B)_\B\\
        &\prec  {_\M L^2(\M)}\otimes_\B L^2(\B)\otimes L^2(\B)_\B
    \end{align*}
    where we  used the fact that ${_\M\H_{\mathbb C}}\subseteq {_\M (L^2(\M)}\otimes \ell^2(S))_{\mathbb C}\prec  {_\M L^2(\M)}_{\mathbb C}$  by the representation theory of von Neumann algebras, where $S$ is a possibly uncountable set.
\end{example}

\begin{example}
Let $G$ be a discrete group, $H$ a subgroup of $G$ and $\pi\colon G\to \mathcal U(\tilde\H)$ a unitary representation. If we let $\M=L(G)$, $\B=L(H)$, then $\H=\tilde\H\otimes\ell^2(G)$ becomes an $\M$-$\M$ correspondence when endowed with the bimodule actions given by $\pi_l^\H(\lambda(g))=\pi(g)\otimes\lambda(g)$ and $\pi_r^\H(\lambda(g)^{\op})=1\otimes \rho(g^{-1})$. With the unitary $U\colon \H\to \H$, $\eta\otimes 1_g\mapsto \pi(g^{-1})\eta\otimes 1_g$ we get $U\pi_l^\H(x)U^\ast=1\otimes x$ and $U\pi_r^\H(\lambda(g)^{\mathrm{op}})U^\ast=\pi(g^{-1})\otimes \rho(g^{-1})$. It follows that $U^\ast(1\otimes T)U\in \pi_r^\H(\B^{\op})^\prime$ for $T\in \langle\M,\B\rangle$. Thus there exists a $\B$-$\B$ correspondence $\K$ such that ${_\M\H}_\B\cong {_\M L^2(\M)}\otimes_\B \K_\B$ by Lemma~\ref{lem:ext_bimodule_action}.
\end{example}

\subsection{Self-polar form on von Neumann algebras}

We also need some facts about the theory of self-polar forms on von Neumann algebras.

Let $\M$ be a $\sigma$-finite von Neumann algebra. A sesquilinear form $s:\M\times\M\to\mathbb C$ is called  a {\em self-polar form} if
 $s$ is bi-positive, that is, $s(x,y)\geq0$ for all $x,y\in \M_+$, 
    $s$ is positive-definite, that is, $s(x,x)\geq 0$ for all $x\in \M$
and  $\phi:x\mapsto s(x,1)$ is a normal state on $\M$ such that for all $\psi\in \M_\ast^+$ with $0\leq \psi\leq \phi$, there is $a\in\M$ with $0\leq a\leq 1$ satisfying $\psi(x)=s(x,a), \forall x\in \M$.

The form $s$ is linear in the first variable while anti-linear in the second. For any normal faithful state $\phi$ on $\M$, there is a canonical self-polar form $s_\phi$ associated to $\phi$ given by 
\begin{align*}
    s_\phi(x,y)=\langle\phi^{1/2}, x\phi^{1/2}y^\ast\rangle,\;\;\forall x,y\in\M.
\end{align*}
The following theorem due to Woronowicz  plays a crucial role in the sequel.
\begin{theorem}[{\cite[Theorem 1.1]{Wor74}}]\label{A result of Woronowicz on self-polar forms}
    Let $\M$  be a von Neumann algebra with a normal faithful state $\phi$. Let $s:\M\times\M\to\mathbb C$ be a bipositive positive-definite sesquilinear form such that $s(x,1)=\phi(x),  x\in \M$. Then $s(x,x)\leq s_\phi(x,x)$ for all $x\in \M$.
\end{theorem}

Let us also state the following elementary Lemma which is a variant of the Hahn-Banach Theorem (cf. \cite[Lemma 8.5]{BMO20}).
\begin{lemma}\label{A Hahn-Banach type theorem}
    Let $C$ be a convex cone of a real vector space $V$, and let $p,q:V\to\mathbb R$ be maps such that $p,-q$  are sublinear  satisfying $q(c)\leq p(c)$ for $c\in C.$  Then there is a linear functional $\psi$ on $V$ such that $\psi(x)\leq p(x)$ and $q(c)\leq \psi(c)$ for all $x\in V,c\in C.$
\end{lemma}

\subsection{The main theorem}

We are now ready to state and prove the main theorem  of this paper. See Notation \ref{notation} used below.

\begin{theorem}\label{main_theorem}
    Let $\M$ be a $\sigma$-finite von Neumann algebra,  let $\delta\colon \dom(\delta)\subseteq L^2(\M)\to \H$ be a densely defined closable  modular derivation with associated quantum Markov semigroup $(\Phi_t)_{t\geq 0}$ and let $\B\subseteq\Null(\bar\delta)$. Assume that there exists an $\langle\M,\B\rangle$-$\langle \M,\B\rangle$ correspondence $\K$ and $\langle \M,\B\rangle$-$\B$ correspondences $\K_t$ for $t\geq 0$ such that $_\M\H_\M\prec {_\M\K}_\M$ and $_\M(\H_{\Phi_t})_\B\prec {_\M(\K_t)}_\B$.
    
   For an inclusion $\Q\subseteq 1_\Q\M 1_\Q$ of von Neumann algebras  with expectation, at least one of the following holds:
    \begin{enumerate}
        \item There is a non-zero projection $p\in {\mathcal{Z}(\Q'\cap1_\Q \M1_\Q)}$ such that $p\Q$ is a\-men\-able relative to $\B$.
        \item $(\theta_\alpha^{(2)})^\omega$ strongly  converges to the identity on $L^2(\Q'\cap 1_\Q\M^\omega1_\Q)$ as $\alpha\to\infty$ for every cofinal ultrafilter $\omega$ on a directed set.
    \end{enumerate} 
\end{theorem}
\begin{proof}
We shall prove the result when $1_\Q=1_\M$ just to avoid notational complications. The general case will follow from the same argument (see Remark \ref{rem:proof of the main theorem:general case} below).

 Therefore, for the remainder of the proof, we assume that $1_\Q=1_\M$ and $(\theta_\alpha^{(2)})^\omega$ does not converge to the identity on $L^2(\Q'\cap\M^\omega)$ as $\alpha\to\infty$ for some cofinal ultrafilter $\omega$.
 Let $\varphi$ be a normal faithful state on $\M$ such that $\Q$  is globally invariant under $\sigma^\varphi$. We emphasize here that the derivation $\delta$ need not be modular with respect to $\varphi$, and in particular, $\sigma^\varphi_t$ need not commute with $\theta_\alpha$, $t\in\mathbb R, \alpha>0$. Nevertheless, the following computations go through smoothly without any trouble. Our goal is to construct a state $\psi$ on $\langle\M,\B\rangle$ and a nonzero projection $p\in\mathcal{Z}(\Q'\cap\M)$  such that $\psi_{|_{\mathcal{Z}(p(\Q'\cap\M))}}$ is faithful and satisfies the KMS conditions of Theorem~\ref{thm:conditions for relative amenability}.
 
Clearly, the algebra $\Q'\cap\M^\omega$ is globally invariant under $\sigma^{\varphi^\omega},$ which means that
  the  algebra $\mathcal{A}$ consisting of $\sigma^{\varphi^\omega}$-analytic elements of $\Q'\cap\M^\omega$ is strongly dense in $\Q'\cap\M^\omega$; hence $\mathcal{A}(\varphi^\omega)^{1/2}$ is norm dense in $L^2(\Q'\cap\M^\omega)$. Therefore, there is some $d\in \Q'\cap\M^\omega$ which is $\sigma^{\varphi^\omega}$-analytic such that $(\theta_\alpha^{(2)})^\omega(d(\varphi^\omega)^{1/2})$ does not converge to $d(\varphi^\omega)^{1/2}$. Here $(\phi^\omega)^{1/2}$ denotes the unique vector in  the positive cone $L^2(\M^\omega)_+$ associated to the state $\phi^\omega$.
   Since
   \[\|\tilde\delta^\omega_{\alpha}(d(\varphi^\omega)^{1/2})\|\geq \|(\theta_\alpha^{(2)})^\omega(d(\varphi^\omega)^{1/2})-d(\varphi^\omega)^{1/2}\|,\] it follows that  $\|\tilde\delta^\omega_{\alpha}(d(\varphi^\omega)^{1/2})\|$ does not converge to $0$. Choose $C>0$ and  an increasing sequence $\alpha_n$ of positive numbers converging to $\infty$ such that $\|\delta_{\alpha_n}(d(\varphi^\omega)^{1/2})\|\geq C$ for all $n\geq1$. Set 
    \begin{align*}
     d_n=\frac{d}{\|\tilde\delta^\omega_{\alpha_n}(d(\varphi^\omega)^{1/2})\|} \in \M^\omega,\;\;\;\;   \xi_n:= d_n(\varphi^\omega)^{1/2}\in L^2(\M^\omega)^\omega
    \end{align*}
    so that each ${\tilde\delta^\omega
     _{\alpha_n}(\xi_n)}$ is a unit vector and $\|d_n\|_\infty\leq \frac{\|d\|_\infty}{C}$.
Now we consider the following spaces:
\begin{align*}
    &\mathcal{V}:=\text{span}\{x\otimes y^{\Op}; x,y\in \M\}\subseteq \langle\M,\B\rangle\otimes_{\min,\B}\langle\M,\B\rangle^{\Op},\\
    &\mathcal{C}:=\{\sum_jx_j^\ast \otimes x_j^{\Op}; x_j \in \M\}.
\end{align*}
Then $\mathcal{C}$ is a cone in the  vector space $\mathcal{V}$.  Now consider the following maps on $\mathcal{V}$ given by 
\begin{align*}
&f(T):=\|\sum_ix_i\otimes_\B y_i^{\Op}\|\\
&g(T)=\Re\Lim_n\langle\tilde\delta_{\alpha_n}^\omega(\xi_n), \sum_i\pi_{\H^\omega}(\zeta_{\alpha_n}(x_i)\otimes\zeta_{\alpha_n}(y_i)^{\Op})\tilde\delta_{\alpha_n}^\omega(\xi_n)\rangle\\
&\qquad=\Re\Lim_n\sum_i\langle\tilde\delta_{\alpha_n}^\omega(\xi_n),\zeta_{\alpha_n}(x_i)\tilde\delta_{\alpha_n}^\omega(\xi_n) \zeta_{\alpha_n}(y_i)\rangle\end{align*}
for $T=\sum_ix_i\otimes y_i^{\Op}\in \mathcal{V}$, where $\Lim_n$ is a Banach limit on $\ell^\infty(\mathbb N)$. By Corollary \ref{cor; corollary to the technical lemma in the ultrapower language},  we have
for all $x,y\in \M$ that
\begin{align*}
    \Lim_n\|\zeta_{\alpha_n}(x)\tilde\delta_{\alpha_n}^\omega(d(\varphi^\omega)^{1/2})\zeta_{\alpha_n}(y)-\tilde\delta_{\alpha_n}^\omega(xd(\varphi^\omega)^{1/2}y)\|=0,
\end{align*}
and since $\|\tilde\delta_{\alpha_n}^\omega(d(\phi^\omega)^{1/2})\|\geq C$ for all $n\geq 1$, it follows that
\begin{align}\label{eq:a consequence of technical lemma}
  \Lim_n\|\zeta_{\alpha_n}(a)&\tilde\delta_{\alpha_n}^\omega(\xi_n)\zeta_{\alpha_n}(b)-\tilde\delta_{\alpha_n}^\omega(a\xi_nb)\|\\
  &\leq \frac{1}{C}   \Lim_n\|\zeta_{\alpha_n}(a)\tilde\delta_{\alpha_n}^\omega(d(\varphi^\omega)^{1/2})\zeta_{\alpha_n}(b)-\tilde\delta_{\alpha_n}^\omega(ad(\varphi^\omega)^{1/2}b)\|=0\nonumber,
\end{align}
which can then be used to rewrite the expression for $g$ as follows:
\begin{align*}
 g(x\otimes y^{\Op})=\Re\Lim_n\langle\tilde\delta^\omega_{\alpha_n}(\xi_n), \tilde\delta^\omega_{\alpha_n}(x \xi_n y)\rangle,\quad x,y\in\M.
\end{align*}
We shall repeatedly use the expression \eqref{eq:a consequence of technical lemma} in the sequel without always referring to it. Since $\K$ is a $\langle\M,\B\rangle$-$\langle\M,\B\rangle$ correspondence, and since ${_\M}{\H^\omega}_\M\prec{_\M}{\H}_\M\prec {_\M}\K_\M$ and $_\M(\H_{\zeta_\alpha})_\B\prec {_\M L^2(\M)}\otimes_\B (\K_\alpha)_\B$ for some $\B$-$\B$ correspondence $\K_\alpha$ as discussed in Remark~\ref{rmk:weak_containment_deformations}, it follows from Lemma \ref{Norm estimate of involving certain ucp maps} that 
\begin{align*}
    \norm{\pi_{\H^\omega}(\sum_i\zeta_{\alpha_n}(x_i^\ast)\otimes\zeta_{\alpha_n}(x_i)^{\op})}&\leq \norm{\pi_\K(\sum_i\zeta_{\alpha_n}(x_i^\ast)\otimes\zeta_{\alpha_n}(x_i)^{\op})}\\
    &\leq \|\sum_i\zeta_{\alpha_n}(x_i^\ast)\otimes_\B\zeta_{\alpha_n}(x_i)^{\op}\|\\
    &\leq \|\sum_ix_i^\ast\otimes_\B {x_i}^{\op}\|
\end{align*}
for all $n\in\mathbb N$; hence we have
$g\leq f$ on $\mathcal{C}.$  Moreover, $f,-g$ are sublinear (in fact, $g$ is linear). We use Lemma \ref{A Hahn-Banach type theorem} by considering $\mathcal{V}$ as a real vector space, so  that  there exists a complex  linear map $\psi$ on $\mathcal{V}$ such that 
\[\Re\psi\leq f\;\;\mbox{on }\mathcal{V} \mbox{ and }g\leq \Re\psi\;\;\mbox{on }\mathcal{C}.\]
This implies   that $\|\psi(T)\|\leq \|T\|$ for $T\in \mathcal{V}$, so  $\psi$ extends to a bounded linear map, still denoted by $\psi$ on $\langle\M,\B\rangle\otimes_{\min,\B}\langle\M,\B\rangle^{\Op}$ such that $\|\psi\|\leq 1$. Moreover, since $\zeta_{\alpha_n}$ is unital, it follows that $1=g(1\otimes_\B 1)\leq \psi(1\otimes_\B 1)$; hence $\psi$ is a state on $\langle\M,\B\rangle\otimes_{\min,\B}\langle\M,\B\rangle^{\Op}$.  Define the state $\tilde\psi:\langle\M,\B\rangle\to \C $ by 
  \[\tilde\psi(T)=\frac{1}{2}\psi(T\otimes_\B1+1\otimes_\B T^{\Op}),\;\;T\in \langle\M,\B\rangle.\]
  We choose to work with the above symmetrization  rather than $\psi(T\otimes 1)$ in order to consider certain bi-positive and positive-definite sesquilinear forms in the sequel.
We make the following two claims:
\\
\\
{\em Claim 1}: $\tilde\psi_{|_\M}\leq   C_1\varphi$ for some constant $C_1>0$, and
\\
  {\em Claim 2:} There is a projection $p\in\mathcal{Z}(\Q)$ such that $\tilde\psi(Ta)=\tilde\psi(\sigma^\varphi_i(a)T)$ for $T\in\langle\M,\B\rangle$ and $a\in p\Q_\infty$.
  \\
  \\
 For this purpose,
 consider the sesquilinear forms $s,t$ on $\M\times\M$ given by 
\begin{align*}
&t(x,y)=\frac{1}{2}\psi(x\otimes_\B{y^\ast}^{\Op}+y^\ast \otimes_\B x^{\Op}),\\
& s(x,y)=\frac{1}{2}\Lim_n\left(\langle\tilde\delta^\omega_{\alpha_n}( \xi_n),\tilde\delta^\omega_{\alpha_n}(x \xi_n y^\ast)\rangle+\langle\tilde\delta^\omega_{\alpha_n}( \xi_n), \tilde\delta^\omega_{\alpha_n}(y^\ast \xi_n x)\rangle\right)
\end{align*}
for $x,y\in \M$. For all $x\in \M$, note that $t(x,x),s(x,x)\in \mathbb R$; indeed  $t(x,x)\in\mathbb R$ follows because $\psi$ is a state, while $s(x,x)\in\mathbb R$ follows from observing that 
\begin{align*}
 \Lim_n\langle\tilde\delta^\omega_{\alpha_n}( \xi_n),\tilde\delta^\omega_{\alpha_n}(x \xi_n x^\ast)\rangle&=\Lim_n\langle\tilde\delta^\omega_{\alpha_n}( \xi_n),\zeta_{\alpha_n}(x)\tilde\delta^\omega_{\alpha_n}( \xi_n)\zeta_{\alpha_n}(x^\ast)\rangle\\
 &=\Lim_n\langle \zeta_{\alpha_n}(x^\ast)\tilde\delta^\omega_{\alpha_n}( \xi_n)\zeta_{\alpha_n}(x),\tilde\delta^\omega_{\alpha_n}( \xi_n)\rangle\\
 &=\Lim_n\langle\tilde\delta^\omega_{\alpha_n}(x^\ast \xi_n x),\tilde\delta^\omega_{\alpha_n}( \xi_n)\rangle
\end{align*}
that is $s(x,x)=g(x\otimes {x^\ast}^{\Op})$.
Moreover,  
\begin{align*}
    t(x,x)=\Re\psi(x\otimes {x^\ast}^{\Op})\geq g(x\otimes {x^{\ast}}^{\Op})=s(x,x).
\end{align*}
Now let $y\in \M$ be a self-adjoint element and $\lambda\in\mathbb R$. Since $t(y,1)=t(1,y), s(y,1)=s(1,y)$ and since $ t(1,1)=1=s(1,1)$,
we obtain
 \begin{align*}
     1+2\lambda t( y,1)+\lambda^2t(y,y)&=t(1+\lambda y, 1+\lambda y)\\
     &\geq s(1+\lambda y,1+\lambda y)\\
     &= 1+2\lambda s( y,1)+\lambda^2s(y,y),
 \end{align*}
that is
 \[      2\lambda (t( y,1)-s(y,1))+\lambda^2(t(y,y)-s(y,y))\geq 0.\]
 Since $\lambda\in\mathbb R$ is arbitrary, and since $t(y,y)-s(y,y)\geq0$, it follows that 
 $t( y,1)= s(y,1)$; hence 
 $t(x,1)=s(x,1),\;\forall x\in \M$ 
and thus the state $\tilde\psi$ can be written as
  \begin{align*}
      \tilde\psi(x)=t(x,1)=s(x,1),\;\;x\in\M.
  \end{align*}
  {\em Proof of Claim 1.}
For all $a\in \M$, we have
  \begin{align*}
     \Lim_n\langle{\tilde\delta}^\omega_{\alpha_n}( \xi_n),{\tilde\delta}^\omega_{\alpha_n}(a^\ast a \xi_n)\rangle
      &=\Lim_n\langle{\tilde\delta}^\omega_{\alpha_n}( \xi_n), \zeta_{\alpha_n}(a^\ast){\tilde\delta}^\omega_{\alpha_n}(a \xi_n)\rangle\\
      &=\Lim_n\langle\zeta_{\alpha_n}(a){\tilde\delta}^\omega_{\alpha_n}( \xi_n), {\tilde\delta}^\omega_{\alpha_n}(a \xi_n)\rangle\\
      &=\Lim_n\langle{\tilde\delta}^\omega_{\alpha_n}(a \xi_n), {\tilde\delta}^\omega_{\alpha_n}(a \xi_n)\rangle\\
    &\leq \Lim_n\|\tilde{\delta}^\omega_\alpha\|^2\|a \xi_n\|^2\\
    &\leq \|a(\varphi^\omega)^{1/2}\sigma_{i/2}^{\varphi^\omega}(d_n)\|^2\\
    &\leq \|a(\varphi^\omega)^{1/2}\|^2\|\sigma_{i/2}^{\varphi^\omega}(d_n)\|^2\\
    &\leq  C_1 \varphi(a^\ast a)
  \end{align*}
  where
  \[C_1= \max \{\frac{1}{C^2}\|\sigma_{i/2}^{\varphi^\omega}(d)\|^2_\infty, \frac{1}{C^2}\|d\|^2_\infty\},\] so that  $C_1\geq  \sup_n\|\sigma_{i/2}^{\varphi^\omega}(d_n)\|_\infty^2$ as well as $C_1\geq \sup_n\|d_n\|^2_\infty$.
  Similarly, we have
  \begin{align*}
       \Lim_n\langle{\tilde\delta}^\omega_{\alpha_n}( \xi_n),{\tilde\delta}^\omega_{\alpha_n}( \xi_n a^\ast a)\rangle\leq\sup_n\|d_n\|^2 \|(\varphi^\omega)^{1/2}a^\ast\|^2\leq C_1\varphi(a^\ast a)
  \end{align*}
yielding the following:
\begin{align*}
    \tilde\psi(a^\ast a)=s(a^\ast a,1)\leq   C_1\varphi(a^\ast a).
\end{align*}
 {\em Proof of Claim 2.}  To prove this claim,  we consider the restrictions of the forms $s,t$ to $\Q\times\Q$, and consider  the positive functional $\varphi_1$ on $\Q$ defined by 
\[\varphi_1(a)=\Lim_n\langle\tilde\delta^\omega_{\alpha_n}( \xi_n), \tilde\delta^\omega_{\alpha_n}(a \xi_n)\rangle=\Lim_n\langle\tilde\delta^\omega_{\alpha_n}( \xi_n), \zeta_{\alpha_n}(a)\tilde\delta^\omega_{\alpha_n}( \xi_n)\rangle,\;\;a\in \Q.\]
By a calculation similar to the
one above,  we have that $\varphi_1\leq   C_1\varphi_{|_\Q}$ so that $\varphi_1$ is a normal state on $\Q$.
By noting that $\sigma_t^{\varphi_{|_\Q}}={\sigma^\varphi_t}_{|_\Q}$, we get that 
\[\Q\cap\M_{\infty,\sigma^\varphi}=\Q_{\infty,\sigma^{\phi_{|_\Q}}}:=\Q_\infty.\] So we identify $\sigma_\lambda^\varphi(a)$ and $\sigma_\lambda^{\varphi_{|_\Q}}(a)$, $a\in \Q_{\infty},\lambda\in\mathbb C$ and write them as $\sigma_\lambda(a)$. If $c\in \Q_{\infty,\sigma^\varphi}$, then $c$ is $\sigma^{\varphi^\omega}$-analytic as an element in $\M^\omega$ and $\sigma^{\varphi^\omega}_\lambda(c)=\sigma_\lambda^\varphi(c)$ for $\lambda\in\mathbb C$ \cite[Theorem 4.1]{AH14}, so that $c(\varphi^\omega)^{1/2}=(\varphi^\omega)^{1/2}\sigma_{i/2}^\varphi(c)$, and since $d_n\in \Q'\cap\M^\omega$, we have
\[cd_n(\varphi^\omega)^{1/2}=d_nc(\varphi^\omega)^{1/2}=d_n(\varphi^\omega)^{1/2}\sigma_{i/2}(c)\] which finally yields
    $c\xi_n=\xi_n\sigma_{i/2}(c).$
Thus for all $a\in \Q$ we obtain
\begin{align*}
\varphi_1(ac)&=\Lim_n\langle\tilde\delta^\omega_{\alpha_n}( \xi_n), \tilde\delta^\omega_{\alpha_n}(ac \xi_n)\rangle\\
&=\Lim_n\langle\tilde\delta^\omega_{\alpha_n}( \xi_n), \tilde\delta^\omega_{\alpha_n}(a\xi_n\sigma_{i/2}(c))\rangle\\
&=\Lim_n\langle\tilde\delta^\omega_{\alpha_n}( \xi_n), \tilde\delta^\omega_{\alpha_n}(a \xi_n)\zeta_{\alpha_n}(\sigma_{i/2}(c))\rangle\\
&=\Lim_n\langle\tilde\delta^\omega_{\alpha_n}( \xi_n)\zeta_{\alpha_n}(\sigma_{-i/2}(c^\ast)), \tilde\delta^\omega_{\alpha_n}(a \xi_n)\rangle\\
&=\Lim_n\langle\tilde\delta^\omega_{\alpha_n}( \xi_n \sigma_{-i/2}(c^\ast)), \tilde\delta^\omega_{\alpha_n}(a \xi_n)\rangle\\
&=\Lim_n\langle\tilde\delta^\omega_{\alpha_n}(\sigma_{-i}(c^\ast) \xi_n ), \tilde\delta^\omega_{\alpha_n}(a \xi_n)\rangle\\
&=\Lim_n\langle\tilde\delta^\omega_{\alpha_n}( \xi_n ), \tilde\delta^\omega_{\alpha_n}(\sigma_i(c)a \xi_n)\rangle\\
&=\varphi_1(\sigma_i(c)a).
\end{align*}
Arguing as in the proof of Proposition \ref{Prop:Equivalent conditions on relative amenability in terms of states with extra assumption on summand}, we get in particular $\varphi_1\circ\sigma_t=\varphi_1$ on $\Q$. Combining this with the fact that $\varphi_1\leq   C_1\varphi_{|_\Q}$ yields  a non-zero positive operator $z\in \Q^{\varphi_{|_\Q}}=\M^\varphi\cap\Q$, the centralizer, such that 
\[\varphi_1(a)=\varphi(za)=\varphi(z^{1/2}az^{1/2}),\] for $a\in \Q.$ Let $\xi_\varphi$ denote the cyclic and separating vector in $L^2(\Q,\varphi_{|_\Q})$ such that $\varphi_{|_\Q}(a)=\langle\xi_\varphi,a\xi_\varphi\rangle,a\in\Q$. We identify $L^2(\Q,\varphi_{|_\Q})$ as a subspace of $L^2(\M,\varphi)$ via the map $a\xi_\varphi\mapsto a\varphi^{1/2}$ for $a\in\Q$.
Under this identification,   for all $a\in\Q,c\in\Q_\infty$ we have
\begin{align*}
    \langle a^\ast\xi_\varphi,zc\xi_\varphi\rangle
    =\varphi(azc)=\varphi(\sigma_i(c)az)=\varphi_1(\sigma_i(c)a)=\varphi_1(ac)=\varphi(acz)=\langle a^\ast\xi_\varphi, cz\xi_\varphi\rangle
\end{align*}
implying that $zc=cz$; hence $z\in\mathcal{Z}(\Q)$. In particular, $z^{1/2}\xi_\varphi\in L^2(\Q,\varphi_{|_\Q})_+,$ the positive cone of $L^2(\Q)$.  Note also that if $p=s(z)$ is the support of $z$ in $\Q$, then $p$ is also the support of $\varphi_1$ (since $\varphi_1(a)=0$ if and only if $za=0$ for any $a\in \Q_+$); in particular $p\in \mathcal{Z}(\Q)\subseteq\mathcal{Z}(\Q'\cap\M)$.
By working and proving the following results with $p\Q$, we may and we will as well assume that $p=1$, that is, $\varphi_1$ is faithful.

We now consider the self-polar form $r$ on $\Q\times\Q$ associated to the normal faithful  state $\varphi_1$, that is, $r(a,c)=\langle\varphi_1^{1/2}, a\varphi_1^{1/2}c^\ast\rangle$ for $a,c\in Q$, where $\varphi_1^{1/2}=z^{1/2}\xi_\varphi\in L^2(\Q)$.   We have
 \begin{align*}r(a,c)=\langle z^{1/2}\xi_\varphi,az^{1/2} \xi_\varphi c^\ast\rangle=\langle \xi_\varphi,az \xi_\varphi c^\ast\rangle=\langle \varphi^{1/2},a z\varphi^{1/2}  c^\ast\rangle,
 \end{align*}
 for $a,c\in\Q$. Since $d_n\in\Q'\cap\M^\omega$, we next observe the following  for $a\in\Q,c\in\Q_\infty$:   
 \begin{align*}
     2s(a,c)&=\Lim_n\left(\langle\tilde\delta^\omega_{\alpha_n}( \xi_n), \tilde\delta^\omega_{\alpha_n}(a \xi_n c^\ast)\rangle+\langle\tilde\delta_{\alpha_n}( \xi_n), \tilde\delta^\omega_{\alpha_n}(c^\ast \xi_n a)\rangle\right)\\
     &=\Lim_n\left(\langle\tilde\delta^\omega_{\alpha_n}( \xi_n), \tilde\delta^\omega_{\alpha_n}(a\sigma_{-i/2}(c^\ast) \xi_n)\rangle+\langle\tilde\delta_{\alpha_n}( \xi_n), \tilde\delta^\omega_{\alpha_n}(c^\ast\sigma_{-i/2}(a) \xi_n)\rangle\right)\\
     &=\varphi_1(a\sigma_{-i/2}(c^\ast)+c^\ast\sigma_{-i/2}(a))\\
     &=\langle \varphi^{1/2}, a\sigma_{-i/2}(c^\ast)z\varphi^{1/2}\rangle+\langle\varphi^{1/2}, c^\ast\sigma_{-i/2}(a) z\varphi^{1/2} \rangle\\
     &=\langle\varphi^{1/2}, az\varphi^{1/2} c^\ast\rangle+\langle\varphi^{1/2}, c^\ast z\varphi^{1/2} a\rangle\\
     &=r(a,c)+r(c^\ast,a^\ast),
 \end{align*}
 and since both $s,r$ are weak$^\ast$-continuous separately in each coordinate and since $\Q_\infty$ is weak$^\ast$-dense in $\Q$, it follows that $r(a,b)+r(b^\ast,a^\ast)=2s(a,b)$ for all $a,b\in\Q.$ Since  $r(a,a)\geq 0$, we have 
 \[r(a,a)=\overline{r(a,a)}={r(a^\ast,a^\ast)}\] which implies that
 \begin{align*}
     s(a,a)=r(a,a),\;\;a\in\Q.
 \end{align*}
Positive definiteness of $r$ also yields positive definiteness of $t_{|_{\Q\times\Q}}$, as
$t(a,a)\geq s(a,a)\geq 0$ for all $a\in\Q$. Since $\psi$ is a positive functional,
$t(a,b)=\frac{1}{2}\psi(a\otimes b^{\Op}+b\otimes a^{\Op})\geq 0$ for all $a,b\in\Q$ positive. Moreover,  we have \[t(a,1)=s(a,1)=\frac{1}{2}(r(a,1)+r(1,a^\ast))=\varphi_1(a),\;\; a\in\Q.\]
We can then invoke Theorem \ref{A result of Woronowicz on self-polar forms} to conclude that  $t(a,a)\leq r(a,a)$ for all $a\in\Q$, and the reverse inequality being also true from above yields
 $t(a,a)=r(a,a)$. By polarization, we arrive at the following:
\begin{align*}
    r(a,b)=t(a,b)=s(a,b),\;\;\forall a,b\in\Q.
\end{align*}
Thus for $a,b\in\Q_\infty$, we infer
\begin{align*}
    t(a,b^\ast)=\varphi_1(a\sigma_{-i/2}(b)).
\end{align*}
Returning to the state  $\psi$ on $\M\otimes_{\B,\min}\M^{\Op}$, for $a\in\Q_\infty$, put 
\begin{align*}
A=a\otimes_\B1-1\otimes_\B \sigma_{i/2}(a)^{\text{op}},\;\;\;B=\sigma_{i/2}(a)\otimes_\B1-1\otimes_\B a^{\text{op}}
\end{align*}
 and compute
\begin{align*}
  &\psi(A^\ast A)=  
      \psi(a^\ast a\otimes_\B1)+\psi(1\otimes_\B (\sigma_{i/2}(a)\sigma_{-i/2}(a^\ast))^{\Op})-2 \Re\psi(a^\ast\otimes_\B \sigma_{i/2}(a)^{\text{op}})\\
      &\psi(BB^\ast)=\psi(\sigma_{i/2}(a)\sigma_{-i/2}(a^\ast) \otimes_\B1)+\psi(1\otimes_\B (a^\ast a)^{\op})-2 \Re\psi(\sigma_{i/2}(a)\otimes_\B {a^\ast}^{\text{op}}).
\end{align*}
Adding the  respective sides of the above equalities as well as using the fact that $s(a,b^\ast)=\varphi_1(a\sigma_{-i/2}(b)+b\sigma_{-i/2}(a))$ for all $a,b\in\Q_\infty$  conclude the following:
\begin{align*}
\psi(A^\ast A+BB^\ast)
  &=2t(a^\ast a,1)+2t(\sigma_{i/2}(a)\sigma_{-i/2}(a^\ast),1)-4\Re t(\sigma_{i/2}(a),a)\\
  &=2\varphi_1(a^\ast a)+2\varphi_1(\sigma_{i/2}(a)\sigma_{-i/2}(a^\ast))-4\Re\varphi_1({\sigma_{i/2}(a)\sigma_{-i/2}(a^\ast)})\\
  &=0.
\end{align*}
This proves that
$  \psi(A^\ast A)=0=\psi(BB^\ast)$
and hence for all $T\in \langle\M,\B\rangle$, we have
  \begin{align*}
    &\psi(Ta\otimes _\B1)=\psi(T\otimes_\B \sigma_{i/2}(a)^{\op}),\\
   & \psi(a\otimes T^{\op})=\psi(1\otimes_\B T^{\op}\sigma_{i/2}(a)^{\op})=\psi(1\otimes_\B (\sigma_{i/2}(a)T)^{\op}),\\
   & \psi(\sigma_{i/2}(a)T\otimes_\B1)=\psi(T\otimes_\B a^{\op}),\\
    &\psi(\sigma_{i/2}(a)\otimes T^{\op})=\psi(1\otimes a^{\op}T^{\op})=\psi(1\otimes_\B(Ta)^{\op}),
    \end{align*}
   which finally yields the following:
   \begin{align*}
       \psi(Ta\otimes_\B1)=\psi(\sigma_i(a)T\otimes_\B1),\;\;\;\psi(1\otimes (Ta)^{\op})=\psi(1\otimes (\sigma_i(a)T)^{\op}).
   \end{align*}
   Thus we have shown that 
  \begin{align*}
     \tilde \psi(Ta)&=\frac{1}{2}\psi(Ta\otimes_\B1)+\frac{1}{2}\psi(1\otimes(Ta)^{\op})\\
     &=\frac{1}{2}\psi(\sigma_i(a)T\otimes_\B1)+\frac{1}{2}\psi(1\otimes(\sigma_i(a)T)^{\op})\\
     &=\tilde\psi(\sigma_i(a)T)
  \end{align*}
  for all $a\in \Q_\infty,$ $T\in \langle\M,\B\rangle$, which concludes  the proof of Claim 2.
  
 To finish the proof, let
 $p\in \mathcal{Z}(\Q'\cap\M)$ be the smallest non-zero projection such that $\tilde\psi(1-p)=0$, that is, $p$ is the support of $\tilde\psi_{|_{\mathcal{Z}(\Q'\cap\M)}}$. Clearly, $p\neq 0$.  By definition of $p$,  $\tilde\psi_{|_{p\mathcal{Z}(\Q'\cap \M)}}$ is faithful. Clearly, the condition $\tilde\psi_{|_{p\M p}}\leq C_1 \varphi(p\cdot p)$ remains valid. Moreover,  if $\phi_p:=\phi(p\cdot p)/\phi(p)$ is the state on $p\M p$, then the condition $p\in \M^\phi$ implies that
 \[\sigma_t^{\phi_p}(x)=\sigma_t^\phi(x)=p\sigma_t^\phi(x)p,\;\;\forall t\in\mathbb R, x\in p\M p\] and that the $\sigma^{\phi_p}$-analytic elements of $\M$ (resp. $p\Q$) are precisely $p\M_{\infty,\sigma^{\phi}}p$ (resp. $p\Q_\infty$). Therefore the equality  \[\tilde\psi(Tx)=\tilde\psi(\sigma_i^{\phi_p}(x)T),\;\; T\in p\langle\M,\B\rangle p, x\in p\Q\] remains valid as well. After normalizing the functional $\tilde\psi_{|_{p\langle\M,\B\rangle p}}$  by $\tilde\psi(p)$, we have thus shown that $p\Q$ is amenable relative to $\B$ inside $\M$ in view of Theorem~\ref{thm:conditions for relative amenability}.
\end{proof}

\begin{remark}\label{rem:proof of the main theorem:general case}
 We assumed $1_\Q=1_\M$   in the proof of Theorem \ref{main_theorem}. The general case will follow from the same argument as follows: Choose a normal faithful state $\phi$ on $\M$ such that $1_\Q$ lies in the centralizer $\M^\phi$ and $\Q$ is globally invariant under $\sigma^\phi$. Indeed, take any normal faithful state $\psi$ on $1_\Q\M1_\Q$ such that $\Q$ is globally invariant under $\sigma^\psi$, let $\nu$ be a normal faithful state on $(1_\M-1_\Q)\M(1_\M-1_\Q)$ and set
 \[\phi(x)=\frac{1}{2}\left(\psi(1_\Q x1_\Q)+\nu((1_\M-1_\Q)x(1_\M-1_\Q))\right),\;\;\;x\in \M.\]
  Then if (2) in Theorem \ref{main_theorem} is not true, one will find $d\in \Q'\cap 1_\Q\M^\omega 1_\Q$ such that $(\theta_\alpha^{(2)})^\omega(d(\phi^\omega)^{1/2})$ does not converge to $d(\phi^\omega)^{1/2}$. Since $1_\Q\phi^\omega=\phi^\omega 1_\Q$, the  same argument will yield a state $\psi$ on $1_\Q\langle\M,B\rangle1_\Q$ satisfying the KMS condition.
\end{remark}

\begin{definition}
    We say that an expected inclusion $\Q\subseteq1_\Q\M1_\Q$ is {\em completely nonamenable relative to $\B$} if $p\Q$ is not amenable relative to $\B$ for every non-zero projection $p\in\mathcal{Z}(\Q'\cap1_\Q\M1_\Q)$.
\end{definition}
Note that if $\Q$ has no nonzero amenable summand, then $\Q$ is completely nonamenable relative to $\mathbb C$. Indeed, if $p\Q$ is  amenable relative to $\mathbb C$ for some non-zero projection $p\in \mathcal{Z}(\Q'\cap\M)$, then $p\Q$ is amenable. Since $p\Q\cong r\Q$ for some central projection $r$ (simply take $r$ to be the orthogonal complement of $\ker\pi$ where $\pi:\Q\to p\Q$ is the normal $\ast$-homomorphism given by $x\mapsto px$). So $r\Q$ is amenable, hence $r=0$. This gives $p=0$.

Let us also explicitly state Theorem \ref{main_theorem} in the  scalar-valued case. In this situation,  the condition $_\M(\H_{\Phi_t})_\B\prec{_\M(\K_t)_\mathbb C}$  is always satisfied, as discussed in Example~\ref{ex:weak_containment_amenable}.  Consequently, we obtain Theorem \ref{main_theorem_amenable_for introduction} in  Introduction, because for any amenable $\B$, we have  $L^2(\M)\otimes_\B L^2(\M)\prec L^2(\M)\otimes L^2(\M)$ as $\M$-$\M$ correspondence.

\begin{corollary}\label{cor:dichotomy result in amenable case}
    Let $\M$ be a $\sigma$-finite von Neumann algebra and $\delta\colon \dom(\delta)\subseteq L^2(\M)\to \H$ a densely defined closable modular derivation with $\H\prec L^2(\M)\otimes L^2(\M)$. If $\Q\subseteq 1_\Q\M1_\Q $ is a von Neumann subalgebra with expectation, at least one of the following is true:
    \begin{enumerate}[1.]
        \item  There is a non-zero projection $p\in {\mathcal{Z}(\Q'\cap1_\Q \M1_\Q)}$ such that $p\Q$ is a\-men\-able.
        \item $(\theta_\alpha^{(2)})^\omega$ converges strongly to the identity on $L^2(\Q^\prime\cap 1_\Q\M^\omega1_\Q)$ as $\alpha\to \infty$ for every cofinal ultrafilter $\omega$ on a directed set.
    \end{enumerate}
\end{corollary}

We can also prove some variant of the above result as follows in terms of the uniform convergence at the level of centralizer algebras. Compare the following result with \cite[Theorem 4.3]{Pet09a}
proved in the tracial setting.

\begin{corollary}\label{cor:dichotomy result, unform convergence at the centralizer level version}
Under the same hypotheses
as  in Theorem \ref{main_theorem},   at least one of the following is true:
    \begin{enumerate}
        \item There is a non-zero projection $p\in {\mathcal{Z}(\Q'\cap1_\Q \M1_\Q)}$ such that $p\Q$ is a\-men\-able relative to $\B$.
        \item For every cofinal ultrafilter $\omega$ on a directed set and for every  state $\varphi$ on $1_\Q\M1_\Q$ with $\varphi=\varphi\circ E_\Q$, $(\theta_\alpha^{(2)})^\omega$   converges uniformly to the identity in $\|\cdot\|_{\phi^\omega}$-norm on the unit ball of $\Q'\cap(1_\Q\M^{\omega}1_\Q)^{\phi^\omega}$. 
    \end{enumerate} 
\end{corollary}
\begin{proof}
    If $(\theta_\alpha^{(2)})^\omega$ does not converge uniformly on $\Q'\cap(\M^\omega)^{\varphi^\omega}$ for some cofinal ultrafilter $\omega$ and a state $\varphi$ with $\varphi\circ E_\Q=\varphi$, then by Lemma \ref{lem:pointwise convergece gives uniform convergence at the level of centralizer} below $(\theta_\alpha^{(2)})^\omega$ does not converge strongly on $L^2(\Q'\cap\M^{\omega'})$ for some cofinal ultrafilter $\omega'$ on a directed set, and so we are in the situation of Theorem \ref{main_theorem} to  conclude the first statement.
\end{proof}

\begin{lemma}\label{lem:pointwise convergece gives uniform convergence at the level of centralizer}
 Let $\Q\subseteq\M$ be an inclusion of $\sigma$-finite von Neumann algebras,  and let $\varphi$ be a normal faithful state on $\M$ such
that $\Q$ is globally invariant under $\sigma^\varphi$.   Let $(\Psi_\alpha)_{\alpha\in \Lambda}$ be a net of GNS-symmetric ucp maps such that $\norm{\xi-\Psi_\alpha^{(2)}\xi}\leq \norm{\xi-\Psi_\beta^{(2)}}$ if $\alpha\geq \beta$ for all $\xi\in L^2(\M)$. Suppose that $(\Psi_\alpha^{(2)})^\omega$ does not converge uniformly in $\|\cdot\|_{\varphi^\omega}$ norm on  $(\Q'\cap(\M^\omega)^{\varphi^{\omega}})_1$ for some cofinal ultrafilter $\omega$ on a directed set $J$. Then there exists a cofinal ultrafilter $\omega_1$ on a directed set such that $(\Psi_\alpha^{(2)})^{\omega_1}$ does not converge strongly  on $L^2(\Q'\cap\M^{\omega_1})$.
\end{lemma}
\begin{proof}
Put $T_\alpha=\Psi_\alpha^{(2)}$. Since   $T_\alpha^{\omega}$ does not converge uniformly in $\|\cdot\|_{\phi^\omega}$ on the unit ball of $\Q'\cap(\M^\omega)^{\varphi^{\omega}}$, after passing to a subnet if necessary,  there is $C>0$ such that for all $\alpha\in \Lambda$, there is some $x_\alpha\in \Q'\cap(\M^{\omega})^{\varphi^{\omega}}$ satisfying $\|x_\alpha\|< 1$ and $\|T_\alpha^{\omega}(x_\alpha(\varphi^\omega)^{1/2})-x_\alpha(\phi^\omega)^{1/2}\|\geq C$. For each $\alpha\in\Lambda$, let $(x_{\alpha,j})_{j\in J}\in \mathcal{M}^{\omega}(\M)$ be such that $\sup_{j\in J}\|x_{\alpha,j}\|\leq 1$ and $x_\alpha=((x_{\alpha,j})_{j\in J})^{\omega}$. Consider the directed set
 \[I:=\{(\alpha, F, n); \alpha\in \Lambda, F\subseteq(\Q)_1 \mbox{ finite},n\in\mathbb N\}\] 
 with the order $(\alpha, F,n)\leq (\alpha', F',n')$ if and only if $\alpha\leq\alpha', F\subseteq F'$ and $n\leq n'$. For each $i=(\alpha,F,n)\in I$, there is a $j_i\in J$ such that, putting $y_i=x_{\alpha,{j_i}}$, we have
  \begin{align*}
      &\|T_\alpha(y_i\phi^{1/2})-y_i\varphi^{1/2}\|\geq C,\\
      &\|y_i\varphi^{1/2}-\varphi^{1/2}y_i\|\leq \frac{1}{n}\\
      &\|(y_i q-qy_i)\phi^{1/2}\|\leq\frac{1}{n},\;\;\forall q\in F.
  \end{align*}
Let $\omega_1$ be a cofinal ultrafilter on the directed set $I$ and consider the element $(y_i)_{i\in I}\in \ell^\infty(I, \M)$. Since $\lim_{i}\|y_i\varphi^{1/2}-\varphi^{1/2}y_i\|=0$, with a similar proof as in \cite[Proposition 2.2]{HR15}, it follows  that $(y_i)_{i\in I}\in \mathcal{M}^{\omega_1}(\M)$. Set $x=(y_i)^{\omega_1}\in \M^{\omega_1}$. Thus we have
\begin{align*}
 \|T_\alpha^{\omega_1}(x(\varphi^{\omega_1})^{1/2})-x(\phi^{\omega_1})^{1/2}\| \geq C\;\;\;\mbox{and }\;\;\;x\in \Q'\cap(\M^{\omega_1})^{\varphi^{\omega_1}}.  
\end{align*}
 So we have shown that $T_\alpha^{\omega_1}$ does not converge strongly along $\alpha$ to the identity  on $L^2(\Q'\cap(\M^{\omega_1})^{\varphi^{\omega_1}})$ and hence does not converge strongly on $L^2(\Q'\cap\M^{\omega_1})$.
\end{proof}

 Finally, if $I=\mathbb N$ and $(\Psi_n)_{n\in\mathbb N}$ is a sequence in the above Lemma, and $\Q_\ast$ is separable, then we can choose an increasing sequence $F_n$ of finite subsets of the unit ball of $\Q$ such that $\cup F_n$ spans a strongly dense set in $\Q$, and then take $\Lambda=\mathbb N$ which can be identified with the directed subset $\{(n, F_n, n); n\in\mathbb N\}$. So we get the following:

\begin{corollary}
    Let $\Q\subseteq\M$ be an inclusion of $\sigma$-finite von Neumann algebras with expectation such that $\Q_\ast$ is separable,  let $\omega\in\beta\mathbb N\setminus\mathbb N$ and    let $(\Psi_n)_{n\geq1}$ be a sequence of GNS-symmetric ucp maps on $\M$ such that $\norm{\xi-\Psi_{n+1}^{(2)}\xi}\leq\norm{\xi- \Psi_{n}^{(2)}\xi}$ for all $\xi\in L^2(\M),n\in\mathbb N$. If $(\Psi_n^{(2)})^{\omega}$ converges strongly to the identity  on $L^2(\Q'\cap\M^{\omega})$, then $(\Psi_n^{(2)})^\omega$  converges uniformly in $\|\cdot\|_{\varphi^\omega}$-norm on  $(\Q'\cap(\M^\omega)^{\varphi^{\omega}})_1$ for every normal faithful state $\varphi$ on $\M$ such that $\Q$ is globally invariant under $\sigma^\varphi$.
\end{corollary}

\section{A dichotomy result on amenability}\label{sec: a dichotomy result}
 In this section, we prove the dichotomy result (Theorem \ref{thm:dichotomy result_Introduction}) between amenability and coarse rigidity for diffuse von Neumann algebras. More precisely, we prove the following:

  \begin{theorem}\label{thm:dichotomy result for subalgebras}
     Let $\M$ be a diffuse a von Neumann algebras with  separable predual, and let $\omega\in\beta\mathbb N\setminus\mathbb N$. Exactly one of the following holds true:
     \begin{enumerate}
         \item $\M$ has a non-zero amenable summand.
         \item For every modular derivation  $\delta:\dom(\delta)\subseteq L^2(\M)\to \H$ with $\H\prec L^2(\M)\otimes L^2(\M)$, we have that $(\theta_\alpha^{(2)})^\omega$ converges strongly to the identity on $L^2(\M'\cap\M^\omega)$, as $\alpha\to\infty$.
     \end{enumerate}
 \end{theorem}

We first prove a few Lemmas. Throughout this section, we denote by $\mathcal{R}$ the unique separable hyperfinite II$_1$ factor.

\begin{lemma}\label{lem: a criteria for existence of certain derivation}
    Let $\M$ be a $\sigma$-finite von Neumann algebra, let $\phi$ be a normal faithful state on $\M$, and let $\alpha:\mathbb R\curvearrowright \M$ be a $\phi$-preserving action  such that $\alpha$ commutes with $\sigma^\phi$. Write $\alpha^{(2)}_t=e^{itD}, t\in \mathbb R$ for a  self-adjoint operator $D$ on $L^2(\M)$. Let $\omega$ be a cofinal ultrafilter on a directed set $I$. If there are a  uniformly bounded net  $\{x_n\}_{n\in I}\in \M$ and $\{\lambda_n\}_{n\in I}$ in $ (0,\infty)$ with $\lambda_n\to\infty$ such that   $D^2(x_n\phi^{1/2})=\lambda_nx_n\phi^{1/2}$ and $x:=(x_n)^\omega$ is a well-defined non-zero element in $\M^\omega$ then there is a $\phi$-modular derivation $\delta:\dom\delta\subseteq L^2(\M)\to L^2(\M)$ such that $\theta_\alpha^\omega(x)$ does not converge strongly to $x$ as $\alpha\to\infty$.
\end{lemma}
\begin{proof}
    Let $\mathcal{A}_\alpha$ and $\mathcal{A}_\phi$ be the algebra of $\alpha$ and $\sigma^\phi$-analytic vectors of $\M$, respectively, and set $\mathcal{A}=\mathcal{A}_\alpha\cap\mathcal{A}_\phi$. Since $\alpha$ and $\sigma^\phi$ commute, $\mathcal{A}$ is a strongly dense unital $\ast$-subalgebra of $\M$: If $x\in \M$ then $x_n:=\frac{n}{\pi}\int_{\mathbb R\times\mathbb R}e^{-nt^2-ns^2}\alpha_t(\sigma_s^\phi(x))dtds$ belongs to $\mathcal{A}$ such that $x_n\to x$ strongly as $n\to\infty$.
    
    Set $\dom(\delta)=\mathcal{A}\phi^{1/2}\subseteq L^2(\M,\phi)$ and define $\delta:\dom(\delta)\to L^2(\M)$ by $\delta=iD_{|_{\dom(\delta)}}$, that is,
\[ \delta(x\phi^{1/2})=\left.\frac{d}{dt}\right|_{t=0}
    \alpha_t(x)\phi^{1/2} =iD(x\phi^{1/2}),\quad x\in\mathcal A.\]
 The formula
 $ \alpha_t(xy)=\alpha_t(x)\alpha_t(y)$ for $ x, y\in \M$ implies the derivation rule for $\delta$.  Since $\alpha$ commutes with $\sigma^\phi$, the operators $D$ and  $\Delta_\phi^{it}$  strongly commute for all $t\in\mathbb R$. It is also easy to see that $J_\phi D=-DJ_\phi$. Hence $\delta$ has the required modular property. Since $D$ is a closed unbounded operator, $\delta$ is closable with the closure being $iD$. Moreover, we have $\mathcal{L}_2=\delta^\ast\bar\delta=D^2$.

 Next, we note that
 \[\theta_\alpha^{(2)}(x_n\phi^{1/2})=\frac{\alpha}{\alpha+\lambda_n}x_n\phi^{1/2},\;\;n\geq1, \alpha>0.\]
Since $\lambda_n\to\infty$ and $\{x_n\}_{n\in I}$ is uniformly bounded, it follows  for all $\alpha>0$ that  $\theta_\alpha^{(2)}(x_n\phi^{1/2})\to 0$ along $n$. 
Thus $ (\theta_\alpha^{(2)})^\omega(x(\phi^\omega)^{1/2})=0$ for all $\alpha>0$, which proves that
$(\theta_\alpha^{(2)})^\omega(x)$ does not converge strongly to $x$, as $\alpha\to\infty$.
\end{proof}

\begin{lemma}\label{lem: a derivation on the hyperfinite}
There exists a trace-symmetric derivation $\delta:\dom\delta\subseteq L^2(\mathcal{R})\to L^2(\mathcal{R})$ such that $(\theta_\alpha^{(2)})^\omega$ does not converge strongly on $L^2(\mathcal{R}'\cap\mathcal{R}^\omega)$ as $\alpha\to\infty$ for any $\omega\in\beta\mathbb N\setminus\mathbb N$.
\end{lemma}
\begin{proof}
Take the infinite tensor product decomposition 
\[\mathcal{R}
  =\overline{\bigotimes_{n=1}^\infty}
   (\mathbb M_2, {\tr}_2)\]
of the hyperfinite $\mathrm{II}_1$ factor, with its trace denoted by $\tau$. Here $\tr_2$ denotes the normalized trace on the complex matrix algebra $\mathbb M_2$.
Choose a sequence $\lambda_n\to\infty$, for example $\lambda_n=n$, and set
\[
  h_n
  =\frac{\lambda_n}{2}
   \begin{pmatrix}
     1&0\\
     0&-1
   \end{pmatrix}\in \mathbb M_2.
\]
There is a unique strongly continuous  action $\alpha$ of $\mathbb R$ on $\mathcal{R}$ such that 
\[
  \alpha_t
  =\bigotimes_{n=1}^\infty\Ad(e^{ith_n}),\;\;t\in\mathbb R.
\]
Clearly, $\alpha$ is $\tau$-preserving.  
For each $n\geq1$, let
\[
  u_n=1\otimes\cdots\otimes
   \begin{pmatrix}
     0&1\\
     1&0
   \end{pmatrix}
   \otimes\cdots,
\]
where the nontrivial matrix occurs in the $n$-th tensor factor.  Each $u_n$
is a self-adjoint unitary.  Moreover, $\|u_nx-xu_n\|_2\to0,$as $n\to\infty,$ for all $x\in\mathcal{R},$
because this is immediate for simple tensors and then follows for arbitrary
$x\in\mathcal{R}$ by approximation. 
Hence $u=(u_n)^\omega\in\mathcal{R}'\cap\mathcal{R}^\omega.$

Let $D$ be the self-adjoint operator on $L^2(\mathcal{R},\tau)$ such that $\alpha_t^{(2)}=e^{-itD}.$ For any $y\in \mathbb M_2$, if we set $x=1\otimes\cdots y\otimes\cdots$, where $y$ is in the $n$-th tensor factor, then one verifies
\begin{align*}
    D(x\tau^{1/2})=(yh_n-h_ny)\tau^{1/2}.
\end{align*}
In particular, if $e_{ij}^{(n)}$ for $1\leq i,j\leq 2$ denotes
matrix units in the $n$-th tensor factor, then 
\[
  D(e_{12}^{(n)}\tau^{1/2})=\lambda_ne_{12}^{(n)}\tau^{1/2},
  \qquad
  D(e_{21}^{(n)}\tau^{1/2})=-\lambda_ne_{21}^{(n)}\tau^{1/2}.
\]
Since $u_n=e_{12}^{(n)}+e_{21}^{(n)},$ we obtain
 $ D^2(u_n\tau^{1/2})=\lambda_n^2u_n\tau^{1/2}.$ 
Thus $(u_n)_{n\geq1}$ satisfies the hypothesis of Lemma \ref{lem: a criteria for existence of certain derivation}; hence $\theta_\alpha^\omega(u)$ does not converge strongly to $u$ as $\alpha\to\infty$.
\end{proof}

\begin{lemma}\label{lem: a derivation on diffuse abelian}
If $\A$ is the unique separable diffuse abelian von Neumann algebra $\mathsf A$, then there exists a trace-symmetric derivation $\delta:\dom(\delta)\subseteq L^2(\A)\to L^2(\A)$  such that $(\theta_\alpha^{(2)})^\omega$ does not converge strongly on $L^2(\A'\cap\A^\omega)$ for any $\omega\in\beta\mathbb N\setminus\mathbb N$.
\end{lemma}
\begin{proof}
 Let $(X,\mu)$ be a non-atomic probability measure space such that $\A=L^\infty(X,\mu)$.    By uniqueness,   we may identify $  (X,\mu)\cong(\mathbb T^{\mathbb N},m^{\mathbb N})$, where $m$ is the one-dimensional Lebesgue measure on the unit circle $\mathbb T$.
Choose $\lambda_n\to\infty$, and define a measure-preserving flow on $\mathbb T^{\mathbb N}$ by
\[
  T_t((z_n)_n)
  =(e^{i\lambda_n t}z_n)_n,\;\;\;\;(z_n)_n\in\mathbb T^{\mathbb N},\;\;t\in\mathbb R.
\]
Set $\alpha_t(f)=f\circ T_t $ for $f\in L^\infty(\mathbb T^{\mathbb N}, m^{\mathbb N})$, which induces a $\tau$-preserving action of $\mathbb R$ on $L^\infty(\mathbb T^{\mathbb N}, m^{\mathbb N})$, where $\tau$ is the integration with respect to $m^{\mathbb N}$.
For $n\geq1$, consider the unitaries $v_n\in L^\infty(\mathbb T^{\mathbb N}, m^{\mathbb N})$ defined by 
\[
  v_n(z_1,z_2,\ldots)=z_n.
\]
Then $ \alpha_t(v_n)=e^{i\lambda_nt}v_n,\;t\in\mathbb R $
and hence $D^2(v_n\tau^{1/2})
  =\lambda_n^2v_n\tau^{1/2}.$
Clearly  $(v_n)^\omega\in \mathcal{Z}(\A^\omega)\subseteq\A'\cap\A^\omega$ and $v_n$ satisfies the hypothesis of Lemma \ref{lem: a criteria for existence of certain derivation}. 
\end{proof}

\begin{theorem}\label{thm:existence of derivation for McDuff type algebras}
    Let $\M$ be a $\sigma$-finite von Neumann algebra either of the form $\M_1\oplus(\M_2\otimes\mathcal{R})$ or of the form $\M_1\oplus(\M_2\otimes\A)$ for some von Neumann algebras $\M_1,\M_2$ such that $\M_2\neq0$,
    where $\A$ is the unique separable diffuse abelian von Neumann algebra. There exists a densely defined closable modular derivation $\delta:\dom\delta\subseteq L^2(\M)\to L^2(\M)$ such that $(\theta_\alpha^{(2)})^\omega$ does not converge strongly on $L^2(\M'\cap\M^\omega)$ for any $\omega\in\beta\mathbb N\setminus\mathbb N$.
\end{theorem}
\begin{proof}
 Assume that $\M=\M_1\oplus(\M_2\otimes\mathcal{R}))$ for $\M_2\neq0$.     Let $\delta$  be the closable densely defined modular derivations on $L^2(\mathcal{R})$
     as constructed in Lemma \ref{lem: a derivation on the hyperfinite} and denote by  $(\theta_{\alpha})_{\alpha>0}$ the associated deformation. Define the unique modular derivation $\partial$ on $L^2(M)\cong L^2(\M_1)\oplus \left.(L^2(\M_2)\otimes L^2(\mathcal{R})\right)$
     such that the associated deformation of $\partial$ is
\[\eta_\alpha:=\id_{\M_1}\oplus(\id_{\M_2}\otimes\theta_{\alpha}),\quad \alpha>0.\] 
Clearly, then $(\eta_\alpha^{(2)})^\omega$ does not converge strongly on $L^2(\M'\cap \M^\omega)$ for any $\omega\in\beta\mathbb N\setminus\mathbb N$,  because $(1_{\M_2}\otimes\mathcal{R})'\cap (1_{\M_2}\otimes\mathcal{R})^\omega\subseteq\M'\cap\M^\omega$ and $(\theta_\alpha^{(2)})^\omega$ does not converge strongly on $\mathcal{R}'\cap\mathcal{R}^\omega$. A similar argument yields the assertion if $\M$ is of the form $\M_1\oplus(\M_2\otimes\A)$ for $\M_2\neq0$ in view of Lemma \ref{lem: a derivation on diffuse abelian}.
\end{proof}

\begin{lemma}\label{lem:hyperfinite factor absorption}
    If $\N$ is a diffuse amenable von Neumann algebra with separable predual of either type II or type III, then $\N\cong\N\overline{\otimes}\mathcal{R}$.
\end{lemma}
\begin{proof}
    By disintegration theory of  separable von Neumann algebras, it suffices to prove the result when  $\N$ is a factor. If $\N$ is an amenable factor of type II or type III$_{\lambda}$, $0<\lambda\leq 1$, then $\N\overline{\otimes}\mathcal{R}$ is also an amenable factor of the same type (see \cite[Proposition 28.4]{Str20}), so   by the uniqueness of separable amenable factors of these types (\cite{Con76, Haa87}), we have $\N\cong\N\overline{\otimes}\mathcal{R}$.
    Finally type III$_0$ amenable factors are classified by their flow of weight (\cite[Corollary 7.6]{Con76}). If $\N$ is a type III$_0$ factor, then  Flow$(\N\overline{\otimes}\mathcal{R})$=Flow$(\N)$. 
    Indeed, let $\varphi$ be a faithful normal state on $\N$, and let $\tau$ be the normalized trace on $\mathcal{R}$. Since
$\sigma_t^{\varphi\otimes\tau}
=\sigma_t^\varphi\otimes\mathrm{id}_\mathcal{R},$ there is a canonical identification
\[(\N\bar\otimes \mathcal{R})\rtimes_{\sigma^{\phi\otimes\tau}}\mathbb R
\cong (\N\rtimes_{\sigma^\phi}\mathbb R)\bar\otimes \mathcal{R}
\]under which the dual action becomes $\theta^{\N}\otimes\mathrm{id}_{\mathcal{R}}$. Thus $z\mapsto z\otimes1$ intertwines the actions on the centers, proving that the flows of weights are conjugate. Hence we have  $\N\cong \N\overline{\otimes}\mathcal{R}$.
\end{proof}

Let $\M$ be a diffuse amenable von Neumann algebra with separable predual, and let  $z_1,z_2,z_3\in\mathcal{Z}(\M)$ be the central projections such that $z_1\M$ is of type I, $z_2\M$ is of type II and  $z_3\M$ is of type III. Since $\M$ is diffuse, there are a family of Hilbert spaces $\{\K_n\}_{n\geq 1}$ and  von Neumann algebras $\N_2,\N_3$ such that 
\[z_1\M\cong\bigoplus_{n\geq 1}\mathbb B(\K_n)\otimes\A,\;\;z_2\M\cong \N_2\otimes\mathcal{R},\;\;z_3\M\cong\N_3\otimes\mathcal{R}\]
in view of Lemma \ref{lem:hyperfinite factor absorption}, where $\A$ is as in Lemma \ref{lem: a derivation on diffuse abelian}. Since one of $\mathbb B(\K_n)$, $\N_2$ or $\N_3$ must necessarily be non-zero,  the following result follows in view of Lemma \ref{lem:hyperfinite factor absorption} and Theorem \ref{thm:existence of derivation for McDuff type algebras}.

\begin{corollary}\label{cor:a derivation on diffuse amenable von Neumann algebras}
If $\M$ is a diffuse amenable von Neumann algebra with separable predual, then
there is a densely defined closable modular derivation $ \delta:\dom(\delta)\subseteq L^2(\M)\to L^2(\M)$ such that  $ (\theta_\alpha^{(2)})^\omega$
does not converge strongly to the identity on $L^2(\M'\cap\M^\omega)$ as $\alpha\to\infty$,  for every free ultrafilter $\omega\in\beta\mathbb N\setminus\mathbb N$.
\end{corollary}

\noindent
\begin{proof}[Proof of Theorem \ref{thm:dichotomy result for subalgebras}.] If $\M$ has no amenable summand,   the statement $(2)$ follows from Corollary \ref{cor:dichotomy result in amenable case}.

Otherwise, assume that there is a non-zero central  projection $z\in\mathcal{Z}(\M)$ such that $z\M$ is amenable. By Corollary  \ref{cor:a derivation on diffuse amenable von Neumann algebras},  choose $\delta:\dom(\delta)\subseteq L^2(z\M)\to L^2(z\M)$ to be a modular derivation such that the associated deformation $(\theta_\alpha^{(2)})^\omega$ does not converge strongly on $L^2((z\M)'\cap(z\M)^\omega)$.  Consider the zero derivation from $L^2((1-z)\M)$ to $ L^2((1-z)\M)\otimes L^2((1-z)\M)$, and let $\partial=\delta\oplus 0$ be the derivation from $L^2(\M)\cong L^2(z\M)\oplus L^2((1-z)\M)$ to $L^2(z\M)\oplus \left(L^2((1-z)\M)\otimes L^2((1-z)\M)\right)$.  Clearly $(z\M)'\cap(z\M)^\omega\subseteq\M'\cap\M^\omega$, so the associated deformation of $\partial$, which is $(\theta_\alpha^{(2)}\oplus \id_{L^2((1-z)\M)})^\omega$, does not converge strongly to the identity on $L^2(\M'\cap\M^\omega)$ as $\alpha\to\infty$, for any $\omega\in \beta\mathbb N\setminus\mathbb N$.

Since $z\M$ is amenable, we have $_{z\M}L^2(z\M)_{z\M}\prec {_{z\M}L^2(z\M)\otimes L^2(z\M)_{z\M}}$. It follows from Lemma \ref{lem:direct sum of bimodule and their weak containment} below that as $\M\cong z\M\oplus(1-z)\M$-correspondence, we have
\begin{align*}
    L^2(z\M)\oplus&  \left(L^2((1-z)\M)\otimes L^2((1-z)\M)\right)\\
    &\prec  \left(L^2(z\M)\otimes L^2(z\M)\right)\oplus\left(L^2((1-z)\M)\otimes L^2(1-z)\M)\right).
\end{align*}
However, the correspondence on the right hand side is an $\M$-sub-correspondence of $L^2(\M)\otimes L^2(\M)$, which means that as $\M$-correspondence we have
\begin{align*}
     L^2(z\M)\oplus  \left(L^2((1-z)\M)\otimes L^2((1-z)\M)\right)\prec L^2(\M)\otimes L^2(\M).
\end{align*}
Thus we have produced  a derivation $\partial$ which satisfies the required conditions.
\end{proof}

If $\H_i$ is a correspondence over $\M_i$ for $i=1,2$, then give $\H_1\oplus\H_2$ a bimodule structure over $\M_1\oplus\M_2$ by
\[(x_1\oplus y_1)(\xi_1\oplus\xi_2)(x_2\oplus y_2)=x_1\xi_1 x_2\oplus y_1\xi_2 y_2\]
for $x_i,y_i\in \M_i, \xi_i\in \H_i$ for $i=1,2$.
\begin{lemma}\label{lem:direct sum of bimodule and their weak containment}
    Let $\H_i,\K_i$ be Hilbert correspondences over a von Neumann algebra $\M_i$ for $i=1,2$. If $\H_i\prec\K_i$ as $\M_i$-correspondence, we have $\H_1\oplus\H_2\prec \K_1\oplus\K_2$ as $\M_1\oplus\M_2$-correspondence.
\end{lemma}
\begin{proof}
 Put $\M=\M_1\oplus\M_2$, $\H=\H_1\oplus\H_2$ and $\K=\K_1\oplus\K_2$, and denote by $\pi_\H,\pi_\K$ the bimodule action of $\M$ on $\H,\K$ respectively. If $T\in \M\otimes_{\alg}\M^{\op}$, then $T=\sum_{i,j=1}^2T_{ij}$ for some $T_{ij}\in \M_i\otimes_{\alg}\M_j^{\op},$ $1\leq i,j\leq 2$. Note that $T_{12}$ and $T_{21}$ act as $0$ on $\H$ and $\K$, that is, $\pi_\H(T_{12})=0=\pi_\H(T_{21})$ and same for $\pi_\K$. Therefore, we have
 \begin{align*}
   \|\pi_\H(T)\|=\norm{\pi_{\H_1}(T_{11})\oplus\pi_{\H_2}(T_{22})}=\max_{i=1,2}\norm{\pi_{\H_i}(T_{ii})}\leq \max_{i=1,2}\norm{\pi_{\K_i}(T_{ii})}=\norm{\pi_\K(T)}
 \end{align*}
which proves that $\H\prec\K$ as $\M$-correspondence.
\end{proof}

\section{Coarse rigidity and  spectral gap}\label{sec:coarse rigidity and spectral gap}
In this section,  we prove a generalization of Peterson's result \cite[Corollary 4.6]{Pet09a} regarding $L^2$-rigidity and property $\Gamma$. We show that combining coarse rigidity with either mixing or compactness gives spectral gap properties.

\begin{definition}[\cite{Ioa15}]
    An inclusion $\Q\subseteq\M$ of von Neumann algebras with expectation is said to have {\em $w$-spectral gap} if $\Q'\cap\M^\omega=\mathbb (\Q'\cap\M)^\omega$ for some free ultrafilter $\omega\in\beta\mathbb N\setminus\mathbb N$.
\end{definition}

Recall that an $\M$-correspondence $\H$ is called {\em mixing} if for any $\xi,\eta\in \H$ and any uniformly bounded $a_i\in \M$ such that $a_i\to 0$ weakly, we have
\begin{align*}
  \lim_i  \sup_{\|x\|\leq 1}|\langle\xi, a_i\eta x\rangle|=0=\lim_i\sup_{\|x\|\leq 1}|\langle\xi, x\eta a_i\rangle|.
\end{align*}
The coarse correspondence $L^2(\M)\otimes L^2(\M)$ and its amplifications are natural examples of mixing $\M$-correspondences. The proof of the following results follows the same arguments as in  \cite[Theorem 4.5]{Pet09a}.

\begin{proposition}
Let $\Q\subseteq\M$ be an  inclusion of irreducible factors  such that $\Q$ has separable predual and $\Q$ is nonamenable. Let $\phi$ be a normal faithful state on $\M$ such that $\Q$ is globally invariant under $\sigma^\phi$. Fix $\omega\in\beta\mathbb N\setminus\mathbb N$. If the inclusion $\Q\subseteq\M$ does not have $w$-spectral gap, then for every $\varphi$-modular derivation $\delta:\dom\delta\subseteq L^2(\M)\to\H$ with $\H\prec L^2(\M)\otimes L^2(\M)$ and $\H$ mixing,
  $(\theta_\alpha^{(2)})^\omega$ converges uniformly to the identity on $(\Q^\omega)_1$  as $\alpha\to\infty$.
\end{proposition}
\begin{proof}
Let $\delta:\dom(\delta)\subseteq L^2(\M)\to\H$ be a $\varphi$-modular derivation with  $\H\prec L^2(\M)\otimes L^2(\M)$. In view of Proposition \ref{prop: uniform convergence of GNS-symmetric maps}, the result follows once we show
\[\lim_{\alpha\to\infty}\sup_{x\in \Q, \|x\|\leq 1}\|\tilde\delta_\alpha(x\varphi^{1/2})\|=0.\]
 For simplicity, write $\psi$ for the state $\varphi^\omega$ on $\M^\omega$.  
   It follows from  Theorem \ref{main_theorem}  that $(\theta_\alpha^{(2)})^\omega$ converges strongly  to the identity on $L^2(\Q'\cap\M^\omega)$ as $\alpha\to\infty$. In other words, 
   \[\tilde\delta_\alpha^\omega(z\psi^{1/2})\to 0 \;\;\mbox{ for all }\;z\in \Q'\cap\M^\omega.\] 
   Since $\Q\subseteq\M$ does not have $w$-spectral gap, we  have $\Q'\cap\M^\omega\neq\mathbb C$; hence by \cite[Lemma 2.5]{HR15}, $\Q'\cap(\M^\omega)^\psi$ is diffuse and there is a unitary $u\in \Q'\cap(\M^\omega)^\psi$ such that $E_\M(u)=0$.
 In particular, for all $\alpha>0$, we have $\zeta_\alpha(u_n)\to 0$ weakly as $n\to\omega$.
Fix $x\in \Q$ with $\|x\|_\infty\leq 1$. Since $\H$ is mixing, 
we have for all $\alpha>0$
\begin{align*}
    \lim_{n\to\omega}\langle\zeta_\alpha(u_n)\tilde\delta_\alpha(x\phi^{1/2})\zeta_\alpha(u_n^\ast),\tilde\delta_\alpha(x\phi^{1/2})\rangle=0
\end{align*}
which further yields
\begin{align*}
     \norm{\tilde\delta_\alpha(x\varphi^{1/2})}&\leq \lim_{n\to\omega}\norm{\zeta_\alpha(u_n^\ast )\tilde\delta_\alpha(x\phi^{1/2})\zeta_\alpha(u_n)-\tilde\delta_\alpha(x\varphi^{1/2})}
\end{align*}
By the technical  Lemma \ref{lem: the main technical lemma}, it follows for all $\alpha>0$ that
 \begin{align*}
        \norm{\zeta_\alpha&(u_n^\ast )\tilde\delta_\alpha(x\phi^{1/2})\zeta_\alpha(u_n)-\tilde\delta_\alpha(x\phi^{1/2})}\\
     &\leq \norm{\zeta_\alpha(u_n^\ast )\tilde\delta_\alpha(x\phi^{1/2})\zeta_\alpha(u_n)-\zeta_\alpha(u_n^\ast )\tilde\delta_\alpha(x\phi^{1/2}u_n)}\\
     &\quad+\norm{\zeta_\alpha(u_n^\ast )\tilde\delta_\alpha(x\phi^{1/2}u_n)-\tilde\delta_\alpha(u_n^*x\varphi^{1/2}u_n)}+\norm{\tilde\delta_\alpha(u_n^*x\varphi^{1/2}u_n)-\tilde\delta_\alpha(x\phi^{1/2})}\\
     &\leq 10\|\tilde\delta_\alpha(\varphi^{1/2}u_n)\|^{1/2}+10\|\tilde\delta_\alpha(u_n^\ast\varphi^{1/2})\|^{1/2}+\norm{\tilde\delta_\alpha(u_n^*x\varphi^{1/2}u_n)-\tilde\delta_\alpha(x\phi^{1/2})}
    \end{align*}
 Thus using $xu^\ast=u^\ast x$ and  $u\psi^{1/2}=\psi^{1/2}u$, we obtain
 \begin{align*}
  \|\tilde\delta_\alpha(x\phi^{1/2})\|&\leq  \lim_{n\to\omega}\norm{\zeta_\alpha(u_n^\ast )\tilde\delta_\alpha(x\phi^{1/2})\zeta_\alpha(u_n)-\tilde\delta_\alpha(x\varphi^{1/2})}\\
  &\leq 20\|\tilde\delta_\alpha^\omega(u\psi^{1/2})\|^{1/2}.
 \end{align*}
 Since the right hand side is independent of $x$ and converges to $0$ as $\alpha\to\infty$, we obtain the required result.
\end{proof}

If $\Q=\M$ in the above proposition, then $\M'\cap\M^\omega=\mathbb C$ is equivalent to fullness of $\M$.

\begin{corollary}\label{cor:fullness and rigidity}
   If $\M$ is  a nonamenable  factor with separable predual which is not full, in particular when $\M$ is of type III$_0$,  then $\M$ is rigid with respect to modular derivations taking values in mixing weakly coarse correspondences, that is, for every modular derivation $\delta:\dom\delta\subseteq L^2(\M)\to\H$ with $\H\prec L^2(\M)\otimes L^2(\M)$ and $\H$ mixing,
  $(\theta_\alpha^{(2)})^\omega$ converges uniformly to the identity on $(\M^\omega)_1$  as $\alpha\to\infty$.
\end{corollary}
We record the following corollary for tracial von  Neumann algebras which  is a relative version of spectral gap result of \cite[Corollary 4.6]{Pet09a}.

\begin{corollary}\label{cor:spectral gap in tracial case for mixing derivation}
    Let $\M$ be a tracial von Neumann algebra, and let $\Q$ be a nonamenable irreducible subalgebra of $\M$ with separable predual. If $\Q\subseteq\M$ does not have $w$-spectral gap, then $\Q\subseteq\M$ is  rigid with respect to modular derivation taking values in mixing  weakly coarse correspondences.
\end{corollary}

In Corollary \ref{cor:fullness and rigidity}, mixing of $\H$ is a necessary assumption.  Indeed, the derivation as in Example \ref{ex:comparision of different rigidity notion} provides a modular derivation $\partial$ on $\dom\delta\subseteq \ell^2(\mathbb F_2)\otimes L^2(\mathcal{R})$ into a weakly coarse correspondence $\H=\left(\ell^2(\mathbb F_2)\otimes\ell^2(\mathbb F_2)\right)^{\oplus\infty}\otimes L^2(\mathcal{R})$ such that the corresponding deformation $(\eta_\alpha^{(2)})^\omega$ does not 
converge strongly to the identity on $L^2((L\mathbb F_2\bar\otimes\mathcal{R})^\omega)$. Clearly, $\H$ is not mixing.

However, mixing condition can be removed if we assume compactness as following result shows.

\begin{proposition}\label{prop:compactness implies spectral gap}
    Let $\M$ be a $\sigma$-finite von Neumann algebra, and let $\B\subseteq\M$ be an expected inclusion such that $\B$ is finite dimensional. Let $\delta:\dom\delta\subseteq L^2(\M)\to\H$ be a modular derivation  with $\H\prec L^2(\M)\otimes L^2(\M)$ such that $\delta^\ast\bar\delta$ has compact resolvent. If $\Q\subseteq\M$ is an inclusion with expectation such that $\Q$ has no amenable summand, then $\Q$ has $w$-spectral gap.
\end{proposition}
\begin{proof}
    Since $\B$ is finite-dimensional, compactness of the resolvent of $\delta^\ast\bar\delta$ is same as it compactness relative to $\B$. Let $\omega\in \beta\mathbb N\setminus\mathbb N$. By Corollary \ref{cor:dichotomy result in amenable case}, the  deformation $(\theta_\alpha^{(2)})^\omega$ associated to $\delta$ converges strongly on $L^2(\Q'\cap\M^\omega)$ as $\alpha\to\infty$. By the proof of Proposition \ref{prop:convergence of deformation with compact resolvent gives atomicity}, it follows that $\Q'\cap\M=\Q'\cap\M^\omega=(\Q'\cap\M)^\omega$, which shows that $\Q$ has $w$-spectral gap property.
\end{proof}

The following is a direct corollary of Proposition \ref{prop:compactness implies spectral gap}. This is also a consequence of Proposition \ref{prop: compactness implies solidity} below, because nonamenable solid factors are always full.
\begin{corollary}\label{cor:compactness implies fullness}
     Let $\M$ be a $\sigma$-finite nonamenable  factor. If there is  a modular derivation $\delta:\dom\delta\subseteq L^2(\M)\to\H$  with $\H\prec L^2(\M)\otimes L^2(\M)$ such that $\delta^\ast\bar\delta$ has compact resolvent, then $\M$ is full.
\end{corollary}

\section{Coarse rigidity and primeness}\label{sec:coarse rigidity and primeness}
In this section, we prove  results on primeness of nonamenable factors satisfying either some compactness or some mixing assumptions.

Recall that a diffuse von Neumann algebra $\M$ is  {\em prime} if whenever we have a tensor product decomposition $\M\cong\M_1\bar\otimes\M_2$,  one of $\M_i$ must be of type I. 

The following result proves primeness for general nonamenable factors when the resolvent of the derivation is  compact relative to a type I von Neumann algebra (in particular when $\B$ is abelian).

\begin{theorem}\label{thm:primeness for type I base and compact resolvent}
    Let $\M$ be a  nonamenable factor with separable predual and let $\B\subseteq \M$ be an inclusion of subalgebras with expectation such that $\B$ is of type I. If there is a modular derivation $\delta:\dom\delta\subseteq L^2(\M)\to \H$ such that $\H\prec L^2(\M)\otimes_\B L^2(\M)$ and $\delta^\ast\bar\delta$ has compact resolvent relative to $\B$, then $\M$ is prime.
\end{theorem}
\begin{proof}
 Let $\omega\in \beta\mathbb N\setminus\mathbb N$, and let $\theta_\alpha$ be the associated deformation of $\delta$.   Assume to the contrary  that $\M=\M_1\bar\otimes\M_2$ with both $\M_1$ and $\M_2$ nontype I diffuse factors and that $\M_1$ is nonamenable.  Since $\B$ is amenable, we have
 \[{}_\M\H_\M{}\prec {}_\M L^2(\M)\otimes_\B L^2(\M)_{\M}{}\prec{}_\M L^2(\M)\otimes L^2(\M)_\M.\] It follows from Corollary \ref{cor:dichotomy result in amenable case} that $(\theta_\alpha^{(2)})^\omega$ converges strongly on $L^2(\M_1'\cap\M^\omega)$. Since $\M_2$ is diffuse and not of type I, it contains a type II von Neumann algebra $\N$ with expectation (in fact, one can take $\N=\M_2^\phi$ for some normal faithful state). Since $\N^\omega\subseteq\M_2^\omega\subseteq\M_1'\cap\M^\omega$, it follows that $(\theta_\alpha^{(2)})^\omega$ converges strongly on $L^2(\N^\omega)$ as $\alpha\to\infty$. It follows from Proposition \ref{prop:convergence of deformation with compact resolvent for tracial Q} that $\N\preceq_\M\B$ in the sense of Definition \ref{def:intertwining-by-bimodule}. However this is impossible as $\N$ is type II while $\B$ is type I, so a corner of $\N$ which is always type II cannot embed in a corner of $\B$. 
\end{proof}

\subsection{Primeness in tracial settings}
We now specialize to results on primeness in tracial settings, where we will assume that the gradient correspondence of a modular derivation is mixing.
Arguing  as in the proof of \cite[Theorem 4.5]{Pet09a}, one can prove the following theorem.

\begin{theorem}
Let $\M$ be a  tracial von Neumann algebra, let $\omega\in \beta\mathbb N\setminus\mathbb N$,  and let $\delta:\dom\delta\subseteq L^2(\M)\to \H$ be a modular derivation with the associated deformation $\theta_\alpha^{(2)}$, $\alpha>0$, such that $\H$ is mixing and $\H\prec L^2(\M)\otimes L^2(\M)$. If $\Q\subseteq\M$ is an inclusion of  subalgebras such that $\Q$ is diffuse and  $(\theta_\alpha^{(2)})^\omega$ converges strongly on $L^2(\Q^\omega)$, then $(\theta_\alpha^{(2)})^\omega$ converges pointwise on $L^2(\mathsf P^\omega)$ where $\mathsf P=\mathcal{N}_\M(\Q)''$ is the von Neumann algebra generated by the normalizer of $\Q$.
\end{theorem}
\begin{proof}
 Because we are in the tracial setting, Proposition \ref{prop: uniform convergence of GNS-symmetric maps} says  that strong convergence of $(\theta_\alpha^{(2)})^\omega$ on $L^2(\Q^\omega)$ is same as uniform convergence on the unit ball of $\Q^\omega$ in the $L^2$-norm. So  the proof follows the same argument as that of \cite[Theorem 4.5]{Pet09a} by using  the mixing condition of $\H$. 
\end{proof}

Similar to the primeness result of \cite[Corollary 4.6]{Pet09a}, one can prove the following result for nonamenable II$_1$ factors.

\begin{theorem}\label{cor:primeness for nonameble tracial algebra with mixing derivation}
    Let $(\M,\tau)$ be a nonamenable II$_1$ factor such that $\M$ is not prime. Then $\M$ is coarsely rigid with respect to modular derivations taking values in weakly coarse mixing correspondences.
\end{theorem}
\begin{proof}
Let $\delta:\dom\delta\subseteq L^2(\M)\to \H$ be a modular derivation  such that $\H$ is mixing and $\H\prec L^2(\M)\otimes L^2(\M)$.  Let $\M=\M_1\bar\otimes\M_2$ be a tensor product decomposition of $\M$ such that $\M_1,\M_2$ are II$_1$ factors such that $\M_1$ is  nonamenable. By Theorem \ref{main_theorem} and Proposition \ref{prop:equivalent conditions for L2-rigidity}, the associated deformation $(\theta_\alpha^{(2)})^\omega$ of $\delta$ converges uniformly on the unit ball of $\M_2^\omega$ as $\alpha\to\infty$, for any $\omega\in\beta\mathbb N\setminus\mathbb N$. Since $\H$ is mixing, by a similar argument as in \cite[Theorem 4.6]{Pet09a}, one can prove that $\theta_\alpha$ converges uniformly on the unit ball of the normalizer of $\M_2$ which is nothing but $\M$.
\end{proof}

\begin{theorem}\label{cor: tensor product of nonamenable factors are coarse rigid}
    Let $\M_1$ and $\M_2$ be two nonamenable II$_1$ factors. Then $\M_1\bar\otimes\M_2$ is coarsely rigid.
\end{theorem}
\begin{proof}
 Write $\M=\M_1\bar\otimes\M_2$, and   let $\delta:\dom\delta\subseteq L^2(\M)\to \H$ be a modular derivation in a weakly coarse $\M$-$\M$ correspondence $\H$ with the associated deformation $(\theta_\alpha^{(2)})_{\alpha>0}$. Let $\omega\in\beta\mathbb N\setminus\mathbb N$. By Theorem \ref{main_theorem}, $(\theta_\alpha^{(2)})^\omega$ converges strongly to the identity on $L^2(\M_1^\omega\bar\otimes1)$ and on $L^2(1\otimes\M_2^\omega)$ as $\alpha\to\infty$. By Lemma \ref{lem:rigidity on commuting subalgebras} below and Proposition \ref{prop:equivalent conditions for L2-rigidity}, it follows that $(\theta_\alpha^{(2)})^\omega$ converges strongly to the identity on $L^2(\M^\omega)$ as $\alpha\to\infty$.
\end{proof}

\begin{lemma}\label{lem:rigidity on commuting subalgebras}
 Let $(\M,\tau)$ be a finite von Neumann algebra, and   let $\A_1,\A_2\subseteq \M$ be inclusion of commuting subalgebras generating $\M$. If $T_{\alpha}$ is a net of $\tau$-preserving completely positive maps such that $T_{\alpha}$ converges uniformly to the identity on the unit ball of both $\A_1,\A_2$ in the $L^2$-norm, then $T_{\alpha}$ converges uniformly on the unit ball of $\M$ in the $L^2$-norm.
\end{lemma}
\begin{proof}
Because we are in the tracial setting, by usual convexity argument (see \cite{Pop86}, \cite[Theorem 4.5]{Pet09a}),    it suffices to prove that $\sup_{u\in \mathcal{U}(\M)}\|T_{\alpha}(u)-u\|_{2,\tau}\to 0$ along $\alpha$. Let $(\H_{\alpha}, \xi_{\alpha})$ be the minimal Stinespring dilation for $T_{\alpha}$, that is, $\H_{\alpha}$ is an $\M$-$\M$ correspondence and $\xi_{\alpha}\in \H_{\alpha}$ is left bounded such that $T_{\alpha}(x)=L_\tau(\xi)^\ast xL_\tau(\xi)$. For all $x\in \M$, we have
\begin{align*}
\|T(x)-x\|_{2,\tau}\leq     \|x\xi_{\alpha}-\xi_{\alpha}x\|\quad\mbox{ and } \quad\|x\xi_{\alpha}-\xi_{\alpha}x\|^2\leq  2\|x\|_{2,\tau}\|T_{\alpha}(x)-x\|_{2,\tau}
\end{align*}
In particular, the pointwise convergence of $T_{\alpha}$ implies that 
   $ \|x\xi_{\alpha}-\xi_{\alpha}x\|\to 0,\;\forall x\in \M.$
Consider the $\A_1$ and $\A_2$ central vectors of $\H_{\alpha}$ given  by 
\[\K_{\alpha,j}:=\{\xi\in \H_{\alpha}; a\xi=\xi a\;\forall a\in \A_j\},\;\;j=1,2.\]
and denote by $P_{\alpha,j}$ the respective projections. Then $P_{\alpha,1}$ and $P_{\alpha,2}$ commute because $\A_1$ and $\A_2$ commute. Indeed,  for any $v\in \mathcal{U}(\M)$, if we define the unitaries $U_v$ on $\H$ by $\eta\mapsto v\eta v^\ast$, then $\K_{\alpha,j}$  is precisely the subspace fixed by the unitaries $U_u$ for $u\in \mathcal{\A_j}$. However $U_u$ and $U_v$ commute whenever $u\in \mathcal{U}(\A_1)$ and $\mathcal{U}(\A_2)$ and hence so does the projection onto their respective fixed point space. Set
\[\eta_{\alpha}=P_{\alpha,1}P_{\alpha,2}\xi\in \K_{\alpha,1}\cap\mathsf K_{\alpha,2}.\]
Then $a\eta_\alpha=\eta_{\alpha} a$ for all $a\in \A_1$ or $a\in\A_2$, and since $\M=\A_1\vee\A_2$, it follows that $x\eta_{\alpha}=\eta_{\alpha}x$ for all $x\in \M$. Moreover, we have
\begin{align*}
    \|\xi_{\alpha}-\eta_{\alpha}\|&\leq \|(1-P_{\alpha,1})\xi_{\alpha}\|+\|(1-P_{\alpha,2})\xi_{\alpha}\|\\
    &=\sup_{u\in \mathcal{U}(\A_1)}\|T_{\alpha}(u)-u\|_{2,\tau}+\sup_{u\in \mathcal{U}(\A_2)}\|T_{\alpha}(u)-u\|_{2,\tau}
\end{align*}
which converges to $0$ by assumption.  Thus we arrive at the following:
\begin{align*}
    \sup_{u\in \mathcal{U}(\M)}\|T_{\alpha}(u)-u\|_{2,\tau}&\leq \sup_{u\in \mathcal{U}(\M)}\|u\xi_{\alpha}-\xi_{\alpha} u\|\\
    &\leq \sup_{u\in \mathcal{U}(\M)}\|u(\xi_{\alpha}-\eta_{\alpha})+(\eta_{\alpha}-\xi_{\alpha})u\|\\
    &\leq 2\|\xi_{\alpha}-\eta_{\alpha}\|
\end{align*}
and the right hand goes to $0$ along $\alpha$.
\end{proof}


\section{Relative solidity}\label{sec:solidity}

In his fundamental paper, Ozawa \cite{Oza04}  introduced a remarkable rigidity property of von Neumann algebras, called  solidity.  He shows that
 group von Neumann algebras of biexact groups
are solid. Any nonamenable solid factors are automatically full and prime.  In \cite{Oz06},  he extended the result to relative solidity, showing that crossed
products by trace preserving actions on abelian algebras are solid relative to the base algebra. This has been further generalized to type III von Neumann algebras \cite{Mar17, Iso25}.

In this section, we prove some general results on $\omega$-solidity and stable solidity of von Neumann algebras with separable predual admitting modular derivations with  compact resolvent and satisfying the hypotheses of Theorem \ref{main_theorem}.  We also introduce new notions of relative $\omega$-solidity and relative stable solidity. We apply our general results to examples such as amalgamated free products, operator-valued free Araki-Woods factors and cocycle crossed products.

\subsection{On $\omega$-solidity}

Recall from \cite{HR15} that if $\omega$ is a cofinal ultrafilter, then a diffuse von Neumann algebra $\M$ is called {\em $\omega$-solid} if whenever $\Q$ is a von Neumann subalgebra of $\M$ with expectation such that $\Q'\cap\M^\omega$ is diffuse, we have that $Q$ is amenable.

\begin{proposition}\label{prop: compactness implies solidity}
    If a von Neumann algebra $\M$ admits a modular derivation $\delta\colon \text{dom}(\delta)\subseteq L^2(\M)\to \H $ with $\H \prec L^2(\M)\otimes L^2(\M)$ such that $\delta^\ast\bar\delta$ has compact resolvent then $\M$ is $\omega$-solid for every cofinal ultrafilter $\omega$ on a directed set  and hence $\M$ is solid.
\end{proposition}
\begin{proof}
Fix a cofinal ultrafilter $\omega$ on a directed set.
Let $\Q$ be a subalgebra of $\M$ with expectation such that $\Q$ is not amenable. Let $p\in \mathcal{Z}(\Q)$ be a non-zero projection such that $p\Q$ has no amenable summand.  By  Theorem \ref{main_theorem}, $(\theta_\alpha^{(2)})^\omega$ converges strongly  on $L^2(p(\Q'\cap\M^\omega))$. Since $\theta_\alpha^{(2)}$ is compact for each $\alpha>0$, $p(\Q'\cap\M^\omega)$ is atomic  by Proposition \ref{prop:convergence of deformation with compact resolvent gives intertwining}; hence $\Q'\cap\M^\omega$ is not diffuse.
\end{proof}

More generally, we propose the following definition of relative $\omega$-solidity.  Let $\omega$ be a cofinal ultrafilter on a directed set and $\B\subseteq\M$ be an inclusion of von Neumann algebras  with expectation. 

We say that $\M$ is {\em $\omega$-solid relative to $\B$} if whenever $\Q\subseteq 1_\Q\M1_\Q$ is an inclusion of subalgebras with expectation such that $\Q$ is completely non-amenable relative to $\B$, then for every non-zero projection $p\in \Q'\cap1_\Q\M^\omega1_\Q$, there is  no unitary $u\in \mathcal{U}((p\Q)'\cap p\M^\omega p))$ such that $E_{\B^\omega}(x^\ast u x) =0$ for all $x\in p\M$.


When $\B=\mathbb C$, $\omega\in\beta\mathbb N\setminus\mathbb N$ and $\M$ has separable predual, then the two definitions of $\omega$-solidity agree. Indeed, if $\M$ is $\omega$-solid in the sense of \cite{HR15}, then for  an expected inclusion $\Q\subseteq\M$ such that  $\Q$ is completely non-amenable,  we must have that  $\Q'\cap\M^\omega$ is atomic. Indeed, if there is a projection $p\in\mathcal{Z}(\Q'\cap\M^\omega)$ such that $p(\Q'\cap\M^\omega)$ is diffuse, then by  the proof \cite[Theorem A]{HR15}, there is a unitary $u\in p(\Q'\cap\M^\omega)p$ such that $E_{\M}(u)=0$, which is the same as saying that $E_{\mathbb C^\omega}(x^\ast u x)=0$ for all $x\in p\M$.

\begin{theorem}\label{thm:omega-solidity}
 Let $\B\subseteq\M$ be an inclusion of  von Neumann algebras  with expectation  such that $\M$ has separable predual, and   let $\delta:\dom\delta\subseteq L^2(\M)\to\H$ be a densely defined closable modular derivation such that $\B\subseteq\Null(\bar\delta)$. Assume that there exist $\B$-$\B$ correspondences $\K$ and $\K_t$ for $t\geq 0$ such that $_\M\H_\M\prec {_\M L^2(\M)}\otimes_\B\K\otimes_\B L^2(\M)_\M$ and $_\M(\H_{\Phi_t})_\B\prec {_\M L^2(\M)}\otimes_\B (\K_t)_\B$ for $t\geq 0$. If  $\delta^\ast\bar\delta$ has compact resolvent relative to $\B$, then $\M$ is $\omega$-solid relative to $\B$ for any cofinal ultrafilter $\omega$.
\end{theorem}
\begin{proof}
 Let $\omega$ be a cofinal ultrafilter on a directed set $I$.    Let $\Q$ be a subalgebra of $1_\Q\M1_\Q$ with expectation $\Q$ is completely nonamenable relative to $\B$. 
 By Theorem \ref{main_theorem}, $(\theta_\alpha^{(2)})^\omega$ converges strongly to the identity on $L^2(\Q'\cap1_\Q\M^\omega1_\Q)$ as $\alpha\to\infty$. By Proposition \ref{prop:convergence of deformation with compact resolvent gives intertwining on relative commutatant}, we have the result.
\end{proof}

\subsection{On stable solidity}
 Let $\B\subseteq\M$ be an inclusion of  von Neumann algebras with expectation and with separable predual, and let $\omega\in \beta\mathbb N\setminus\mathbb N$. Following \cite{DV25},  we say that $\M$ is {\em stably solid relative to $\B$ along $\omega$} if for every  von Neumann algebra $\N$ with separable predual and a von Neumann subalgebra $\Q\subseteq (\M\overline{\otimes}\N)$ with expectation such that $\Q'\cap (\M\overline{\otimes}\N)\npreceq_{\M\overline{\otimes}\N}^\omega\B\otimes\N$, we have that $\Q$ is amenable relative to $\B\otimes\N.$

When $\B=\mathbb C$ and $\M$ is a finite von Neumann algebra, then this definition agrees with that in \cite{DV25} in view of Proposition \ref{prop: comparion of two inytertwining notion}.

\begin{theorem}\label{thm:stable solidity}
    Let $\M$ be a   von Neumann algebra with separable predual and with a densely defined closable modular  derivation $\delta:\dom(\delta)\subseteq L^2(\M)\to\H$, let $\B\subseteq\Null(\bar\delta)$. Assume that there exist $\B$-$\B$ correspondences $\K$ and $\K_t$ for $t\geq 0$ such that $_\M\H_\M\prec {_\M L^2(\M)}\otimes_\B\K\otimes_\B L^2(\M)_\M$, $_\M(\H_{\Phi_t})_\B\prec{_\M L^2(\M)}\otimes_\B(\K_t)_\B$ and that $\delta^\ast\bar\delta$ has a compact resolvent relative to $\B$. Then $\M$ is stably solid relative to $\B$ along $\omega$, for any $\omega\in \beta\mathbb N\setminus\mathbb N$.
\end{theorem}
\begin{proof}
  Let $\N$ be a $\sigma$-finite von Neumann algebra with faithful normal state $\psi$ and $\Q\subseteq \M\bar\otimes\N$  be a subalgebra with expectation such that 
\[\Q'\cap(\M\overline{\otimes\N})\npreceq^\omega_{\M\overline{\otimes}\N} \B\otimes\N.\] Consider the tensor product derivation $\delta_1:=\delta\otimes \mathrm{id}$ on $\dom(\delta)\odot\N_\infty \psi^{1/2}$ taking values in $\H_1:=\H\otimes L^2(\N)$ whose semigroup is given by $e^{-t\delta^\ast\bar\delta}\otimes 1.$ Then $\delta_1$ is a modular derivation and $\B\overline\otimes\N\subseteq\Null(\bar\delta_1).$ Moreover, $\delta_1^\ast\bar\delta_1$ has compact resolvent relative to $\B\bar\otimes\N$. Give $\K\otimes L^2(\N)$ the obvious $\B\overline{\otimes}\N$-bimodule structure. Then  as an $\M\bar\otimes\N$ correspondence, we have 
  \begin{align*}
       \H_1&\prec (L^2(\M)\otimes_\B \K\otimes_\B L^2(\M))\otimes L^2(\N)\\
       &\cong (L^2(\M)\otimes L^2(\N))\otimes_{\B\overline{\otimes}\N}(\K\otimes L^2(\N))\otimes_{\B\overline{\otimes}\N} (L^2(\M)\otimes L^2(\N)).
  \end{align*}
  Now if $\Q$ is not amenable relative to $\B\bar\otimes\N$, there is a projection $q\in \mathcal{Z}(\Q'\cap (\M\bar\otimes\N))$ such that $q\Q$ has no amenable summand relative to $\B\bar\otimes\N$. This means by Theorem \ref{main_theorem} that $(\theta_\alpha^{(2)}\otimes 1_\N)^\omega$ converge strongly on $q\Q'\cap (q(\M\overline{\otimes}\N)q)^\omega$ for every cofinal ultrafilter $\omega$ on a directed set.  By Proposition \ref{prop:convergence of deformation with compact resolvent gives intertwining}, it follows that $\Q'\cap(\M\overline{\otimes\N})\preceq_{\M\overline{\otimes}\N}^\omega\B\bar\otimes\N$, a contradiction.
\end{proof}

In view of Proposition \ref{prop: comparion of two inytertwining notion}, the following stable solidity result can be proven for finite von Neumann algebras. In fact, one can remove the separability assumption.
\begin{corollary}\label{thm:stable solidity_finite case}
    Let $\M$ be a tracial  von Neumann algebra with   a densely defined closable modular  derivation $\delta:\dom(\delta)\subseteq L^2(\M)\to\H$, let $\B=\Null(\bar\delta)$. Assume that there exist $\B$-$\B$ correspondences $\K$ and $\K_t$ for $t\geq 0$ such that $_\M\H_\M\prec {_\M L^2(\M)}\otimes_\B\K\otimes_\B L^2(\M)_\M$, $_\M(\H_{\Phi_t})_\B\prec{_\M L^2(\M)}\otimes_\B(\K_t)_\B$ and that $\delta^\ast\bar\delta$ has a compact resolvent relative to $\B$. Then for every finite von Neumann algebra $\N$ and an unital inclusion $\Q\subseteq (\M\overline{\otimes}\N)$  such that $\Q'\cap (\M\overline{\otimes}\N)\npreceq_{\M\overline{\otimes}\N}\B\otimes\N$, we have that $\Q$ is amenable relative to $\B\otimes\N.$
\end{corollary}

We also explicitly state the result when $\B$ is amenable.

\begin{corollary}\label{thm:stable solidity_amenable_case}
    Let $\B\subseteq\M$ be an inclusion of   von Neumann algebras with separable predual and with expectation such that $\B$ is amenable. If  $\delta:\dom\delta\subseteq L^2(\M)\to\H$ is a densely defined closable modular  derivation  such that $\H\prec L^2(\M)\otimes_\B L^2(\M)$ and $\delta^\ast\bar\delta$ has a compact resolvent relative to $\B$, then $\M$ is stably solid relative to $\B$ along $\omega$ for $\omega\in \beta\mathbb N\setminus\mathbb N$.
\end{corollary}

\section{A few examples}\label{chp:Examples}

In this  final Section, we present four classes of examples for which we verify the hypotheses of Theorem \ref{main_theorem} and then apply the general solidity results from the previous sections.

\subsection{Finite-index inclusions and amalgamated free products}

Let $\B\subseteq \M$ be a unital inclusion of von Neumann algebras and $\phi$ a faithful normal state on $\M$ such that $\B$ is globally invariant under $\sigma^\phi$. Suppose that $\mathrm{id}_{L^2(\M)}\in\mathbb K(\B\subseteq\M,\phi)$. If $\M$ is a $\mathrm{II}_1$ factor, this condition is equivalent to $[\M:\B]<\infty$ by \cite[Proposition 1.3]{PP86}.

Let $\phi_\B$ denote the restriction of $\phi$ to $\B$. We have a bounded densely defined modular derivation $\delta\colon \M_\infty\phi^{1/2}\to L^2(\M)\otimes_\B L^2(\M)$ given by $\delta(a)=a\otimes_{\phi_\B}\phi^{1/2}-\phi^{1/2}\otimes_{\phi_\B}a$. The associated quantum Markov semigroup is given by $\Phi_t=e^{-t}\mathrm{id}_\M+(1-e^{-t})E_\B\in\mathbb K(\B\subseteq \M,\phi)$ and the Stinespring correspondence is $\H_{\Phi_t}=L^2(\M)\otimes_\B L^2(\M)$. Thus the solidity results from Section~\ref{sec:solidity} are applicable.

\subsubsection{Amalgamated free products}
Examples of infinite index can be created using amalgamated free products. The setup is as follows. Let $\M_1$, $\M_2$ be $\sigma$-finite von Neumann algebras with a common von Neumann subalgebra $\B$ and assume that there exists a faithful normal conditional expectation $E_j\colon \M_j\to \B$ for $j\in \{1,2\}$. Let $\phi_\B$ be a faithful normal state on $\B$ and let $\phi_j=\phi_\B\circ E_j$ for $j\in\{1,2\}$.

Let $L^2_0(\M_j)=L^2(\M_j)\ominus \overline{\B\phi_j^{1/2}}$, which is a $\B$-$\B$ correspondence, and let $\mathbf J_n=\{\mathbf j\in \{1,2\}^n\mid j_k\neq j_{k+1}\text{ for }1\leq k\leq n-1\}$. Let
\begin{align*}
    \H&=L^2(\B)\oplus\bigoplus_{n\geq 1}\bigoplus_{\mathbf j\in \mathbf J_n}L^2_0(\M_{j_1})\otimes_\B\dots\otimes_\B L^2_0(\M_{j_n})
    \intertext{and}
    \H(j)&=L^2(\B)\oplus \bigoplus_{n\geq 1}\bigoplus_{\substack{\mathbf j\in\mathbf J_n\\j_1\neq j}}L^2_0(\M_{j_1})\otimes_\B\dots\otimes_\B L^2_0(\M_{j_n})
\end{align*}
for $j\in \{1,2\}$. With the canonical isomorphism $_\B L^2(\B)\otimes_\B \K_\Q\cong _\B\K_\Q$ valid for any $\B$-$\Q$ correspondence $\K$, we see that
\begin{align*}
    L^2(\M_j)\otimes_\B \H(j)\cong L^2(\B)\otimes_\B \H(j)\oplus L^2_0(M_j)\otimes_\B \H_j\cong \H.
\end{align*}
Using this isomorphism $U_j\colon L^2(\M_j)\otimes_\B \H(j)\to \H$, we get a faithful $\ast$-representation 
\begin{equation*}
    \lambda_j\colon \langle\M_j,e_\B\rangle\to \mathbb B(\H),\,\lambda_j(x)=U_j(x\otimes_\B 1)U_j^\ast
\end{equation*}
for $j\in \{1,2\}$. The \emph{amalgamated free product} $\M_1\bar\ast_\B \M_2$ of $\M_1$ and $\M_2$ is defined to be the von Neumann algebra generated by $\lambda_1(\M_1)\cup\lambda_2(\M_2)$. In the following, we identify $\M_j$ with its image under $\lambda_j$. The vector $\phi_\B^{1/2}\in L^2(\B)$ is cyclic and separating for $\M_1\bar\ast_\B M_2$, which gives a canonical identification $\H\cong L^2(\M_1\bar\ast_\B M_2)$. If $(j_1,\dots,j_n)\in\mathbf J_n$ and $x_k\in \ker E_{j_k}$ for $1\leq k\leq n$, then
\begin{equation*}
    \lambda_{j_1}(x_1)\dots \lambda_{j_n}(x_n)\phi_\B^{1/2}=x_1\phi_{j_1}^{1/2}\otimes_{\phi_\B}\otimes\dots\otimes_{\phi_\B} x_{j_n}\phi_{j_n}^{1/2}.
\end{equation*}
If $\iota_\B\colon L^2(\B)\to \H$ denotes the inclusion map, then $E_\B(x)=\iota_\B^\ast x\iota_\B$ for $x\in \M_1\bar\ast_\B \M_2$ defines a faithful normal conditional expectation from $\M_1\bar\ast_\B \M_2$ onto $\B$. The von Neumann algebras $\M_1$ and $\M_2$ are \emph{free} with respect to $E_\B$, that is, if $(j_1,\dots,j_n)\in\mathbf J_n$ and $x_{k}\in \ker E_{j_k}$ for $1\leq k\leq n$, then $E_\B(x_1\dots x_n)=0$. The free product state $\phi_1\ast_\B \phi_2$ is defined as $\phi_1\ast_\B \phi_2=\phi_\B\circ E_\B$, which is the same as the vector state induced by $\phi_\B^{1/2}$ on $\M_1\bar\ast_\B \M_2$.

The associated modular conjugation is  characterized by $J_\phi|_{L^2(\B)}=J_{\phi_\B}$ and
\begin{align*}
    J_\phi(a_1\otimes_{\phi_\B}\dots\otimes_{\phi_\B}a_n)=J_{\phi_{j_n}}a_n\otimes_{\phi_\B}\dots_{\phi_\B}J_{\phi_{j_1}}a_1
\end{align*}
for $a_k\in L^2_0(\M_{j_k})$, $1\leq k\leq n$, and $(j_1,\dots,j_n)\in\mathbf J_n$. The modular operator is characterized by
    \begin{equation*}
        \Delta_\phi^{it}=\Delta_{\phi_\B}^{it}\oplus\bigoplus_{n\geq 1}\bigoplus_{(j_1,\dots,j_n)\in \mathbf J_n}\Delta_{\phi_{j_1}}^{it}\otimes_{\phi_\B}\dots\otimes_{\phi_\B}\Delta_{\phi_{j_n}}^{it},
    \end{equation*}
for $t\in\mathbb R$.

\subsubsection{Free products of CP maps}
For $j\in \{1,2\}$ let $\Phi_j\colon \M_j\to \M_j$ be a $\phi_j$-preserving completely positive map such that $\Phi_j|_\B=\id_\B$. In this case, $\B$ belongs to the multiplicative domain of $\Phi_j$, which implies that $\Phi_j$ and $\Phi_j^{(2)}$ are $\B$-modular. It follows from \cite[Theorem 3.8]{BD01} that there exists a unique $\phi_1\ast_\B\phi_2$-preserving completely positive map $\Phi_1\ast_\B \Phi_2$ on $\M_1\bar\ast_\B \M_2$ such that
\[\Phi_1\ast_\B \Phi_2|_\B=\id_\B\;\;\mbox{and }\;\;(\Phi_1\ast_\B \Phi_2)(x_1\dots x_n)=\Phi_{j_1}(x_1)\dots \Phi_{j_n}(x_n)\] whenever $(j_1,\dots,j_n)\in\mathbf J_n$ and $x_k\in \ker E_k$ for $1\leq k\leq n$. Its $L^2$ implementation is given by
\begin{equation*}
    \id_{L^2(\B)}\oplus\bigoplus_{n\geq 1}\bigoplus_{(j_1,\dots,j_n)\in\mathbf J_n}\Phi_{j_1}^{(2)}|_{L^2_0(\M_{j_1})}\otimes_\B\dots\otimes_\B \Phi_{j_n}^{(2)}|_{L^2_0(\M_{j_n})}.
\end{equation*}
If the map $\Phi_1$ (or $\Phi_2$) has a spectral gap, then relative compactness is preserved under amalgamated free products. This is \cite[Theorem 8.1]{CKSVW23} (note that the finiteness assumption of $\B$ is not used in the proof).

\begin{proposition}
    If $\Phi_j\colon \M_j\to \M_j$ is a $\phi$-preserving ucp map with $\Phi_j^{(2)}\in \mathbb K(\B\subseteq \M_j,\phi_j)$ for $j\in \{1,2\}$ and $\norm{\Phi_1^{(2)}|_{L^2_0(\M_1)}}<1$, then $(\Phi_1\ast_\B \Phi_2)^{(2)}\in\mathbb K(\B\subseteq \M_1\ast_\B \M_2,\phi))$.
\end{proposition}

\subsubsection{Semigroups on AFP algebras}
For $j\in \{1,2\}$ let $(\Phi_{j,t})_{t\geq 0}$ be a quantum Markov semigroup on $\M_j$ that is GNS-symmetric with respect to $\phi_j$ and leaves $\B$ invariant. Note that $E_j \Phi_{j,t}E_j=\Phi_{j,t}E_j$, which together with the GNS symmetry implies that $\Phi_{j,t}$ commutes with $E_j$. It is not hard to see that $(\Phi_{1,t}\ast_\B \Phi_{2,t})_{t\geq 0}$ is a quantum Markov semigroup on $\M_1\bar\ast_\B \M_2$ that is GNS-symmetric with respect to $\phi$. In the following, we write $\Phi_t$ for $\Phi_{1,t}\ast_\B \Phi_{2,t}$.

For $j\in \{1,2\}$ let $\mathcal E_j$ be the quantum Dirichlet form associated with $(\Phi_{j,t})_{t\geq 0}$, let $\mathfrak A_j=\mathfrak A_{\mathcal E_j}$ and $\mathfrak A$ be the linear span of $\B_\infty \phi_\B^{1/2}$ and
\begin{equation*}
    \bigcup_{n\geq 1}\bigcup_{(j_1,\dots,j_n)\in\mathbf J_n}\{a_1\otimes_{\phi_\B}\dots\otimes_{\phi_\B} a_n\mid a_k\in \mathfrak A_{j_k}\cap L^2_0(\M_{j_k})\text{ for }1\leq k\leq n\}.
\end{equation*}

\begin{lemma}
    If $\mathcal E$ is the quantum Dirichlet form associated with $(\Phi_t)_{t\geq 0}$, then $\mathfrak A$ is a Tomita subalgebra of the maximal Tomita algebra associated with $\phi$, it satisfies $\mathfrak A\subseteq \mathfrak A_{\mathcal E}$ and it is a form core for $\mathcal E$.
\end{lemma}
\begin{proof}
    From the description of the modular conjugation and modular operator, it is not hard to see that $\mathfrak A$ is a Tomita subalgebra of the maximal Tomita algebra associated with $\phi$.
    
    Since $\Phi_t(b)=b$ for $b\in \B$, we have $\B_\infty \phi_\B^{1/2}\subseteq \mathfrak A_{\mathcal E}$.
    To show that 
    \begin{equation*}
        \bigcup_{(j_1,\dots,j_n)\in\mathbf J_n}\{a_1\otimes_{\phi_\B}\dots\otimes_{\phi_\B} a_n\mid a_k\in \mathfrak A_{j_k}\cap L^2_0(\M_{j_k})\text{ for }1\leq k\leq n\}\subseteq \mathfrak A_{\mathcal E},
    \end{equation*}
    we proceed by induction over $n\in\mathbb N$. Since 
    \begin{equation*}
        \Phi_t^{(2)}=\id_{L^2(\B)}\oplus\bigoplus_{n\geq 1}\bigoplus_{(j_1,\dots,j_n)\in\mathbf J_n}\Phi_{j_1,t}^{(2)}|_{L^2_0(\M_{j_1})}\otimes_\B\dots\otimes_\B \Phi_{j_n,t}^{(2)}|_{L^2_0(\M_{j_n})}
    \end{equation*}
    and
    \begin{equation*}
        \Delta_\phi^{it}=\Delta_{\phi_\B}^{it}\oplus\bigoplus_{n\geq 1}\bigoplus_{(j_1,\dots,j_n)\in \mathbf J_n}\Delta_{\phi_{j_1}}^{it}\otimes_{\phi_\B}\dots\otimes_{\phi_\B}\Delta_{\phi_{j_n}}^{it},
    \end{equation*}
    the case $n=1$ is immediate.
    
    Assume that we have proven the claim for $n\in\mathbb N$. If $(j_1,\dots,j_{n+1})\in\mathbf J_{n+1}$ and $a_k\in \mathfrak A_{j_k}\cap L^2_0(\M_{j_k})$ for $1\leq k\leq n+1$, then 
    \begin{equation*}
        a_1\otimes_{\phi_\B}\otimes \dots \otimes_{\phi_\B}a_{n+1}=\pi_l(a_1)a_2\otimes_{\phi_\B}\otimes \dots \otimes_{\phi_\B}a_{n+1}.
    \end{equation*}
    Since $\mathfrak A_{\mathcal E}$ is a left Hilbert algebra and $a_1,a_2\otimes_{\phi_\B}\otimes \dots \otimes_{\phi_\B}a_{n+1}\in \mathfrak A_{\mathcal E}$, we conclude that $a_1\otimes_{\phi_\B}\otimes \dots \otimes_{\phi_\B}a_{n+1}\in\mathfrak A_{\mathcal E}$.

    To show that $\mathfrak A$ is a form core for $\mathcal E$, we first show that it is dense in $\H$. For this purpose, we prove by induction over $n\in\mathbb N$ that for every $(j_1,\dots,j_n)\in\mathbf J_n$ the set
    \begin{equation*}
        \operatorname{span} \{a_1\otimes_{\phi_\B}\dots\otimes_{\phi_\B} a_n\mid a_k\in \mathfrak A_{j_k}\cap L^2_0(\M_{j_k})\text{ for }1\leq k\leq n\}
    \end{equation*}
    is dense in $L^2_0(\M_{j_1})\otimes_{\phi_\B}\otimes\dots\otimes_{\phi_\B}L^2_0(\M_{j_n})$.

    For the case $n=1$, let $P_j\colon L^2(\M_j)\to L^2_0(\M_j)$ be the orthogonal projection. Since $\Phi_{j,t}$ commutes with $E_j$, it follows easily that $P_j(\mathfrak A_j)\subseteq \mathfrak A_j$. Combined with the fact that $\mathfrak A_j$ is dense in $L^2(\M_j)$, we obtain the claim for $n=1$.
    Now suppose that the claim is true for $n$ and let 
    \[(j_1,\dots,j_{n+1})\in \mathbf J_{n+1},\;\;\;\xi\in \mathcal D(L^2_0(\M_{j_1})_\B,\phi_\B),\;\;\; \eta\in \mathcal{D}'({_\B}\K,\phi_\B).\]  where $\K=(L^2_0(\M_{j_2})\otimes_{\phi_\B}\otimes\dots\otimes_{\phi_\B}L^2_0(\M_{j_{n+1}}).$ Recall that $\mathcal{D}(\mathsf L_\B,\phi_\B)$ and $\mathcal{D}'({_\B}\mathsf L,\phi_\B)$ respectively denote the left and right bounded vectors in a $\B$-correspondence $\mathsf L$.
    
   Note that $L_{\phi_\B}(\xi)\in \mathcal L(L^2(\B)_\B,L^2(\M_{j_1})_\B)=\mathcal L(L^2(\M_{j_1})_\B)\iota_{j_1}$, where $\iota_{j_1}$ denotes the inclusion of $L^2(\B)$ inside $L^2(\M_j)$.   Fix $\epsilon>0$.  Since $\mathcal L(L^2(\M_{j_1})_\B)$ is the weak$^\ast$ closed linear span of $\{ye_\B z\mid y,z\in \M_{j_1}\}$ and $y e_\B z\iota_{j_1}=yE_{j_1}(z)\iota_{j_1}\in \M_{j_1}\iota_{j_1}$, by Kaplansky's density theorem, there exists $x\in \M_{j_1}$ such that $\|x\|\leq \|L_{\phi_\B}(\xi)\|$ and  $\|\xi-x\phi_{j_1}^{1/2}\|\,\|R_{\phi_\B}(\eta)\|<\epsilon/3$. Moreover, by \cite[Theorem 6.3]{Wir22} there exists an element $a$ in $\mathfrak A_1$ such that $\norm{\pi_l(a)}\leq \norm{x}$ and $\|(x-a)\phi_{j_1}^{1/2}\|\,\|R_{\phi_\B}(\eta)\|<\epsilon/3.$ Thus we have $\|\xi-a\phi_{j_1}^{1/2}\|\,\|R_{\phi_\B}(\eta)\|<2\epsilon/3$.
    By the induction hypothesis, there exists  \[\zeta\in \operatorname{span}\{a_2\otimes_{\phi_\B}\dots\otimes_{\phi_\B} a_{n+1}\mid a_k\in \mathfrak A_{j_k}\cap L^2_0(\M_{j_k})\text{ for }2\leq k\leq n+1\}\] 
    such that $\|\eta-\zeta\|\,\|L_{\phi_\B}(\xi)\|<\epsilon/3$.  Therefore,
    \begin{align*}
        \norm{\xi\otimes_{\phi_\B}\eta-a\phi_{j_1}^{1/2}\otimes_{\phi_\B}\zeta}&\leq \norm{(\xi-a\phi_{j_1}^{1/2})\otimes_{\B}\eta}+\norm{a\phi_{j_1}^{1/2}\otimes_\B(\eta-\zeta)}\\
        &\leq \|\xi-a\phi_{j_1}^{1/2}\|\,\|R_{\phi_\B}(\eta)\|+\|\pi_\ell(a)\|\,\|\eta-\zeta\|\\
&<\epsilon.
    \end{align*}
  Since right and left bounded vectors are dense,  we have thus proven that $\mathfrak A$ is dense in $\H$. Moreover, it is evidently invariant under $\Phi_t^{(2)}$ for all $t\geq 0$. Hence it is a form core by \cite[Lemma 3.4]{Wir22}.
\end{proof}

\subsubsection{Derivations on AFP algebras}
By the previous result, the modular derivation $\delta$ associated with $(\Phi_t)_{t\geq 0}$ is uniquely determined by its values on $\mathfrak A$. For the next lemma, recall that for a Tomita correspondence $\K$ over $(\M_j,\phi_j)$ there exists an expected inclusion $\N_j\supset \M_j$ such that $\K_j\subseteq L^2(\N_j)\ominus L^2(\M_j)$ as Tomita correspondences. Hence we can define the codomain of the modular derivation $\delta_j$ to be $L^2(\N_j)\ominus L^2(\M_j)$.

\begin{lemma}\label{lem:derivation on the amalagamated free product}
    If $\delta_j\colon \mathfrak A_{\mathcal E_j}\to L^2(\N_j)\ominus L^2(\M_j)$ is the $\phi_j$-modular derivation associated with $(\Phi_t^{(j)})_{t\geq 0}$ and
    \begin{equation*}
        \delta\colon \mathfrak A\to L^2(\N_1\ast_\B \N_2)\ominus L^2(\M_1\ast_\B \M_2)
    \end{equation*}
    is given by $\delta|_{\B_\infty\phi^{1/2}}=0$ and
    \begin{equation*}
        \delta(a_1\otimes_{\phi_\B}\dots\otimes_{\phi_\B}a_n)=\sum_{k=1}^n a_1\otimes_{\phi_\B}\dots\otimes_{\phi_\B}a_{k-1}\otimes_{\phi_\B}\delta_{j_k}(a_k)\otimes_{\phi_\B}a_{k+1}\otimes_{\phi_\B}\dots\otimes_{\phi_\B}a_n
    \end{equation*}
    for $a_k\in \mathfrak A_{j_k}\cap L^2_0(\M_{j_k})$ and $(j_1,\dots,j_n)\in\mathbf J_n$, then $\delta$ is a $\phi$-modular derivation and
    \begin{equation*}
        \mathcal E(a)=\norm{\delta(a)}^2
    \end{equation*}
    for $a\in\mathfrak A$.
\end{lemma}
\begin{proof}
    From the description of the modular data for $\phi$, it is not hard to see that $\delta$ is a $\phi$-modular derivation. Moreover, if $a_k\in \mathfrak A_{j_k}\cap L^2_0(\M_{j_k})$ for $1\leq k\leq n$ and $(j_1,\dots,j_n)\in\mathbf J_n$ and $a_k^\prime\in \mathfrak A_{_k}\cap L^2_0(\M_{i_k})$ for $1\leq k\leq m$ and $(i_1,\dots,i_m)\in\mathbf J_n$, then 
    \begin{equation*}
        \langle\delta(a_1\otimes_{\phi_\B} \dots\otimes_{\phi_\B}a_n),\delta(a_1^\prime\otimes_{\phi_\B}\dots\otimes_{\phi_\B}a_m^\prime)\rangle=0
    \end{equation*}
    unless $m=n$ and $j_1=i_1$.

    In this case, write $a_{<k}$ for $a_1\otimes_{\phi_\B}\dots\otimes_{\phi_\B}a_{k-1}$ and $a_{>k}$ for $a_{k+1}\otimes_{\phi_\B}\dots\otimes_{\phi_\B}a_n$. We have
    \begin{align*}
       \quad &\langle\delta(a_1\otimes_{\phi_\B} \dots\otimes_{\phi_\B}a_n),\delta(a_1^\prime\otimes_{\phi_\B}\dots\otimes_{\phi_\B}a_n^\prime)\rangle\\
        &=\sum_{k=1}^n \langle a_{<k}\otimes_{\phi_\B}\delta_{j_k}(a_k)\otimes_{\phi_\B}a_{>k}, a_{<k}^\prime\otimes_{\phi_\B}\delta_{j_k}(a_k^\prime)\otimes_{\phi_\B}a_{>k}^\prime\rangle\\
        &=\sum_{k=1}^n \langle \delta_{j_k}(a_k),\delta_{j_k}(L_{\phi_\B}(a_{<k})^\ast L_{\phi_\B}(a_{<k}^\prime)\cdot a_k^\prime\cdot J_{\phi_\B}R_{\phi_\B}(a_{>k}^\prime)^\ast R_{\phi_\B}(a_{>k})J_{\phi_\B})\rangle
        \end{align*}
        which by derivative formula further yields
        \begin{align*}
       &\quad\langle\delta(a_1\otimes_{\phi_\B} \dots\otimes_{\phi_\B}a_n),\delta(a_1^\prime\otimes_{\phi_\B}\dots\otimes_{\phi_\B}a_n^\prime)\rangle \\
       &=-\frac{d}{dt}\bigg|_{t=0}\sum_{k=1}^n \langle a_k,\Phi_{t,j_k}^{(2)}(L_{\phi_\B}(a_{<k})^\ast L_{\phi_\B}(a_{<k}^\prime)\cdot a_k^\prime\cdot J_{\phi_\B}R_{\phi_\B}(a_{>k}^\prime)^\ast R_{\phi_\B}(a_{>k})J_{\phi_\B})\rangle\\
        &=-\frac{d}{dt}\bigg|_{t=0}\sum_{k=1}^n \langle a_k,L_{\phi_\B}(a_{<k})^\ast L_{\phi_\B}(a_{<k}^\prime)\cdot(\Phi_{t,j_k}^{(2)} a_k^\prime)\cdot J_{\phi_\B}R_{\phi_\B}(a_{>k}^\prime)^\ast R_{\phi_\B}(a_{>k})J_{\phi_\B}\rangle\\
        &=-\frac{d}{dt}\bigg|_{t=0}\sum_{k=1}^n \langle a_{<k}\otimes_{\phi_\B}a_k\otimes_{\phi_\B}a_{>k},a_{<k}^\prime\otimes_{\phi_\B}\Phi_{t,j_k}^{(2)}(a_k^\prime)\otimes_{\phi_\B}a_{>k}^\prime\rangle\\
        &=-\frac{d}{dt}\bigg|_{t=0}\langle a_1\otimes_{\phi_\B}\dots\otimes_{\phi_\B}a_n,\Phi_t^{(2)}(a_1\prime\otimes_{\phi_\B}\dots\otimes_{\phi_\B}a_n^\prime)\rangle\\
        &=\mathcal E(a_1\otimes_{\phi_\B}\dots\otimes_{\phi_\B}a_n,a_1\prime\otimes_{\phi_\B}\dots\otimes_{\phi_\B}a_n^\prime),
    \end{align*}
    where we used that $\Phi_{t,j_k}^{(2)}$ and $\delta_{j_k}$ are $\B$-bimodule maps.
\end{proof}

\begin{remark}
    As discussed in Subsection \ref{subsec:bimodules}, it is not restrictive to assume that $\delta_j$ takes values in $L^2(\N_j)\ominus L^2(\M_j)$. In the special case when $\delta_j\colon \mathfrak A_{\mathcal E_j}\to L^2(\M_j)\otimes_\B L^2(\M_j)$, the formula from Lemma \ref{lem:derivation on the amalagamated free product} shows that
    \[\operatorname{ran}\delta\subseteq (L^2(\M_1\bar\ast_\B M_2)\otimes_\B L^2(\M_1\bar\ast_\B \M_2))^{\oplus 2}.\] In particular, this is the case for the derivation $\delta_j$ associated with the dephasing semigroup $(\Phi_{j,t}=e^{-t}\mathrm{id}_{\M_j}+(1-e^{-t})E_\B$ discussed in the beginning of the section. Thus, if $\mathrm{id}_{L^2(\M_j)}\in\mathbb K(\B\subseteq \M_j,\phi_j)$, then the solidity results from Section~\ref{sec:solidity} apply to $\M_1\bar\ast_\B \M_2$.

\subsubsection{Solidity in AFP algebras} We consider the following lemma which fits into the hypothesis of Theorem \ref{main_theorem}.
\begin{lemma}
For $j=1,2$, let $\Phi_j:\M_j\to\M_j$ be  $\phi_i$-preserving ucp maps such that ${\Phi_j}_{|_\B}=\id_\B$. Assume that  there are $\B$-$\B$ correspondences $\K_{j}$ such that    the Stinespring correspondences $\H_{\Phi_j}$ of $\Phi_j$ satisfy ${_{\M_j}}(\H_{j})_\B\prec {_{\M_j}}(L^2(\M_j)\otimes_\B\K_{j})_\B$. Then we have ${_\M}(\H_{\Phi})_\B\prec {_\M}L^2(\M)\otimes_\B\K_\B$ for some $\B$-$\B$ correspondence $\K$, where $\M=\M_1\ast_\B\M_2$ and $\Phi=\Phi_1\ast_\B\Phi_2$.
\end{lemma}
\begin{proof}
 For $j=1,2$, write $\H_j=\H_{\Phi_j}$, and denote by $\xi_j$ the cyclic left bounded vector in $\H_j$ implementing $\Phi_j$. Set $\H:=(\H_1,\xi_1)\ast_\B(\H_2,\xi_2)$. We shall show that the Stinespring correspondence $\H_\Phi$ of $\Phi$ is contained in $\H$ as $\M$-$\B$ correspondence, and that ${_\M}\H_\B\prec{_\M}L^2(\M)\otimes_\B\K_\B$ for some $\B$-$\B$ correspondence. This will show that ${_\M}(\H_\Phi)_\B\prec{_\M}L^2(\M)\otimes_\B\K_\B$ as required.
 
 We first recall the construction of $\H$. Set $\H_j^0=\H_j\ominus\overline{\B\xi_j}$. The equality ${\Phi_j}_{|_\B}=\id_\B$ implies that $b\xi_j=\xi_j\sigma_{i/2}^{\phi_\B}(b)$ for any $\sigma^{\phi_\B}$-analytic element $b$ of $\B$. This shows that  $\overline{\B\xi_j}=\overline{\xi_j\B}$, and hence it is invariant under the left and  right action of $\B$.  Thus $\H_j^0$ is an $\B$-$\B$ sub-correspondence of $\H_j$. 
Set
 \begin{align*}
     \mathsf H=L^2(\B)\oplus\bigoplus_{n\geq1}\bigoplus_{(j_1,\ldots,j_n)\in \mathrm J_n}\H^0_{j_1}\otimes_\B\ldots\otimes_\N \H^0_{j_n}.
 \end{align*}
 Then $\mathsf H$ has an obvious right $\B$ action by diagonal operation. Moreover, following  the proof of \cite[Theorem 2.2, Theorem 3.8]{BD01}, one gives a left $\M$ action on $\H$ as follows: let $\lambda_j:\M_j\to \mathcal L((\H_{j})_\B)$ be the left action on the Stinespring correspondence. Define the tail correspondence
 \begin{align*}
     \mathsf H(j)&=L^2(\B)\oplus \bigoplus_{n\geq 1}\bigoplus_{\substack{\mathbf j\in\mathbf J_n\\j_1\neq j}}\mathsf H_j^0\otimes_\B\dots\otimes_\B \H_{j_n}^0 
 \end{align*}
Because $\overline{\B\xi_j}\cong L^2(\B)$, we have $\H_j\cong L^2(\B)\oplus\H_j^0$ as $\B$-$\B$ correspondence; hence there is a canonical  $\B$-bilinear isometry $V_j: \H_j\otimes\H(j)\to \H$ by using the fact that $L^2(\B)\otimes_\B\H(j)\cong\H(j)$. Define $\pi_j:\M_j\to \mathcal L (\H_\B)$ by 
\[\pi_j(a)=V_j(\lambda_j(a)\otimes1)V_j^\ast,\;\;\;a\in \M_j.\] 
  By \cite[Theorem 3.7]{BD01}, there is a unique normal representation $\pi_1\ast_\B\pi_2:\M\to \mathcal{L}(\H_\B)$  extending $\pi_j$. This induce a $\M$-$\B$ correspondence on $\H$. 

By construction of $\H$ and $\Phi$, it is clear that the minimal Stinespring correspondence $\H_\Phi$ is contained in $\H$ (actually as $\M$-$\M$ correspondence, but below we need only as $\M$-$\B$ correspondence).

To show the last claim, set $\mathsf L_j=\H_j\ominus\overline{\M_j\xi_j}$.  As above, $\overline{\M_j\xi_j}$ is invariant under left $\M_j$ action and right $\B$ action. Hence $\mathsf L_j$ is a $\M_j$-$\B$ sub-correspondence of $\H_j$. Since $L^2(\M_j)\cong\overline{\M_j\xi_j}$ as $\M_j$-$\B$ correspondence, and $L^2(\B)\cong\overline{\B\xi_j}$ as $\B$-$\B$ correspondence, we have
\begin{align*}
    \H_j^0\cong\H_j\ominus L^2(\B)\cong(\H_j\ominus L^2(\M_j))\oplus (L^2(\M_j)\ominus L^2(\B))\cong \mathsf L_j\oplus(L^2(\M_j)\ominus L^2(\B))
\end{align*}
where the isomorphism is as $\B$-$\B$ correspondence. By regrouping, it is then easy to see that 
\begin{align*}
    \H\cong L^2(\M)\oplus\bigoplus_{j=1,2}L^2(\M)\otimes_{\M_j} \mathsf L_j\otimes_\B\H(j)
\end{align*}
as $\M$-$\B$ correspondence. By assumption, we have ${_{\M_j}}(\mathsf  L_j)_\B\prec{_{\M_j}}(L^2(\M_j)\otimes_\B\K_j)_\B$
which implies that ${_\M}\H_\B\prec {_\M}L^2(\M)\otimes_\B\K_\B$ where
\begin{align*}
    \K=L^2(\B)\oplus\bigoplus_{j=1,2}L^2(\M)\otimes_\B\K_j\otimes_\B\H(j).
\end{align*}
This completes the proof.
\end{proof}
    
   Now let $\delta_j$ and $\delta$ be as in Lemma \ref{lem:derivation on the amalagamated free product} and assume that the Stinespring correspondence satisfies ${_{\M_j}}(\H_{\Phi_j})_\B\prec {_{\M_j}}L^2(\M_j)\otimes_\B\K_j)_\B$ for some $\B$-$\B$ correspondence $\K_j$. Consider the dephasing semigroup $\hat\Phi_{j,t}:=e^{-t}{{\Phi_{j,t}}}+(1-e^{-t})E_j$. Then we have
    \begin{align*}
        {_{\M_j}}(\H_{\hat\Phi_j})_\B&\cong {_{\M_j}}(\H_{\Phi_j}\oplus \H_{E_j})_\B\\
        &\cong {_{\M_j}}(\H_{\Phi_j}\oplus L^2(\M_j)\otimes_\B L^2(\M_j))_\B\\
        &\prec {_{\M_j}}L^2(\M_j)\otimes_\B\K_j)_\B
    \end{align*}
 Moreover, $\|\hat\Phi_{j,t}\|\leq e^{-t}<1$.   So again  the solidity results from Section~\ref{sec:solidity} apply to $\M_1\bar\ast_\B \M_2$ whenever $\M_j$ admits such $\Phi_{j,t}$. We record these observations in the following theorem.
\end{remark}

\begin{theorem}\label{thm:solidity in Amalgamated free products}
    Let $\M_1,\M_2$ be  von Neumann algebras with separable predual and with a common von Neumann subalgebra $\B$ with expectations such that for each $j=1,2$, there is a modular derivation $\delta_j:\dom(\delta_j)\subseteq L^2(\M_j)\to \H_j$ satisfying the hypothesis of Theorem \ref{main_theorem} and such that $\delta_j^\ast\bar\delta_j$ has compact resolvent relative to $\B$. Then $\M_1\bar\ast_\B\M_2$ is $\omega$-solid relative to $\B$ and stable solid relative to $\B$ along $\omega$ for any $\omega\in \beta\mathbb N\setminus\mathbb N$. 
\end{theorem}
For tracial AFP, much stronger results than Theorem \ref{thm:solidity in Amalgamated free products} are known where general  spectral-gap theorems apply without derivations or compactness assumptions (see \cite[Theorem 6.3]{Ioa15}). Therefore, the main novelty lies in the relative solidity results with nontracial base algebras in our ultrapower formulation.

Since certain hypothesis of Theorem \ref{main_theorem} is automatically satisfied for amenable $\B$, we get the following:
\begin{corollary}
     Let $\M_1,\M_2$ be  von Neumann algebra with separable predual and with a common amenable von Neumann subalgebra $\B$ with expectations such that for each $j=1,2$, there is a modular derivation $\delta_j:\dom(\delta_j)\subseteq L^2(\M_j)\to \H_j$ satisfying ${_{\M_j}}(\H_j)_{\M_j}\prec{_{\M_j}}L^2(\M_j)\otimes_\B L^2(\M_j)_{\M_j}$ such that $\delta_j^\ast\bar\delta_j$ has compact resolvent relative to $\B$. Then $\M_1\bar\ast_\B\M_2$ is $\omega$-solid relative to $\B$ and stable solid relative to $\B$ along $\omega$ for any $\omega\in \beta\mathbb N\setminus\mathbb N$. 
\end{corollary}

In the special case when $\Phi_{j,t}=\id_{L^2(\M)}$, we obtain the following.

\begin{corollary}\label{thm:solidity for amalgamated free products}
    Let $\M_1,\M_2$ be $\sigma$-finite von Neumann algebras with a common von Neumann subalgebra $\B$ with expectations such that $\mathrm{id}_{L^2(\M_j)}\in\mathbb K(\B\subseteq \M_j,\phi_j)$ for $j=1,2$. Then $\M_1\bar\ast_\B\M_2$ is $\omega$-solid relative to $\B$ and stable solid relative to $\B$ along $\omega$ for any $\omega\in \beta\mathbb N\setminus\mathbb N$.
\end{corollary}

\subsubsection{Primeness in AFP algebras}
We state some direct consequences of our general Theorem \ref{thm:primeness for type I base and compact resolvent} on primeness.
However, the following statement has a stronger known version (see \cite[Theorem E]{HV13}). 

\begin{theorem}
     Let $\M$ be a  von Neumann algebra with separable predual and let $\B\subseteq\M$ be an expected inclusion such that $\B$ is of type I (e.g. $\B$ is finite dimensional or abelian). If there is a modular derivation $\delta:\dom(\delta)\subseteq L^2(\M)\to \H$ satisfying $\H\prec L^2(\M)\otimes_\B L^2(\M)$ and $\B\subseteq\Null(\bar\delta)$  such that $\delta^\ast\bar\delta$ has compact resolvent relative to $\B$, and if $\M$ is a nonamenable factor then $\M\bar\ast_\B\M$ is prime.
\end{theorem}

The following corollary also has a known stronger known version (\cite[Theorem 5.7]{HV13}).

\begin{corollary}
     Let $\M_1,\M_2$ be  von Neumann algebras with separable predual and with a common von Neumann subalgebra $\B$ with expectations such that  $\B$ is of type I (e.g. $\B$ is finite dimensional or abelian). If $\M$ is nonamenable and $\id_{L^2(\M_i)}\subseteq \mathbb K(\B\subseteq \M_j,\phi)$ for $j=1,2$ then $\M\bar\ast_\B\M$ is prime.
\end{corollary}
    
\subsection{Operator-valued free Araki--Woods algebras}\label{subsec:op-valued_Araki-Woods}

Let $\B$ be a von Neumann algebra with a faithful normal state $\phi_\B\in \B_\ast$ and let $(\H,\mathcal J,(\mathcal U_t)_{t\in\mathbb R})$ be a Tomita correspondence over $(\B,\phi_\B)$. The Fock correspondence $\mathcal F(_\B\H_\B)$ is defined as
\begin{equation*}
    \mathcal F(_\B\H_\B)=L^2(\B)\oplus\bigoplus_{n\geq 1}\H^{\otimes_{\phi_\B} n}
\end{equation*}
If $\xi\in \mathcal D(\H,\phi_\B)$, define the left creation operator $a^\ast(\xi)$ on $\mathcal F(_\B\H_\B)$ by $a^\ast(\xi)|_{{L^2(\B)}}=L_{\phi_\B}(\xi)$ and 
\begin{equation*}
    a^\ast(\xi)\xi_1\otimes_{{\phi_\B}}\dots\otimes_{\phi_\B}\xi_n=\xi\otimes_{\phi_\B}\xi_1\otimes_{\phi_\B}\dots\otimes_{\phi_\B}\xi_n
\end{equation*}
for $\xi_1,\dots,\xi_n\in \mathcal D(\H,{\phi_\B})$. The left annihilation operator $a(\xi)$ is the adjoint of $a^\ast(\xi)$ and the left field operator $s(\xi)$ is defined as $s(\xi)=a^\ast(\xi)+a(\xi)$.

The {\em operator-valued free Araki--Woods algebra} $\Gamma(_\B\H_\B,\mathcal J,(\mathcal U_t)_{t\in\mathbb R})$ is the von Neumann algebra generated by $\pi_l^{\mathcal F(_\B\H_\B)}(\B)$ and 
\begin{equation*}
\{s(\xi)\mid \xi\in \mathcal D(\H,{\phi_\B})\cap\dom(\mathcal U_{-i/2}),\,\mathcal J\mathcal U_{-i/2}\xi=\xi\}.
\end{equation*}In the following, we denote $\Gamma(_\B\H_\B,\mathcal J,(\mathcal U_t)_{t\in\mathbb R})$ by $\M$. The vector ${\phi_\B}^{1/2}\in L^2(\B)\subseteq\mathcal F(_\B\H_\B)$ is cyclic and separating for $\M$, and we denote the induced vector state on $\M$ by $\phi$.

If $T\colon \H\to \H$ is a contractive $\B$-bimodule map, then the second quantization $\mathcal F(T)\in \mathbb B(\mathcal F(_\B\H_\B))$ is the map given by $\mathcal F(T)_{|_{L^2(\B)}}=\mathrm{id}_{L^2(\B)}$ and $\mathcal F(T)_{|_{\H^{\otimes_\B n}}}=T^{\otimes_\B n}$ for $n\in\mathbb N$.

By \cite[Theorem 3.2.5]{Ske01} there exists a quasi-orthonormal basis for $\mathcal D(\H,\phi_\B)$, that is, a family $(\xi_\lambda,p_\lambda)_{\lambda\in \Lambda}$ with $\xi_\lambda\in \mathcal D(\H,{\phi_\B})$ and  a projection $p_\lambda\in \B$ such that $L_{\phi_\B}(\xi_\lambda)^\ast L_{\phi_\B}(\xi_\mu)=\delta_{\lambda,\mu}p_\lambda$ for all $\lambda,\mu\in\Lambda$ and
\begin{equation*}
\sum_{\lambda\in \Lambda}L_{\phi_\B}(\xi_\lambda)L_{\phi_\B}(\xi_\lambda)^\ast=1_\H
\end{equation*}
in the strong operator topology. In the following, we work under the assumption that $\mathcal D(\H,\phi_\B)$ admits a finite quasi-orthonormal basis that is also compatible with the Tomita correspondence structure.

\begin{lemma}\label{lem:rel_compact_Fock_space}
    Assume that there exists a finite family $(\xi_\lambda,p_\lambda)_{\lambda\in \Lambda}$ with $\xi_\lambda\in \mathcal D(\H,\phi_\B)\cap \dom(\mathcal U_{-i/2})$, $\mathcal J\mathcal U_{-i/2}\xi_\lambda\in \mathcal D(\H,\phi_\B)$ and $p_\lambda\in \B$ a projection for $\lambda \in \Lambda$ such that $L_{\phi_\B}(\xi_\lambda)^\ast L_{\phi_\B}(\xi_\mu)=\delta_{\lambda,\mu}p_\lambda$ for all $\lambda,\mu\in \Lambda$ and
    \begin{equation*}
        \sum_{\lambda\in \Lambda}L_{\phi_\B}(\xi_\lambda) L_{\phi_\B}(\xi_\lambda)^\ast=1_\H.
    \end{equation*}
    The operator $\mathcal F(e^{-t})$ belongs to $\mathbb K(\B\subseteq \M,\phi)$ for all $t>0$.
\end{lemma}
\begin{proof}
    Let $P_n$ be the orthogonal projection from $\mathcal F(_\M \H _M)$ onto $\H^{\otimes_\B n}$. Since $\mathcal F(e^{-t})=\sum_{n=0}^\infty e^{-nt}P_n$, it suffices to prove that $P_n\in \mathbb K(\B\subseteq M,\phi)$ for all $n\in\mathbb N$.

    For $(i_1,\dots,i_n)\in \Lambda^n$ let $\xi_{(i_1,\dots,i_n)}=\xi_{i_1}\otimes_{\phi_\B}\dots\otimes_{\phi_\B}\xi_{i_n}$and define inductively $x_{\emptyset}=1$, $x_i=a^\ast(\xi_i)+a(\mathcal J\mathcal U_{-i/2})$ and $x_{(i_1,\dots,i_n)}=x_{i_1} x_{(i_2,\dots,i_n)}-(L_{\phi_\B}(\xi_{i_1})^\ast L_{\phi_\B}(\xi_{i_2}))\cdot x_{(i_3,\dots,i_n)}$ for $n\geq 2$. By definition, $x_{(i_1,\dots,i_n)}\in \M$ and one checks recursively that $x_{(i_1,\dots,i_n)}\phi_\B^{1/2}=\xi_{(i_1,\dots,i_n)}$ for all $(i_1,\dots,i_n)\in \Lambda^n$ and $n\in\mathbb N$.
    
    The set $(\xi_{(i_1,\dots,i_n)},p_{(i_1,\dots,i_n)})_{(i_1,\dots,i_n)\in \Lambda^n}$ forms a quasi-orthonormal basis of $\Lambda^n$, which implies 
    \begin{equation*}
        P_n=\sum_{(i_1,\dots,i_n)\in \Lambda^n}x_{(i_1,\dots,i_n)} e_\B x_{(i_1,\dots,i_n)}^\ast\in C^\ast(\M e_\B \M)
    \end{equation*}
    for all $n\in\mathbb N$.
\end{proof}

We will show next that there exists a GNS-symmetric quantum Markov semigroup $(\Phi_t)_{t\geq 0}$ with fixed-point algebra $\B$ such that $\Phi_t^{(2)}=\mathcal F(e^{-t})$ for $t\geq 0$.

\begin{proposition}\label{prop:construction of QMS on Free-Araki Woods}
    Let $T\colon \H\to \H$ be a bounded self-adjoint $\B$-bimodule map with $\norm{T}\leq 1$ such that $\mathcal J T=T\mathcal J$ and $\mathcal U_t T=T\mathcal U_t$ for all $t\in\mathbb R$. There exists a unique GNS-symmetric ucp map $\Phi$ on $\Gamma(_\B\H_\B,\mathcal J,(\mathcal U_t)_{t\in\mathbb R})$ such that $\mathcal F(T)=\Phi^{(2)}$.

    Moreover, if $\H_\Phi$ denotes the minimal Stinespring correspondence for $\Phi$, then there exists a $\B$-$\B$ correspondence $\K$ such that $_\M(\H_\Phi)_\B\subseteq {_\M \mathcal F(\H)}\otimes_\B \K_\B$.
\end{proposition}
\begin{proof}
    The proof of the first part is a simple modification of \cite[Theorem 2.11]{BKS97}. Let $_\B\tilde \H_\B={_\B\H_\B}\oplus {_\B\H_\B}$, $\tilde{\mathcal J}=\mathcal J\oplus \mathcal J$, $\tilde{\mathcal U_t}=\mathcal U_t\oplus\mathcal U_t$, $\tilde\M=\Gamma(_\B\tilde \H_\B,\tilde{\mathcal J},(\tilde{\mathcal U}_t)_{t\in\mathbb R})$ and
    \begin{equation*}
        \tilde T=\begin{pmatrix}T&(1-T^2)^{1/2}\\1-T^2)^{1/2}&-T\end{pmatrix}\in \mathbb B(\tilde\H).
    \end{equation*}
    Clearly, $\tilde T$ is a $\B$-bimodule map that commutes with $\tilde{\mathcal J}$ and $\tilde{\mathcal U_t}$ for all $t\in\mathbb R$. Moreover, $\tilde T$ is unitary.

    Let $\iota\colon \M\to\tilde \M$ be the $\ast$-homomorphism that maps $s(\xi)$ to $s(\xi,0)$ for $\xi\in \mathcal D(\H,\phi_\B)\cap \dom(\mathcal U_{-i/2})$ with $\mathcal J\mathcal U_{-i/2}\xi=\xi$ and $\pi_l^{\mathcal F(_\B\H_\B)}(b)$ to $\pi_l^{\mathcal F(_\B\tilde \H_\B)}(b)$ for $b\in \B$.
    
    Let $\tilde\phi$ be the vector state induced by $\phi_\B^{1/2}$ on $\tilde\M$. By the description of the modular theory of $\phi$ in \cite[Section 5.2]{KW25}, the image $\iota(\M)$ is globally invariant under $\sigma^{\tilde\phi}$. Let $E\colon \tilde\M\to\iota(\M)$ be the $\tilde\phi$-preserving conditional expectation.
    
    Finally, let $\alpha=\mathrm{ad}(\mathcal F(\tilde T))$. We have $\alpha(s(\xi))=s(\tilde T \xi)$ for $\xi\in \mathcal D(\tilde\H,\phi_\B)\cap\dom(\tilde{\mathcal U}_{-i/2})$ with $\tilde{\mathcal J}\tilde{\mathcal U}_{-i/2}\xi=\xi$ and $\alpha(b)=b$ for $b\in \B$ so that $\alpha$ leaves $\M$ globally invariant. We define $\Phi=\iota^{-1}E\alpha \iota$. Note that $\Phi$ commutes with $\sigma^\phi$. The fact that $\Phi^{(2)}=\mathcal F(T)$ follows by a direct calculation.

    Let $W\colon \H\to \tilde\H$, $\xi\mapsto (\xi,0)$. For $x,y\in \M$ we have
    \begin{align*}
        \langle \phi_\B^{1/2},\alpha(\iota(x))\phi_\B^{1/2}\iota(y)\rangle&=\langle \mathcal F(W)(\phi_\B^{1/2}y^\ast),\mathcal F(\tilde T) \mathcal F(W)x\phi_\B^{1/2}\rangle\\
        &=\langle \phi_\B^{1/2}\iota(y^\ast),\iota(\Phi(x))\phi_\B^{1/2}\rangle\\
        &=\langle \phi_\B^{1/2},\Phi(x)\phi_\B^{1/2}y\rangle.
    \end{align*}
    Hence a minimal Stinespring dilation of $\Phi$ is given by $\H_\Phi=\overline{\mathrm{span}}\{\alpha(\iota(x))\phi_\B^{1/2}\iota(y)\mid x,y\in \M\}$, $V=\mathcal F(W)$ and $\pi=\alpha\iota$. Since $\alpha|_\B=\id_\B$, $\mathcal F(\tilde T)^\ast\colon \H_\Phi\to \mathcal F(_\B\tilde \H_\B)$ is a right $\B$-module map such that $\mathcal F(\tilde T)^\ast \alpha(\iota(x))=\iota(x)\mathcal F(\tilde T)^\ast$. Hence ${_\M(\H_\Phi)_\B}\subseteq {_\M\mathcal F(_\B\tilde \H_\B)_\B}$, where the left action of $\M$ on $\mathcal F(_\B\tilde \H_\B)$ is given by $\iota$. Finally,
    \begin{equation*}
        _\M\mathcal F(_\B\tilde \H_\B)_\B\cong {_\M}\mathcal F(_\B\H_\B)\otimes_\B(L^2(\B)_\B\oplus \H\otimes_\B\mathcal F(_\B\H_\B)_\B).\qedhere
    \end{equation*}
\end{proof}

\subsubsection{Derivations associated to the number operator}
Let $\mathcal F^{(0)}(_\B\H_\B)$ denote the linear span of $\B_\infty {\phi_\B}^{1/2} $ and $\cup_{n\geq1}\mathcal{S}_n$ where
\begin{align*}\mathcal{S}_n&=\{\xi_1\otimes_{\phi_\B}\dots\otimes_{\phi_\B}\xi_n\mid \xi_k\in \bigcap_{z\in \mathbb C}\dom(\mathcal U_z),\\
&\qquad\qquad\mathcal U_z\xi_k\in\mathcal D(\H,{\phi_\B})\cap \mathcal D^\prime(\H,{\phi_\B})\text{ for all }z\in\C,1\leq k\leq n\}.
\end{align*}
As a consequence of \cite[Proposition 4.8]{Wir22}, the set $\mathcal F^{(0)}(_\B\H_\B)$ is dense in $\mathcal F(_\B\H_\B)$. Since it is clearly invariant under $\mathcal F(e^{-t})$, it follows from \cite[Lemma 3.4]{Wir22} that $\mathcal F^{(0)}(_\B\H_\B)$ is a form core for $\mathcal E$. If $\xi\in \bigcap_{z\in \mathbb C}\dom(\mathcal U_z)$ with $\mathcal U_z\xi\in\mathcal D(\H,{\phi_\B})\cap \mathcal D^\prime(\H,{\phi_\B})$ for all $z\in\C$, then it follows from \cite[Theorem 5.18]{KW25} that $s(\xi)\phi_\B^{1/2}\in \mathfrak T_\phi$. Then one can show by induction that $\mathcal F^{(0)}(_\B\H_\B)\subseteq \mathfrak T_\phi$.

Let
\begin{equation*}
    \delta\colon \mathcal F^{(0)}(_\B\H_\B)\to \mathcal F(_\B\H_\B)\otimes_{\phi_\B}\H\otimes_{\phi_\B}\mathcal F(_\B\H_\B),
\end{equation*}
be given by $\delta|_{B_\infty \phi_\B^{1/2}}=0$ and 
\begin{equation*}
    \delta(\xi_1\otimes_{\phi_\B}\dots\otimes_{\phi_\B}\xi_n)=\sum_{k=1}^n \underbrace{\xi_1\otimes_{\phi_\B \dots}\otimes_{\phi_\B}\xi_{k-1}}_{\in \mathcal F(_\B\H_\B)}\otimes_{\phi_\B}\underbrace{\xi_k}_{\in \H}\otimes_{\phi_\B}\underbrace{\xi_{k+1}\otimes_{\phi_\B}\dots\otimes_{\phi_\B}\xi_n}_{\in\mathcal F(_\B\H_\B)}.
\end{equation*}
As in the scalar-valued case \cite[Section 8.4]{Wir22}, one shows that $\delta$ is a $\phi$-modular derivation such that $\mathcal E(a)=\norm{\delta(a)}^2$ for $a\in \mathcal F^{(0)}(_\B\H_\B)$. Note that the codomain $\mathcal F(_\B\H_\B)\otimes_{\phi_\B}\H\otimes_{\phi_\B}\mathcal F(_\B\H_\B)$ of $\delta$ carries a natural structure of an $\langle \M,\B\rangle$-$\langle \M,\B\rangle$ correspondence. It is easy to verify that $\delta^\ast\bar\delta$ is the number operator, that is, $\delta^\ast\bar\delta(\xi)=n\xi$ for $\xi\in \dom(\bar\delta)\cap\H^{\otimes_\B n}, n\geq1$.

\subsubsection{Relative solidity in operator-valued free Araki-Woods algebras}
Below we fix  a Tomita correspondence  $(\H,\mathcal J,(\mathcal U_t)_{t\in\mathbb R})$  over a  von Neumann algebra $\B$ with separable predual.

\begin{theorem}\label{cor:solidity of free Araki-Woods with finitely generated bimodule}
        If $\B$ has separable predual and $(\H,\mathcal J,(\mathcal U_t)_{t\in\mathbb R})$ $\B$  satisfies the condition from Lemma~\ref{lem:rel_compact_Fock_space}, then $\Gamma(_\B\H_\B,\mathcal J,(\mathcal U_t)_{t\in\mathbb R})$ is stably solid relative to $\B$ along $\omega$ and is $\omega$-solid relative to $\B$ for any $\omega\in\beta\mathbb N\setminus\mathbb N$.
\end{theorem}

\begin{corollary}\label{cor:solidity of free Araki-Woods}
    If $\B$ has separable predual and  $\H=\oplus_{j\in \Lambda}\H_j$ for some countable set $\Lambda$, where each $\H_j$
is a Tomita subcorrespondence  satisfying the condition from Lemma~\ref{lem:rel_compact_Fock_space},
    then $\Gamma(_\B\H_\B,\mathcal J,(\mathcal U_t)_{t\in\mathbb R})$ is stably solid relative to $\B$ along $\omega$ and is $\omega$-solid relative to $\B$ for any $\omega\in\beta\mathbb N\setminus\mathbb N$.
\end{corollary}
\begin{proof}
We assume that $\Lambda$ is infinite. Choose positive real numbers $\lambda_j\geq 1$ such that $\lambda_j\to\infty$ as $j\to\infty$, and define the operator $S:\H\to \H$ and $(T_t:\H\to\H)_{t\geq 0}$ by 
\[
S_{|_{H_j}}=\lambda_j,\qquad T_t=e^{-tS}.
\]
For $t>0$, each  ${T_t}_{|_{\H_j}}$ is a relatively compact $\B$-bimodule map because $\H_j$ is finitely generated as right $\B$-module, and since $e^{-t\lambda_j}\to 0$, it follows that  $T_t$ is compact relative to $\B$.  Moreover, it is clear that $T_t$ commutes with the Tomita data, and $\|T_t\|\leq \sup_{j\in \Lambda}e^{-t\lambda_j}<1$.

On the Fock module,
\[
\mathcal F(T_t)=1_{L^2(B)}\oplus\bigoplus_{n\geq1}T_t^{\otimes_B n}
\]
is relatively compact by Lemma~\ref{lem:rel_compact_Fock_space}. Second quantization gives a GNS-symmetric QMS by the construction of Proposition \ref{prop:construction of QMS on Free-Araki Woods}. As above, consider the subspace $\mathfrak A$ spanned by  $\B_\infty {\phi_\B}^{1/2}$ and $\bigcup_{n\geq 1}\{\xi_1\otimes_{\phi_\B}\dots\otimes_{\phi_\B}\xi_n\mid $ where
\[\xi_k\in \bigcap_{z\in \mathbb C}\dom(\mathcal U_z)\cap\dom(S^{1/2}),\,\,\,\mathcal U_z\xi_k\in\mathcal D(\H,{\phi_\B})\cap \mathcal D^\prime(\H,{\phi_\B})\]
for all $z\in\C,1\leq k\leq n$ and define the linear map $ \delta_S\colon \mathfrak A\to \mathcal F(_\B\H_\B)\otimes_{\phi_\B}\H\otimes_{\phi_\B}\mathcal F(_\B\H_\B),$
by $\delta|_{B_\infty \phi_\B^{1/2}}=0$ and
\begin{equation*}
    \delta_S(\xi_1\otimes_{\phi_\B}\dots\otimes_{\phi_\B}\xi_n)=\sum_{k=1}^n \underbrace{\xi_1\otimes_{\phi_\B \dots}\otimes_{\phi_\B}\xi_{k-1}}_{\in \mathcal F(_\B\H_\B)}\otimes_{\phi_\B}\underbrace{S^{1/2}\xi_k}_{\in \H}\otimes_{\phi_\B}\underbrace{\xi_{k+1}\otimes_{\phi_\B}\dots\otimes_{\phi_\B}\xi_n}_{\in\mathcal F(_\B\H_\B)}.
\end{equation*}
Again as above, one shows that $\delta_S$ is a $\phi$-modular derivation such that $\mathcal{F}(T_t)=e^{-t{\delta_S}^\ast\bar\delta_S}$. Note that the codomain $\mathcal F(_\B\H_\B)\otimes_{\phi_\B}\H\otimes_{\phi_\B}\mathcal F(_\B\H_\B)$ of $\delta_S$ carries a natural structure of an $\langle \M,\B\rangle$-$\langle \M,\B\rangle$ correspondence.
The results now follow from the general results of Section \ref{sec:solidity}.
\end{proof}

In the tracial setting, a stronger version of Corollary \ref{cor:solidity of free Araki-Woods} is known (see \cite[Theorem 4.1]{KV19}). Therefore,   the main novelty of the above results are the  general nontracial-base theorem of relative solidity in our ultrapower formulation. 

As an immediate corollary, for $\B=\mathbb C$, we get the following well-known $\omega$-solidity result about free Araki-Woods factors from \cite[Theorem A]{HR15} where it is known for arbitrary $(U_t)$ and almost periodicity is unnecessary. However, the result about stable solidity seems to be new as we could not locate a precise reference.

\begin{corollary}
    If $(U_t)_{t\in \mathbb R}$ is a one-parameter group of orthogonal transformations on a separable Hilbert space $\H_{\mathbb R}$ such that $U_t$ is almost-periodic, then $\Gamma(U_t,\H_{\mathbb R})$ is stably solid along $\omega$ and is $\omega$-solid for any $\omega\in\beta\mathbb N\setminus\mathbb N$
\end{corollary}

\begin{example}
    Let $G$ be a discrete group and let $\pi\colon G\to O(H)$ be a strongly continuous orthogonal representation. If $(U_t)_{t\in \mathbb R}$ is a strongly continuous almost periodic orthogonal group on $H$ that commutes with $\pi(g)$ for all $g\in G$, then we obtain a Gaussian action $g\mapsto \Gamma(\pi(g))$ on the free Araki--Woods factor $\Gamma(H_\mathbb C,J,(U_t)_{t\in\mathbb R})$, where $H_\mathbb C$ is the complexification of $H$ and $J$ is the induced complex conjugation.

    Let $\B=L(G)$ and $\H=H_\mathbb C\otimes \ell^2(G)$ with left and right actions given by $\pi_l^\H(\lambda(g))=\pi(g)\otimes \lambda(g)$ and $\pi_r^\H(\lambda(g)^\op)=1\otimes\rho(g^{-1})$. One can turn $\H$ into a Tomita correspondence over $(\B,\tau_G)$ by defining $\mathcal J(\eta\otimes \delta_g)=J\eta\otimes \delta_{g^{-1}}$ and $\mathcal U_t=U_t\otimes 1$. As shown in \cite[Proposition 6.7]{KW25}, there exists an isomorphism $\alpha\colon \Gamma(H,J,(U_t)_{t\in\mathbb R})\rtimes_{\Gamma(\pi)} G\to \Gamma(_\B\H_\B,\mathcal J,(\mathcal U_t)_{t\in\mathbb R})$ that preserves the canonical copies of $\B$.
    
    Since $(U_t)_{t\in\mathbb R}$ is almost periodic, we have $H_{\mathbb C}=\bigoplus_{i\in I}(H_i)_{\mathbb C}$ with $(H_i)_{\mathbb C}$ a finite-dimensional eigenspace of $(U_t)_{t\in\mathbb R}$ \cite{Shl99}. If $(\xi_k^{(i)})_{1\leq k\leq n_i}$ is an orthonormal basis of $(H_i)_\mathbb C$, then $(\xi_k^{(i)}\otimes 1_e,1_\B)_{1\leq k\leq n_i}$ is a finite quasi-orthonormal basis of $\mathcal D((H_i)_\mathbb C\otimes \ell^2(G),\tau_G)$. Thus $\Gamma(\H,J,(U_t))\rtimes_{\Gamma(\pi)}G$ is stably solid relative to $L(G)$ along $\omega\in\beta\mathbb N\setminus\mathbb N$. This result generalizes \cite[Theorem B.3]{HT21} for almost periodic $(U_t)$.
\end{example}

\subsubsection{Fullness of operator-valued free Araki-Woods factors}
We  fix  a Tomita correspondence  $(\H,\mathcal J,(\mathcal U_t)_{t\in\mathbb R})$  over a $\sigma$-finite von Neumann algebra $\B$.
\begin{lemma}\label{lem:mixing of H implies mixing of K in operator valued free Araki-Woods}
  If ${}_\B\H_\B{}$ is mixing, then $\K:
=
{}_\M L^2(\M)\otimes_\B \H\otimes_\B L^2(\M)_\M$ is mixing over $\M$, where $\M=\Gamma(_\B\H_\B,\mathcal J,(\mathcal U_t)_{t\in\mathbb R})$.  
\end{lemma}
\begin{proof}
    Let $\phi$ be a normal faithful state on $\M$ such that $\B$ is globally invariant under $\sigma^\phi$. Let $(x_i)$ be a net of bounded elements in $\M$ weakly converging to $0$. Fix $a,b\in\M$ and $\xi\in \H$ and set $\eta=a\phi^{1/2}\otimes_\phi\xi\otimes_\phi\phi^{1/2}b\in \K$. We have
    \begin{align*}
        \sup_{\underset{\|y\|_\infty\leq 1}{y\in\M}}|\langle\eta, x_i\eta y\rangle|=\sup_{\underset{\|y\|_\infty\leq 1}{y\in\M}}|\langle\xi, E_\B(a^\ast x_i a)\xi E_\B(byb^\ast)\rangle|=\sup_{\underset{\|z\|_\infty\leq 1}{z\in\B}}|\langle\xi, E_\B(a^\ast x_i a)\xi z\rangle|.
    \end{align*}
    However, the right hand side goes to zero along $i$ because ${}_\B H_\B$ is mixing and $E_\B(a^\ast x_i a)\to 0$ weakly. Similarly, one prove that $\sup_{\underset{\|y\|_\infty\leq 1}{y\in\M}}|\langle\eta, y\eta x_i\rangle|\to 0$ along $i$. This shows that ${}_\M\mathcal \K_\M$  is mixing because elements of the form $\eta$ is dense in $\K$.
\end{proof}

\begin{lemma}\label{lem:failure of coarse rigidity in operator-valued free Araki-Woods}
Suppose  $\H$
contains the standard  correspondence
$(L^2(\B),J_{{\phi_\B}},\Delta_{{\phi_\B}}^{it})$ as a Tomita subcorrespondence. Then $(\theta_\alpha^{(2)})^\omega$ associated to $\delta$ constructed above does not converge to $1$ on $L^2(\M)$ for any $\omega\in\beta\mathbb N\setminus\mathbb N$.
\end{lemma}
\begin{proof}
Write $T_\alpha=\theta_\alpha^{(2)}$. Let $V:L^2(\B)\to{\H}$ be the isometric bimodule embedding preserving  the Tomita data, and put $\xi=V({\phi_\B}^{1/2})$. Then $\xi$ is a $\phi_\B$-left bounded vector such that
$L_{{\phi_\B}}(\xi)=V$. Set $\A:=\{s(\xi)\}''\subseteq\M$. Then $\A$ is globally invariant under $\sigma^\phi$ such that  $\{\phi_\B^{1/2}, \xi^{\otimes_\B n}; n\geq 1\}$ forms an orthogonal basis of $L^2(\A)$. Since $T_\alpha(\xi^{\otimes_\B n})=\frac{\alpha}{\alpha+n}\xi^{\otimes_\B n}$, it follows that ${T_\alpha}_{|_{L^2(\A)}}$ is a compact operator for all $\alpha>0$.

Set $u=e^{is(\xi)}$ which is a Haar unitary in $\A$. Since $u^n\to0$ weakly, compactness of $T_\alpha$ implies that $\|T_\alpha(u^n\phi^{1/2})\|\to 0$ as $n\to\infty$. Clearly, $(u^n)\in\mathcal{M}^\omega(\M)$. Set $\xi=(u^n\phi^{1/2})\in L^2(\M^\omega)$. Clearly, then $T_\alpha^\omega\xi=0$ while $\|\xi\|=1$. Therefore, $T_\alpha^\omega\xi$ does not converge to $\xi$ as $\alpha\to\infty$.
\end{proof}

Combining Lemmas \ref{lem:mixing of H implies mixing of K in operator valued free Araki-Woods} and \ref{lem:failure of coarse rigidity in operator-valued free Araki-Woods} with Corollary \ref{cor:fullness and rigidity} yields the following result on fullness.
\begin{theorem}\label{thm:fullness of operator-valued-free Araki-Woods factors}
   Assume that ${}_\B\H_\B{}$ is mixing and $\H$ contains
$(L^2(\B),J_{{\phi_\B}},\Delta_{{\phi_\B}}^{it})$ as a Tomita subcorrespondence. Set $\M=\Gamma({}_\B\H_\B{},\mathcal J,(\mathcal U_t))$ and $\K=L^2(\M)\otimes_\B\H\otimes_\B L^2(\M)$.  Assume that ${}_\M\K_\M{}\prec{}_\M L^2(\M)\otimes L^2(\M)_\M$ (e.g. when ${}_\B\H_\B{}\prec {}_\B L^2(\B)\otimes L^2(\B)_\B{}$ or if $\B$ is amenable). If $\M$ is a nonamenable factor, then $\Gamma({}_\B\H_\B{})$ is full.
\end{theorem}
    
When $\B$ is amenable and  $\mathcal{U}_t$ is trivial, them above result follows from \cite[Theorem A]{KV19} because any nonamenable solid factor is full. However, the key novelty of the above result lies in its nontracial generalization or the case when $\B$ is nonamenable.

\subsubsection{Primeness of free Araki-Woods algebras} Combining Theorem \ref{thm:primeness for type I base and compact resolvent} with the relative compactness argument in the proof of Corollary \ref{cor:solidity of free Araki-Woods} yields the following result on primeness of operator-valued free Araki-Woods factors. Again, when $\B$ is amenable and  $\mathcal{U}_t$ is trivial, then it follows from \cite[Theorem A]{KV19} because any nonamenable solid factors are prime. However, again the key novelty  lies in its nontracial generalization or when $\B$ is an arbitrary type I algebra.

\begin{theorem}\label{thm:primeness of free Araki-Woods}
    If $\B$ is a type I von Neumann algebra with separable predual and  $\H=\oplus_{j\in \Lambda}\H_j$ for some countable set $\Lambda$, where each $\H_j$ is a Tomita subcorrespondence  satisfying the condition from Lemma~\ref{lem:rel_compact_Fock_space},
    then $\Gamma(_\B\H_\B,\mathcal J,(\mathcal U_t)_{t\in\mathbb R})$ is prime whenever it is a nonamenable factor.
\end{theorem}

\subsection{Cocycle crossed product von Neumann algebras}\label{sec:crossed product}

We develop a unified approach to relative rigidity results for arbitrary cocycle actions in the nontracial setting. To fix notation, we recall the construction of cocycle crossed products in detail (see \cite{Sut80}).

Let $\Gamma$ be a countable discrete group   and let $\alpha:\Gamma\to \text{Aut}(\B)$ and $\nu:\Gamma\times\Gamma\to\mathcal{U}(\B)$ be a cocycle action on a von Neumann algebra $\B$, that is,
\begin{align}
&\alpha_e=\id,\quad \nu_{e,g}=\nu_{g,e}=1,\label{eq:normalization}\\
&\alpha_g\circ\alpha_h=\Ad(\nu_{g,h})\circ\alpha_{gh},\label{eq:action}\\
&\nu_{g,h}\nu_{gh,k}=\alpha_g(\nu_{h,k})\nu_{g,hk}.\label{eq:nu}
\end{align}
for any $g,h,k\in \Gamma$.
Here $\Ad(v)(a)=vav^*$ for $v\in\mathcal{U}(\B), a\in \B$.
 We write $(\alpha, \nu)\curvearrowright \B$ for such a cocycle action. One must distinguish the inverse automorphism $\alpha_g^{-1}$ from the automorphism  $\alpha_{g^{-1}}$;  in fact,
\[\alpha_{g^{-1}}=\Ad(\nu_{g^{-1},g})\circ \alpha_g^{-1},\;\;\;\alpha_g(\nu_{g^{-1},g})=\nu_{g,g^{-1}}.\]
 Let $\M=\B\rtimes_{\alpha,\nu} \Gamma$ (or simply $\B\rtimes\Gamma$) be the {\em (cocycle) crossed product} which is the von Neumann subalgebra of $\mathbb B(\ell^2(\Gamma)\otimes L^2(\B))$ generated by the set $\{L(a), u_g; a\in \B, g\in \Gamma\}$ where the operators $L(a)$ and $ u_g$ are defined on $\ell^2(\Gamma)\otimes L^2(\B)$ by 
\begin{align*}
L(a)(\delta_h\otimes\xi)
 &=\delta_h\otimes\alpha_h^{-1}(a)\xi\\
u_g(\delta_h\otimes\xi)
 &=\delta_{gh}\otimes\alpha_{gh}^{-1}(\nu_{g,h})\xi.
\end{align*}
for $a\in \B, g,h\in \Gamma$ and $\xi\in L^2(\B)$. Clearly, $u_g$ is a unitary.
The following identities hold:
\begin{align}
&u_gL(a)u_g^*=L(\alpha_g(a))\label{eq:formula involving both u and L},\\
&u_gu_h=L(\nu_{g,h})u_{gh},\label{eq:formulas for products of u}\\ 
&u_g^\ast
    =u_{g^{-1}}L(\nu_{g,g^{-1}}^\ast).\label{eq:formula for adjoint}
\end{align}
To verify this, consider the map $\kappa:\Gamma\times\Gamma\to\mathcal{U}(\B)$ given by 
\[\kappa_{g,h}=\alpha_{gh}^{-1}(\nu_{g,h}),\;\;\;g,h\in \Gamma.\]
\begin{lemma}\label{lem:automorphisms}
For $g,h,k\in\Gamma$,
\begin{align}
&\alpha_{gh}^{-1}\circ\alpha_g
 =\Ad(\kappa_{g,h})\circ\alpha_h^{-1},\label{eq:intertwiner}\\
&\alpha_h^{-1}\circ\alpha_g^{-1}
 =\Ad(\kappa_{g,h}^*)\circ\alpha_{gh}^{-1}.\label{eq:inverseaction}\\
 &\kappa_{g,hk}\kappa_{h,k}
 =\alpha_{ghk}^{-1}(\nu_{g,h})\kappa_{gh,k}.\label{eq:blockcocycle}
\end{align}
\end{lemma}
\begin{proof}
Compose \eqref{eq:action}  on the left with $\alpha_{gh}^{-1}$.
Since any automorphism $\theta$ on $\B$ and $v\in\mathcal{U}(\B)$ satisfies
$\theta\circ\Ad(v)=\Ad(\theta(v))\circ\theta$, this gives
\[
\alpha_{gh}^{-1}\circ\alpha_g\circ\alpha_h
 =\alpha_{gh}^{-1}\circ\Ad(\nu_{g,h})\circ\alpha_{gh}
 =\Ad(\alpha_{gh}^{-1}(\nu_{g,h}))
 =\Ad(\kappa_{g,h}).
\]
Compose on the right with $\alpha_h^{-1}$ to obtain
\eqref{eq:intertwiner}. The second identity easily follows from the first by composing by $\alpha_g^{-1}$ from the right.

Finally we  prove \eqref{eq:blockcocycle} by applying the 
automorphism $\alpha_{ghk}$ to both sides.
On one hand,
the twisted action identity \eqref{eq:action} for the pair $(g,hk)$  yields
$\alpha_{ghk}\circ\alpha_{hk}^{-1}
 =\Ad(\nu_{g,hk}^*)\circ\alpha_g$
which further yields
\begin{align*}
\alpha_{ghk}\bigl(\kappa_{g,hk}\kappa_{h,k}\bigr)
 =\nu_{g,hk}\,
       \alpha_{ghk}\bigl(\alpha_{hk}^{-1}(\nu_{h,k})\bigr)=\nu_{g,hk}\,
       \nu_{g,hk}^*\alpha_g(\nu_{h,k})\nu_{g,hk}=\alpha_g(\nu_{h,k})\nu_{g,hk}.
\end{align*}
On the other hand, apply \eqref{eq:inverseaction} with $(g,h)$ replaced by $(gh,k)$ to obtain:
\[
\alpha_k^{-1}\circ\alpha_{gh}^{-1}
 =\Ad(\kappa_{gh,k}^*)\circ\alpha_{ghk}^{-1}.
\]
Composing on the left with $\alpha_{ghk}$  gives
\begin{equation*}
\alpha_{ghk}\circ\alpha_k^{-1}\circ\alpha_{gh}^{-1}
=\Ad(\alpha_{ghk}(\kappa_{gh, k}^\ast)) =\Ad(\nu_{gh,k}^*).
\end{equation*}
Consequently,
\begin{align*}
\alpha_{ghk}\bigl(\kappa_{gh,k}\alpha_k^{-1}(\kappa_{g,h})\bigr)
&=\nu_{gh,k}\,
  (\alpha_{ghk}\circ\alpha_k^{-1}\circ\alpha_{gh}^{-1})(\nu_{g,h})=\nu_{gh,k}\,\nu_{gh,k}^*\nu_{g,h}\nu_{gh,k}\\
&=\nu_{g,h}\nu_{gh,k}=\nu_{g,h}\alpha_{ghk}(\kappa_{gh,k}).
\end{align*}
Thus we obtain
\eqref{eq:blockcocycle}.
\end{proof}

We now prove \eqref{eq:formula involving both u and L}.
For any $g,k\in \Gamma$ and $a\in \B$, we use \eqref{eq:intertwiner} to obtain 
\begin{align*}
    \alpha_{gk}^{-1}(\alpha_g(a))\kappa_{g,k}=\kappa_{g,k}\alpha_k^{-1}(a).
\end{align*}
which further yields
\begin{align*}u_gL(a)(\delta_k\otimes\xi)&=u_g(\delta_k\otimes\alpha_k^{-1}(a)\xi)=\delta_{gk}\otimes \kappa_{g,k}\alpha_k^{-1}(a)\xi\\
    &=\delta_{gk}\otimes\alpha_{gk}^{-1}(\alpha_g(a))\kappa_{g,k}\xi=L(\alpha_g(a))u_g(\delta_k\otimes\xi).
\end{align*}
proving that $u_gL(a)=L(\alpha_g(a))u_g$. To show \eqref{eq:formulas for products of u}, we have
\begin{align*}
u_gu_h(\delta_k\otimes\xi)
 &=u_g(\delta_{hk}\otimes\kappa_{h,k}\xi)\\
 &=\delta_{ghk}\otimes\kappa_{g,hk}\kappa_{h,k}\xi\\
 &=\delta_{ghk}\otimes
    \alpha_{ghk}^{-1}(\nu_{g,h})\kappa_{gh,k}\xi
       &&\text{by \eqref{eq:blockcocycle}}\\
 &=L(\nu_{g,h})(\delta_{ghk}\otimes\kappa_{gh,k}\xi)\\
 &=L(\nu_{g,h})u_{gh}(\delta_k\otimes\xi).
\end{align*}
Equality on the dense span of elementary tensors proves the bounded
operator identity.

In the sequel, we shall suppress the symbol $L$ and simply write $a$ instead of $L(a)$.

\subsubsection{Modular theory of crossed product}
Let $(L^2(\B), L^2(\B)_+,J_\B)$ be the standard form for $\B$.
Let $\varphi$ be a normal faithful state on $\B$, and we do not assume that $\alpha$ is $\varphi$-preserving. 
Consider the dual state $\tilde{\varphi}$ on $\B\rtimes\Gamma$ determined by the formula
\begin{align*}
    \tilde{\varphi}(\sum_{g\in \Gamma}u_ga_g)=\varphi(a_e).
\end{align*}
The canonical faithful normal conditional expectation $E_\B:\M\to\B$
is characterized by
\[
E_\B(u_ga)=\begin{cases}a,&g=e,\\0,&g\ne e\end{cases}
\]
so that we have $\tilde\phi=\phi\circ E_\B$. For each $g\in \Gamma$, let $ V_g $ be the unique unitary on $L^2(\B)$ implementing $\alpha_g$, that is,
\[\alpha_g(a)=V_g a V_g ^\ast,\;\;\;V_g(L^2(\B)_+)=L^2(\B)_+,\;\;\;J_{\B} V_g =V_g J_{\B},\;\;\; a\in \B.\]
The uniqueness of the unitary implementation yields
\begin{equation}\label{eq:rel_nu and V}
V_gV_h=\nu_{g,h}J_{\B}\nu_{g,h}J_{\B}\,V_{gh},
\quad
V_g^*=\nu_{g^{-1},g}^\ast J_{\B}\nu_{g^{-1},g}^\ast J_{\B} V_{g^{-1}}.
\end{equation}
Indeed, the standard implementation of $\Ad(v)$ is $vJ_{\B} vJ_{\B}$ for any $v\in \mathcal{U}(\M)$.
Consequently $V_g^*$ need not equal $V_{g^{-1}}$, and
$V_gV_h=\nu_{g,h}V_{gh}$ is not the standard implementation identity.
In the untwisted case, $g\mapsto V_g$ is  a unitary representation.

\begin{proposition}
If we set $d_g(t)=\left(\frac{D(\varphi\circ\alpha_g)}{D\varphi}\right)_t $ for $g\in \Gamma, t\in\mathbb R$ to be the Connes cocycle derivative, then the modular objects of $\tilde\phi$ satisfy
\begin{align}
\Delta_{\tilde\phi}^{i t}(\delta_g\otimes\xi)
 &=\delta_g\otimes d_g(t)\Delta_\varphi^{i t}\xi,
       \label{eq:Delta}\\
J_{\tilde\phi}(\delta_g\otimes\xi)
 &=\delta_{g^{-1}}\otimes
       \nu_{g,g^{-1}}^*V_gJ_{\B}\xi,
       \label{eq:J}
\end{align}
for $\xi\in L^2(\B)$. Moreover,
\begin{equation}\label{eq:sigma}
\sigma_t^{\tilde\phi}(a)=\sigma_t^\varphi(a),\quad
\sigma_t^{\tilde\phi}(u_g)=u_gd_g(t),\quad a\in\B.
\end{equation}
\end{proposition}
\begin{proof}
    We have $\tilde\phi^{1/2}=\delta_e\otimes\phi^{1/2}$ and $u_ga\tilde\phi^{1/2}=\delta_g\otimes a\phi^{1/2}$ for $g\in \Gamma, a\in \B$, so the set
    \[\mathcal{D}_0:=\text{span}\{\delta_g\otimes a\phi^{1/2}; g\in \Gamma, a\in \B\}\] lies in the domain of the Tomita operator $S_{\tilde\phi}$ such that
\begin{align*}
S_{\tilde\phi}(\delta_g\otimes a\phi^{1/2})
 =S(u_ga\tilde\phi^{1/2})
 =(u_ga)^*\tilde\phi^{1/2}
 =\delta_{g^{-1}}\otimes
       \nu_{g,g^{-1}}^*\alpha_g(a^*)\phi^{1/2}.
\end{align*}
By Kaplansky's density theorem, the algebraic crossed product is strongly-$\ast$ dense in $\M$ with
bounded approximants. Hence for every $x\in\M$, finite Fourier
sums $x_i$ can be chosen so that both
$x_i\tilde\phi^{1/2}\to x\tilde\phi^{1/2}$ and $x_i^*\tilde\phi^{1/2}\to x^*\tilde\phi^{1/2}$.
It follows that $\mathcal D_0$ is a core for $S_{\tilde\phi}$.

For each $g\in \Gamma$, write $H_g=\delta_g\otimes L^2(\B)$ so that we have $\ell^2(\Gamma)\otimes L^2(\B)=\oplus_{g\in \Gamma}H_g$. Put
\[
\psi_g=\varphi\circ\alpha_g,\quad
D_g=\Delta_{\psi_g,\varphi},\quad
C_g=\nu_{g,g^{-1}}^*V_gJ_{\B}.
\]
where $D_g$ is the operator $S_{\psi_g,\varphi}^\ast S_{\psi_g,\varphi}$. Here the relative Tomita operator $S_{\psi_g,\varphi}$ is the closure of the map 
$a\phi^{1/2}\mapsto a^*\psi_g^{1/2}, a\in\B,$
and has polar decomposition
$S_{\psi_g,\varphi}=J_{\B} D_g^{1/2}$ (see \cite[Proposition 10.3]{Hia21}).
The operator $C_g$ is antiunitary on $L^2(\B)$ and $\B\phi^{1/2}$ is a core for each $D_g^{1/2}$, by the
definition of the relative Tomita operator.
Since $V_g$ is the standard implementation of $\alpha_g$,
$V_g\psi_g^{1/2}=\phi^{1/2}.$
Thus
\begin{align*}
\nu_{g,g^{-1}}^*\alpha_g(a^*)\phi^{1/2}
 =\nu_{g,g^{-1}}^*V_g(a^*\psi_g^{1/2})=\nu_{g,g^{-1}}^*V_gS_{\psi_g,\varphi}(a\phi^{1/2})
 =C_gD_g^{1/2}(a\phi^{1/2}).
\end{align*} 
Thus we obtain
\begin{equation}\label{eq:Sclosed}
S_{\tilde\phi}(\oplus_{g\in \Gamma}\delta_g\otimes \xi_g)=\oplus_{g\in \Gamma}\delta_{g^{-1}}\otimes C_gD_g^{1/2}\xi_g
\end{equation}
with domain consisting of all $\oplus_{g\in \Gamma}\delta_g\otimes\xi_g\in \ell^2(\Gamma)\otimes L^2(\B)$ for which $\xi_g\in \dom(D_g^{1/2})$ for all $g\in\Gamma$ and $\sum_{g\in \Gamma}\|D_g^{1/2}\xi_g\|^2<\infty$.

We now compute the formula for the adjoint $S_{\tilde\phi}^\ast$.
Because $C_g$ is antiunitary and $D_g^{1/2}$ is self-adjoint,
$
(C_gD_g^{1/2})^*=D_g^{1/2}C_g^{-1}.
$
The adjoint reverses the direction of the original block:
\begin{equation}\label{eq:Sstarblock}
S^*(\delta_{g^{-1}}\otimes\eta)
 =\delta_g\otimes D_g^{1/2}C_g^{-1}\eta,
\quad C_g^{-1}\eta\in\dom(D_g^{1/2}).
\end{equation}
More explicitly, $\oplus_{g\in \Gamma}\delta_g\otimes\eta_g\in\dom(S^*)$ precisely when
\[
C_g^{-1}\eta_{g^{-1}}\in\dom(D_g^{1/2})\quad\text{for every }g,
\quad
\sum_g\|D_g^{1/2}C_g^{-1}\eta_{g^{-1}}\|^2<\infty.
\]
For $\xi\in\dom(D_g)$, formulas \eqref{eq:Sclosed} and
\eqref{eq:Sstarblock} now give
\begin{align}\label{eq:directsum}
S_{\tilde\phi}^*S_{\tilde\phi}(\delta_g\otimes\xi)
 =S_{\tilde\phi}^*(\delta_{g^{-1}}\otimes C_gD_g^{1/2}\xi)
 =\delta_g\otimes D_g^{1/2}C_g^{-1}C_gD_g^{1/2}\xi
 =\delta_g\otimes D_g\xi.
\end{align}
This shows that $\Delta_{\tilde\phi}=\oplus_{g\in \Gamma}{1\otimes D_g}$
with
\[
\dom(\Delta_{\tilde\phi})=\left\{\oplus_{g\in \Gamma}\delta_g\otimes\xi_g\in \ell^2(\Gamma)\otimes L^2(\B):
\xi_g\in\dom(D_g),\;\sum_g\|D_g\xi_g\|^2<\infty\right\}.
\]
Define the antiunitary $C$ on $\ell^2(\Gamma)\otimes L^2(\B)$ by
\[
C(\delta_g\otimes\xi)=\delta_{g^{-1}}\otimes C_g\xi.
\]
Equations \eqref{eq:Sclosed} and \eqref{eq:directsum} say that
$S_{\tilde\phi}=C\Delta_{\tilde\phi}^{1/2}$. Since $\Delta_{\tilde\phi}$ is nonsingular,
uniqueness of the polar decomposition yields $C=J_{\tilde\phi}$,
proving \eqref{eq:J}. In particular $C^2=1$.
Functional calculus in \eqref{eq:directsum} gives
\[
\Delta_{\tilde\phi}^{i  t}(\delta_g\otimes\xi)
 =\delta_g\otimes D_g^{i  t}\xi.
\]
The Connes derivative identity
$d_g(t)=D_g^{i  t}\Delta_\varphi^{-i  t}
$
implies $D_g^{i  t}=d_g(t)\Delta_{\varphi}^{i  t}$ (see \cite[Proposition 10.11]{Hia21}),
which proves \eqref{eq:Delta}.
Finally, we have
\begin{align*}
\sigma_t^{\tilde\phi}(u_ga)\tilde\phi^{1/2}
 =\Delta_{\tilde\phi}^{i  t}(u_ga\tilde\phi^{1/2})=\delta_g\otimes d_g(t)\Delta_\varphi^{i  t}(a\phi^{1/2})=\delta_g\otimes d_g(t)\sigma_t^\varphi(a)\phi^{1/2}=u_gd_g(t)\sigma_t^\varphi(a)\tilde\phi^{1/2}.
\end{align*}
and the separating property of $\tilde\phi^{1/2}$ yields \eqref{eq:sigma}.
\end{proof}

\begin{remark}
If $\varphi\circ\alpha_g=\varphi$ for every $g$, then $d_g(t)=1$
and $\Delta_{\tilde\phi}^{i t}=1\otimes\Delta_\phi^{i t}$.
The cocycle factors in the formulas for the modular operators are still present.
\end{remark}

\subsubsection{Correspondences of crossed product}
The right action of $\M$ on $L^2(\M)$ is determined by $J_{\tilde{\varphi}}.$ 
Concretely, under the identification $L^2(\M)\cong \ell^2(\Gamma)\otimes L^2(\B)$, the formula for the right action  is given by 
\begin{align*}
   (\delta_h\otimes \xi)u_ga=J_{\tilde\phi}(u_ga)^\ast J_{\tilde\phi}(\delta_h\otimes \xi)=\delta_{hg}\otimes\kappa_{h,g}J_{\B} a^\ast J_{\B} V_g^\ast \xi
\end{align*}
for $\xi\in L^2(\B), a\in \B$, $g,h\in \Gamma$. Indeed, 
\begin{align*}
    (\delta_h\otimes \xi)a
    =J_{\tilde\phi}a^\ast J_{\tilde\phi}u_h(\delta_e\otimes \xi)=u_hJ_{\tilde\phi}a^\ast J_{\tilde\phi}(\delta_e\otimes \xi)=u_h(\delta_e\otimes J_{\B} a^\ast J_{\B} \xi)=\delta_h\otimes J_{\B} a^\ast J_{\B} \xi
\end{align*}
and 
\begin{align*}
    (\delta_h\otimes \xi)u_g
    &=u_hJ_{\tilde\phi}u_g^\ast J_{\tilde\phi}(\delta_e\otimes \xi)=u_hJ_{\tilde\phi}u_{g^{-1}}\nu_{g,g^{-1}}^\ast (\delta_e\otimes J_{\B} \xi)\\
    &=u_hJ_{\tilde\phi}(\delta_{g^{-1}}\otimes  \nu_{g,g^{-1}}^\ast J_{\B} \xi)=u_h(\delta_g\otimes\nu_{g^{-1},g}^\ast V_{g^{-1}}J_{\B} \nu_{g,g^{-1}}^\ast J_{\B} \xi)\\
    &=\delta_{hg}\otimes\kappa_{h,g} V_g^\ast \xi
\end{align*}
where we have used \eqref{eq:rel_nu and V}  to write
\begin{align*}
    V_g^\ast
    =\nu_{g^{-1},g}^\ast J_{\B}\nu_{g^{-1},g}^\ast J_{\B} V_{g^{-1}}
    =\nu_{g^{-1},g}^\ast V_{g^{-1}} J_{\B}\alpha_{g^{-1}}^{-1}(\nu_{g^{-1},g}^\ast) J_{\B}=\nu_{g^{-1},g}^\ast V_{g^{-1}}J_{\B} \nu_{g,g^{-1}}^\ast J_{\B}.    
\end{align*}
More generally, if $(\pi,\H_\pi)$ is a unitary representation of the group $\Gamma$, then $L^2(\M)\otimes\H_\pi\cong\ell^2(\Gamma)\otimes L^2(\B)\otimes\H_\pi$ can be turned into an $\M$-correspondence by the following actions: 
\begin{align*}
    \pi_\ell(u_ga)f\otimes \eta&=(u_ga)f\otimes\pi_g\eta,\\
    \pi_r(u_ga)f\otimes\eta&=(J_{\tilde\phi}(u_ga)^\ast J_{\tilde\phi}f)\otimes\eta
\end{align*}
for $g\in \Gamma,f\in L^2(\M),\eta\in\H_\pi$, that is, for $\xi\in L^2(\B)$, $h\in \Gamma$,
\begin{align*}
    &\pi_\ell(u_ga)(\delta_h\otimes \xi\otimes\eta)=\delta_{gh}\otimes \kappa_{g,h}\alpha_h^{-1}(a)\xi\otimes\pi_g\eta,\\
    &\pi_r(u_ga)(\delta_h\otimes \xi\otimes\eta)=\delta_{hg}\otimes \kappa_{h,g}J_{\B} a^\ast J_{\B} V_g^\ast \xi\otimes\eta.
\end{align*}
Thus,  on the second leg of the tensor $L^2(\M)\otimes\H_\pi$, $\M$ acts trivially from   the right.

\subsubsection{Construction of a derivation on the crossed product}

Let $\pi:\Gamma\to B(\K_{\mathbb R})$ be an orthogonal representation of $\Gamma$ on a Hilbert space $\K_{\mathbb R}$. Set $\K=\K_{\mathbb R}\otimes_{\mathbb R}\mathbb C$ to be its complexification and write $I$ for the conjugation map $\xi\mapsto\bar\xi$ on $\K$. Consider $\H=L^2(\M)\otimes \K=\ell^2(\Gamma)\otimes L^2(\B)\otimes\K$ as an $\M$-correspondence with the action defined above. On $\H$, define the operators
\begin{align*}
    \mathcal{U}_t&=\Delta_{\tilde\varphi}^{it}\otimes 1,\\
    \mathcal{J}(\delta_h\otimes\xi\otimes\eta)&=J_{\tilde\varphi}(\delta_h\otimes\xi)\otimes \pi_{h^{-1}} (I\eta)=\delta_{h^{-1}}\otimes \nu_{h,h^{-1}}^*V_hJ_{\B}\xi\otimes\pi_{h^{-1}}(I\eta).
\end{align*}
Then the triple $(\H, (\mathcal{U}_t)_{t\in \mathbb R}, \mathcal{J})$ is a Tomita correspondence. 

We now identify a Tomita algebra inside $L^2(\M,\tilde\phi)$.
Let $\M_{\mathrm{an}}$ be the entire analytic elements for
$\sigma^{\tilde\phi}$ and put
\begin{equation}\label{eq:analyticcore}
\mathcal{A}=\text{span}\{u_ga; g\in \Gamma, a\in \B \mbox{ such that }u_ga\in \M_{\mathrm{an}}\},
\quad \mathfrak A=\mathcal{A}\tilde\phi^{1/2}.
\end{equation}
Then $\mathcal{A}$ is a unital, $\ast$- subalgebra of $\M$ with finite Fourier
support and  it is closed under $\sigma_z^{\tilde\phi}$ for every $z\in\C$. Moreover, $\mathcal{A}$ is strongly dense in $\M$. Indeed, density can be obtained  by Gaussian
smoothing: if $x=u_g a$, then
\begin{equation}\label{eq:smoothing}
x_n=\sqrt{\frac n\pi}\int_{\R}e^{-nt^2}\sigma_t^{\tilde\phi}(x)\,dt=\sqrt{\frac n\pi}u_g\int_{\R}e^{-nt^2}d_g(t)\sigma_t^{\phi}(a)\,dt
 \in u_g\B\cap\M_{\mathrm{an}} \subseteq\mathcal{A}
\end{equation}
such that $x_n\to x$ strongly. In particular, $\mathfrak A$ is a  dense Tomita algebra.

Now we define a modular derivation associated to a 1-cocycle.
Let $c:\Gamma\to\K_\mathbb R$ be a 1-cocycle of the representation $\pi$, that is, $c$ satisfies $c(gh)=c(g)+\pi_g c(h)$ for $g,h\in \Gamma$.  
Consider the densely defined operator $\partial$ on $L^2(\M,\tilde\phi)$ with domain
\[\dom\partial:=\text{span}\{\delta_g\otimes \xi; g\in \Gamma, \xi\in L^2(\B)\}\]
and the formula  
\begin{align*}
\partial(\delta_g\otimes\xi)=i\delta_g\otimes\xi\otimes c(g),\;\;\;\;g\in\Gamma,\xi\in L^2(\B).
\end{align*}
The smoothing argument above also shows that $\partial(x_n\tilde\phi^{1/2})\to \partial(x\tilde\phi^{1/2})$ for $x=u_ga$, which means  that $\mathfrak A$ forms a core for $\partial$.
\begin{lemma}
  The operator  $\partial_{|_{\mathfrak{A}}}$ is a $\tilde\phi$-modular derivation.
\end{lemma}
\begin{proof}
For any $g\in\Gamma, \xi\in L^2(\B),\eta\in \K$, one verifies that $\delta_g\otimes\xi\otimes\eta\in \dom(\partial^\ast)$ such that 
\[\partial^\ast(\delta_g\otimes\xi\otimes\eta)=-i\langle c(g),\eta\rangle \delta_g\otimes\xi\]
This shows that $\dom(\partial^\ast)$ is dense in $L^2(\M)\otimes\K$; hence $\partial$ is closable and so is $\partial_{|_\mathfrak{A}}$. This also shows that
\[\partial^\ast\bar\partial(\delta_g\otimes\xi)=\|c(g)\|^2\delta_g\otimes\xi,\;\;g\in \Gamma,\xi\in L^2(\B).\]
Since $\partial\subseteq \overline{\partial_{|_\mathfrak{A}}}$, it suffices to show the required derivation formulas  for $\partial$ restricted to the set $\text{span}\{\delta_g\otimes b\phi^{1/2}; g\in \Gamma, b\in \B\}$. The following identities are straightforward to verify:
\begin{align*}
    \partial\circ\Delta_{\tilde\varphi}^{it}=\mathcal{U}_t\circ\partial,\;\;\partial\circ J_{\tilde\varphi}=\mathcal{J}\circ\partial,\;\;t\in\mathbb R.
\end{align*}
To show the derivation rule, we first write the following formula for $g,h\in \Gamma$ and $a,b\in \B$: 
\begin{align*}
   u_g a\tilde\varphi^{1/2}u_hb&=u_ga(\delta_h\otimes J_{\B} b^\ast J_{\B} V_h^\ast\phi^{1/2})=\delta_{gh}\otimes\kappa_{g,h}\alpha_h^{-1}(a)J_{\B} b^\ast J_{\B} V_h^\ast\phi^{1/2}\\
   &=(\delta_g\otimes a\phi^{1/2})u_h b=\delta_{gh}\otimes \kappa_{g,h} J_{\B} b^\ast J_{\B} V_h^\ast a\phi^{1/2}
\end{align*}
On the one hand, we have 
\begin{align*}
    \partial(u_g a\tilde\varphi^{1/2}u_hb)=i\delta_{gh}\otimes\kappa_{g,h}\alpha_h^{-1}(a)J_{\B} b^\ast J_{\B} V_h^\ast\phi^{1/2}\otimes c(gh).
\end{align*}
On the other hand, a similar calculation of the left-right action on $\H$ yields  
\begin{align*}
    \mathcal{J}&(u_h b)^\ast\mathcal{J}
\partial(u_ga\tilde\varphi^{1/2})+u_ga\partial(\tilde\varphi^{1/2}u_hb)
    \\
    &= \mathcal{J}(u_h b)^\ast\mathcal{J}(i\delta_g\otimes a\phi^{1/2}\otimes c(g))+iu_ga(\delta_h\otimes J_{\B} b^\ast J_{\B} V_h^\ast\phi^{1/2}\otimes c(h))\\
    &=i\delta_{gh}\otimes \kappa_{g,h} J_{\B} b^\ast J_{\B} V_h^\ast a\phi^{1/2}\otimes c(g)\\
    &\qquad+i\delta_{gh}\otimes \kappa_{g,h}\alpha_{h^{-1}}(a)J_{\B} b^\ast J_{\B} V_h^\ast\phi^{1/2}\otimes\pi_g c(h)\\
    &=i\delta_{gh}\otimes\kappa_{g,h}\alpha_h^{-1}(a)J_{\B} b^\ast J_{\B} V_h^\ast\phi^{1/2}\otimes c(gh).
\end{align*}
So we get $ \partial(u_g a\tilde\varphi^{1/2}u_hb)=\mathcal{J}(u_h b)^\ast\mathcal{J}
\partial(u_ga\tilde\varphi^{1/2})+u_ga\partial(\tilde\varphi^{1/2}u_hb)$ which is precisely the required derivation condition.
\end{proof}


\subsubsection{Solidity results on crossed products}
We now prove relative solidity results for crossed products with under additional hypotheses on the cocycle and the representation. 
Let $(\alpha,\nu)\curvearrowright\B$ be a cocycle action and $\M:=\B\rtimes\Gamma$ be the cocycle crossed product.

We now assume that $c$ is a  cocycle of a representation $(\pi,\H_\pi)$ such that $\pi\prec\lambda_\Gamma$, then verify the main hypothesis of Theorem \ref{main_theorem} for the derivation $\partial$ constructed above on $\B\rtimes\Gamma$. As noted above, we have
\[\mathcal{L}_2(\delta_g\otimes\xi)=\partial^\ast\bar\partial(\delta_g\otimes\xi)=\|c(g)\|^2\delta_g\otimes\xi,\;\;g\in\Gamma, \xi\in L^2(\B).\]
Therefore, by Lemma \ref{lem:multiplies of positive definite functiona}, the associated  deformations are  the Fourier multipliers given by 
\[
  \Phi_t(u_ga)=e^{-t\lVert c(g)\rVert^2}u_ga,\quad \theta_\alpha(u_ga)=\frac{\alpha}{\alpha+\|c(g)\|^2}u_ga,\quad g\in\Gamma, t,\alpha>0.
\]
\begin{lemma}\label{lem:multiplies of positive definite functiona}
    Let $f:\Gamma\to\mathbb C$ be a normalized positive definite function. Then  the assignment $\Phi:u_ga\mapsto f(g)u_ga$ for $g\in \Gamma, a\in \B$ extends to a normal $\tilde\phi$-preserving unital completely positive map on $\B\rtimes\Gamma$. If $f$ is real-valued, then $\Phi$ is $\tilde\phi$-GNS symmetric.
\end{lemma}
\begin{proof}
   Define the bounded operator $\Phi$ on $\mathbb B(\ell^2(\Gamma)\otimes L^2(\B))$ by 
    the assignment 
    \[\langle\delta_h\otimes\eta,\Phi(T)(\delta_g\otimes\xi)\rangle=f(hg^{-1})\langle\delta_h\otimes\eta,T(\delta_g\otimes\xi)\rangle,\] 
    $g,h\in \Gamma, \eta,\xi\in L^2(\B), T\in \mathbb B(\ell^2(\Gamma)\otimes L^2(\B))$
 One easily verifies that $\Phi$ is a normal completely positive map using the positive-definiteness of $f$.  A simple calculation shows that $\Phi$ maps $\M$ to itself and we have $\Phi(u_ga)=f(g)u_ga$ for $g\in \Gamma, a\in \B$. The condition $f(e)=1$ implies that $\Phi$ is $\tilde\phi$-preserving. When $f$ is real-valued,  the $L^2$-implementation of $\Phi$ on $L^2(\B\rtimes\Gamma)$ is self-adjoint implying that $\Phi$ is GNS symmetric.
\end{proof}

Coming back to $\partial$, the exact null algebra is $\Null(\overline\partial)=B\rtimes\ker c$,
where  $\ker c:=\{g\in\Gamma; c(g)=0\}$ is a  subgroup. In particular, the inclusion $\B\subseteq \B\rtimes\Sigma$ is with expectation.

Since $\pi\prec{\lambda_\Gamma}$ by assumption and the right action on the second leg of the tensor is trivial, it follows that as $\M$-correspondence,
\[L^2(\M)\otimes\H_\pi\prec L^2(\M)\otimes\ell^2(\Gamma)\cong L^2(\M)\otimes_\B L^2(\M),\]
where the last isomorphism is a consequence of  Fell-absorption principle. Indeed, if we identify  $L^2(\M)=\ell^2(\Gamma)\otimes L^2(\B)$, then the map $U:L^2(\M)\otimes_\B L^2(\M)\to L^2(\M)\otimes\ell^2(\Gamma)$ by
\[U((\delta_g\otimes b\phi^{1/2})\otimes_\B(\delta_h\otimes \xi))=\delta_{gh}\otimes\kappa_{g,h}\alpha_{h}^{-1}(b)\xi\otimes\delta_{g} ,\;\;\;b\in\B,\xi\in L^2(\B)\;g,h\in \Gamma\] 
extends to a $\M$-bilinear unitary operator.

For $t\geq0$, let $(\rho_t,\mathsf K_t,\xi_t)$ be the GNS representation of the positive-definite function $f_t(g)=e^{-t\|c(g)\|^2},g\in \Gamma$. Then $(L^2(\M)\otimes\K_t, \tilde\phi^{1/2}\otimes\xi_t)$ is the minimal Stinespring dilation of $\Phi_t$. Now make $\mathsf L_t:=L^2(\B)\otimes\K_t$ into a $\B$-$\B$ correspondence with the standard left-right action of $\B$ on $L^2(\B)$ and trivial action on $\K_t$.  Since the right action of $\B$  on the second leg of $L^2(\M)\otimes\K_t$ is trivial, again by Fell absorption principle,  we have the following isomorphism as $\M$-$\B$-correspondence:
\begin{align*}
    L^2(\M)\otimes_\B\mathsf L_t\cong L^2(\M)\otimes\K_t. 
\end{align*}
Indeed, the map $V:L^2(\M)\otimes_\B\mathsf L_t\to L^2(\M)\otimes\K_t$ by 
\begin{align*}
    V((\delta_g\otimes b\phi^{1/2})\otimes_\B(\xi\otimes\eta))=\delta_g\otimes b\xi\otimes\rho_t(g)\eta,\;\;\;g\in \Gamma,\; b\in\B,\xi\in L^2(\B),\;\eta\in \K_t
\end{align*}
extends to a $\M$-$\B$ linear unitary.
Thus we have shown that the Stinespring representation $\H_{\Phi_t}$ is isomorphic to $ L^2(\M)\otimes_\B\mathsf L_t$ for a $\B$-correspondence $\mathsf L_t$. Therefore, the hypothesis of Theorem  is satisfied.

We next assume that the cocycle is {\em proper}, that is,  for every $R>0$, the set $\{g\in\Gamma; \|c(g)\|\leq R\}$ is finite. In other words, $\|c(g)\|\to\infty$ as $g\to\infty$. Since
\[\theta_\alpha^{(2)}(\delta_g\otimes \xi)=\frac{\alpha}{\alpha+\|c(g)\|^2}(\delta_g\otimes\xi),\;\;g\in \Gamma, \xi\in L^2(\B),\]
it follows that 
\[\theta_\alpha^{(2)}=\sum_{g\in \Gamma}\frac{\alpha}{\alpha+\|c(g)\|^2} u_ge_\B u_g^\ast\]
where the sum converges in norm because of properness of $c$. This proves that $\theta_\alpha^{(2)}$ is compact relative to $\B$.

Combining this with Theorems \ref{thm:omega-solidity} and \ref{thm:stable solidity} gives the following result.

\begin{theorem}\label{prop:omega solidity in crossed products}
    Let $\Gamma$ be a countable discrete group admitting a cocycle $c$ into a representation weakly contained in the left regular representation, and let $(\alpha,\nu)$ be a cocycle  action on a von Neumann algebra $\B$ with separable predual. 
    If $c$ is proper, then $\B\rtimes\Gamma$ is stably solid and $\omega$-solid relative to $\B$ for $\omega\in\beta\mathbb N\setminus\mathbb N$.
\end{theorem}

In the tracial setting, much stronger results are known about relative solidity with less assumptions (see \cite{PV14}). In the nontracial setting,
Theorem \ref{prop:omega solidity in crossed products} extends the relative solidity framework of \cite[Theorem E]{Iso25} to possibly nonamenable base algebras, under a different hypothesis on the acting group. It also accommodates cocycle actions and yields both relative $\omega$-solidity and relative stable solidity. Its principal contribution is thus a unified rigidity result for cocycle crossed products over arbitrary bases in the nontracial setting.

The following result is a consequence of Theorem \ref{thm:primeness for type I base and compact resolvent}.
\begin{theorem}\label{prop:primeness in crossed products}
    Let $\Gamma$ be a countable discrete group admitting a cocycle $c$ into a representation weakly contained in the left regular representation, and let $(\alpha,\nu)$ be a cocycle  action on a type I von Neumann algebra $\B$ with separable predual. 
    If $c$ is proper and if $\M:=\B\rtimes\Gamma$ is a nonamenable factor, then $\M$ is prime.
\end{theorem}

Let us also explicitly state the above result for $\B=\mathbb C$. Although our definition of stable solidity requires tensoring with von Neumann algebras with separable predual, the  following result on finite von Neumann algebras can be proven in a similar fashion by requiring the tensoring with only with finite von Neumann algebras possibly without the separability assumptions (cf. Proposition \ref{prop:convergence of deformation with compact resolvent for tracial Q}).  Much stronger version of this result is known (see \cite[Theorem 3.1 (iii)]{DonVae25}, \cite[Theorem 4.13]{DV25}). 

\begin{corollary}
     Let $\Gamma$ be a countable discrete group admitting a proper 1-cocycle $c$ into a representation weakly contained in the left regular representation, let $\nu:\Gamma\times\Gamma\to \mathbb T$ be a 2-cocycle, and let $\omega\in\beta\mathbb N\setminus\mathbb N$. Then $L_\nu(\Gamma)$ is $\omega$-solid and stably solid in the sense of \cite{DV25}.
\end{corollary}

\subsubsection{Fullness and primeness in crossed products}
We also record the following folklore results which is a  direct consequence of Corollary \ref{cor:fullness and rigidity} and Theorem \ref{cor:primeness for nonameble tracial algebra with mixing derivation}.

\begin{corollary}
    Let $\Gamma$ be a nonamenable countable group admitting an unbounded cocycle $c$ into a mixing representation weakly contained in the left regular representation. If $\nu:\Gamma\times\Gamma\to \mathbb T$ is a $2$-cocycle, $L_\nu(\Gamma)$ is full and prime.
\end{corollary}

A relative version of the above spectral gap result can be given as follows by invoking Corollary \ref{cor:spectral gap in tracial case for mixing derivation}. It appears to be a new result as we could not locate any reference.
\begin{corollary}
     Let $\Lambda\leq\Gamma$ be a subgroup of a countable group $\Gamma$ such that $\Lambda$ is nonamenable, and let $\nu:\Gamma\times\Gamma\to \mathbb T$ be a $2$-cocycle. If $\Gamma$ admits a cocycle $c$ into a mixing representation weakly contained in the left regular representation such that $c_{|_\Lambda}$ is unbounded,  then $L_\nu(\Lambda)\subseteq L_\nu(\Gamma)$ has w-spectral  gap property.
\end{corollary}

\subsection{Discrete measured groupoid von Neumann algebras}

Our final examples are modular derivations on discrete measured groupoid von Neumann algebras (\cite{AD13, ADR00}).
Given a groupoid $\G$, denote by $X$ its unit space and by $\G^{(2)}$ the set of composable pairs. The  source and target maps from $\G$ to ${X}$ are denoted respectively by $s$ and $t$. To make our conventions explicit, we have $s(\gamma)={\gamma}^{-1}\gamma$ and $t(\gamma)={\gamma}{\gamma^{-1}}$ for $\gamma\in \G$, and $(\gamma_1, \gamma_2)\in \mathcal{G}^{(2)}$ if and only if $s(\gamma_1)=t(\gamma_2)$, in which case $\gamma_1\gamma_2$ makes sense. For an element $x\in {X}$, the {\em source} and {\em target fibers} are given by 
    $\G_x:=s^{-1}(\{x\}),\;\;\G^x:=t^{-1}(\{x\}).$
    
A {\em discrete Borel  groupoid} is a groupoid $\G$ endowed with a standard Borel structure such
that the unit space ${X}$ is a Borel subset of $\G$ and the source, target, inverse and product are Borel maps, and each fiber $\G_x$ (equivalently $\G^x$) is at most countable for  $x\in {X}$. Here $\mathcal{G}^{(2)}$ has the Borel
structure induced by $\G\times\G$.

Let $\mu$ be a probability measure on the unit space ${X}$. Define the $\sigma$-finite measures $\nu_s$ and $\nu_t$ on $\G$ by 
\begin{align*}
    \nu_s(A):=\int_{{X}}|\G_x\cap A|\,d\mu(x),\qquad\nu_t(A):=\int_{{X}}|\G^x\cap A|\,d\mu(x)
\end{align*}
for a Borel set $A\subseteq\G$.
Equivalently, for any non-negative function $F$ on $\G$,
\begin{align*}
\int_\G F\,d\nu_s=\int_X\sum_{\gamma\in\G_x}F(\gamma)\,d\mu(x),\qquad
\int_\G F\,d\nu_t=\int_X\sum_{\gamma\in\G^x}F(\gamma)\,d\mu(x).
\end{align*}
Note that $\nu_t$ is the push-forward of $\nu_s$ under inversion $\iota:\gamma\mapsto {\gamma^{-1}}$, that is, $\nu_t=\iota_\ast\nu_s$.
The measure $\mu$ is called {\em quasi-invariant} if $\nu_s$ and $\nu_t$ are equivalent, that is, their null sets are the same. In this case, write 
\[\mathfrak d=\frac{d\nu_t}{d\nu_s}:\G\to (0,\infty)\]
for the Radon-Nikodym derivative of the equivalent measures. Then $\mathfrak d$ satisfies the following multiplicative $1$-cocycle rule:
\[\mathfrak d({{\gamma_1\gamma_2}})=\mathfrak d(\gamma_1)\mathfrak d(\gamma_2),\;\;\; \mathfrak d({\gamma^{-1}})=\mathfrak d(\gamma)^{-1},\;\;\mathfrak d(x)=1,\]
 a.e. $({\gamma_1,\gamma_2})\in \mathcal{G}^{(2)}$, $ \gamma\in \G,x\in {X}.$
The pair $(\G,\mu)$ is called a 
{\em (non-singular) discrete measured
groupoid} if $\mu$ is quasi-invariant and $({X},\mu)$ is a standard probability space.

\subsubsection{Twisted groupoid von Neumann algebra}
Let $(\G,\mu)$ be a discrete measured groupoid.
A {\em normalized measurable $2$-cocycle} is a map $\sigma:\G^{(2)}\to\mathbb T$ satisfying
\begin{align*}
\sigma({\gamma,\alpha})\sigma({{\gamma\alpha}},\beta)
 =\sigma(\alpha,\beta)\sigma(\gamma,{\alpha\beta}),\qquad
\sigma(t(\gamma),\gamma)=\sigma(\gamma,s(\gamma))=1.
\end{align*}
 Substituting $(\gamma,{\gamma^{-1}},\gamma)$ gives
$\sigma(\gamma,{\gamma^{-1}})=\sigma({\gamma^{-1}},\gamma).$
The untwisted case is $\sigma=1$.  
Consider the set $ I(\G)$ of all Borel functions $f:\G\to\mathbb C$ such that
$\|f\|_I:=\max\{\|x\mapsto \sum_{\gamma\in \G_x}|f(\gamma)|\|_\infty, \|x\mapsto \sum_{\gamma\in \G^x}|f(\gamma)|\|_\infty\}$
is finite. Then $( I(\G), \|\cdot\|_I)$ forms a $\ast$-Banach algebra with the following product and adjoint:
\begin{align*}
    f\ast_\sigma g(\gamma)=\sum_{\alpha\beta={\gamma}}f(\alpha)g(\beta)\sigma(\alpha,\beta),\quad f^\sharp(\gamma)=\overline{\sigma({\gamma^{-1}},\gamma)}\,\overline{f({\gamma^{-1}})},\;\gamma\in \G
\end{align*}
for $f,g\in  I(\G)$. 
For each $f\in  I(\G)$, define a bounded operator $L_\sigma(f)$ on $L^2(\G,\nu_s)$ by 
\begin{align*}
    (L_\sigma(f)\xi)(\gamma)=f\ast_\sigma\xi(\gamma)=\sum_{\alpha\beta={\gamma}}f(\alpha)\xi(\beta)\sigma(\alpha,\beta),\;\;\;\;\xi\in L^2(\G,\nu_s),\, \gamma\in \G.
\end{align*}
Then  $L_\sigma(f)$ is a bounded operator with
    $\|L_\sigma(f)\|\leq \|f\|_I$. Moreover, we have $L_\sigma(f)^\ast=L_\sigma(f^\sharp),$ $ L_\sigma({f})L_\sigma({g})=L_\sigma({f}\ast_\sigma {g})$ for ${f},{g}\in I(\G)$.
     The {\em twisted measured-groupoid von Neumann algebra} of  $(\G,\mu)$ is defined to be
 \[L(\G,\mu,\sigma):=\{L_\sigma(f); f\in  I(\G)\}^{''}\subseteq \mathbb B(L^2(\G,\nu_s)).\]
 We  write $L_\sigma(\G)$ when the measure class is understood.
  Each element of $L_\sigma(\G)$ is of the form $L_\sigma(f)$ for some $f\in L^2(\G,\nu_s)$ such that $\xi\mapsto f\ast_\sigma\xi$ is well defined on $L^2(\G,\nu_s)$ (\cite[Proposition 5.1]{BCDK26}). Such $f$ is called {\em left convolver}.  In particular,  $L_\sigma(\G)=\{L_\sigma(\mathbf1_{B}); B\subseteq\G \mbox{ is a bisection}\}''$, where a {\em bisection} $B$ is a Borel subset of $\G$ such that the restrictions of $s$
and $t$ to $B$ are injective.
Moreover,
\[L^2(\G,\nu_s)\cong \int_{{X}}^\infty\ell^2(\G_x)\,d\mu(x)\;\;\mbox{ and }\;\;L_\sigma(f)=\int_X^\infty L_x(f)\,d\mu(x),\]
  where $L_x(f):\ell^2(\G_x)\to \ell^2(\G_x)$ is the operator 
$L_x(f)\delta_\gamma=\sum_{\alpha\in \G_x}f(\alpha\gamma^{-1})\sigma(\alpha\gamma^{-1},\gamma)\delta_\alpha$ for $\gamma\in\G_x$. Here $\{\delta_\gamma; \gamma\in\G_x\}$ is the canonical orthonormal basis for $\ell^2(\G_x)$.

\subsubsection{The vacuum state and the diagonal algebra}
Set 
$\M:=L_\sigma(\G)$ and $\B:=L^\infty(X,\mu)$. 
 Then $\B$ can be realized as a unital subalgebra of $\M$ as functions supported on the unit space $X$, that is,  $(L_\sigma(f)\xi)(\gamma)=f(t(\gamma))\xi(\gamma)$ is the multiplication operator $M_{f\circ t}$  for $f\in \B$, $\xi\in L^2(\G,\nu_s), \gamma\in \G.$
 Note that $X$ is a bisection of $\G$ and $\mathbf1_{X}$ is a unit vector of $L^2(\G,\nu_s)$. Consider the state $\varphi$ on $L_\sigma(\G,\mu)$ by 
 \begin{align*}
     \phi(T)=\langle \mathbf1_{X}, T \mathbf1_{X}\rangle,\;\;T\in L_\sigma(\G,\mu).
 \end{align*}
By normalization of $\sigma$, we have $L_\sigma(f)\mathbf1_X=f$, so that $\phi(L_\sigma(f))=\int_{X}f(x)d\mu(x),$ and $ \phi(L_\sigma(f)^\ast L_\sigma(\gamma))=\langle f,g\rangle$
for $f\in  I(\G), \gamma\in\G$. 
The state $\phi$ is faithful. Moreover, $\mathbf 1_X$ is cyclic, so $L^2(\G,\nu_s)$ will be realized as standard space for $L_\sigma(\G)$.

Compression to $L^2(X,\mu)$ gives the faithful normal expectation
$E_\B:L_\sigma(\G)\to \B$ by $ E_\B(L_\sigma(f))(x)=f(x)$ such that $\varphi=\mu\circ E_\B$. In particular, the Jones projection $e_\B$ satisfies $e_\B\xi=\xi_{|_X}$  and $e_\B L_\sigma(f) e_\B=M_{f_{_X}}e_\B$. 

\subsubsection{Tomita theory}\label{subsec:groupoid-modular}

We now compute the modular theory of the state $\phi$ on $\M:=L_\sigma(\G,\mu)$. 
The untwisted modular formulas  are recorded in \cite{AD13}.

We identify $I(\G)$ as a subspace of $L^2 (\G,\nu_s)$ via $f\mapsto L_\sigma(f)\mathbf{1}_X$. Let $I_\infty(\G)$ be the subspace of $I(\G)$ consisting of functions $f$ such that $f\mathfrak d$ is bounded from below and above.
For $f\in I_\infty(\G)$, we have  $L_\sigma(f)^*=L_\sigma(f^\sharp)$, so it follows that $f\in \dom(S_\phi)$ such that 
 $S_\phi f=f^\sharp$.
Since $|\sigma(\gamma^{-1},\gamma)|=1$, we get
\begin{equation*}
 \|S_\phi f\|_2^2
 =\int_\G |f(\gamma^{-1})|^2\,d\nu_s(\gamma)
 =\int_\G\mathfrak d(\gamma)|f(\gamma)|^2\,d\nu_s(\gamma).
\end{equation*}
 The modular data can now be easily proven as follows:
\begin{proposition}\label{prop:groupoid-closed-S}
The domains and formulas for $S_\phi$, $\Delta_\phi$  are 
\begin{align*}
 &\dom S_\varphi
 =\left\{\xi\in L^2(\G,\nu_s):\int_\G\mathfrak d(\gamma)|\xi(\gamma)|^2\,d\nu_s(\gamma)<\infty\right\},\;\; \\
& \dom\Delta_\varphi
 =\left\{\xi\in L^2(\G,\nu_s):\int_\G\mathfrak d(\gamma)^2|\xi(\gamma)|^2\,d\nu_s(\gamma)<\infty\right\},
 \end{align*}
with
\[(S_\varphi\xi)(\gamma)
=\overline{\sigma(\gamma^{-1},\gamma)}\,\overline{\xi(\gamma^{-1})},\;\;\;(\Delta_\varphi\xi)(\gamma)
 =\mathfrak d(\gamma)\xi(\gamma).\]
 Moreover, the formula for  $J_\phi$ and modular automorphism groups are:
\begin{align*} 
 (J_\varphi\xi)(\gamma)
 =\mathfrak d(\gamma)^{1/2}\overline{\sigma(\gamma^{-1},\gamma)}\,
   \overline{\xi(\gamma^{-1})},\qquad \sigma_t^\phi(L_\sigma(f))=L_\sigma(\mathfrak d^{it}f)
\end{align*}
 for $\xi\in L^2(\G,\nu_s)$ and $f\in I(\G),\gamma\in\G,t\in\mathbb R.$
\end{proposition}

In particular, $\Delta_\phi$ is a multiplication operator $M_{\mathfrak d}$. Hence 
the functional calculus is  explicit here. For $z\in\C$, define
$\mathfrak d^z=\exp(z\log\mathfrak d)$. Then
\begin{equation*}
 \Delta_\varphi^z\xi=\mathfrak d^z\xi,\qquad
 \dom\Delta_\varphi^z
 =\left\{\xi\in L^2(\G,\nu_s):\int_\G\mathfrak d^{\,2\operatorname{Re}z}|\xi|^2\,d\nu_s<\infty\right\}.
\end{equation*}
In particular, since $\mathfrak d_{|_X}\equiv1$, it follows that $\sigma_t^{\phi}{(L(f))}=L(f)$ for $f\in \B=L^\infty(X,\mu)$. Thus 
 $\B\subseteq \M^\varphi,$ $ E_{\B}\circ\sigma_t^\varphi=E_{\B} $ and $\phi\circ E_\B=\phi.$
The right action of $\M$ on $L^2(\G,\nu_s)$ is given by 
$  \xi\cdot L_\sigma(f)=\xi\ast_\sigma(\mathfrak d^{1/2}f),$ $ f\in I(\G),\xi\in L^2(\G,\nu_s)$. 

Thus $(L^2(\G,\nu_s), L^2(\G,\nu_s)_+, J_\phi)$ is a standard form for $\M$ where $L^2(\G,\nu_s)_+$ is the closure of the set consisting of $f\ast_\sigma J_\phi f, f\in I(\G)$. In the sequel, we shall identify $L^2(\M)$ with this standard form. 

\subsubsection{Groupoid representations}\label{subsec:groupoid-gradient}
For the notions of measurable Hilbert bundles, their direct integral and measurable representations, see \cite{TakII, AD13, FH25}.

A {\em orthogonal (resp. unitary) representation} of a groupoid $(\G,\mu)$ consists of a measurable field of real (resp. complex)
separable Hilbert spaces $(\H_{x})_{x\in X}$
and  orthogonal (resp. unitary)  operators
$\rho(\gamma):\H_{s(\gamma)}\longrightarrow \H_{t(\gamma)}$
such that
\begin{equation*}
   \rho(x)=\id_{\H_{x}},\;\;\;
   \rho({\gamma_1\gamma_2})=\rho(\gamma_1)\rho(\gamma_2),\;\;\; \mbox{for a.e.} \;\;x\in X,\,\; ({\gamma_1,\gamma_2})\in\mathcal{G}^{(2)}
\end{equation*}
and the operators are measurable: $\gamma\mapsto\langle\xi_{t(\gamma)},\rho(\gamma)\xi_{s(\gamma)}\rangle$ is measurable for any measurable section $(\xi_x)_{x\in X}$ of vectors.

\begin{example}
    The {\em left regular representation $\lambda_\G$} of $\G$ is defined by putting $\H_{x}=\ell^2(\G^x),$  and  $\lambda_G(\gamma)\delta_{\alpha}=\delta_{{{\gamma\alpha}}},$ $ \gamma,\alpha\in\G, t(\alpha)=s(\gamma).$ 
\end{example}

\subsubsection{Correspondence of $\M=L(\G,\mu,\sigma)$}
Given a representation $\rho$ on $H=(H_x)_{x\in X}$, consider the pullback bundle $t^*H\to\G$ which is the measurable field $(H_{t(g)})_{g\in\G}$ of Hilbert spaces with respect to $(\G,\nu_s)$. Define its $L^2$-space by
\begin{equation*}
   \H_\rho=L^2(\G,\nu_s;t^*H)
   :=\int_\G^\oplus H_{t(\gamma)}\,d\nu_s(\gamma).
\end{equation*}
If $\rho$ is an orthogonal representation on a real Hilbert space bundle $H_{\mathbb R}=(H_{\mathbb R, x})_{x\in X}$, then $\rho$ extends to a unitary representation  on $H=(H_x)_{x\in X}$, still denoted by $\rho$, where $H_x=H_{\mathbb R,x}\otimes_{\mathbb R}\mathbb C$  is the complexification of $H_{\mathbb R,x}$. Write $I_x$ for the conjugation map $\xi+i\eta\mapsto\xi-i\eta$ on $H_x$.
We turn $\H_\rho$ into a Tomita $\M$-$\M$ correspondence, where $\M=L_\sigma(\G,\mu)$ as follows.

\begin{proposition}
For $f\in I_\infty(\G)$ and a measurable field of vectors    $\xi=(\xi_{\gamma})_{\gamma\in \G}\in \H_\rho$ with $\xi_\gamma\in\H_{t(\gamma)}, \gamma\in\G$. the assignment
\begin{align*}
 (\pi_\ell^{\H_\rho}(L_\sigma(f))\xi))_\gamma
 &=\sum_{\alpha\beta=\gamma}
      f(\alpha)\sigma(\alpha,\beta)
      \rho(\alpha){\xi_\beta},\\
 (\pi_r^{\H_\rho}(L_\sigma(f))\xi)_\gamma &=\sum_{\alpha\beta=\gamma}
      {\xi_\alpha}\,\mathfrak d(\beta)^{1/2}
      f(\beta)\sigma(\alpha,\beta).
\end{align*}
extends to a well-defined normal bimodule actions of $\M$ on $\H_\rho$.  Moreover, if $\rho$ is the complexification of an orthogonal representation, then   $(\H_\rho, (\mathcal{U}_t)_{t\in\mathbb R},\mathcal{J})$ is a Tomita correspondence over $(\M,\phi)$ with  the formulas 
\begin{align*}
 (\mathcal U_t{\xi})_{\gamma}
   &=\mathfrak d(\gamma)^{it}{\xi}(\gamma),
   \\
 (\mathcal J{\xi})_{\gamma}
   &=\mathfrak d(\gamma)^{1/2}
      \overline{\sigma(\gamma^{-1},\gamma)}\,
      \rho(\gamma)I_{s(\gamma)}{{\xi}(\gamma^{-1})}
\end{align*}
\end{proposition}
\begin{proof}
Firstly, note that all the formulas are fiber-wise meaningful.
To prove the normal extension of $\pi_\ell=\pi_\ell^{\H_\rho}$, one can use a Fell-absorption type principal. Indeed, there is a unitary $W:\H_\rho\to\int_{X}^\oplus (\ell^2(\G_x, \H_x)\,d\mu(x)$ given by 
$(W\xi)_{x}(\gamma)=\rho(\gamma)^{-1}\xi_{\gamma}$ for $ \xi=(\xi_\gamma)_{\gamma\in \G}\in \H_\rho$. For $\zeta=(\zeta_x)_{x\in X}\in \int_\G^\oplus\ell^2(\G_x,\H_x)d\mu(x)$ with $\zeta_x\in \ell^2(\G_x,\H_x)$, we have
\begin{equation*}
 (W \pi_\ell(L_\sigma(f))W^*\zeta)_{x}(\gamma)=\sum_{\alpha\beta=\gamma}
     f(\alpha)\sigma(\alpha,\beta)\zeta(\beta),\qquad x\in X,\gamma\in\G_x.
\end{equation*}
However, the above formula  is exactly the ordinary twisted left regular
representation, tensored with $\id_{H_x}$, so it is easy to see that $L_\sigma(f)\mapsto W \pi_\ell(L_\sigma(f))W^*$ extends to a normal unital $\ast$-homomorphism from $\M$ to $\mathbb B(W(\H_\rho))$. Thus $\pi_\ell$ is a well-defined representation of $\M$ on $\H_\rho$. 
Similarly, after conjugating by the unitary $W_r:\H_\rho\to\int_X^\oplus(\ell^2(\G^x)\otimes\H_x)d\mu(x)$ given by $(W_r\xi)_x(\alpha)=\mathfrak d(\alpha)^{-1/2}\xi_\alpha$, for $\xi=(\xi_\gamma)_{\gamma\in\G}\in\H_\rho,x\in X, \alpha\in \G^x$, the right action becomes the twisted right regular representation tensored with $\id_{\H_x},$ and therefore extends normally to $\M$.

Lastly, if $\rho$ is a complexification,  one easily verifies that $(\mathcal{U}\xi)_{\gamma\in\G}$ and $(\mathcal{J}\xi)_{\gamma\in\G}$ are measurable fields of vectors preserving the norms, and that  $\mathcal{U}_t$ and $\mathcal{J}$ are unitaries. The respective compatibility of these operators with $\Delta_\phi^{it}$ and $J_\phi$ are simple calculations.
\end{proof}

\begin{lemma}\label{prop:groupoid-fusion}
The  correspondence $\H_{\lambda_\G}=\int_{\G}^{\oplus}\ell^2(\G^{t(\gamma)})d\nu_s(g)$ associated to the left regular representation $\lambda_\G$ is isomorphic to $L^2(\M)\otimes_\B L^2(\M)$ as $\M$-$\M$ correspondence via the unitary
$V:  L^2(\M)\otimes_\B L^2(\M)\to 
 \H_{\lambda_\G}$
satisfying
$( V(f\otimes_{\B}\xi))_{\gamma}(\alpha)
 =f(\alpha)\xi(\alpha^{-1}\gamma)
   \sigma(\alpha,\alpha^{-1}\gamma)$ for $f\in I(\G),\xi\in L^2(\M)$ and $\gamma,\alpha\in\G, t(\gamma)=t(\alpha)$.
\end{lemma}
\begin{proof}
  Using the $\B$-valued inner product $\langle {f},{g}\rangle_\B (x)
=
\sum_{\gamma\in\G_x}\overline{{f}(\gamma)}{g}(\gamma)$ for  $f,g\in I(\G)$,  we have
\[
\langle {f}\otimes_{\B}\xi,{g}\otimes_{\B}\eta\rangle
=
\int_X
\sum_{\beta\in\G_x}
\sum_{s(\alpha)=t(\beta)}
\overline{{f}(\alpha)}{g}(\alpha)
\overline{\xi(\beta)}\eta(\beta)\,d\mu(x).
\]
On the other hand, since $|\sigma(\alpha,\beta)|\equiv1$,
\begin{align*}
  \langle {f}\otimes_{\B}\xi,{g}\otimes_{\B}\eta\rangle&=\int_X\sum_{\gamma\in\G_x} \sum_{t(\alpha)=t(\gamma)}\overline{ f(\alpha)}\,\overline{\xi(\alpha^{-1}\gamma)}\,g(\alpha)\,\eta(\alpha^{-1}\gamma) \,d\mu(x),
\end{align*}
and by substituting $\beta=\alpha^{-1}\gamma$ in the above integral, we get $\langle V({f}\otimes_{\B}\xi),V({g}\otimes_{\B}\eta)\rangle=\langle {f}\otimes_{\B}\xi,{g}\otimes_{\B}\eta\rangle$ which proves well-definedness and isometry of $V$.

To prove surjectivity, choose a countable Borel partition
$\G=\bigsqcup_{n\ge1}T_n$
into bisections (\cite[Proposition 3.1]{BCDK26}). Then the sections of $\H_{\lambda_\G}$ of the form
$\gamma\longmapsto a(\gamma)\delta_{\alpha_n(\gamma)}$ span a dense subspace of $\H_{\lambda_\G}$,
where $a$ is bounded and supported on a bisection $S$, and $\alpha_n(\gamma)\in T_n$ is the unique element  with $t(\alpha_n(\gamma))=t(\gamma)$, 
For such a section, put  $R=T_n^{-1}S$, and define $\eta(\beta)=a(\alpha\beta)\overline{\sigma(\alpha,\beta)},$
where $\alpha\in T_n$ is the unique element such that $s(\alpha)=t(\beta)$. Then a direct calculation gives
$V(\mathbf1_T\otimes_\B\eta)_\gamma
=
a(\gamma)\delta_{\alpha_n(\gamma)}.$
Hence the range of $V$ is dense. 
 Bimodularity of $V$ is an easy calculation.
\end{proof}

\subsubsection{Weak containment of representations and  the associated correspondences}\label{subsec:groupoid-weak-containment}

Let $\pi$ and $\rho$ be unitary representations of $(\G,\mu)$ on measurable Hilbert fields
$(\H_{\pi,x})_{x\in X}$ and $(\H_{\rho,x})_{x\in X}$.  For bounded measurable
sections $\xi=(\xi_x)_{x\in X},\eta=(\eta_x)_{x\in X}$ of the field of $\pi$, put
\begin{equation*}
 c^{\pi}_{\xi,\eta}(\gamma)
 =\langle \xi_{t(\gamma)},\pi(\gamma)\eta_{s(\gamma)}\rangle,
 \qquad \gamma\in\G .
\end{equation*}
As usual, diagonal coefficients are written $c^{\pi}_{\xi}=c^{\pi}_{\xi,\xi}$.

We say that $\pi$ is {\em weakly contained} in $\rho$, denoted $\pi\prec\rho$,  if for every $\varepsilon>0$, every measurable section
$\xi$ of the field of $\pi$, and every measurable set
$E\subseteq\G$ with $\nu_s(E)<\infty$, there are measurable sections
$\eta^1,\ldots,\eta^k$ of the field of $\rho$ such that
\[
 \nu_s\!(\{\gamma\in E:|
 c^{\pi}_{\xi}(\gamma)-\sum_{j=1}^k c^{\rho}_{\eta^j}(\gamma)
 |\geq\varepsilon\})<\varepsilon .
\]

\begin{proposition} \label{thm:groupoid-bimodule-weak-equivalence}
Let $\pi,\rho$ be measurable unitary representations of  $(\G,\mu)$, and let $\sigma$ be a  scalar $2$-cocycle on $\mathcal{G}^{(2)}$.  If $
 \pi\prec\rho$, then
 ${}_\M (\H_\pi){}_\M\prec{}_\M(\H_\rho){}_\M .$
\end{proposition}
\begin{proof}
This is \cite[Proposition 4.6]{FH25} in the untwisted pmp setting. We give a sketch. Let $\mathcal{D}_\pi$ be the linear span of bounded sections in $\H_\pi$ supported on bisections.  Then  $\mathcal{D}_\pi$ is dense in $\H_\pi$. Therefore, it suffices to approximate a coefficient in $\H_\pi$ supported on a single bisection by coefficients in $\H_\rho$.  Let a section $\xi=(\xi_{\gamma})_{\gamma\in \G}\in\mathcal{D}_\pi$ be supported on a bisection $S$, and  let $\epsilon>0$.   If $f,g\in I(\G)$ are functions such that $f\mathfrak d, g\mathfrak d$ are bounded from below and above, then
\begin{equation}\label{eq:coefficient in Hpi}
    \langle\xi, L_\sigma(f)\xi L_\sigma(g)\rangle=\int_{S} g(\gamma) c_\xi^\pi(\gamma)\,d\nu_s(\gamma)
\end{equation}
 where $g$ is a bounded integrable function. By definition, $c_\xi^\pi$ and hence $\langle\xi, L_\sigma(f)\xi L_\sigma(\mathbf g)\rangle$ are approximated by coefficients of $\H_\rho$.  Since the set consisting of $L_\sigma(f)$ of the above form generates $\M$, we conclude that ${}_\M (\H_\pi){}_\M\prec{}_\M(\H_\rho){}_\M .$
\end{proof}

\subsubsection{Construction of a derivation  via Hilbert-bundle cocycles}
Let $\rho$ be a real orthogonal representation of $(\G,\mu)$ on $(\H_{\mathbb R,x})_{x\in X}$.
Let $   \H_x=\H_{\R,x}\otimes_{\R}\C$
be the complexification and denote by $I_x$ its canonical conjugation. Then $\rho(\gamma)$ naturally extends to a unitary representation of $\G$, which is also denoted by $\rho(\gamma)$. 

A {\em measurable $1$-cocycle} for the real orthogonal representation $\rho$ is a measurable section
   $c(\gamma)\in \H_{\R,t(\gamma)}$
satisfying
   $c({\gamma_1\gamma_2})=c(\gamma_1)+\rho(\gamma_1)c(\gamma_2)$ for all $ ({\gamma_1,\gamma_2})\in\mathcal{G}^{(2)}.$
We have $c(x)=0,
   c(\gamma^{-1})=-\rho(\gamma^{-1})c(\gamma),$ $ x\in X,\gamma\in\G.$
We next show that  $c$  induces a modular derivation into the $\M$-$\M$ correspondence $\H_\rho$. 

We first set some notations.
As before,  consider $I(\G)$ as a subspace via $f\mapsto L_\sigma(f)\mathbf 1_X$.
For each $n\in\mathbb N$, set $C_n=\{\gamma\in\G; 1/n\leq \mathfrak d(\gamma)\leq n, \|c(\gamma)\|\leq n\}$ and let $I_n(\G)$ be functions of $I(\G)$ supported on $C_n$.
Consider
$I_{\infty,c}(\G):=\cup_{n\in\mathbb N}I_n(\G)$.
We turn $I_{\infty,c}(\G)$ into a  Tomita algebra in $L^2(\M,\phi)$.  Indeed, for  $z\in\mathbb C$, every $f\in I_{\infty,c}(\G)$ belongs  $\dom(\Delta_\phi^z)$ with $\Delta_\phi^z f=\mathfrak d^z f$ because on each $C_n$, we have
  $|\mathfrak d(\gamma)^z|
  \leq n^{|\Re z|}.$ Moreover, this inequality says that $I_{\infty,c}(\G)$ is invariant under $\Delta_\phi^z$. It  is dense in $L^2(\M,\phi)$.
Consequently $I_{\infty,c}(\G)$ is a dense unital Tomita subalgebra of the maximal Tomita
algebra of $\phi$.

\begin{proposition}\label{prop:groupoid-cocycle-derivation}
The operator $\partial:I_{\infty,c}(\G)\subseteq L^2(\M,\varphi)\to \H_{\rho}$ given  by 
\begin{equation*}
   (\partial f)_{\gamma}=i f(\gamma)c(\gamma),
   \qquad f\in I_{\infty,c}(\G),\;\gamma\in \G
\end{equation*}
is a densely defined closable $\phi$-modular derivation. Its closure is the fiberwise multiplication by $c(\gamma)$ 
\begin{equation*}
 \dom\overline\partial=
 \left\{\xi\in L^2(\G,\nu_s):
 \int_\G \|c(\gamma)\|^2|\xi(\gamma)|^2\,d\nu_s(\gamma)<\infty\right\},
\end{equation*}
with $(\overline\partial\xi)_{\gamma}=i\xi(\gamma)c(\gamma)$ for $\xi\in \dom(\bar\partial),\gamma\in\G$. Moreover
\[ \|\overline\partial\xi\|^2
 =\int_\G\|c(\gamma)\|^2\xi(\gamma)|^2\,d\nu_s(\gamma),\]
and $\partial^\ast\bar\partial$ is the fiber-wise multiplication  by $\|c(\gamma)\|^2$.
\end{proposition}
\begin{proof}
Clearly, $(f(\gamma)c(\gamma))_{\gamma\in \G}$ is a measurable filed of vectors, and that it belongs to $\H_\rho$ for any $f\in I_{\infty,c}(\G)$.
All assertions about closures and the formulas for $\bar\partial$ and $\partial^\ast\bar\partial$ are standard facts for a measurable multiplication operator. 
The derivation rule follows from the cocycle operation:
for ${f},{g}  \in I_{\infty,c}(\G)$ and $\gamma\in\G$, on one hand we have
\begin{align*}
    (\partial({f}\ast_\sigma {g}))_{\gamma}=i({f}\ast_\sigma {g})(\gamma)c(\gamma)=i\sum_{{\alpha\beta}={\gamma}}{f}(\alpha){g}(\beta)\sigma(\alpha,\beta)c(\gamma)
\end{align*}
and on the other hand, by setting $x_i=L_\sigma(f_j), j=1,2$, we use  $\sigma_{i/2}^\phi(x_2)=L_\sigma(\mathfrak d^{-1/2}{g})$ to obtain
\begin{align*}
 (x_1(\partial {g})+(\partial {f})\cdot \mathcal{J}x_2^\ast\mathcal{J})_{\gamma}   &=(\pi_\ell^{\H_\rho}(x_1)\partial({g}))_{\gamma}+(\pi_r^{\H_\rho}(\sigma_{i/2}^\phi(x_2))\partial {f})_{\gamma}\\
    &=i\sum_{{\alpha\beta}={\gamma}}{f}(\alpha)\sigma(\alpha,\beta)\rho(\alpha){g}(\beta)c(\beta)\\&\quad\qquad+i\sum_{{\alpha\beta}={\gamma}}{f}(\alpha)c(\alpha)\mathfrak d(\beta)^{1/2} \mathfrak d^{1/2}(\beta)^{-1/2}{g}(\beta)\sigma({\alpha,\beta})\\
    &=i\sum_{{\alpha\beta}={\gamma}}{f}(\alpha){g}(\beta)\sigma(\alpha,\beta)(\rho(\alpha)c(\beta)+c(\alpha))\\
    &=(\partial({f}\ast_\sigma {g}))_{\gamma}.
\end{align*}
Finally the identities $\partial\circ J_\phi=\mathcal{J}\circ\partial$ and $\partial\circ\Delta_\phi^{it}=\mathcal{U}_t\circ\partial$ for $t\in\mathbb R$ follows by definition and cocycle rules.
\end{proof}

\begin{remark}
    The diagonal algebra $\B=L^\infty(X,\mu)$ lies in the null algebra $\Null(\partial)$, because
$c(x)=0$ for any $x\in X$. More generally, the zero set
  $\G_0=\{\gamma\in\G:c(\gamma)=0\}$
is a  discrete measurable subgroupoid of $\G$. Therefore, we have $\Null(\partial)=L_\sigma(\G_0,\mu)$.
\end{remark}

\subsubsection{The associated groupoid Markov semigroup}
Recall from \cite[Definition]{AD13} that a Borel function $\psi:\G\to \mathbb C$ is called {\em positive definite} if there is a $\mu$-null set $A\subseteq X$ such that the matrix $[\psi(g_i^{-1}g_j)]$ is non-negative for each $x\in X\setminus A$ and $\gamma_1,\ldots,\gamma_n\in \G^x$.

Now, let $c$ be a $1$-cocycle of a real orthogonal representation $\rho$ on $(\H_{\mathbb R,x})_{x\in X}$. Then for each $t>0$, the Borel map $\psi:\G\to (0,\infty)$ given by $\psi_t(\gamma)=e^{-t\|c(\gamma)\|^2}, \gamma\in \G$ is a positive definite function. Indeed, if we set $b_x(\gamma)=c({\gamma})$ for $x\in X, \gamma\in \G^x$, then
$\|c(\gamma_1^{-1}\gamma_2)\|^2=\|b_x(\gamma_1)-b_x(\gamma_2)\|^2,$ $\gamma_1,\gamma_2\in \G^x$
which immediately implies that $({\gamma_1,\gamma_2})\mapsto e^{\|c({\gamma_1^{-1}}\gamma_2)\|^2}$ for ${\gamma_1,\gamma_2}\in \G^x$ is a positive definite kernel. Moreover,  we have ${\psi_t}_{|_X}\equiv1$

By a similar argument as in \cite[Proposition 5.3]{AD13} for cocycle von Neumann algebras, since $\psi_t$ is real-valued, there is a unique normal ucp map $\Phi_t:\M\to\M$  such that
\[\Phi_t(L_\sigma(f))=L_\sigma(\psi_t f),\;\;{\Phi_t}_{|_\B}=\id_\B,\;\mbox{ and }E_\B\circ\Phi_t=E_\B\] for $f\in I(\G)$. In particular, we have $\phi\circ\Phi_t=\varphi$.
This is the cocycle multiplier construction, and in fact   the operator $\Phi_t^{(2)}$ is the multiplication operator $M_{\psi_t}$ on $L^2(\G,\nu_s)$.
By functional calculus, it follows from Proposition \ref{prop:groupoid-cocycle-derivation} that $\Phi_t^{(2)}=e^{-t\partial^\ast\bar\partial}.$ In particular, $(\Phi_t)_{t\geq0}$ forms a $\phi$-GNS symmetric quantum Markov semigroup.

\subsubsection{Relative compactness from finite bisection covers}\label{subsec:bisection-compact}
We give necessary and sufficient criterion for the compact resolvent of the derivation considered above to be relatively compact.

We say that a measurable  $1$-cocycle $c$ is {\em uniformly proper on source fibers} if the function  
\[
N_R(c):=\left\|x\mapsto|\{\gamma\in \G_x:\|c(\gamma)\|\le R\}|\right|_{L^\infty(X,\mu)}<\infty
\]
for all $R>0$. The cocycle operation implies that $\|c(\gamma^{-1})\|=\|c(\gamma)\|$ for all $\gamma\in\G$; hence the above condition is equivalent to say that $x\mapsto \sharp\{\gamma\in\G^x; \|c(\gamma)\|\leq R\}\in L^\infty(X,\mu)$ for all $R>0$.

Let $(\Phi_t)_{t>0}$ be the semigroup and let $(\theta_\alpha)_{\alpha>0}$ be the deformation constructed on $\M$ associated to $c$.

\begin{lemma}\label{lem:resolvent-companess-groupoid}
 The following are equivalent:
\begin{enumerate}
    \item $\Phi_t^{(2)}\in C^*(\M e_\B \M)$ for all $t>0$,
    \item $\theta_\alpha^{(2)}\in C^\ast(\M e_\B\M)$ for all $\alpha>0$.
\item $c$ is uniformly proper on source fibers.
\end{enumerate}
\end{lemma}

Since each $\Phi_t^{(2)}$ and $\theta_\alpha^{(2)}$ are respectively the multiplication functions $M_{\psi_t}$ and $M_{q_\alpha}$ where $\psi_t(\gamma)=e^{-t\|c(\gamma)\|^2}$ and $q_\alpha(\gamma)=\frac{\alpha}{\alpha+\|c(\gamma)\|^2}$ for $\gamma\in \G$, 
Lemma \ref{lem:resolvent-companess-groupoid} follows from the following more general result, which extends \cite[Proposition 4.7]{AD13} and also provides a necessary criteria for  a multiplier operator to be relatively compact.

\begin{lemma}
Let $F\in L^\infty(\mathcal G,\nu_s)$ and let $M_F$ denote multiplication by $F$ on $L^2(\mathcal G,\nu_s).$ Then $
M_F\in C^\ast(\M e_\B\M)$ if and only if  $\forall\varepsilon>0$, the functions  $x\mapsto \#\{\gamma\in \G_x:|F(\gamma)|\ge\varepsilon\}$ and $x\mapsto \#\{\gamma\in \G_x:|F(\gamma)|\ge\varepsilon\}$ lie in $L^\infty(X,\mu)$.
\end{lemma}
\begin{proof}
 Firstly,  $M_F$ is right $\B$-linear. Indeed, if $\xi\in L^2(\G,\nu_s)$ and $a\in \B$, then $(\xi a)(\gamma)=(\xi\ast_\sigma\mathfrak d^{1/2}a)(\gamma)=\xi(\gamma)a(s(\gamma))$  for all $\gamma\in \G$, so 
\[(M_F(\xi a))(\gamma)=F(\gamma)\xi(\gamma)a(s(\gamma))=(M_F\xi)(\gamma)a(s(\gamma))= ((M_F\xi)a)(\gamma).\]
which shows that $M_F(\xi a)=(M_F\xi)a$.
Now assume that for every $\varepsilon>0$, 
\[
N_\varepsilon
:=\norm{x\mapsto
\#\{\gamma\in\mathcal G_x:|F(\gamma)|\ge\varepsilon\}}<\infty.
\]
Set
$E_\varepsilon
=
\{\gamma\in\mathcal G:|F(\gamma)|\ge\varepsilon\}\subseteq\G.$ Since $\|M_F-M_{F\mathbf 1_{E_\varepsilon}}\|<\epsilon$, it suffices to show that $M_{F\mathbf{1}_{E_\epsilon}}\in C^\ast(\M e_\B\M)$ to conclude that $M_F\in C^\ast(\M e_\B\M)$. For $\mu$-a.e. $x$, we have $\# (\G_x\cap E_\epsilon)\leq N_\varepsilon$ and $\sharp(\G^x\cap E_\epsilon)\leq N_\epsilon$.
Since $s,t:E_\varepsilon\to X$ are countable-to-one Borel maps with uniformly bounded fibers, a theorem of Lusin-Novikov gives a decomposition
$E_\varepsilon
=
E_1\sqcup\cdots\sqcup E_{N_\epsilon}$ into mutually disjoint bisections. A simple calculation, using injectivity of $s_{|_{E_j}}, t_{|_{E_j}}$  show that $L(F\mathbf{1}_{E_j})$ is bounded such that
    $M_{F1_{E_j}}=L_\sigma(F\mathbf 1_{E_j})e_\B L(\mathbf 1_{E_j})^\ast$, so we get 
\[M_{F\mathbf{1}_{E_\varepsilon}}=\sum_{j=1}^{n_\varepsilon}L_\sigma(F\mathbf 1_{E_j})e_\B L(\mathbf 1_{E_j})^\ast \in C^\ast(\M e_\B\M).\]   
 Conversely, assume that
$M_F\in C^\ast(\M e_\B\M).$ Fix $\varepsilon>0$. By   \cite[Proposition 5.1]{BCDK26}, every element of $L_\sigma(\G)$ is of the form $L_\sigma(f)$ for some left convolver $f\in L^2(\G,\nu_s)$. Therefore, there exists a finite set $f_j,g_j\in L^2(\G,\nu)$  of left convolvers  such that  such that
$\|M_F-\sum_{j=1}^k L_\sigma(f_j)e_\B L_\sigma(g_j)\|\leq \epsilon/2.$ We claim that $\# V_x\leq k$ for each $x\in X$, where $V_x:=\{\gamma\in \G_x; |F(\gamma)|\geq \epsilon\}.$ Let $(M_{F,x})_{x\in X}$ and $(T_x)_{x\in X}$  be the measurable field of operators such that
\[M_F=\int_X^{\oplus} M_{F,x}\,d\mu(x),\qquad \sum_{j=1}^k L_\sigma(f_j)e_\B L_\sigma(g_j)=\int_X^\oplus T_x\,d\mu(x).\]
Then $\|M_{F,x}-T_x\|\leq \varepsilon/2$ for $\mu$ a.e. $x$.
Note that $\operatorname{rank}(T_x)\le k$ for $\mu$ a.e. $x$, because $\operatorname{Ran}(T_x)
\subset
\operatorname{span}
\{g_1(x),\dots,g_k(x)\}
$. Suppose to the contrary that $\# V_x>k$  for a positive-measure set of $x$'s. Since
$\operatorname{rank}(T_x)\le k,$ there exists a unit vector
$v\in \ker T_x\cap \ell^2(V_x)\subseteq \ell^2(\G_x).$ This implies
\[
\|M_{F,x}v\|=\|M_{F,x}v-T_xv\|\leq \epsilon/2.
\]
However, $|F(\gamma)|\geq\varepsilon$ for $\gamma\in V_x$, so that $\|M_{F,x}v\|\geq \epsilon$. Thus we get contradiction, and hence we have proved that $\# V_x\leq k<\infty$ for $\mu$- a.e. $x\in X$. This completes the proof.
\end{proof}


\subsubsection{Examples of groupoids}\label{chp:Examples-groupoid}
\paragraph{Transformation groupoids.}
Let $\Gamma\curvearrowright(X,\mu)$ be a nonsingular action of a discrete countable group $\Gamma$ on a probability space $(X,\mu)$. Its {\em transformation groupoid} is the set $\G=\Gamma\times X$ with unit space $\{e\}\times X$, $((g,x),(h,y))\in \mathcal{G}^{(2)}$ if $x=hy$ and
\[
s(g,x)=x,\quad t(g,x)=gx,\quad (g,hx)(h,x)=({gh},x).
\]
Let $v:\Gamma\times\Gamma\to \mathcal{U}(L^\infty(X,\mu))$ be a normalized $2$-cocycle (as defined in Sec.  \ref{sec:crossed product}). This induces a scalar $2$-cocycle $\sigma$ on $\mathcal{G}^{(2)}$ by  
$\sigma\bigl((g,hx),(h,x)\bigr)=v_{g,h}(ghx).$ All scalar $2$-cocycle of $\G$ arises this way. On shows that $L_\sigma(\G,\mu)\cong L^\infty(X,\mu)\rtimes_{\alpha,v}\Gamma$ and then we are in the setting of the cocycle crossed product of sec \ref{sec:crossed product}.

\paragraph{Equivalence relations.}
Let $(X,\mu)$ be a standard probability space and let
$\mathcal R\subseteq X\times X$ 
be a countable Borel equivalence relation.
  Then $\mathcal{R}$ becomes a measured discrete groupoid with unit space $X$ and $((x,y),(z,w))\in\mathcal{R}^{(2)}$ iff $y=z$ and
$t(x,y)=x,$ $ s(x,y)=y,$ $ (x,y)(y,z)=(x,z).$
We write $[x]_{\mathcal{R}}$ for the equivalence class of $x$.
A {\em normalized $2$-cocycle} of $\mathcal{R}$ is now written $\sigma(x,y,z)$, and satisfies
\[
\sigma(x,y,z)\sigma(x,z,w)=\sigma(y,z,w)\sigma(x,y,w),
\quad \sigma(x,x,y)=\sigma(x,y,y)=1.
\]

\subsubsection{Solidity and primeness of groupoid von Neumann algebras}

We first state an unboundedness condition on a cocycle which forces the associated deformation to be non-rigid. We let $(\G,\mu,\sigma)$ to be a discrete measured groupoid with a scalar $2$-cocycle $\sigma:\mathcal{G}^{(2)}\to\mathbb T$. We denote by $\phi$ the canonical state on $L_\sigma(\G,\mu)$.

\begin{lemma}\label{lem:strongly bounded cocycles}
 Let  $c$ be a $1$-cocycle into an orthogonal representation of $(\G,\mu)$. Let $\theta_\alpha$ be the  deformation associated to the derivation constructed above.  Let $\omega\in\beta\mathbb N\setminus\mathbb N$. Assume that $c$ is strongly unbounded in the following sense: there exist $\varepsilon>0$ and $C\geq 1$ such that for every $R>0$ there is a bisection $S_R$ with
 $\mu(s(S_R))\ge\varepsilon$ such that
 $\norm{c(\gamma)}>R$ and $ 
 1/C\leq{\mathfrak d}(\gamma)\leq C$ for all $\gamma\in S_R$. Then  $\theta_\alpha^{(2)}$ does not converge to $1$ on $L^2(L_\sigma(\G,\mu)^\omega)$ as $\alpha\to\infty$.
\end{lemma}
\begin{proof}
Write $\M=L_\sigma(\G,\mu)$.    For each $n\in \mathbb N$, let $S_n$ be the bisection satisfying the hypothesis. Set $v_n=L_\sigma(\mathbf 1_{S_n})$. Clearly, $\sup_{n\geq1}\|v_n\|\leq 1$ and we have $(v_n)_{n\geq1}\in \mathcal{M}^\omega(\M)$. Indeed, $v_n$ is $\sigma^\phi$-analytic such that $\Delta_\phi^z (v_n\phi^{1/2})=\mathfrak d^{iz}\mathbf{1}_{S_n}=\mathfrak d^{z}v_n\phi^{1/2}\leq C$. In particular, we have $\sup_{n\geq 1}\|\Delta_\phi^{1/2}(v_n\phi^{1/2})\|\leq C$. Similarly, we have $\sup_{n\geq 1}\|\Delta_\phi^{1/2}(v_n^\ast\phi^{1/2})\|\leq C$. This implies that if $(x_n)\in\mathcal{I}_\omega(\M)$, then $\|x_nv_n\phi^{1/2}\|\to 0$ and $\|v_nx_n\phi^{1/2}\|\to0$ as $n\to\omega$ showing that $(x_nv_n)\in\mathcal{I}_\omega$. Similarly, $(v_nx_n)\in\mathcal{I}_\omega$; hence $(v_n)\in\mathcal{M}^\omega(\M)$. Set $v=(v_n)^\omega\in\M^\omega$ and $\psi=\phi^\omega$. Now,  for fixed $\alpha>0$, write $T_\alpha=\theta_\alpha^{(2)}$, so
\begin{align*}
 \norm{T_\alpha(v_n\phi^{1/2})}^2
 =\int_{S_n}\frac{\alpha^2}{(\alpha+\norm{c(\gamma)}^2)^2}d\nu_s(\gamma)\notag\leq\left(\frac\alpha{\alpha+n^2}\right)^2\mu(s(S_n))
\end{align*}
It follows that $T_\alpha(v\psi^{1/2})=\lim_{n\to\omega}\|T_\alpha(v_n\phi^{1/2})\|=0$.
On the other hand, $\|v\psi^{1/2}\|\geq\epsilon$, which shows that $T_\alpha^\omega(v\psi^{1/2})$ does not converges to $v\psi^{1/2}$ as $\alpha\to\infty$.
\end{proof}


\begin{remark}
    If the measured groupoid $(\G,\mu)$ is pmp, then the notion of strongly unbounded cocycle agree with that of \cite[Definition 2.24]{FH25}.
\end{remark}

Combing Lemma \ref{lem:strongly bounded cocycles}, Lemma \ref{lem:resolvent-companess-groupoid} with Theorem \ref{thm:primeness for type I base and compact resolvent} and Corollary \ref{cor:fullness and rigidity}, we obtain the following primeness and fullness result. 
\begin{theorem}\label{thm:fullness and primeness_groupoid}
         Let $(\G,\mu,\sigma)$  admit a strongly unbounded $1$-cocycle $c$ into a representation $\pi$ weakly contained in the left regular representation. Then $L_\sigma(\G)$ is not coarsely rigid relative to $L^\infty(X,\mu)$. Moreover, if $c$ is uniformly proper on its source fibers and $L_\sigma(\G,\mu)$ is a nonamenable factor then it is full and prime.
\end{theorem}

When $(\G,\mu)$ is a pmp discrete measured groupoid, the above theorem is known \cite[Theorem 5.2, Theorem 6.6]{FH25}. Therefore, the key novelty in our results lies in its extension to the nonpmp groupoid case.

Finally, by combining the relative compactness  of Lemma \ref{lem:resolvent-companess-groupoid} and solidity results from Section \ref{sec:solidity}, the following theorem  follows. 

\begin{theorem}\label{thm:solidity_measured_groupoid}
    Let $(\G,\mu,\sigma)$ admit a uniformly proper $1$-cocycle into a representation weakly contained in the left regular representation. Then
 $L_\sigma(\G,\mu)$ is $\omega$ solid and stably solid relative to $L^\infty(X,\mu)$ along an ultrafilter $\omega\in\beta\mathbb N\setminus\mathbb N$.
\end{theorem}

Relative solidity for groupoid von Neumann algebras have been studied in special cases such as crossed products and equivalence relations  \cite{Oz06, CI10, Dri22, Mar17, Iso25, DV25}. However, the above results are a novel application of our derivation-based theory to  the general groupoid scope, particularly nonsingular setting and twisting to their relative \(\omega\) or tensor-stability conclusions.

\begin{example}
    A treeing example satisfying the hypotheses of the above results can be found in \cite[Section 7]{AD13}.
\end{example}

\section*{Acknowledgment}
M.\,K.  was partially supported by the Research Foundation - Flanders (FWO)
through a Postdoctoral fellowship (1221025N).  M.\,W. was partially funded by the Deutsche Forschungsgemeinschaft
(DFG, German Research Foundation) – Project number 565405322. The authors thank Stefaan Vaes for useful comments on the intial draft of this paper. Part of this work was completed when M.K. was visiting  Leipzig University in September 2025. He gratefully acknowledges the university’s hospitality.

ChatGPT 6.0 Astra was used for proofreading a final draft version and to construct Example~\ref{ex:the notions of intertwining are not equivalent}, which was afterwards verified by the authors. All the other content of this article was created by the authors without LLM assistance.

\bibliography{ref}

@misc{Wir22,
      title={{The Differential Structure of Generators of GNS-symmetric Quantum Markov Semigroups}}, 
      author={Melchior Wirth},
      year={2022},
      howpublished={arXiv.2207.09247},
      eprint={2207.09247},
      archivePrefix={arXiv},
      primaryClass={math.OA},
      url={https://arxiv.org/abs/2207.09247}, 
}

@book {BH91,
    AUTHOR = {Bouleau, Nicolas and Hirsch, Francis},
     TITLE = {Dirichlet forms and analysis on {W}iener space},
    SERIES = {De Gruyter Studies in Mathematics},
    VOLUME = {14},
 PUBLISHER = {Walter de Gruyter \& Co., Berlin},
      YEAR = {1991},
     PAGES = {x+325},
      ISBN = {3-11-012919-1},
   MRCLASS = {60H07 (31C25 60G30 60J45)},
MRREVIEWER = {David\ Nualart},
       DOI = {10.1515/9783110858389},
       URL = {https://doi.org/10.1515/9783110858389},
}

@article {OP10,
    AUTHOR = {Ozawa, Narutaka and Popa, Sorin},
     TITLE = {On a class of {${\rm II}_1$} factors with at most one {C}artan
              subalgebra},
   JOURNAL = {Ann. of Math. (2)},
  FJOURNAL = {Annals of Mathematics. Second Series},
    VOLUME = {172},
      YEAR = {2010},
    NUMBER = {1},
     PAGES = {713--749},
      ISSN = {0003-486X,1939-8980},
   MRCLASS = {46L10 (37A20)},
MRREVIEWER = {Stuart\ A.\ White},
       DOI = {10.4007/annals.2010.172.713},
       URL = {https://doi.org/10.4007/annals.2010.172.713},
}

@article {Cip97,
    AUTHOR = {Cipriani, Fabio},
     TITLE = {Dirichlet forms and {M}arkovian semigroups on standard forms
              of von {N}eumann algebras},
   JOURNAL = {J. Funct. Anal.},
  FJOURNAL = {Journal of Functional Analysis},
    VOLUME = {147},
      YEAR = {1997},
    NUMBER = {2},
     PAGES = {259--300},
      ISSN = {0022-1236,1096-0783},
   MRCLASS = {46L50 (31C25 47D07)},
MRREVIEWER = {Charles\ Batty},
       DOI = {10.1006/jfan.1996.3063},
       URL = {https://doi.org/10.1006/jfan.1996.3063},
}

@book {TakII,
    AUTHOR = {Takesaki, M.},
     TITLE = {Theory of operator algebras. {II}},
    SERIES = {Encyclopaedia of Mathematical Sciences},
    VOLUME = {125},
 PUBLISHER = {Springer-Verlag, Berlin},
      YEAR = {2003},
     PAGES = {xxii+518},
      ISBN = {3-540-42914-X},
   MRCLASS = {46L10 (47L35 47L55)},
MRREVIEWER = {Robert\ S.\ Doran},
       DOI = {10.1007/978-3-662-10451-4},
       URL = {https://doi.org/10.1007/978-3-662-10451-4},
}

@book {Pis,
    AUTHOR = {Pisier, Gilles},
     TITLE = {Tensor products of {$C^*$}-algebras and operator spaces---the
              {C}onnes-{K}irchberg problem},
    VOLUME = {96},
 PUBLISHER = {Cambridge University Press, Cambridge},
      YEAR = {2020},
     PAGES = {x+484},
      ISBN = {978-1-108-74911-4; 978-1-108-47901-1},
   MRCLASS = {46-02 (46L06 46L07 47L25)},
MRREVIEWER = {Hiroyuki\ Osaka},
       DOI = {10.1017/9781108782081},
       URL = {https://doi.org/10.1017/9781108782081},
}

@article {KW25,
    AUTHOR = {Kumar, R. Rahul and Wirth, Melchior},
     TITLE = {Operator-valued twisted {A}raki-{W}oods algebras},
   JOURNAL = {Comm. Math. Phys.},
  FJOURNAL = {Communications in Mathematical Physics},
    VOLUME = {406},
      YEAR = {2025},
    NUMBER = {5},
     PAGES = {Paper No. 110, 47},
      ISSN = {0010-3616,1432-0916},
   MRCLASS = {46L10 (46L65 81R10)},
MRREVIEWER = {Manish\ Kumar},
       DOI = {10.1007/s00220-025-05285-7},
       URL = {https://doi.org/10.1007/s00220-025-05285-7},
}

@article {CKSVW23,
    AUTHOR = {Caspers, Martijn and Klisse, Mario and Skalski, Adam and Vos,
              Gerrit and Wasilewski, Mateusz},
     TITLE = {Relative {H}aagerup property for arbitrary von {N}eumann
              algebras},
   JOURNAL = {Adv. Math.},
  FJOURNAL = {Advances in Mathematics},
    VOLUME = {421},
      YEAR = {2023},
     PAGES = {Paper No. 109017, 61},
      ISSN = {0001-8708,1090-2082},
   MRCLASS = {46L10},
MRREVIEWER = {Martin\ Bohata},
       DOI = {10.1016/j.aim.2023.109017},
       URL = {https://doi.org/10.1016/j.aim.2023.109017},
}

@article{PP86,
 author = {Pimsner, Mihai and Popa, Sorin},
 title = {Entropy and index for subfactors},
 fjournal = {Annales Scientifiques de l'{\'E}cole Normale Sup{\'e}rieure. Quatri{\`e}me S{\'e}rie},
 journal = {Ann. Sci. {\'E}c. Norm. Sup{\'e}r. (4)},
 issn = {0012-9593},
 volume = {19},
 number = {1},
 pages = {57--106},
 year = {1986},
 language = {English},
 doi = {10.24033/asens.1504},
 url = {https://eudml.org/doc/82174},
 zbMATH = {4054342},
 Zbl = {0646.46057}
}

@article {BKS97,
    AUTHOR = {Bo\.zejko, Marek and K\"ummerer, Burkhard and Speicher,
              Roland},
     TITLE = {{$q$}-{G}aussian processes: non-commutative and classical
              aspects},
   JOURNAL = {Comm. Math. Phys.},
  FJOURNAL = {Communications in Mathematical Physics},
    VOLUME = {185},
      YEAR = {1997},
    NUMBER = {1},
     PAGES = {129--154},
      ISSN = {0010-3616,1432-0916},
   MRCLASS = {81S05 (46L50 81S25)},
MRREVIEWER = {Philippe\ Biane},
       DOI = {10.1007/s002200050084},
       URL = {https://doi.org/10.1007/s002200050084},
}

@article {Wor74,
    AUTHOR = {Woronowicz, S. L.},
     TITLE = {Selfpolar forms and their applications to the
              {$C\sp*$}-algebra theory},
   JOURNAL = {Rep. Mathematical Phys.},
  FJOURNAL = {Reports on Mathematical Physics},
    VOLUME = {6},
      YEAR = {1974},
    NUMBER = {3},
     PAGES = {487--495},
      ISSN = {0034-4877},
   MRCLASS = {46L05},
MRREVIEWER = {W.\ B.\ Arveson},
       DOI = {10.1016/s0034-4877(74)80012-1},
       URL = {https://doi.org/10.1016/s0034-4877(74)80012-1},
}

@article {Sau83,
    AUTHOR = {Sauvageot, Jean-Luc},
     TITLE = {Sur le produit tensoriel relatif d'espaces de {H}ilbert},
   JOURNAL = {J. Operator Theory},
  FJOURNAL = {Journal of Operator Theory},
    VOLUME = {9},
      YEAR = {1983},
    NUMBER = {2},
     PAGES = {237--252},
      ISSN = {0379-4024},
   MRCLASS = {46L10 (46M05 46M15)},
MRREVIEWER = {Marc\ A.\ Rieffel},
}

@article {AD95,
    AUTHOR = {Anantharaman-Delaroche, C.},
     TITLE = {Amenable correspondences and approximation properties for von
              {N}eumann algebras},
   JOURNAL = {Pacific J. Math.},
  FJOURNAL = {Pacific Journal of Mathematics},
    VOLUME = {171},
      YEAR = {1995},
    NUMBER = {2},
     PAGES = {309--341},
      ISSN = {0030-8730,1945-5844},
   MRCLASS = {46L10 (46L55)},
MRREVIEWER = {Paul\ Jolissaint},
       URL = {http://projecteuclid.org/euclid.pjm/1102368918},
}

@online{Ske01,
Author={Skeide, Michael},
Title={{Hilbert Modules and Applications in Quantum Probability}},
Institution={TU Cottbus},
Type={Habilitationsschrift},
Year={2001},
URL={http://web.unimol.it/skeide/_MS/downloads/habil.pdf},
}

@article {Con76,
    AUTHOR = {Connes, Alain},
     TITLE = {Classification of injective factors. {C}ases {$II\sb{1},$}
              {$II\sb{\infty },$} {$III\sb{\lambda },$} {$\lambda \not=1$}},
   JOURNAL = {Ann. of Math. (2)},
  FJOURNAL = {Annals of Mathematics. Second Series},
    VOLUME = {104},
      YEAR = {1976},
    NUMBER = {1},
     PAGES = {73--115},
      ISSN = {0003-486X},
   MRCLASS = {46L10},
MRREVIEWER = {Fran\c cois\ Combes},
       DOI = {10.2307/1971057},
       URL = {https://doi.org/10.2307/1971057},
}

@article {CS03,
    AUTHOR = {Cipriani, Fabio and Sauvageot, Jean-Luc},
     TITLE = {Derivations as square roots of {D}irichlet forms},
   JOURNAL = {J. Funct. Anal.},
  FJOURNAL = {Journal of Functional Analysis},
    VOLUME = {201},
      YEAR = {2003},
    NUMBER = {1},
     PAGES = {78--120},
      ISSN = {0022-1236,1096-0783},
   MRCLASS = {46L53 (47D07 60J25)},
MRREVIEWER = {Michael\ Skeide},
       DOI = {10.1016/S0022-1236(03)00085-5},
       URL = {https://doi.org/10.1016/S0022-1236(03)00085-5},
}

@article {Wir24,
    AUTHOR = {Wirth, Melchior},
     TITLE = {Modular completely {D}irichlet forms as squares of
              derivations},
   JOURNAL = {Int. Math. Res. Not. IMRN},
  FJOURNAL = {International Mathematics Research Notices. IMRN},
      YEAR = {2024},
    NUMBER = {14},
     PAGES = {10597--10614},
      ISSN = {1073-7928,1687-0247},
   MRCLASS = {46L10 (46L57)},
MRREVIEWER = {Ioannis\ Zarakas},
       DOI = {10.1093/imrn/rnae092},
       URL = {https://doi.org/10.1093/imrn/rnae092},
}

@article {Shl99,
    AUTHOR = {Shlyakhtenko, Dimitri},
     TITLE = {{$A$}-valued semicircular systems},
   JOURNAL = {J. Funct. Anal.},
  FJOURNAL = {Journal of Functional Analysis},
    VOLUME = {166},
      YEAR = {1999},
    NUMBER = {1},
     PAGES = {1--47},
      ISSN = {0022-1236,1096-0783},
   MRCLASS = {46L54 (46L10 46L35)},
MRREVIEWER = {Ken\ Dykema},
       DOI = {10.1006/jfan.1999.3424},
       URL = {https://doi.org/10.1006/jfan.1999.3424},
}

@article {Wat79,
    AUTHOR = {Watanabe, Seiji},
     TITLE = {Ergodic theorems for dynamical semigroups on operator
              algebras},
   JOURNAL = {Hokkaido Math. J.},
  FJOURNAL = {Hokkaido Mathematical Journal},
    VOLUME = {8},
      YEAR = {1979},
    NUMBER = {2},
     PAGES = {176--190},
      ISSN = {0385-4035},
   MRCLASS = {46L55},
MRREVIEWER = {Gr.\ Arsene},
       DOI = {10.14492/hokmj/1381758269},
       URL = {https://doi.org/10.14492/hokmj/1381758269},
}

@article {Iso22,
    AUTHOR = {Isono, Yusuke},
     TITLE = {Unitary conjugacy for type {${\rm III}$} subfactors and {$\rm
              W^*$}-superrigidity},
   JOURNAL = {J. Eur. Math. Soc. (JEMS)},
  FJOURNAL = {Journal of the European Mathematical Society (JEMS)},
    VOLUME = {24},
      YEAR = {2022},
    NUMBER = {5},
     PAGES = {1679--1721},
      ISSN = {1435-9855,1435-9863},
   MRCLASS = {46L36 (37A55 46L10 46L55)},
MRREVIEWER = {Mateusz\ Wasilewski},
       DOI = {10.4171/jems/1135},
       URL = {https://doi.org/10.4171/jems/1135},
}

@article {HR15,
    AUTHOR = {Houdayer, Cyril and Raum, Sven},
     TITLE = {Asymptotic structure of free {A}raki-{W}oods factors},
   JOURNAL = {Math. Ann.},
  FJOURNAL = {Mathematische Annalen},
    VOLUME = {363},
      YEAR = {2015},
    NUMBER = {1-2},
     PAGES = {237--267},
      ISSN = {0025-5831,1432-1807},
   MRCLASS = {46L10 (46L54)},
MRREVIEWER = {Liguang\ Wang},
       DOI = {10.1007/s00208-015-1168-1},
       URL = {https://doi.org/10.1007/s00208-015-1168-1},
}

@article {AH14,
    AUTHOR = {Ando, Hiroshi and Haagerup, Uffe},
     TITLE = {Ultraproducts of von {N}eumann algebras},
   JOURNAL = {J. Funct. Anal.},
  FJOURNAL = {Journal of Functional Analysis},
    VOLUME = {266},
      YEAR = {2014},
    NUMBER = {12},
     PAGES = {6842--6913},
      ISSN = {0022-1236,1096-0783},
   MRCLASS = {46M07 (46L10)},
MRREVIEWER = {E.\ St\o rmer},
       DOI = {10.1016/j.jfa.2014.03.013},
       URL = {https://doi.org/10.1016/j.jfa.2014.03.013},
}

@article {BMO20,
    AUTHOR = {Bannon, Jon and Marrakchi, Amine and Ozawa, Narutaka},
     TITLE = {Full factors and co-amenable inclusions},
   JOURNAL = {Comm. Math. Phys.},
  FJOURNAL = {Communications in Mathematical Physics},
    VOLUME = {378},
      YEAR = {2020},
    NUMBER = {2},
     PAGES = {1107--1121},
      ISSN = {0010-3616,1432-0916},
   MRCLASS = {46L10 (22F10 46L36 46L40 46L55)},
MRREVIEWER = {Robert\ S.\ Doran},
       DOI = {10.1007/s00220-020-03816-y},
       URL = {https://doi.org/10.1007/s00220-020-03816-y},
}

@misc {DV25,
      title={{${\rm W}^*$}-correlations of {${\rm II}_1$} factors and rigidity of tensor products and graph products}, 
      author={Drimbe, Daniel and Vaes, Stefaan},
      year={2025},
      howpublished={arXiv:2507.04691},
      eprint={2507.04691},
      archivePrefix={arXiv},
      primaryClass={math.OA},
      url={https://arxiv.org/abs/2507.04691}, 
}

@misc {Iso25,
      title={Weak relative Dixmier property and Popa's intertwining technique for type {${\rm III}$} subfactors}, 
      author={Isono, Yusuke},
      year={2025},
      howpublished={arXiv:2508.17592},
      eprint={2508.17592},
      archivePrefix={arXiv},
      primaryClass={math.OA},
      url={https://arxiv.org/abs/2508.17592}, 
}

@article {Mar25,
    AUTHOR = {Marrakchi, Amine},
     TITLE = {Kadison's problem for type {III} subfactors and the bicentralizer conjecture},
   JOURNAL = {Invent. Math.},
  FJOURNAL = {Inventiones Mathematicae},
    VOLUME = {239},
      YEAR = {2025},
    NUMBER = {1},
     PAGES = {79--163},
      ISSN = {0020-9910,1432-1297},
   MRCLASS = {46L10 (46L36 46L37 46M07)},
MRREVIEWER = {Sofya\ S.\ Masharipova},
       DOI = {10.1007/s00222-024-01299-5},
       URL = {https://doi.org/10.1007/s00222-024-01299-5},
}

@article {BD01,
    AUTHOR = {Blanchard, Etienne F. and Dykema, Kenneth J.},
     TITLE = {Embeddings of reduced free products of operator algebras},
   JOURNAL = {Pacific J. Math.},
  FJOURNAL = {Pacific Journal of Mathematics},
    VOLUME = {199},
      YEAR = {2001},
    NUMBER = {1},
     PAGES = {1--19},
      ISSN = {0030-8730,1945-5844},
   MRCLASS = {46L09 (46L05)},
MRREVIEWER = {Emmanuel\ C.\ Germain},
       DOI = {10.2140/pjm.2001.199.1},
       URL = {https://doi.org/10.2140/pjm.2001.199.1},
}

@article {Pet09a,
    AUTHOR = {Peterson, Jesse},
     TITLE = {{$L^2$}-rigidity in von {N}eumann algebras},
   JOURNAL = {Invent. Math.},
  FJOURNAL = {Inventiones Mathematicae},
    VOLUME = {175},
      YEAR = {2009},
    NUMBER = {2},
     PAGES = {417--433},
      ISSN = {0020-9910,1432-1297},
   MRCLASS = {46L10 (46L54 46L55 46L57)},
MRREVIEWER = {Narutaka\ Ozawa},
       DOI = {10.1007/s00222-008-0154-6},
       URL = {https://doi.org/10.1007/s00222-008-0154-6},
}

@article {KFGV77,
    AUTHOR = {Kossakowski, Andrzej and Frigerio, Alberto and Gorini,
              Vittorio and Verri, Maurizio},
     TITLE = {Quantum detailed balance and {KMS} condition},
   JOURNAL = {Comm. Math. Phys.},
  FJOURNAL = {Communications in Mathematical Physics},
    VOLUME = {57},
      YEAR = {1977},
    NUMBER = {2},
     PAGES = {97--110},
      ISSN = {0010-3616,1432-0916},
   MRCLASS = {82.47 (46L99)},
MRREVIEWER = {Panayotis\ Tsilimigras},
       URL = {http://projecteuclid.org/euclid.cmp/1103901281},
}

@article {Ioa15,
    AUTHOR = {Ioana, Adrian},
     TITLE = {Cartan subalgebras of amalgamated free product {${\rm II}_1$}
              factors},
      NOTE = {With an appendix by Ioana and Stefaan Vaes},
   JOURNAL = {Ann. Sci. \'Ec. Norm. Sup\'er. (4)},
  FJOURNAL = {Annales Scientifiques de l'\'Ecole Normale Sup\'erieure.
              Quatri\`eme S\'erie},
    VOLUME = {48},
      YEAR = {2015},
    NUMBER = {1},
     PAGES = {71--130},
      ISSN = {0012-9593,1873-2151},
   MRCLASS = {46L36 (28D15 37A20 46L10 60B15)},
MRREVIEWER = {Sven\ Raum},
       DOI = {10.24033/asens.2239},
       URL = {https://doi.org/10.24033/asens.2239},
}

@article {CJ85,
    AUTHOR = {Connes, A. and Jones, V.},
     TITLE = {Property {$T$} for von {N}eumann algebras},
   JOURNAL = {Bull. London Math. Soc.},
  FJOURNAL = {The Bulletin of the London Mathematical Society},
    VOLUME = {17},
      YEAR = {1985},
    NUMBER = {1},
     PAGES = {57--62},
      ISSN = {0024-6093,1469-2120},
   MRCLASS = {46L35 (22D25)},
MRREVIEWER = {Colin\ E.\ Sutherland},
       DOI = {10.1112/blms/17.1.57},
       URL = {https://doi.org/10.1112/blms/17.1.57},
}

@incollection {Sau90,
    AUTHOR = {Sauvageot, Jean-Luc},
     TITLE = {Quantum {D}irichlet forms, differential calculus and
              semigroups},
 BOOKTITLE = {Quantum probability and applications, {V} ({H}eidelberg,
              1988)},
    SERIES = {Lecture Notes in Math.},
    VOLUME = {1442},
     PAGES = {334--346},
 PUBLISHER = {Springer, Berlin},
      YEAR = {1990},
      ISBN = {3-540-53026-6},
   MRCLASS = {46Lxx},
       DOI = {10.1007/BFb0085527},
       URL = {https://doi.org/10.1007/BFb0085527},
}

@article {dHV89,
    AUTHOR = {de la Harpe, Pierre and Valette, Alain},
     TITLE = {La propri\'et\'e{} {$(T)$} de {K}azhdan pour les groupes
              localement compacts (avec un appendice de {M}arc {B}urger)},
      NOTE = {With an appendix by M. Burger},
   JOURNAL = {Ast\'erisque},
  FJOURNAL = {Ast\'erisque},
    NUMBER = {175},
      YEAR = {1989},
     PAGES = {158},
      ISSN = {0303-1179,2492-5926},
   MRCLASS = {22-02 (05C99 22D10 22E99 28C10)},
MRREVIEWER = {Michael\ Cowling},
}

@article {PP05,
    AUTHOR = {Peterson, Jesse and Popa, Sorin},
     TITLE = {On the notion of relative property ({T}) for inclusions of von
              {N}eumann algebras},
   JOURNAL = {J. Funct. Anal.},
  FJOURNAL = {Journal of Functional Analysis},
    VOLUME = {219},
      YEAR = {2005},
    NUMBER = {2},
     PAGES = {469--483},
      ISSN = {0022-1236,1096-0783},
   MRCLASS = {46L10 (22E40)},
MRREVIEWER = {Paul\ Jolissaint},
       DOI = {10.1016/j.jfa.2004.04.017},
       URL = {https://doi.org/10.1016/j.jfa.2004.04.017},
}

@article {PV14,
    AUTHOR = {Popa, Sorin and Vaes, Stefaan},
     TITLE = {Unique {C}artan decomposition for {$\rm II_1$} factors arising
              from arbitrary actions of free groups},
   JOURNAL = {Acta Math.},
  FJOURNAL = {Acta Mathematica},
    VOLUME = {212},
      YEAR = {2014},
    NUMBER = {1},
     PAGES = {141--198},
      ISSN = {0001-5962,1871-2509},
   MRCLASS = {46L36 (37A05 46L10 46L37)},
MRREVIEWER = {Liguang\ Wang},
       DOI = {10.1007/s11511-014-0110-9},
       URL = {https://doi.org/10.1007/s11511-014-0110-9},
}

@article {Pop06,
    AUTHOR = {Popa, Sorin},
     TITLE = {On a class of type {${\rm II}_1$} factors with {B}etti numbers
              invariants},
   JOURNAL = {Ann. of Math. (2)},
  FJOURNAL = {Annals of Mathematics. Second Series},
    VOLUME = {163},
      YEAR = {2006},
    NUMBER = {3},
     PAGES = {809--899},
      ISSN = {0003-486X,1939-8980},
   MRCLASS = {46L10 (37A20)},
MRREVIEWER = {Paul\ Jolissaint},
       DOI = {10.4007/annals.2006.163.809},
       URL = {https://doi.org/10.4007/annals.2006.163.809},
}

@misc{Pop86,
      title={Correspondences}, 
      author={Popa, Sorin},
      year={1986},
      howpublished={unpublished},
      primaryClass={math.OA}, 
}

@article {Haa87,
    AUTHOR = {Haagerup, Uffe},
     TITLE = {Connes' bicentralizer problem and uniqueness of the injective
              factor of type {${\rm III}_1$}},
   JOURNAL = {Acta Math.},
  FJOURNAL = {Acta Mathematica},
    VOLUME = {158},
      YEAR = {1987},
    NUMBER = {1-2},
     PAGES = {95--148},
      ISSN = {0001-5962,1871-2509},
   MRCLASS = {46L35},
MRREVIEWER = {Steve\ Wright},
       DOI = {10.1007/BF02392257},
       URL = {https://doi.org/10.1007/BF02392257},
}

@book {Str20,
    AUTHOR = {Str\u{a}til\u{a}, \c{S}erban Valentin},
     TITLE = {Modular theory in operator algebras},
    SERIES = {Cambridge-IISc Series},
   EDITION = {Second},
 PUBLISHER = {Cambridge University Press, Delhi},
      YEAR = {2020},
     PAGES = {xii+447},
      ISBN = {978-1-108-48960-7},
   MRCLASS = {46-02 (46L10 46L37 46L40 46L45)},
       DOI = {10.1017/9781108489607},
       URL = {https://doi.org/10.1017/9781108489607},
}

@article {HI17,
    AUTHOR = {Houdayer, Cyril and Isono, Yusuke},
     TITLE = {Unique prime factorization and bicentralizer problem for a
              class of type {III} factors},
   JOURNAL = {Adv. Math.},
  FJOURNAL = {Advances in Mathematics},
    VOLUME = {305},
      YEAR = {2017},
     PAGES = {402--455},
      ISSN = {0001-8708,1090-2082},
   MRCLASS = {46L36 (46L10 46L30 46L54)},
MRREVIEWER = {Alain\ Valette},
       DOI = {10.1016/j.aim.2016.09.030},
       URL = {https://doi.org/10.1016/j.aim.2016.09.030},
}

@book {Hia21,
    AUTHOR = {Hiai, Fumio},
     TITLE = {Lectures on selected topics in von {N}eumann algebras},
    SERIES = {EMS Series of Lectures in Mathematics},
 PUBLISHER = {EMS Press, Berlin},
      YEAR = {[2021] \copyright 2021},
     PAGES = {viii+241},
      ISBN = {978-3-98547-004-4},
   MRCLASS = {46-02 (46L10 46L51)},
MRREVIEWER = {Eusebio\ Gardella},
       DOI = {10.4171/ELM/32},
       URL = {https://doi.org/10.4171/ELM/32},
}

@article {DonVae25,
    AUTHOR = {Donvil, Milan and Vaes, Stefaan},
     TITLE = {{${\rm W}^*$}-superrigidity for cocycle twisted group von
              {N}eumann algebras},
   JOURNAL = {Invent. Math.},
  FJOURNAL = {Inventiones Mathematicae},
    VOLUME = {240},
      YEAR = {2025},
    NUMBER = {1},
     PAGES = {193--260},
      ISSN = {0020-9910,1432-1297},
   MRCLASS = {46L10 (20E22 20E34 20F65)},
MRREVIEWER = {Paul\ Jolissaint},
       DOI = {10.1007/s00222-025-01320-5},
       URL = {https://doi.org/10.1007/s00222-025-01320-5},
}

@article {Oza04,
    AUTHOR = {Ozawa, Narutaka},
     TITLE = {Solid von {N}eumann algebras},
   JOURNAL = {Acta Math.},
  FJOURNAL = {Acta Mathematica},
    VOLUME = {192},
      YEAR = {2004},
    NUMBER = {1},
     PAGES = {111--117},
      ISSN = {0001-5962,1871-2509},
   MRCLASS = {46L10},
MRREVIEWER = {Dorin\ Ervin\ Dutkay},
       DOI = {10.1007/BF02441087},
       URL = {https://doi.org/10.1007/BF02441087},
}

@article {Oz06,
    AUTHOR = {Ozawa, Narutaka},
     TITLE = {A {K}urosh-type theorem for type {$\rm II_1$} factors},
   JOURNAL = {Int. Math. Res. Not.},
  FJOURNAL = {International Mathematics Research Notices},
      YEAR = {2006},
     PAGES = {Art. ID 97560, 21},
      ISSN = {1073-7928,1687-0247},
   MRCLASS = {46L10 (20F67 46L09 46L35 46L55)},
MRREVIEWER = {Claire\ Anantharaman-Delaroche},
       DOI = {10.1155/IMRN/2006/97560},
       URL = {https://doi.org/10.1155/IMRN/2006/97560},
}

@article {Mar17,
    AUTHOR = {Marrakchi, Amine},
     TITLE = {Solidity of type {III} {B}ernoulli crossed products},
   JOURNAL = {Comm. Math. Phys.},
  FJOURNAL = {Communications in Mathematical Physics},
    VOLUME = {350},
      YEAR = {2017},
    NUMBER = {3},
     PAGES = {897--916},
      ISSN = {0010-3616,1432-0916},
   MRCLASS = {37A20 (37A15 46L80)},
MRREVIEWER = {Lewis\ Bowen},
       DOI = {10.1007/s00220-016-2717-5},
       URL = {https://doi.org/10.1007/s00220-016-2717-5},
}

@article {Pet09b,
    AUTHOR = {Peterson, Jesse},
     TITLE = {A 1-cohomology characterization of property ({T}) in von
              {N}eumann algebras},
   JOURNAL = {Pacific J. Math.},
  FJOURNAL = {Pacific Journal of Mathematics},
    VOLUME = {243},
      YEAR = {2009},
    NUMBER = {1},
     PAGES = {181--199},
      ISSN = {0030-8730,1945-5844},
   MRCLASS = {22D25 (46L10 46L54)},
MRREVIEWER = {Robert\ S.\ Doran},
       DOI = {10.2140/pjm.2009.243.181},
       URL = {https://doi.org/10.2140/pjm.2009.243.181},
}

@article {HV13,
    AUTHOR = {Houdayer, Cyril and Vaes, Stefaan},
     TITLE = {Type {III} factors with unique {C}artan decomposition},
   JOURNAL = {J. Math. Pures Appl. (9)},
  FJOURNAL = {Journal de Math\'ematiques Pures et Appliqu\'ees. Neuvi\`eme
              S\'erie},
    VOLUME = {100},
      YEAR = {2013},
    NUMBER = {4},
     PAGES = {564--590},
      ISSN = {0021-7824,1776-3371},
   MRCLASS = {46L10 (37A40 46L36)},
MRREVIEWER = {Bachir\ Bekka},
       DOI = {10.1016/j.matpur.2013.01.013},
       URL = {https://doi.org/10.1016/j.matpur.2013.01.013},
}

@article {KV19,
    AUTHOR = {Krogager, Anna Sofie and Vaes, Stefaan},
     TITLE = {Thin {${\rm II}_1$} factors with no {C}artan subalgebras},
   JOURNAL = {Kyoto J. Math.},
  FJOURNAL = {Kyoto Journal of Mathematics},
    VOLUME = {59},
      YEAR = {2019},
    NUMBER = {4},
     PAGES = {815--867},
      ISSN = {2156-2261,2154-3321},
   MRCLASS = {46L36 (43A46)},
MRREVIEWER = {Liguang\ Wang},
       DOI = {10.1215/21562261-2019-0028},
       URL = {https://doi.org/10.1215/21562261-2019-0028},
}

@article {HT21,
    AUTHOR = {Houdayer, Cyril and Trom, Benjamin},
     TITLE = {Structure of extensions of free {A}raki-{W}oods factors},
   JOURNAL = {Ann. Inst. Fourier (Grenoble)},
  FJOURNAL = {Universit\'e{} de Grenoble. Annales de l'Institut Fourier},
    VOLUME = {71},
      YEAR = {2021},
    NUMBER = {2},
     PAGES = {473--514},
      ISSN = {0373-0956,1777-5310},
   MRCLASS = {46L36 (46L10 46L40 46L54 46L55)},
MRREVIEWER = {Liguang\ Wang},
       DOI = {10.5802/aif.3391},
       URL = {https://doi.org/10.5802/aif.3391},
}

@article {Sut80,
    AUTHOR = {Sutherland, Colin E.},
     TITLE = {Cohomology and extensions of von {N}eumann algebras. {I},
              {II}},
   JOURNAL = {Publ. Res. Inst. Math. Sci.},
  FJOURNAL = {Kyoto University. Research Institute for Mathematical
              Sciences. Publications},
    VOLUME = {16},
      YEAR = {1980},
    NUMBER = {1},
     PAGES = {105--133, 135--174},
      ISSN = {0034-5318,1663-4926},
   MRCLASS = {46L10 (46L40 46M20)},
MRREVIEWER = {B.\ E.\ Johnson},
       DOI = {10.2977/prims/1195187501},
       URL = {https://doi.org/10.2977/prims/1195187501},
}

@incollection {AD13,
    AUTHOR = {Anantharaman-Delaroche, Claire},
     TITLE = {The {H}aagerup property for discrete measured groupoids},
 BOOKTITLE = {Operator algebra and dynamics},
    SERIES = {Springer Proc. Math. Stat.},
    VOLUME = {58},
     PAGES = {1--30},
 PUBLISHER = {Springer, Heidelberg},
      YEAR = {2013},
      ISBN = {978-3-642-39459-1; 978-3-642-39458-4},
   MRCLASS = {22A22 (46L55)},
MRREVIEWER = {Alain\ Valette},
       DOI = {10.1007/978-3-642-39459-1\_1},
       URL = {https://doi.org/10.1007/978-3-642-39459-1_1},
}

@article {BCDK26,
    AUTHOR = {Berendschot, Tey and Chakraborty, Soham and Donvil, Milan and
              Kim, Se-Jin},
     TITLE = {Factoriality of groupoid von {N}eumann algebras},
   JOURNAL = {J. Operator Theory},
  FJOURNAL = {Journal of Operator Theory},
    VOLUME = {95},
      YEAR = {2026},
    NUMBER = {1},
     PAGES = {21--49},
      ISSN = {0379-4024,1841-7744},
   MRCLASS = {46L10},
       DOI = {10.7900/jot.2024feb29.2472},
       URL = {https://doi.org/10.7900/jot.2024feb29.2472},
}

@book {ADR00,
    AUTHOR = {Anantharaman-Delaroche, C. and Renault, J.},
     TITLE = {Amenable groupoids},
    SERIES = {Monographies de L'Enseignement Math\'ematique [Monographs of
              L'Enseignement Math\'ematique]},
    VOLUME = {36},
      NOTE = {With a foreword by Georges Skandalis and Appendix B by E.
              Germain},
 PUBLISHER = {L'Enseignement Math\'ematique, Geneva},
      YEAR = {2000},
     PAGES = {196},
      ISBN = {2-940264-01-5},
   MRCLASS = {22A22 (22D25 43A07 46L05 46L10 46L80)},
MRREVIEWER = {Robert\ S.\ Doran},
}

@misc{FH25,
      title={Discrete measured groupoid von Neumann algebras via the Gaussian deformation}, 
      author={Flores, Felipe and Harbour, James},
      year={2025},
      howpublished={ 	arXiv:2509.15161},
      eprint={2509.15161},
      archivePrefix={arXiv},
      primaryClass={math.OA},
      url={ 	
https://doi.org/10.48550/arXiv.2509.15161}, 
}

@article {CI10,
    AUTHOR = {Chifan, Ionut and Ioana, Adrian},
     TITLE = {Ergodic subequivalence relations induced by a {B}ernoulli
              action},
   JOURNAL = {Geom. Funct. Anal.},
  FJOURNAL = {Geometric and Functional Analysis},
    VOLUME = {20},
      YEAR = {2010},
    NUMBER = {1},
     PAGES = {53--67},
      ISSN = {1016-443X,1420-8970},
   MRCLASS = {37A20 (37A15 46L10)},
MRREVIEWER = {Alain\ Valette},
       DOI = {10.1007/s00039-010-0058-7},
       URL = {https://doi.org/10.1007/s00039-010-0058-7},
}

@article {Dri22,
    AUTHOR = {Drimbe, Daniel},
     TITLE = {Solid ergodicity and orbit equivalence rigidity for coinduced
              actions},
   JOURNAL = {Int. Math. Res. Not. IMRN},
  FJOURNAL = {International Mathematics Research Notices. IMRN},
      YEAR = {2022},
    NUMBER = {11},
     PAGES = {8251--8279},
      ISSN = {1073-7928,1687-0247},
   MRCLASS = {37A25 (37A20)},
MRREVIEWER = {Wouter\ van Limbeek},
       DOI = {10.1093/imrn/rnaa325},
       URL = {https://doi.org/10.1093/imrn/rnaa325},
}

@article {BV97,
    AUTHOR = {Bekka, Mohammed E. B. and Valette, Alain},
     TITLE = {Group cohomology, harmonic functions and the first
              {$L^2$}-{B}etti number},
   JOURNAL = {Potential Anal.},
  FJOURNAL = {Potential Analysis. An International Journal Devoted to the
              Interactions between Potential Theory, Probability Theory,
              Geometry and Functional Analysis},
    VOLUME = {6},
      YEAR = {1997},
    NUMBER = {4},
     PAGES = {313--326},
      ISSN = {0926-2601,1572-929X},
   MRCLASS = {20J05 (31C05 58G12)},
MRREVIEWER = {William\ Paschke},
       DOI = {10.1023/A:1017974406074},
       URL = {https://doi.org/10.1023/A:1017974406074},
}

@article {MOSS15,
    AUTHOR = {Mimura, Masato and Ozawa, Narutaka and Sako, Hiroki and
              Suzuki, Yuhei},
     TITLE = {Group approximation in {C}ayley topology and coarse geometry,
              {III}: {G}eometric property ({T})},
   JOURNAL = {Algebr. Geom. Topol.},
  FJOURNAL = {Algebraic \& Geometric Topology},
    VOLUME = {15},
      YEAR = {2015},
    NUMBER = {2},
     PAGES = {1067--1091},
      ISSN = {1472-2747,1472-2739},
   MRCLASS = {20F65 (46M20)},
MRREVIEWER = {J\"org\ Lehnert},
       DOI = {10.2140/agt.2015.15.1067},
       URL = {https://doi.org/10.2140/agt.2015.15.1067},
}
\bibliographystyle{amsplain}

\end{document}